\date{} 
\title{Imaginary geometry IV:\\ interior rays, whole-plane reversibility, and space-filling trees}
\author{Jason Miller and Scott Sheffield}
\documentclass[12pt,naturalnames]{article}
\usepackage{amsmath}
\usepackage{amssymb}
\usepackage{morefloats}
\usepackage{amsthm}
\usepackage{amsfonts}
\usepackage{hyperref}
\usepackage{graphicx}
\usepackage{tabularx}
\usepackage{caption}
\usepackage{enumerate}
\usepackage{microtype}
\usepackage[normalem]{ulem}
\usepackage{comment}
\usepackage[usenames,dvipsnames]{xcolor}

\def\@rst #1 #2other{#1}
\newcommand\MR[1]{\relax\ifhmode\unskip\spacefactor3000 \space\fi
  \MRhref{\expandafter\@rst #1 other}{#1}}
\newcommand{\MRhref}[2]{\href{http://www.ams.org/mathscinet-getitem?mr=#1}{MR#2}}

\usepackage[margin=1.2in]{geometry}

\newcommand{\giv}{\,|\,}

\allowdisplaybreaks

\newif\ifdraft
\drafttrue
\def\note#1/{\ifdraft {\bf [#1]}\fi}
\long\def\comment#1{}
\numberwithin{equation}{section}
\numberwithin{figure}{section}

\newtheorem{theorem}{Theorem}
\numberwithin{theorem}{section}

\newtheorem{lemma}[theorem]{Lemma}
\newtheorem{proposition}[theorem]{Proposition}

\theoremstyle{remark}
\theoremstyle{remark}\newtheorem{remark}[theorem]{Remark}

\usepackage{subfigure}
\usepackage{graphicx}

\newcommand{\C}{\mathbf{C}}
\newcommand{\D}{\mathbf{D}}
\newcommand{\E}{\mathbf{E}}

\newcommand{\N}{\mathbf{N}}
\newcommand{\Z}{\mathbf{Z}}
\newcommand{\p}{\mathbf{P}}
\newcommand{\Q}{\mathbf{Q}}
\newcommand{\R}{\mathbf{R}}
\newcommand{\s}{\mathbf{S}}
\newcommand{\h}{\mathbf{H}}

\newcommand{\Fh}{\mathfrak {h}}

\newcommand{\CA}{\mathcal {A}}

\newcommand{\CC}{\mathcal {C}}
\newcommand{\CD}{\mathcal {D}}

\newcommand{\CF}{\mathcal {F}}
\newcommand{\CI}{\mathcal {I}}
\newcommand{\CJ}{\mathcal {J}}

\newcommand{\CP}{\mathcal {P}}

\newcommand{\CR}{\mathcal {R}}
\newcommand{\CS}{\mathcal {S}}

\newcommand{\CV}{\mathcal {V}}
\newcommand{\CZ}{\mathcal {Z}}

\newcommand{\CG}{\mathcal {G}}
\newcommand{\CH}{\mathcal {H}}

\newcommand{\dist}{{\rm dist}}

\newcommand{\var}{{\rm Var}}
\newcommand{\im}{{\rm Im}}
\newcommand{\re}{{\rm Re}}
\newcommand{\cov}{{\rm Cov}}

\newcommand{\SLE}{{\rm SLE}}
\newcommand{\CLE}{{\rm CLE}}

\newcommand{\strip}{\CS}
\newcommand{\striptop}{\partial_U \CS}
\newcommand{\stripbot}{\partial_L \CS}

\newcommand{\vstrip}{\CV}

\newcommand{\wh}{\widehat}
\newcommand{\wt}{\widetilde}
\newcommand{\ol}{\overline}
\newcommand{\ul}{\underline}
\newcommand{\one}{{\mathbf 1}}

\def\dist{\mathop{\mathrm{dist}}}

\def\Ito/{It\^o}

\def \E {{\bf E}}

\includecomment{ficomment} 

\begin{document}
\maketitle

\begin{abstract}
We establish existence and uniqueness for Gaussian free field flow lines started at {\em interior} points of a planar domain.  We interpret these as rays of a random geometry with imaginary curvature and describe the way distinct rays intersect each other and the boundary.

Previous works in this series treat rays started at {\em boundary} points and use Gaussian free field machinery to determine which chordal $\SLE_\kappa(\rho_1; \rho_2)$ processes are time-reversible when $\kappa < 8$.  Here we extend these results to whole-plane $\SLE_\kappa(\rho)$ and establish continuity and transience of these paths.  In particular, we extend ordinary whole-plane SLE reversibility (established by Zhan for $\kappa \in [0,4]$) to all $\kappa \in [0,8]$.

We also show that the rays of a given angle (with variable starting point) form a space-filling planar tree.  Each branch is a form of $\SLE_\kappa$ for some $\kappa \in (0, 4)$, and the curve that traces the tree in the natural order (hitting $x$ before $y$ if the branch from $x$ is left of the branch from $y$) is a space-filling form of $\SLE_{\kappa'}$ where $\kappa':= 16/\kappa \in (4, \infty)$.  By varying the boundary data we obtain, for each $\kappa'>4$, a family of space-filling variants of $\SLE_{\kappa'}(\rho)$ whose time reversals belong to the same family.  When $\kappa' \geq 8$, ordinary $\SLE_{\kappa'}$ belongs to this family, and our result shows that its time-reversal is $\SLE_{\kappa'}(\kappa'/2 - 4; \kappa'/2 - 4)$.

As applications of this theory, we obtain the local finiteness of $\CLE_{\kappa'}$, for $\kappa' \in (4,8)$, and describe the laws of the boundaries of $\SLE_{\kappa'}$ processes stopped at stopping times.
\end{abstract}

\newpage
\setlength{\parskip}{0.01cm plus1mm minus1mm}
\tableofcontents
\newpage

\medbreak {\noindent\bf Acknowledgments.}  We thank Bertrand Duplantier, Oded Schramm, Wendelin Werner, David Wilson, and Dapeng Zhan for helpful discussions.  We also thank several anonymous referees for many helpful comments on an earlier version of this article.

\parindent 0 pt
\setlength{\parskip}{0.25cm plus1mm minus1mm}



\section{Introduction}
\subsection{Overview}
\label{ss::overview}

This is the fourth in a series of papers that also includes \cite{MS_IMAG, MS_IMAG2, MS_IMAG3}.  Given a real-valued function $h$ defined on a subdomain $D$ of the complex plane $\C$, constants $\chi, \theta \in \R$ with $\chi \neq 0$, and an initial point $z \in \ol{D}$, one may construct a {\em flow line} of the complex vector field $e^{i(h/\chi+\theta)}$, i.e., a solution to the ODE
\begin{equation} \label{eqn::ode}
\frac{d}{dt} \eta(t) = e^{i \left(h(\eta(t))/\chi+\theta\right)}\quad\text{for}\quad t > 0, \,\,\,\eta(0) = z.
\end{equation}
In \cite{MS_IMAG, MS_IMAG2, MS_IMAG3} (following earlier works such as \cite{DUB_PART, She_SLE_lectures, SchrammShe10, SHE_WELD}) we fixed $\chi$ and interpreted these flow lines as the rays of a so-called {\em imaginary geometry}, where $z$ is the starting point of the ray and $\theta$ is the {\em angle}.  The ODE~\eqref{eqn::ode} has a unique solution when $h$ is smooth and $z \in D$.  However  \cite{MS_IMAG, MS_IMAG2, MS_IMAG3} deal with the case that $h$ is an instance of the {\em Gaussian free field} (GFF) on $D$, in which case $h$ is a random {\em generalized function} (or distribution) and~\eqref{eqn::ode} cannot be solved in the usual sense.  These works assume that the initial point $z$ lies on the boundary of $D$ and use tools from $\SLE$ theory to show that, in some generality, the solutions to~\eqref{eqn::ode} can be defined in a canonical way and exist almost surely.  By considering different initial points and different values for $\theta$ (which corresponds to the ``angle'' of the geodesic ray) one obtains an entire family of geodesic rays that interact with each other in interesting but comprehensible ways.

In this paper, we extend the constructions of \cite{MS_IMAG, MS_IMAG2, MS_IMAG3} to rays that start at points in the interior of $D$.  This provides a much more complete picture of the imaginary geometry.  Figure~\ref{fig::spokes} illustrates the rays (of different angles) that start at a single interior point when $h$ is a discrete approximation of the GFF.  Figure~\ref{fig::spokes2} illustrates the rays (of different angles) that start from each of two different interior points, and Figure~\ref{fig::spokes4} illustrates the rays (of different angles) starting at each of four different interior points.  We will prove several results which describe the way that rays of different angles interact with one another.  We will show in a precise sense that while rays of different angles can sometimes intersect and bounce off each other at multiple points (depending on $\chi$ and the angle difference), they can only ``cross'' each other at most once before they exit the domain.  (When $h$ is smooth, it is also the case that rays of different angles cross at most once; but if $h$ is smooth the rays cannot bounce off each other without crossing.)  Similar results were obtained in \cite{MS_IMAG} for paths started at boundary points of the domain.

It was also shown in \cite{MS_IMAG} that two paths with the same angle but different initial points can ``merge'' with one another.  Here we will describe the entire family of flow lines with a given angle (started at all points in some countable dense set).  This collection of merging paths can be understood as a kind of rooted space-filling tree; each branch of the tree is a variant of $\SLE_\kappa$, for $\kappa \in (0,4)$, that starts at an interior point of the domain.  These trees are illustrated for a range of $\kappa$ values in Figure~\ref{fig::tree_dual}.  It turns out that there is an a.s.\ continuous space-filling curve\footnote{We will in general write $\eta$ to denote an $\SLE_{\kappa}$ process (or variant) with $\kappa \in (0,4)$ and $\eta'$ an $\SLE_{\kappa'}$ process (or variant) with $\kappa' > 4$, except when making statements which apply to all $\kappa > 0$ values.} $\eta'$ that traces the entire tree and is a space-filling form of $\SLE_{\kappa'}$ where $\kappa'=16/\kappa > 4$ (see Figure~\ref{fig::space_filling} and Figure~\ref{fig::space_filling2} for an illustration of this construction for $\kappa' = 6$ as well as Figure~\ref{fig::space_filling3} and Figure~\ref{fig::bigger_than_8_reversible} for simulations when $\kappa' \in \{8,16,128 \}$).  In a certain sense, $\eta'$ traces the {\em boundary} of the tree in counterclockwise order.  The left boundary of $\eta'([0,t])$ is the branch of the tree started at $\eta'(t)$, and the right boundary is the branch of the dual tree started at $\eta'(t)$.  This construction generalizes the now well-known relationship between the GFF and uniform spanning tree scaling limits (whose branches are forms of $\SLE_2$ starting at interior domain points, and whose outer boundaries are forms of $\SLE_8$) \cite{KEN01, LSW04}.  Based on this idea, we define a new family of space-filling curves called {\bf space-filling $\SLE_{\kappa'}(\underline \rho)$} processes, defined for $\kappa' > 4$.

Finally, we will obtain new time-reversal symmetries, both for the new space-filling curves we introduce here and for a three-parameter family of whole-plane variants of $\SLE$ (which are random curves in $\C$ from $0$ to $\infty$) that generalizes the whole-plane $\SLE_\kappa(\rho)$ processes.

In summary, this is a long paper, but it contains a number of fundamental results about $\SLE$ and $\CLE$ that have not appeared elsewhere. These results include the following:
\begin{enumerate}
\item The first complete description of the collection of GFF flow lines. In particular, the first construction of the flow line rays emanating from interior points (including points with logarithmic singularities).
\item The first proof that, when $\kappa' > 8$, the time reversal of an $\SLE_{\kappa'}(\rho_1; \rho_2)$ process is a process that belongs to the same family. It has been known for some time \cite{RS_REVERSE} that $\SLE_{\kappa'}$ itself should {\em not} have time-reversal symmetry when $\kappa' > 8$. However the fact that its time reversal {\em can} be described by an $\SLE_{\kappa'}(\rho_1; \rho_2)$ process was not known, or even conjectured, before the current work.
\item The first proof that when $\kappa' \in (4,8)$ the {\em space-filling} $\SLE_{\kappa'}(\rho)$ processes are well-defined, are continuous, and have time-reversal symmetry. (The reversibility of chordal $\SLE$ was proved for $\kappa \in (0,4]$ in \cite{Z_R_KAPPA}, for the non-boundary intersecting $\SLE_\kappa(\rho)$ processes with $\kappa \in (0,4]$ in \cite{DUB_DUAL,Z_R_KAPPA_RHO}, for the entire class of $\SLE_\kappa(\rho_1;\rho_2)$ processes in \cite{MS_IMAG2}, and for the $\SLE_{\kappa'}(\rho_1;\rho_2)$ processes with $\kappa' \in (4,8]$ in \cite{MS_IMAG3}.)
\item The first proof of the time-reversal symmetry of whole-plane $\SLE_\kappa(\rho)$ processes that applies for general $\kappa$ and $\rho$.  This extends the main result of \cite{zha_whole_plane} (using very different techniques), which gives the reversibility of whole-plane $\SLE_\kappa$ for $\kappa \in (0,4]$, to the entire class of whole-plane $\SLE_\kappa(\rho)$ processes which have time-reversal symmetry.
\item The first complete development of $\SLE$ duality.  In particular, we give a complete description of the outer boundary of an $\SLE_{\kappa'}$ process stopped at an arbitrary stopping time.  ($\SLE$ duality was first proved in certain special cases in \cite{ZHAN_DUALITY_1,ZHAN_DUALITY_2,DUB_DUAL,MS_IMAG}.)
\item The first proof that the conformal loop ensembles $\CLE_{\kappa'}$, for $\kappa' \in (4,8)$, are actually well-defined as random collections of loops.  We also give the first proof that these random loop ensembles are locally finite and invariant under all conformal automorphisms of the domains on which they are defined.  (Similar results were proved in \cite{SHE_WER_CLE} in the case that $\kappa \in (8/3,4]$ using Brownian loop soups.)
\end{enumerate}

This paper is cited very heavily in works by the authors concerning Liouville quantum gravity, scaling limits of FK-decorated planar maps, the peanosphere, the Brownian map, and so forth. Basically, this is because there are many instances in which understanding what happens when Liouville quantum gravity surfaces are welded together turns out to be equivalent to understanding how GFF flow lines interact with each other. Moreover, the space-filling paths constructed and studied here for $\kappa' \in (4,8)$ are the foundation of several other constructions.

To elaborate on some of these points in more detail, let us first consider the program for relating FK weighted random planar maps to $\CLE$-decorated Liouville quantum gravity (LQG) \cite{DS08,SHE_WELD,SHE_INVENTORY}. It is shown in \cite{SHE_INVENTORY} that it is possible to encode such a random planar map in terms of a discrete tree/dual-tree pair which are glued together along a space-filling path and that these trees converge jointly to a pair of correlated continuum random trees (CRTs) \cite{ald1991crt1,ald1991crt2,ald1993crt3} as the size of the map tends to $\infty$.  It is then shown in \cite{LQG_MATING} that a certain type of LQG surface decorated with a {\em space-filling $\SLE$} of the sort introduced in this paper (which describes the interface between a tree/dual-tree pair constructed using GFF flow lines as described above) can be interpreted as a gluing of a pair of correlated CRTs.  This gives that LQG decorated with a space-filling $\SLE$ is the scaling limit of FK weighted random planar maps where two spaces are close when the contour functions of the associated tree/dual-tree pair are close. The duality between flow line trees and space-filling curves developed here is the basis for the proofs of the main results about mating trees in \cite{LQG_MATING}, and the results from this paper are extensively cited there. The results in this article (including results about reversibility and duality) also feature prominently in a program announced in \cite{qle2013} and carried out in \cite{quantum_spheres,map_making,qlebm,qle_continuity,qle_determined} to construct the metric space structure of $\sqrt{8/3}$-LQG and relate it to the Brownian map.

The results here will also be an important part of the proofs of several results in joint work by the authors and with Wendelin Werner \cite{cle_percolations,lightcones} about continuum analogs of FK models, conformal loop ensembles, and $\SLE_\kappa(\rho)$ processes with $\rho < -2$.  For example, the first proof that the $\SLE_\kappa(\rho)$ processes with $\rho < -2$ are continuous will be derived as a consequence of the continuity of the space-filling $\SLE$ processes introduced here.

\begin{figure}[ht!]
\begin{center}
\begin{ficomment}
\includegraphics[width=0.95\textwidth]{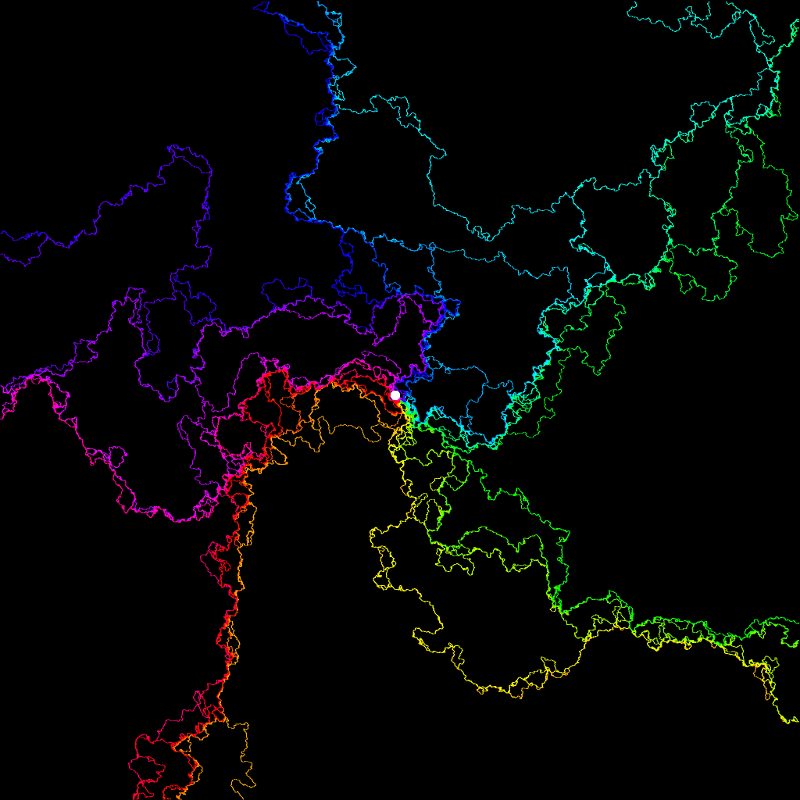}
\end{ficomment}
\end{center}
\caption{ \label{fig::spokes} Numerically generated flow lines, started at a common point, of $e^{i(h/\chi+\theta)}$ where $h$ is the projection of a GFF onto the space of functions piecewise linear on the triangles of an $800 \times 800$ grid; $\kappa=4/3$ and $\chi = 2/\sqrt{\kappa} - \sqrt{\kappa}/2 = \sqrt{4/3}$.  Different colors indicate different values of $\theta \in [0,2\pi)$.  We expect but do not prove that if one considers increasingly fine meshes (and the same instance of the GFF) the corresponding paths converge to limiting continuous paths.}
\end{figure}

\begin{figure}[ht!]
\begin{center}
\begin{ficomment}
\includegraphics[width=0.95\textwidth]{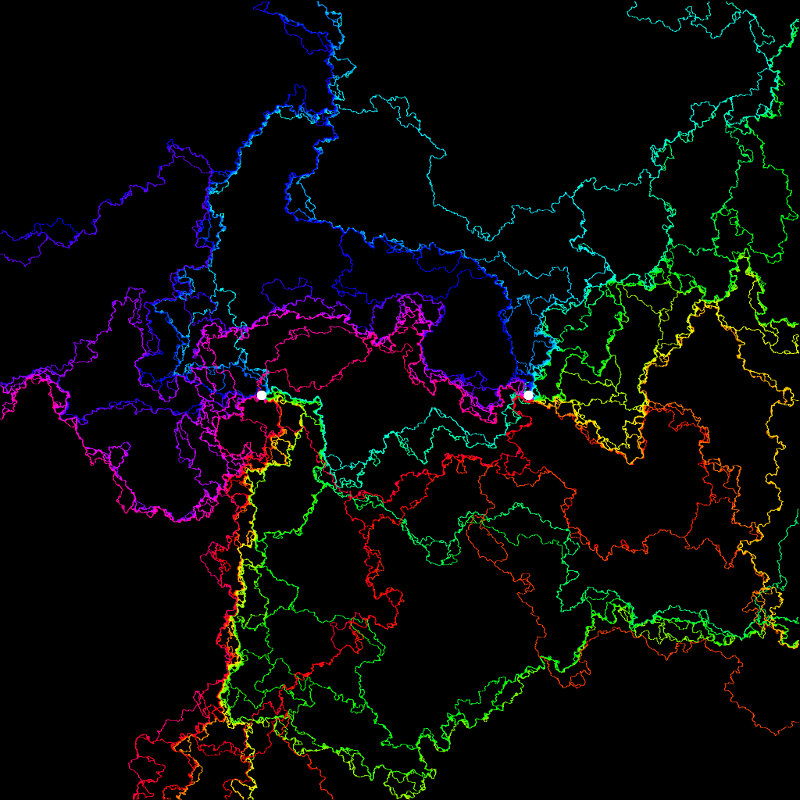}
\end{ficomment}
\end{center}
\caption{ \label{fig::spokes2} Numerically generated flow lines, emanating from two points, of $e^{i(h/\chi+\theta)}$ generated using the same discrete approximation $h$ of a GFF as in Figure~\ref{fig::spokes}; $\kappa=4/3$ and $\chi = 2/\sqrt{\kappa} - \sqrt{\kappa}/2 = \sqrt{4/3}$.  Flow lines with the same angle (indicated by the color) started at the two points appear to merge upon intersecting.}
\end{figure}

\begin{figure}[ht!]
\begin{center}
\begin{ficomment}
\includegraphics[width=0.95\textwidth]{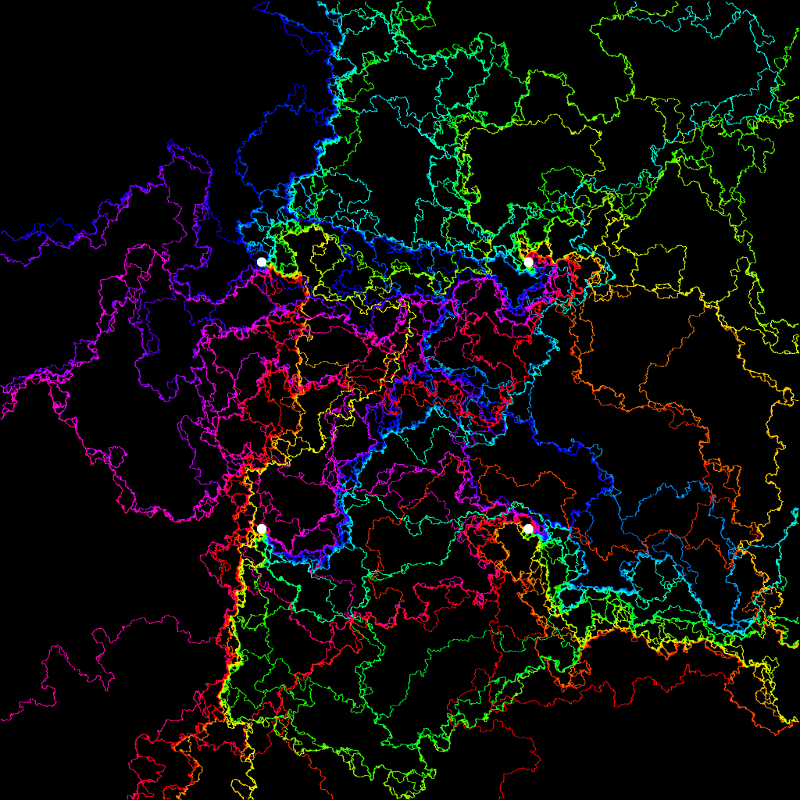}
\end{ficomment}
\end{center}
\caption{ \label{fig::spokes4} Numerically generated flow lines, emanating from four points, of $e^{i(h/\chi+\theta)}$ where $h$ is the same discrete approximation of the GFF used in Figure~\ref{fig::spokes} and Figure~\ref{fig::spokes2}; $\kappa=4/3$ and $\chi = 2/\sqrt{\kappa} - \sqrt{\kappa}/2 = \sqrt{4/3}$.}
\end{figure}



\begin{figure}[ht!]
\begin{center}
\begin{ficomment}
\subfigure[$\kappa=1/2$]{\includegraphics[width=0.48\textwidth]{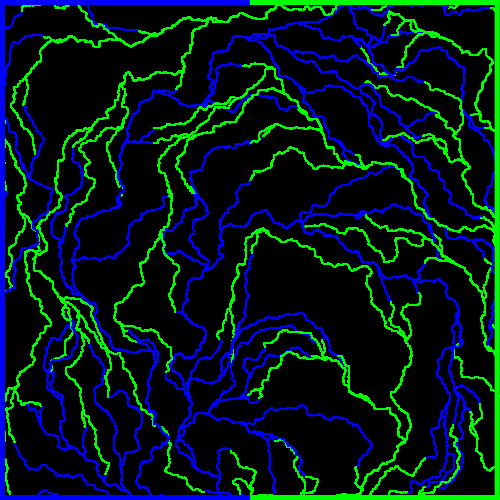}}
\hspace{0.02\textwidth}
\subfigure[$\kappa=1$]{\includegraphics[width=0.48\textwidth]{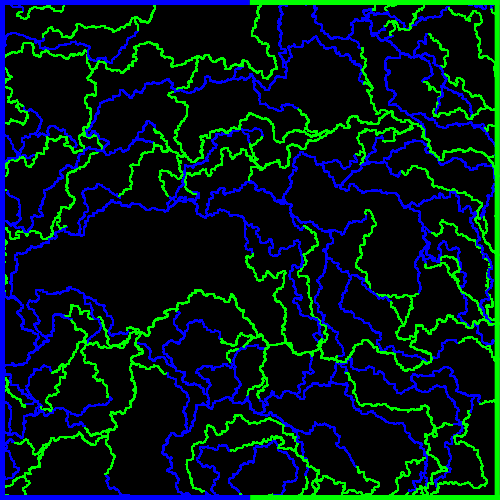}}
\subfigure[$\kappa=2$]{\includegraphics[width=0.48\textwidth]{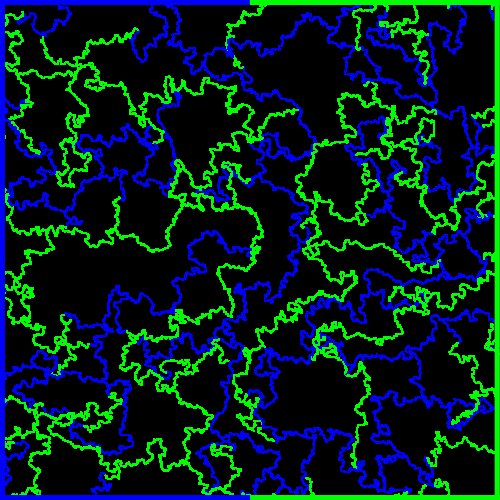}}
\hspace{0.02\textwidth}
\subfigure[\label{fig::tree_dual_8_3}$\kappa=8/3$]{\includegraphics[width=0.48\textwidth]{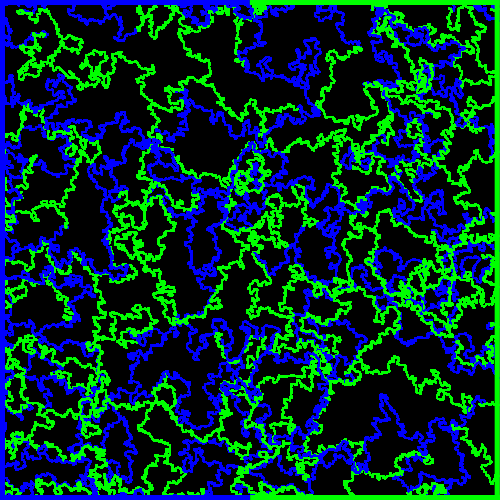}}
\end{ficomment}
\end{center}
\caption{\label{fig::tree_dual} \small{ Numerically generated flow lines of $e^{i(h/\chi+\theta)}$ where $h$ is the projection of a GFF onto the space of functions piecewise linear on the triangles of a $800 \times 800$ grid with various $\kappa$ values.  The flow lines start at $100$ uniformly chosen random points in $[-1,1]^2$.  The same points and approximation of the free field are used in each of the simulations.  The blue paths have angle $\tfrac{\pi}{2}$ while the green paths have angle $-\tfrac{\pi}{2}$.  The collection of blue and green paths form a pair of intertwined trees.  We will refer to the green tree as the ``dual tree'' and likewise the green branches as ``dual branches.''}}
\end{figure}


\begin{figure}[ht!]
\begin{center}
\begin{ficomment}
\subfigure{
\includegraphics[width=0.7\textwidth]{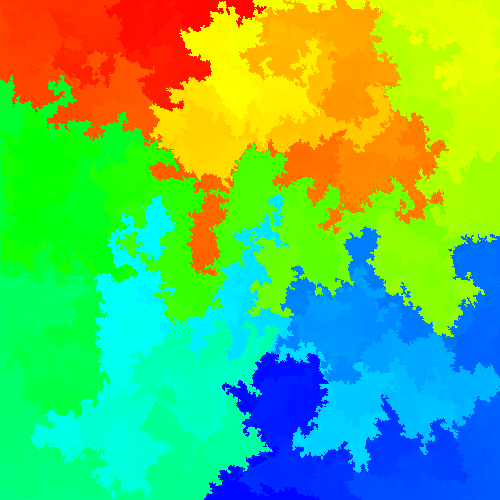}}

\subfigure{\includegraphics[scale=.9]{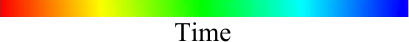}}

\subfigure{
\includegraphics[width=0.29\textwidth]{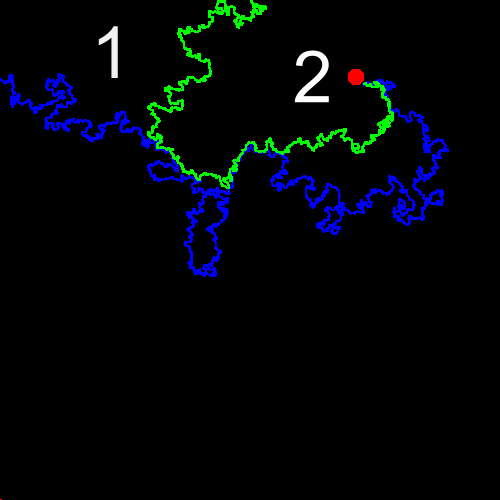}}
\hspace{0.01\textwidth}
\subfigure{
\includegraphics[width=0.29\textwidth]{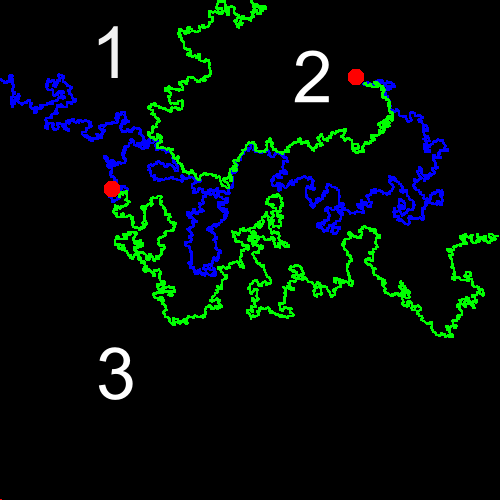}}
\hspace{0.01\textwidth}
\subfigure{
\includegraphics[width=0.29\textwidth]{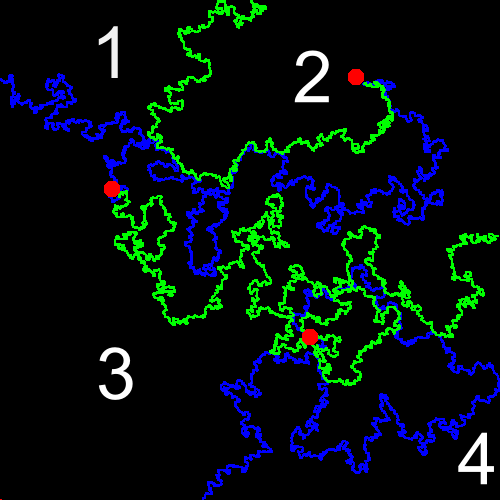}}
\end{ficomment}
\end{center}
\vspace{-0.027\textheight}
\caption{\label{fig::space_filling} \small{The intertwined trees of Figure~\ref{fig::tree_dual} can be used to generate space-filling $\SLE_{\kappa'}(\ul{\rho})$ for $\kappa'=16/\kappa$.  The branch and dual branch from each point divide space into those components whose boundary consists of part of the right (resp.\ left) side of the branch (resp.\ dual branch) and vice-versa.  The space-filling $\SLE$ visits the former first, as is indicated by the numbers in the lower illustrations after three successive subdivisions.  The top contains a simulation of a space-filling $\SLE_6$ in $[-1,1]^2$ from $i$ to $-i$. The colors indicate the time at which the path visits different points.  This was generated from the  same approximation of the GFF used to make Figure~\ref{fig::tree_dual_8_3}.}}
\end{figure}

\begin{figure}[ht!]
\begin{center}
\begin{ficomment}
\subfigure[$25\%$]{\includegraphics[width=0.48\textwidth]{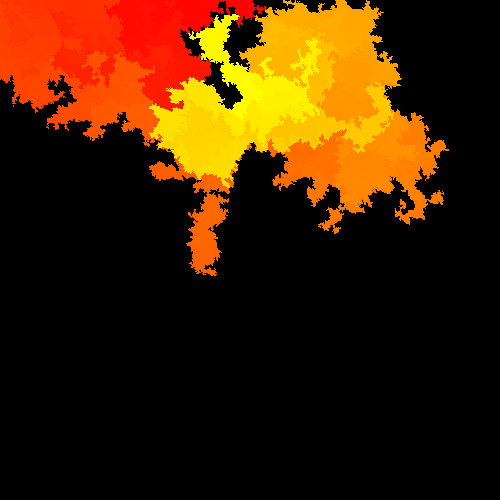}}
\hspace{0.02\textwidth}
\subfigure[$50\%$]{\includegraphics[width=0.48\textwidth]{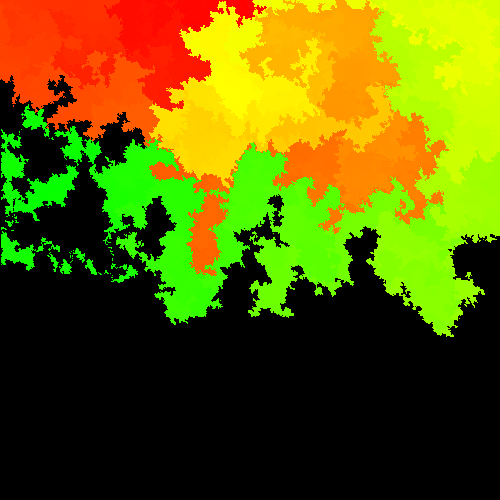}}
\subfigure[$75\%$]{\includegraphics[width=0.48\textwidth]{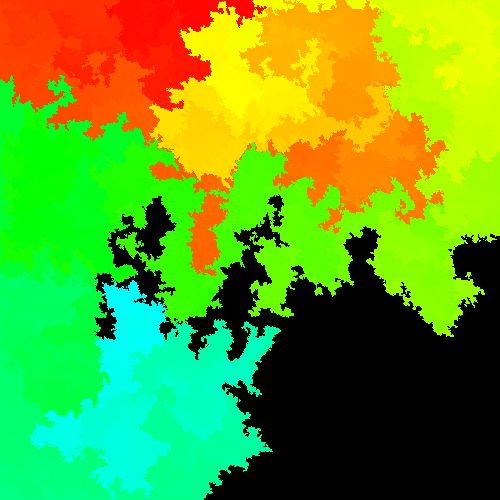}}
\hspace{0.02\textwidth}
\subfigure[$100\%$]{\includegraphics[width=0.48\textwidth]{figures/gff_pictures/space_filling_sle/flowline_animation2_0000000500.png}}

\subfigure{\includegraphics[scale=.9]{figures/hue_scale_time.pdf}}
\end{ficomment}
\end{center}
\vspace{-0.025\textheight}
\caption{\label{fig::space_filling2} \small{The space-filling $\SLE_6$ from Figure~\ref{fig::space_filling} parameterized according to area drawn up to different times.  Thousands of shades are used in the figure.  The visible interfaces between colors (separating green from orange, for example) correspond to points that are hit by the space-filling curve at two very different times.  (The orange side of the interface is filled in first, the green side on a second pass much later.)  See also Figure~\ref{fig::space_filling3} and Figure~\ref{fig::bigger_than_8_reversible} for related simulations with $\kappa'=8,16,128$.}}
\end{figure}

\subsection{Statements of main results}

\subsubsection{Constructing rays started at interior points}
A brief overview of imaginary geometry (as defined for general functions $h$) appears in \cite{SHE_WELD}, where the rays are interpreted as geodesics of an ``imaginary'' variant of the Levi-Civita connection associated with Liouville quantum gravity.  One can interpret the $e^{ih/\chi}$ direction as ``north'' and the $e^{i(h/\chi + \tfrac{\pi}{2})}$ direction as ``west'', etc.  Then $h$ determines a way of assigning a set of compass directions to every point in the domain, and a ray is determined by an initial point and a direction.  When $h$ is constant, the rays correspond to rays in ordinary Euclidean geometry.  For more general smooth functions $h$, one can still show that when three rays form a triangle, the sum of the angles is always $\pi$ \cite{SHE_WELD}.

\begin{figure}[ht!]
\begin{center}
\includegraphics[scale=0.85]{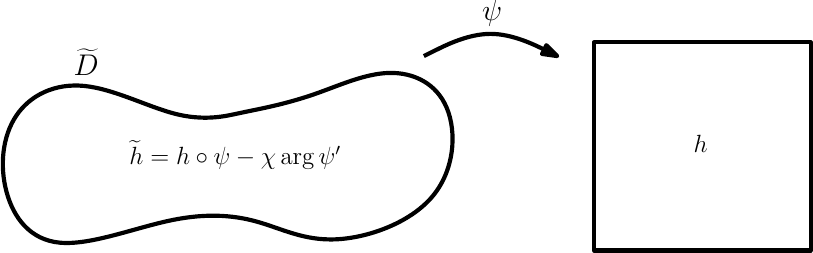}
\caption{\label{fig::coordinatechange} The set of flow lines in $\wt{D}$ will be the pullback via a conformal map $\psi$ of the set of flow lines in $D$ provided $h$ is transformed to a new function $\wt{h}$ in the manner shown.}
\end{center}
\end{figure}

If $h$ is a smooth function, $\eta$ a flow line of $e^{ih/\chi}$, and $\psi \colon \wt D \to D$ a conformal transformation, then by the chain rule, $\psi^{-1}(\eta)$ is a flow line of $h \circ \psi - \chi \arg \psi'$, as in Figure~\ref{fig::coordinatechange}. With this in mind, we define an {\bf imaginary surface}\footnote{We remark that, for readers familiar with this terminology, an imaginary surface can also be understood as a simply connected domain together with a section of its {\em orthonormal frame bundle}.} to be an equivalence class of pairs $(D,h)$ under the equivalence relation
\begin{equation}
\label{eqn::ac_eq_rel}
 (D,h) \rightarrow (\psi^{-1}(D), h \circ \psi - \chi \arg \psi') = (\wt{D},\wt{h}).
\end{equation}
We interpret $\psi$ as a (conformal) {\em coordinate change} of the imaginary surface.
In what follows, we will generally take $D$ to be the upper half-plane, but one can map the flow lines defined there to other domains using~\eqref{eqn::ac_eq_rel}.

Although~\eqref{eqn::ode} does not make sense as written (since $h$ is an instance of the GFF, not a function), one can construct these rays precisely by solving~\eqref{eqn::ode} in a rather indirect way: one begins by constructing explicit couplings of $h$ with variants of $\SLE$ and showing that these couplings have certain properties.  Namely, if one conditions on part of the curve, then the conditional law of $h$ is that of a GFF in the complement of the curve with certain boundary conditions.  Examples of these couplings appear in \cite{She_SLE_lectures, SchrammShe10, DUB_PART, SHE_WELD} as well as variants in \cite{MakarovSmirnov09,HagendorfBauerBernard10,IzyurovKytola10}.  This step is carried out in some generality in \cite{DUB_PART, SHE_WELD,MS_IMAG}.  The next step is to show that in these couplings the path is almost surely {\em determined by the field} so that we really can interpret the ray as a path-valued function of the field.  This step is carried out for certain boundary conditions in \cite{DUB_PART} and in more generality in \cite{MS_IMAG}.  Theorem~\ref{thm::existence} and Theorem~\ref{thm::uniqueness} describe analogs of these steps that apply in the setting of this paper.

\begin{figure}[ht!]
\begin{center}
\subfigure{\includegraphics[scale=0.85]{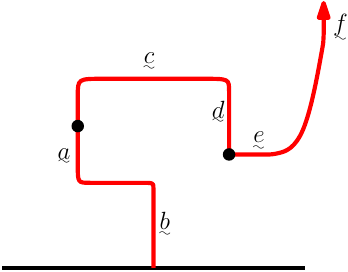}}
\hspace{0.05\textwidth}
\subfigure{\includegraphics[scale=0.85]{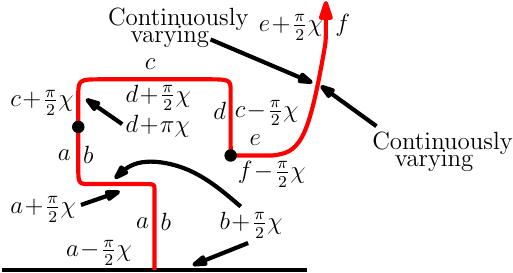}}
\caption{\label{fig::winding}  The notation on the left is a shorthand for the boundary data indicated on the right.  We often use this shorthand to indicate GFF boundary data.  In the figure, we have placed some black dots on the boundary $\partial D$ of a domain $D$.  On each arc $L$ of $\partial D$ that lies between a pair of black dots, we will draw either a horizontal or vertical segment $L_0$ and label it with $\uwave{x}$.  This means that the boundary data on $L_0$ is given by $x$, and that whenever $L$ makes a quarter turn to the right, the height goes down by $\tfrac{\pi}{2} \chi$ and whenever $L$ makes a quarter turn to the left, the height goes up by $\tfrac{\pi}{2} \chi$.  More generally, if $L$ makes a turn which is not necessarily at a right angle, the boundary data is given by $\chi$ times the winding of $L$ relative to $L_0$.  If we just write $x$ next to a horizontal or vertical segment, we mean just to indicate the boundary data at that segment and nowhere else.  The right side above has exactly the same meaning as the left side, but the boundary data is spelled out explicitly everywhere.  Even when the curve has a fractal, non-smooth structure, the {\em harmonic extension} of the boundary values still makes sense, since one can transform the figure via the rule in Figure~\ref{fig::coordinatechange} to a half-plane with piecewise constant boundary conditions. The notation above is simply a convenient way of describing what the constants are.  We will often include horizontal or vertical segments on curves in our figures (even if the whole curve is known to be fractal) so that we can label them this way.  This notation makes sense even for multiply connected domains.
}
\end{center}
\end{figure}

\begin{figure}[ht!]
\begin{center}
\includegraphics[scale=0.85]{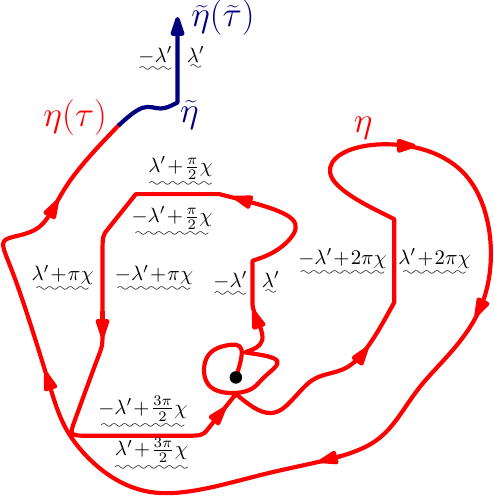}
\end{center}
\vspace{-0.2in}
\caption{\label{fig::interior_path_bd} {\footnotesize Suppose that $\eta$ is a non-self-crossing and non-self-tracing path in $\C$ starting from $0$ with the property that for all $t > 0$, the point $\eta(t)$ is not equal to the origin and lies on the boundary of the infinite component of $\C \setminus \eta([0,t])$, and $\eta$ has a continuous whole-plane Loewner driving function. Let us assume further $\eta$ has the property that for all $t$, a Brownian motion started at a point off $\eta([0,t])$ does not hit $\eta([0,t])$ for the first time at a {\em double} point of the path $\eta$. This implies that $\eta([0,t])$ has a well-defined ``left side'' and ``right side'' in the harmonic sense ---i.e., if one runs a Brownian motion from a point in $\C \setminus \eta([0,t])$, stopped at the first time it hits $\eta([0,t])$, one can a.s.\ make sense of whether it first hits $\eta([0,t])$ from the left or from the right. 
Let $\tau \in (0,\infty)$. We let $\wt{\eta}$ be a non-self-crossing path which agrees with $\eta$ until time $\tau$ and parameterizes a north-going vertical line segment in the time interval $[\tau+\tfrac{1}{2},\wt{\tau}]$ which is disjoint from $\wt{\eta}([0,\tau+\tfrac{1}{2}])$, as illustrated, where $\wt{\tau}=\tau+1$.  We then take $f$ to be the function which is harmonic in $\C \setminus \wt{\eta}([0,\wt{\tau}])$ whose boundary conditions are $-\lambda'$ (resp.\ $\lambda'$) on the left (resp.\ right) side of the vertical segment $\wt{\eta}([\tau+\tfrac{1}{2},\wt{\tau}])$.  The boundary data of $f$ on the left and right sides of $\wt{\eta}([0,\wt{\tau}])$ then changes by $\chi$ times the winding of $\wt{\eta}$, as explained in Figure~\ref{fig::winding} and indicated in the illustration above.  Explicitly, if $\varphi$ is a conformal map from the unbounded component $U$ of $\C \setminus \wt{\eta}([0,\wt{\tau}])$ to $\h$ which takes the left (resp.\ right) side of $\wt{\eta}|_{[0,\wt{\tau}]}$ which forms part of $\partial U$ to $\R_-$ (resp.\ $\R_+$) with $\varphi(\wt{\eta}(\wt{\tau})) = 0$ and $\Fh$ is the function which harmonic in $\h$ with boundary values given by $-\lambda$ (resp.\ $\lambda$) in $\R_-$ (resp.\ $\R_+$) then $f|_U$ has the same boundary data on $\partial U$ as $\Fh \circ \varphi - \chi \arg \varphi'$.  We define $f$ similarly in the other components of $\C \setminus \wt{\eta}([0,\wt{\tau}])$.  Note that $f$ is only defined up to a global additive constant in $2\pi \chi \Z$ since one has to choose the branch of $\arg$.  Given a domain $D$ in $\C$, we say that a GFF on $D \setminus \eta([0,\tau])$ has {\bf flow line boundary conditions} on $\eta([0,\tau])$ up to a global additive constant in $2\pi \chi \Z$ if the boundary data of $h$ agrees with $f$ along $\eta([0,\tau])$, up to a global additive constant in $2\pi \chi \Z$ (this specifies the boundary data up to a harmonic function which is $0$ on $\partial D$ and a multiple of $2\pi \chi$ on $\eta([0,\tau])$).  
This definition does not depend on the choice of $\wt{\eta}$.  More generally, we say that $h$ has flow line boundary conditions on $\eta([0,\tau])$ with angle $\theta$ if the boundary data of $h+\theta \chi$ agrees with $f$ on $\eta([0,\tau])$, up to a global additive constant in $2\pi \chi \Z$.
}}
\end{figure}

Before we state these theorems, we recall the notion of boundary data that tracks the ``winding'' of a curve, as illustrated in Figure~\ref{fig::winding}.  For $\kappa \in (0,4)$ fixed, we let
\begin{equation}
\label{eqn::constants}
 \lambda = \frac{\pi}{\sqrt{\kappa}},\quad \lambda' =\frac{\pi}{\sqrt{16/\kappa}} = \frac{\pi \sqrt{\kappa}}{4},\quad\text{and}\quad\chi = \frac{2}{\sqrt{\kappa}} - \frac{\sqrt{\kappa}}{2}.
\end{equation}
Note that $\chi > 0$ for this range of $\kappa$ values.  Given a path starting in the interior of the domain, we use the term {\bf flow line boundary conditions} to describe the boundary conditions that would be given by $-\lambda'$ (resp.\ $\lambda'$) on the left (resp.\ right) side of a north-going vertical segment of the curve and then changes according to $\chi$ times the winding of the path, up to an additive constant in $2\pi\chi \Z$.  We will indicate this using the notation of Figure~\ref{fig::winding}.  See the caption of Figure~\ref{fig::interior_path_bd} for further explanation.

Note that if $\eta$ solves~\eqref{eqn::ode} when $h$ is smooth, then this will remain the case if we replace $h$ by $h + 2 \pi \chi$.  This will turn out to be true for the flow lines defined from the GFF as well, and this idea becomes important when we let $h$ be an instance of the whole-plane GFF on $\C$.  Typically, an instance $h$ of the whole-plane GFF is defined modulo a global additive constant in $\R$, but it turns out that it is also easy and natural to define $h$ modulo a global additive multiple of $2\pi \chi$ (see Section~\ref{subsec::gff} for a precise construction).  When we know $h$ modulo an additive multiple of $2 \pi \chi$, we will be able to define its flow lines.  Before we show that $\eta$ is a path-valued function of $h$, we will establish a preliminary theorem that shows that there is a unique coupling between $h$ and $\eta$ with certain properties.  Throughout, we say that a domain $D \subseteq \C$ has {\bf harmonically non-trivial boundary} if a Brownian motion started at a point in $D$ hits $\partial D$ almost surely.

\begin{theorem}
\label{thm::existence}
Fix a connected domain $D \subsetneq \C$ with harmonically non-trivial boundary and let $h$ be a GFF on $D$ with some boundary data.  Fix a point $z \in D$.  There exists a unique coupling between $h$ and a random path $\eta$ (defined up to monotone parameterization) started at $z$ (and stopped when it first hits $\partial D$) such that the following is true.  For any $\eta$-stopping time $\tau$, the conditional law of $h$ given $\eta|_{[0,\tau]}$ is given by that of  the sum of a GFF $\wt{h}$ on $D \setminus \eta([0,\tau])$ with zero boundary conditions and a random\footnote{We recall that flow line boundary conditions are only defined up to a global additive constant in $2\pi \chi \Z$.  We thus emphasize that saying that the boundary data along $\eta([0,\tau])$ itself is given by flow line boundary conditions only specifies the boundary data along $\eta([0,\tau])$ up to a global additive constant in $2\pi \chi \Z$.  In the case that $D = \C$, flow line boundary conditions specify the boundary data up to a global additive constant in $2\pi \chi \Z$.  In the case that $D$ has harmonically non-trivial boundary, flow line boundary conditions along $\eta([0,\tau])$ specify the boundary data up to a harmonic function which is $0$ on $\partial D$ and a multiple of $2\pi \chi$ on $\eta([0,\tau])$.} harmonic function $\Fh$ on $D \setminus \eta([0,\tau])$ whose boundary data agrees with the boundary data of $h$ on $\partial D$ and is given by flow line boundary conditions on $\eta([0,\tau])$ itself.  Moreover, $\wt{h}$ and $\Fh$ are conditionally independent given $\eta|_{[0,\tau]}$.  The path is simple when $\kappa \in (0,8/3]$ and is self-touching for $\kappa \in (8/3,4)$.  Similarly, if $D = \C$ and $h$ is a whole-plane GFF (defined modulo a global additive multiple of $2 \pi \chi$) there is also a unique coupling of a random path $\eta$ and $h$ satisfying the property described in the $D \subsetneq \C$ case above.  In this case, the law of $\eta$ is that of a whole-plane $\SLE_\kappa(2-\kappa)$ started at $z$.  Finally, in all cases the set $\eta([0,\tau])$ is local for $h$ in the sense of \cite{SchrammShe10}.
\end{theorem}

We emphasize that the harmonic function $\Fh$ in the statement of Theorem~\ref{thm::existence} is \emph{not} determined by $\eta|_{[0,\tau]}$ in the case that $D \neq \C$ and $\tau$ occurs before $\eta$ first hits $\partial D$.  However, $\Fh$ is determined by $\eta|_{[0,\tau]}$ and the $\sigma$-algebra $\CF$ which is given by $\cap_{\epsilon > 0} \sigma( h|_{B(z,\epsilon)})$.  The uniform spanning tree (UST) height function provides a discrete analogy of this statement.  Namely, if one picks a lattice point $z$ and starts to explore a branch of the UST starting from $z$ back to the boundary, then the UST height function along the path is not determined by the path before the path has hit the boundary.  However, if one conditions on both the path and the height function at one point along the path, then the heights are determined along the entire path even before it has hit the boundary.

In this article, we will often use the term ``self-touching'' to describe a curve which is both self-intersecting and non-crossing.

We will give an overview of the whole-plane $\SLE_\kappa(\rho)$ and related processes in Section~\ref{subsec::SLEoverview} and, in particular, show that these processes are almost surely generated by continuous curves.  This extends the corresponding result for chordal $\SLE_\kappa(\ul{\rho})$ processes established in \cite{MS_IMAG}.  In Section~\ref{subsec::gff}, we will explain how to make sense of the GFF modulo a global additive multiple of a constant $r > 0$.  The construction of the coupling in Theorem~\ref{thm::existence} is first to sample the path $\eta$ according to its marginal distribution and then, given $\eta$, to pick $h$ as a GFF with the boundary data as described in the statement.  Theorem~\ref{thm::existence} implies that when one integrates over the randomness of the path, the marginal law of $h$ on the whole domain is a GFF with the given boundary data.  Our next result is that $\eta$ is in fact determined by the resulting field, which is not obvious from the construction.  Similar results for ``boundary emanating'' GFF flow lines (i.e., flow lines started at points on the boundary of $D$) appear in \cite{DUB_PART, MS_IMAG,SchrammShe10}.

\begin{theorem}
\label{thm::uniqueness}
In the coupling of a GFF $h$ and a random path $\eta$ as in Theorem~\ref{thm::existence}, the path $\eta$ is almost surely determined by $h$ viewed as a distribution modulo a global multiple of $2\pi \chi$.  (In particular, the path does not change if one adds a global additive multiple of $2\pi \chi$ to $h$.)
\end{theorem}

Theorem~\ref{thm::existence} describes a coupling between a whole-plane $\SLE_\kappa(2-\kappa)$ process for $\kappa \in (0,4)$ and the whole-plane GFF.  In our next result, we will describe a coupling between a whole-plane $\SLE_\kappa(\rho)$ process for general $\rho > -2$ (this is the full range of~$\rho$ values for which ordinary $\SLE_\kappa(\rho)$ makes sense) and the whole-plane GFF plus an appropriate multiple of the argument function.  We motivate this construction with the following.  Suppose that $h$ is a smooth function on the cone $\CC_{\overline{\theta}}$ obtained by identifying the two boundary rays of the wedge $\{z : \arg z \in [0,\overline{\theta}] \}$.  (When defining this wedge, we consider $z$ to belong to the universal cover of $\C \setminus \{0\}$, on which $\arg$ is continuous and single-valued; thus the cone $\CC_{\overline{\theta}}$ is defined even when $\overline{\theta} > 2\pi$.)  Note that there is a $\overline{\theta}$ range of angles of flow lines of $e^{ih/\chi}$ in $\CC_{\overline{\theta}}$ starting from $0$ and a $2\pi$ range of angles of flow lines starting from any point $z \in \CC_{\overline{\theta}} \setminus \{0\}$.  We can map $\CC_{\overline{\theta}}$ to $\C$ with the conformal transformation $z \mapsto \psi_{\overline{\theta}}(z) \equiv z^{2\pi/\overline{\theta}}$.  Applying the change of variables formula~\eqref{eqn::ac_eq_rel}, we see that $\eta$ is a flow line of $h$ if and only if $\psi_{\overline{\theta}}(\eta)$ is a flow line of $h \circ \psi_{\overline{\theta}}^{-1} - \chi \big(\overline{\theta} / 2\pi  - 1\big) \arg(\cdot)$.  Therefore we should think of $h-\alpha \arg(\cdot)$ (where $h$ is a GFF) as the conformal coordinate change of a GFF with a conical singularity.  The value of $\alpha$ determines the range $\overline{\theta}$ of angles for flow lines started at $0$: indeed, by solving $\chi(\overline{\theta}/2\pi - 1) = \alpha$, we obtain
\begin{equation}
\label{eqn::alphatheta}
\overline{\theta} = 2\pi\left(1+\frac{\alpha}{\chi}\right),
\end{equation}
which exceeds zero as long as $\alpha > - \chi$.  See Figure~\ref{fig::varying_alpha} for numerical simulations.

Theorem~\ref{thm::alphabeta}, stated just below, implies that analogs of Theorem~\ref{thm::existence} and Theorem~\ref{thm::uniqueness} apply in the even more general setting in which we replace $h$ with $h_{\alpha\beta} \equiv h-\alpha \arg(\cdot-z) - \beta \log|\cdot-z|$, $\alpha >- \chi$ (where $\chi$ is as in~\eqref{eqn::constants}), $\beta \in \R$, and $z$ is fixed.  When $\beta = 0$ and $h$ is a whole-plane GFF, the flow line of $h_{\alpha} \equiv h_{\alpha 0}$ starting from  $z$ is a whole-plane $\SLE_\kappa(\rho)$ process where the value of $\rho$ depends on $\alpha$.  Non-zero values of $\beta$ cause the flow lines to spiral either in the clockwise $(\beta < 0)$ or counterclockwise ($\beta > 0$) direction;  see Figure~\ref{fig::spiral_fan}.  In this case, the flow line is a variant of whole-plane $\SLE_\kappa(\rho)$ in which one adds a constant drift whose speed depends on $\beta$.  As will be shown in Section~\ref{sec::timereversal}, the case that $\beta \neq 0$ will arise in our proof of the reversibility of whole-plane $\SLE_\kappa(\rho)$.

\begin{figure}[ht!]
\begin{center}
\includegraphics[scale=0.85]{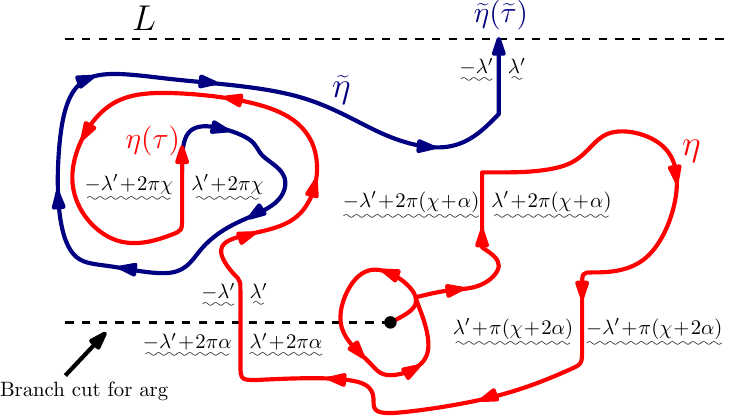}
\end{center}
\caption{\label{fig::interior_path_bd2}  Suppose that $\eta$ is a non-self-crossing path in $\C$ starting from~$0$ and let $\tau \in (0,\infty)$.  Fix $\alpha \in \R$ and a horizontal line $L$ which lies above $\eta([0,\tau])$.  We let $\wt{\eta}$ be a non-self-crossing path whose range lies below $L$, and which agrees with $\eta$ up to time $\tau$, terminates in $L$ at time $\wt{\tau}=\tau+1$, and parameterizes an up-directed vertical line segment in the time interval $[\tau+\tfrac{1}{2},\wt{\tau}]$.  Let $f$ be the harmonic function on $\C \setminus \wt{\eta}([0,\wt{\tau}])$ which is $-\lambda'$ (resp.\ $\lambda'$) on the left (resp.\ right) side of $\wt{\eta}([\tau+\tfrac{1}{2},\wt{\tau}])$ and changes by $\chi$ times the winding of $\wt{\eta}$ as in Figure~\ref{fig::winding}, except jumps by $2\pi \alpha$ (resp.\ $-2\pi\alpha$) if $\eta$ passes though $(-\infty,0)$ from above (resp.\ below).  Whenever $\eta$ wraps around $0$ in the counterclockwise (resp.\ clockwise) direction, the boundary data of $f$ increases (resp.\ decreases) by $2\pi (\chi+\alpha)$.  If $\eta$ winds around a point $z \neq 0$ in the counterclockwise (resp.\ clockwise) direction, then the boundary data of $f$ increases (resp.\ decreases) by $2\pi \chi$.  We say that a GFF $h$ on $D \setminus \eta([0,\tau])$, $D \subseteq \C$ a domain, has {\bf $\alpha$-flow line boundary conditions} along $\eta([0,\tau])$ (modulo $2\pi(\chi+\alpha)$) if the boundary data of $h$ agrees with $f$ on $\eta([0,\tau])$, up to a global additive constant in $2\pi(\chi+\alpha) \Z$.  This definition does not depend on the choice of $\wt{\eta}$.  More generally, we say that $h$ has $\alpha$-flow line boundary conditions on $\eta$ with angle $\theta$ if the boundary data of $h+\theta \chi$ agrees with $f$ on $\eta([0,\tau])$, up to a global additive constant in $2\pi(\chi+\alpha) \Z$.  The boundary conditions are defined in an analogous manner in the case that $\eta$ starts from $z \neq 0$.}
\end{figure}

Before stating Theorem~\ref{thm::alphabeta}, we will first need to generalize the notion of flow line boundary conditions; see Figure~\ref{fig::interior_path_bd2}.  We will assume without loss of generality that the starting point for $\eta$ is given by $z = 0$ for simplicity; the definition that we will give easily extends to the case $z \neq 0$.  We will define a function $f$ that describes the boundary behavior of the conditional expectation of $h_{\alpha \beta}$ along $\eta([0,\tau])$ where $\eta$ is a flow line and $\tau$ is a stopping time for $\eta$.  To avoid ambiguity, we will focus throughout on the branch of $\arg$ given by taking $\arg(\cdot) \in (-\pi, \pi]$ and we place the branch cut on $(-\infty,0)$.  In the case that $\alpha = \beta = 0$, the $f$ we defined (recall Figure~\ref{fig::interior_path_bd}) was only determined modulo a global additive multiple of $2 \pi \chi$ since in this setting, each time the path winds around $0$, the height of $h$ changes by $\pm 2\pi \chi$.  For general values of $\alpha, \beta \in \R$, each time the path winds around $0$, the height of $h_{\alpha \beta}$ changes by $\pm 2\pi (\chi+\alpha) = \pm \overline{\theta} \chi$, for the $\overline{\theta}$ defined in~\eqref{eqn::alphatheta}.  Therefore it is natural to describe the values of $f$ modulo global multiple of $2\pi(\chi+\alpha) = \overline{\theta} \chi$.  As we will explain in more detail later, adding a global additive constant that changes the values of $f$ modulo $2\pi(\chi+\alpha)$ amounts to changing the ``angle'' of $\eta$.  Since $h_{\alpha\beta}$ has a $2 \pi \alpha$ size ``jump'' along $(-\infty, 0)$ (coming from the discontinuity in $-\alpha \arg$), the boundary data for $f$ will have an analogous jump.

 In order to describe the boundary data for $f$, we fix a horizontal line $L$ which lies above $\eta([0,\tau])$, we let $\wt{\tau} = \tau+1$, and $\wt{\eta} \colon [0,\wt{\tau}] \to \C$ be a non-self-crossing path contained in the half-space which lies below $L$ with $\wt{\eta}|_{[0,\tau]} = \eta$ and $\wt{\eta}(\wt{\tau}) \in L$.  We moreover assume that the final segment of $\wt{\eta}$ is a north-going vertical line.  We set the value of $f$ to be $-\lambda'$ (resp.\ $\lambda'$) on the left (resp.\ right) side of the terminal part of $\wt{\eta}$ and then extend to the rest of $\wt{\eta}$ as in Figure~\ref{fig::winding} except with discontinuities each time the path crosses $(-\infty,0)$.  Namely, if $\wt{\eta}$ crosses $(-\infty,0)$ from above (resp.\ below), the height increases (resp.\ decreases) by $2\pi \alpha$.  Note that these discontinuities are added in such a way that the boundary data of $f+\alpha \arg(\cdot) + \beta \log|\cdot|$ changes continuously across the branch discontinuity.  We say that a GFF $h$ on $D \setminus \eta([0,\tau])$ has {\bf $\alpha$-flow line boundary conditions} (modulo $2\pi(\chi+\alpha)$) along $\eta([0,\tau])$ if the boundary data of $h$ agrees with $f$ along $\eta([0,\tau])$, up to a global additive constant in $2\pi(\chi+\alpha) \Z$.  We emphasize that this definition does not depend on the particular choice of $\wt{\eta}$.  The reason is that although two different choices may wind around $0$ a different number of times before hitting $L$, the difference only changes the boundary data of $f$ along $\eta([0,\tau])$ by an integer multiple of $2\pi(\chi+\alpha)$.

 \begin{remark}
\label{rem::discontinuous}
The boundary data for the $f$ that we have defined jumps by $2\pi \alpha$ when $\eta$ passes through $(-\infty,0)$ due to the branch cut of the argument function.  If we treated $\arg$ and $h_{\alpha \beta}$ as multi-valued (generalized) functions on the universal cover of $\C \setminus \{0\}$, then we could define $f$ in a continuous way on the universal cover of $\C \setminus \wt{\eta}([0,\wt{\tau}])$.  However, we find that this approach causes some confusion in our later arguments (as it is easy to lose track of which branch one is working in when one considers various paths that wind around the origin in different ways). We will therefore consider $h_{\alpha \beta}$ to be a single-valued generalized function with a discontinuity along $(-\infty, 0)$, and we accept that the boundary data for $f$ has discontinuities.
\end{remark}

\begin{theorem}
\label{thm::alphabeta}
Suppose that $\kappa \in (0,4)$, $\alpha > -\chi$ with $\chi$ as in~\eqref{eqn::constants}, and $\beta \in \R$.  Let $h$ be a GFF on a domain $D \subseteq \C$.  If $D \not = \C$, we assume that some fixed boundary data for $h$ on $\partial D$ is given.  If $D = \C$, then we let $h$ be a GFF on $\C$ defined modulo a global additive multiple of $2\pi(\chi+\alpha)$.    Let $h_{\alpha \beta} = h-\alpha \arg(\cdot - z) - \beta\log|\cdot - z|$.  Then there exists a unique coupling between $h_{\alpha \beta}$ and a random path $\eta$ starting from $z$ so that for every $\eta$-stopping time $\tau$ the following is true.  The conditional law of $h_{\alpha \beta}$ given $\eta|_{[0,\tau]}$ is given by that of the sum of a GFF $\wt{h}$ on $D \setminus \eta([0,\tau])$ with zero boundary conditions and a harmonic function $\Fh$ on $D \setminus \eta([0,\tau])$ with $\alpha$-flow line boundary conditions\footnote{We recall that $\alpha$-flow line boundary conditions are only defined up to a global additive constant in $2\pi(\chi+\alpha) \Z$.  So, saying that the boundary data along $\eta([0,\tau])$ itself is given by $\alpha$-flow line boundary conditions only specifies the boundary data along $\eta([0,\tau])$ up to a global additive constant in $2\pi (\chi+\alpha) \Z$.  In the case that $D = \C$, $\alpha$-flow line boundary conditions specify the boundary data up to a global additive constant in $2\pi(\chi+\alpha) \Z$.  In the case that $D$ has harmonically non-trivial boundary, $\alpha$-flow line boundary conditions along $\eta([0,\tau])$ specify the boundary data up to a harmonic function which is $0$ on $\partial D$ and a multiple of $2\pi(\chi+\alpha)$ on $\eta([0,\tau])$.} along $\eta([0,\tau])$, the same boundary conditions as $h_{\alpha \beta}$ on $\partial D$, and a $2\pi \alpha$ discontinuity along $(-\infty,0)+z$, as described in Figure~\ref{fig::interior_path_bd2}.  Given $\eta([0,\tau])$, $\wt{h}$ and $\Fh$ are conditionally independent.  Moreover, if $\beta = 0$, $D = \C$, and $h_\alpha=h_{\alpha 0}$, then the corresponding path $\eta$ is a whole-plane $\SLE_\kappa(\rho)$ process with $\rho = 2-\kappa + 2\pi\alpha/\lambda$.  Regardless of the values of $\alpha$ and $\beta$, $\eta$ is a.s.\ locally self-avoiding in the sense that its lifting to the universal cover of $D \setminus \{z\}$ is self-avoiding.  Finally, the random path $\eta$ is almost surely determined by the distribution~$h_{\alpha \beta}$ modulo a global additive multiple of $2\pi(\chi+\alpha)$.  (In particular, even when $D \neq \C$, the path~$\eta$ does not change if one adds a global additive multiple of $2\pi(\chi+\alpha)$ to $h_{\alpha\beta}$.)  In all cases the set $\eta([0,\tau])$ is local for $h$ in the sense of \cite{SchrammShe10}.
\end{theorem}

As explained just after the statement of Theorem~\ref{thm::existence}, the harmonic function $\Fh$ is not determined by $\eta|_{[0,\tau]}$ if $\tau$ occurs before $\eta$ has hit $\partial D$ for the first time.  However, it is determined if one conditions on both $\eta|_{[0,\tau]}$ and the $\sigma$-algebra $\CF$ which is given by $\cap_{\epsilon > 0} \sigma(h_{\alpha \beta}|_{B(z,\epsilon)})$.

In the statement of Theorem~\ref{thm::alphabeta}, in the case that $D = \C$ we interpret the statement that the conditional law of $h_{\alpha \beta}$ given $\eta|_{[0,\tau]}$ has the same boundary conditions as $h_{\alpha \beta}$ on $\partial D$ as saying that the behavior of the two fields at $\infty$ is the same.  By this, we mean that the total variation distance of the laws of the two fields (as distributions modulo a global additive multiple of $2\pi(\chi + \alpha)$) restricted to the complement of $B(0,R)$ tends to $0$ as $R \to \infty$.

Using Theorem~\ref{thm::alphabeta}, for each $\theta \in [0, \overline{\theta}) = [0,2\pi(1+\alpha/\chi))$ we can generate the ray $\eta_\theta$ of $h_{\alpha \beta}$ starting from $z$ by taking $\eta_\theta$ to be the flow line of $h_{\alpha \beta} + \theta \chi$.  The boundary data for the conditional law of $h_{\alpha \beta}$ given $\eta_\theta$ up to some stopping time $\tau$ is given by $\alpha$-flow line boundary conditions along $\eta([0,\tau])$ with angle $\theta$ (i.e., $h+\theta \chi$ has $\alpha$-flow line boundary conditions, as described in Figure~\ref{fig::interior_path_bd2}).  Note that we can determine the angle $\theta$ from these boundary conditions along $\eta([0,\tau])$ since the boundary data along a north-going vertical segment of $\eta$ takes the form $\pm \lambda' - \theta \chi$, up to an additive constant in $2\pi(\chi+\alpha)\Z$.  This is the fact that we need in order to prove that the path is determined by the field in Theorem~\ref{thm::alphabeta}.  If we had taken the field modulo a (global) constant other than $2\pi(\chi+\alpha)$, then the boundary values along the path would not determine its angle since the path winds around its starting point an infinite number of times (see Remark~\ref{rem::alpha_beta_determined} below).  The range of possible angles starting from $z$ is determined by $\alpha$.  If $\alpha > 0$, then the range of angles is larger than $2\pi$ and if $\alpha < 0$, then the range of angles is less than $2\pi$; see Figure~\ref{fig::varying_alpha}.  If $\alpha < -\chi$ then we can draw a ray from $\infty$ to $z$ instead of from $z$ to $\infty$.  (This follows from Theorem~\ref{thm::alphabeta} itself and a $w \to 1/w$ coordinate change using the rule of Figure~\ref{fig::coordinatechange}.)  If $\alpha = -\chi$ then a ray started away from $z$ can wrap around $z$ and merge with itself.  In this case, one can construct loops around $z$ in a natural way, but not flow lines connecting $z$ and $\infty$.  A non-zero value for $\beta$ causes the flow lines starting at $z$ to spiral in the counterclockwise (if $\beta > 0$) or  clockwise (if $\beta < 0$) direction, as illustrated in Figure~\ref{fig::spiral_fan} and Figure~\ref{fig::spiral_lightcone}.

\begin{remark}
\label{rem::alpha_beta_determined}
(This remark contains a technical point which should be skipped on a first reading.)
In the context of Figure~\ref{fig::interior_path_bd2}, Theorem~\ref{thm::alphabeta} implies that it is possible to specify a particular flow line starting from $z$ (something like a ``north-going flow line'') provided that the values of the field are known up to a global multiple of $2\pi(\chi + \alpha)$.  What happens if we try to start a flow line from a different point $w \neq z$?  In this case, in order to specify a ``north-going'' flow line starting from $w$, we need to know the values of the field modulo a global multiple of $2\pi\chi$, not modulo $2\pi(\chi+\alpha)$.  If $\alpha = 0$ and $\beta \in \R$ and we know the field modulo a global multiple of $2\pi\chi$, then there is no problem in defining a north-going line started at $w$.

But what happens if $\alpha \neq 0$ and we only know the field modulo a global multiple of $2\pi(\chi+\alpha)$?  In this case, we do not have a way to single out a {\em specific} flow line started from $w$ (since changing the multiple of $2\pi(\chi+\alpha)$ changes the angle of the north-going flow line started at $w$).  On the other hand, suppose we let $U$ be a random variable (independent of $h_{\alpha\beta}$) in $[0,2\pi (\chi + \alpha))$, chosen uniformly from the set $A$ of multiples of $2\pi\chi$ taken modulo $2\pi(\chi+\alpha)$ (or chosen uniformly on all of $[0,2\pi (\chi + \alpha))$ if this set is dense, which happens if $\alpha/\chi$ is irrational).  Then we consider the law of a flow line, started at $w$, of the field $h_{\alpha \beta}+U$.  The conditional law of such a flow line (given $h_{\alpha\beta}$ but not $U$) does not change when we add a multiple of $2\pi(\chi+\alpha)$ to $h_{\alpha\beta}$.  So this {\em random} flow line from $w$ can be defined canonically even if $h_{\alpha\beta}$ is only known modulo an integer multiple of $2\pi(\chi+\alpha)$.

Similarly, if we only know $h_{\alpha \beta}$ modulo $2\pi(\chi+\alpha)$, then the collection of all possible flow lines of $h_{\alpha \beta}+U$ (where $U$ ranges over all the values in its support) starting from $w$ is a.s.\ well-defined.
\end{remark}

It is also possible to extend Theorem~\ref{thm::alphabeta} to the setting that $\kappa' > 4$.

\begin{theorem}
\label{thm::alphabeta_counterflow}
Suppose that $\kappa' > 4$, $\alpha < -\chi$, and $\beta \in \R$.  Let $h$ be a GFF on a domain $D \subseteq \C$ and let $h_{\alpha \beta} = h-\alpha \arg(\cdot - z) - \beta\log|\cdot - z|$.  If $D = \C$, we view $h_{\alpha \beta}$ as a distribution defined up to a global multiple of $2\pi(\chi+\alpha)$.  There exists a unique coupling between $h_{\alpha \beta}$ and a random path $\eta'$ starting from $z$ so that for every $\eta'$-stopping time $\tau$ the following is true.  The conditional law of $h_{\alpha \beta}$ given $\eta'|_{[0,\tau]}$ is a GFF on $D \setminus \eta'([0,\tau])$ with $\alpha$-flow line boundary conditions with angle $\tfrac{\pi}{2}$ (resp.\ $-\tfrac{\pi}{2}$) on the left (resp.\ right) side of $\eta'([0,\tau])$, the same boundary conditions as $h_{\alpha \beta}$ on $\partial D$, and a $2\pi \alpha$ discontinuity along $(-\infty,0)+z$, as described in Figure~\ref{fig::interior_path_bd2}.   Moreover, if $\beta = 0$, $D = \C$, and $h_\alpha=h_{\alpha 0}$ is a whole-plane GFF viewed as a distribution defined up to a global multiple of $2\pi(\chi+\alpha)$, then the corresponding path $\eta'$ is a whole-plane $\SLE_{\kappa'}(\rho)$ process with $\rho = 2-\kappa' - 2\pi\alpha/\lambda'$.  Finally, the random path $\eta'$ is almost surely determined by $h_{\alpha \beta}$ provided $\alpha \leq -\tfrac{3}{2}\chi$ and we know its values up to a global multiple of $2\pi(\chi+\alpha)$ if $D = \C$.
\end{theorem}

The value $\alpha = -\tfrac{3}{2}\chi$ is the critical threshold at or above which $\eta'$ almost surely fills its own outer boundary.  While we believe that $\eta'$ is still almost surely determined by $h_{\alpha \beta}$ for $\alpha \in (-\tfrac{3}{2}\chi,-\chi)$, establishing this falls out of our general framework so we will not treat this case here.  (See also Remark~\ref{rem::bigger_than_8_whole_plane} below.)  By making a $w \mapsto 1/w$ coordinate change, we can grow a path from~$\infty$ rather than from $0$.  For this to make sense, we need $\alpha > -\chi$ --- this makes the coupling compatible with the setup of Theorem~\ref{thm::alphabeta}.  In this case, $\eta'$ is a whole-plane $\SLE_{\kappa'}(\kappa'-6+2\pi \alpha / \lambda')$ process from $\infty$ to $0$ provided $\beta = 0$.  Moreover, the critical threshold at or below which the process fills its own outer boundary is $-\tfrac{\chi}{2}$.  That is, the process almost surely fills its own outer boundary if $\alpha \leq -\tfrac{\chi}{2}$ and does not if $\alpha > -\tfrac{\chi}{2}$.


\begin{figure}[h!]
\begin{center}
\begin{ficomment}
\subfigure[$\alpha=-\tfrac{1}{2} \chi$; $\pi$ range of angles]{\includegraphics[width=0.48\textwidth]{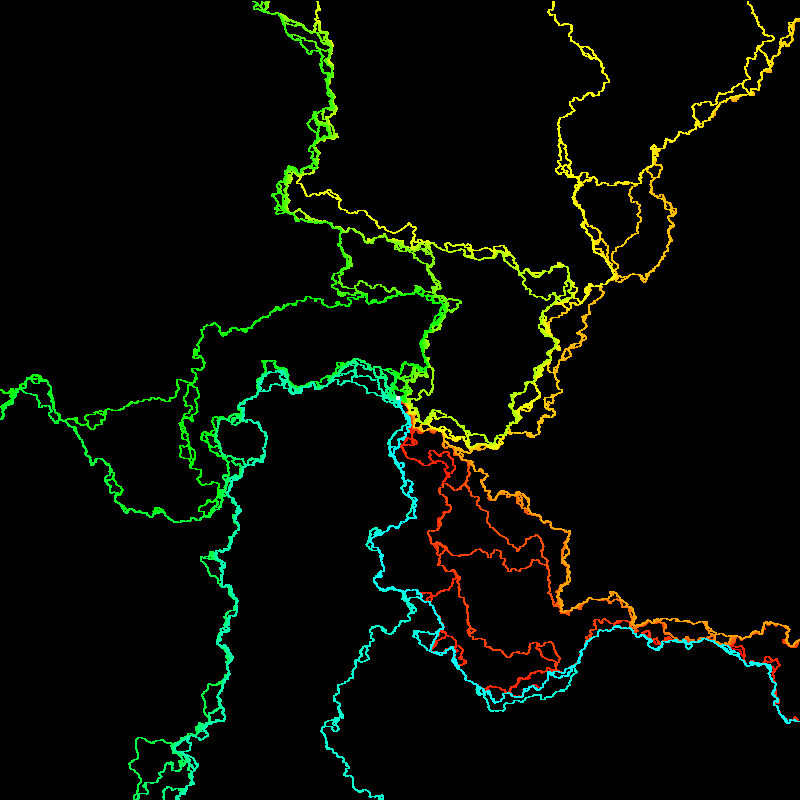}}
\hspace{0.02\textwidth}
\subfigure[$\alpha=0$; $2\pi$ range of angles]{\includegraphics[width=0.48\textwidth]{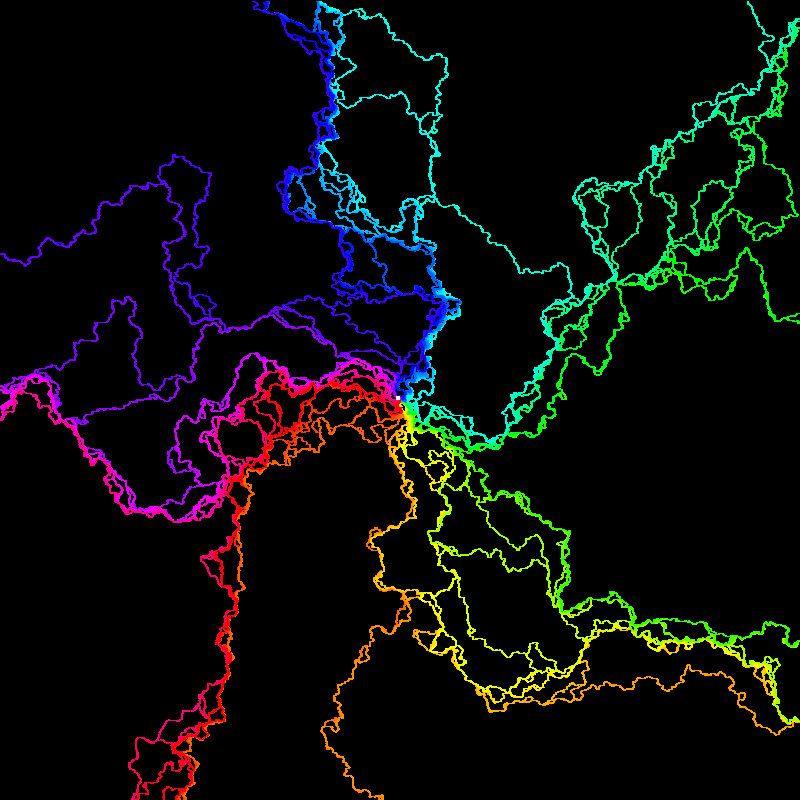}}
\subfigure[$\alpha=\chi$; $4\pi$ range of angles]{\includegraphics[width=0.48\textwidth]{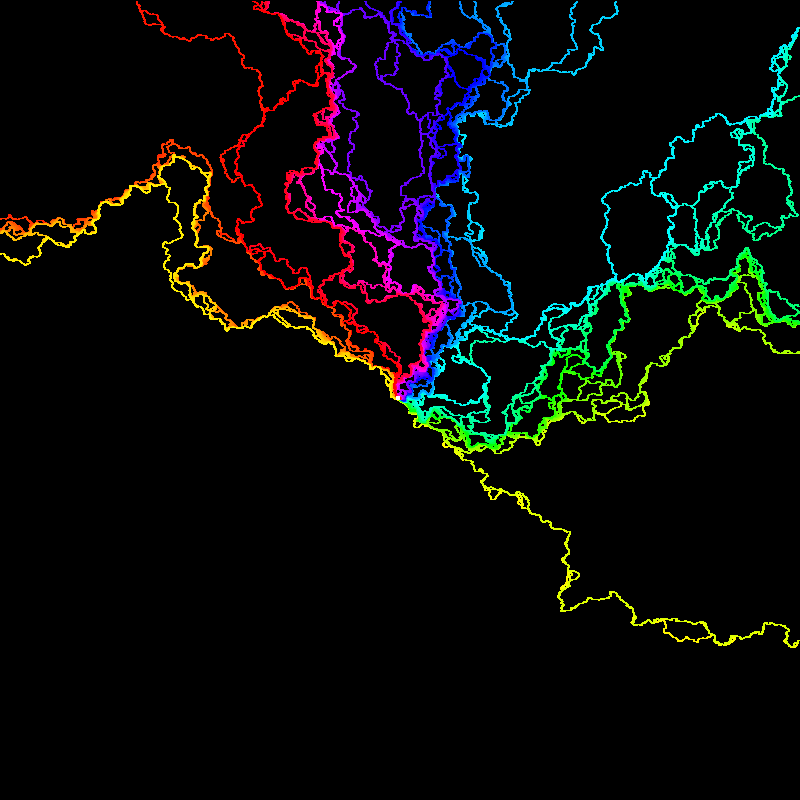}}
\hspace{0.02\textwidth}
\subfigure[$\alpha=2\chi$; $6\pi$ range of angles]{\includegraphics[width=0.48\textwidth]{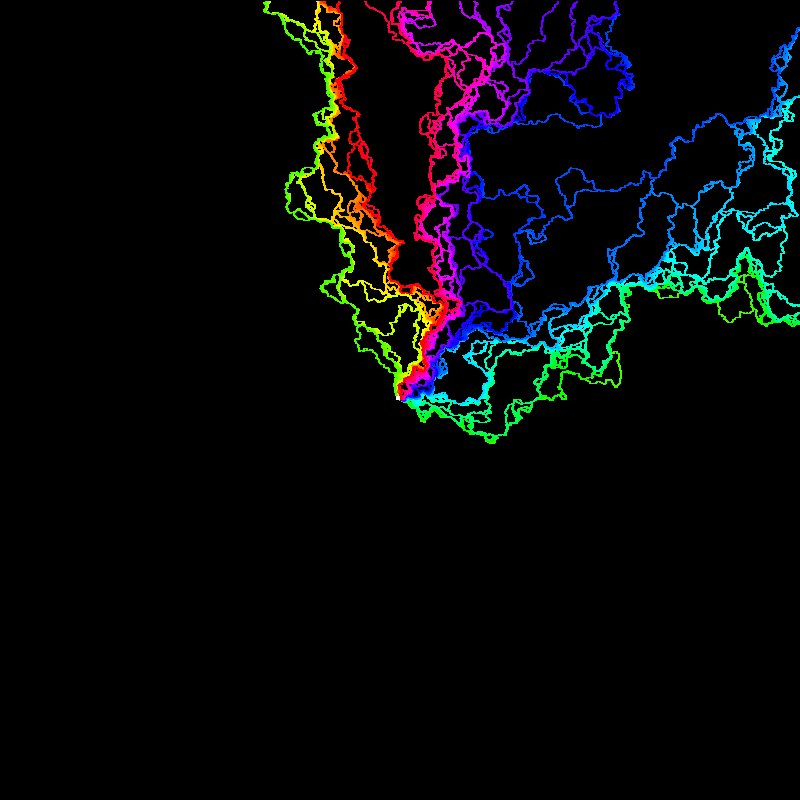}}
\end{ficomment}
\end{center}
\caption{\label{fig::varying_alpha} Numerical simulations of the set of points accessible by traveling along the flow lines of $h-\alpha\arg(\cdot)$ starting from the origin with equally spaced angles ranging from $0$ to $2\pi$ with varying values of $\alpha$; $\kappa=1$.  Different colors indicate paths with different angles.  For a given value of $\alpha > -\chi$, there is a $2\pi(1+\alpha/\chi)$ range of angles.  This is why the entire range of colors is not visible for $\alpha < 0$ and the paths shown represent only a fraction of the possible different possible directions for $\alpha > 0$.}
\end{figure}



\begin{figure}[h!]
\begin{center}
\begin{ficomment}
\includegraphics[width=0.95\textwidth]{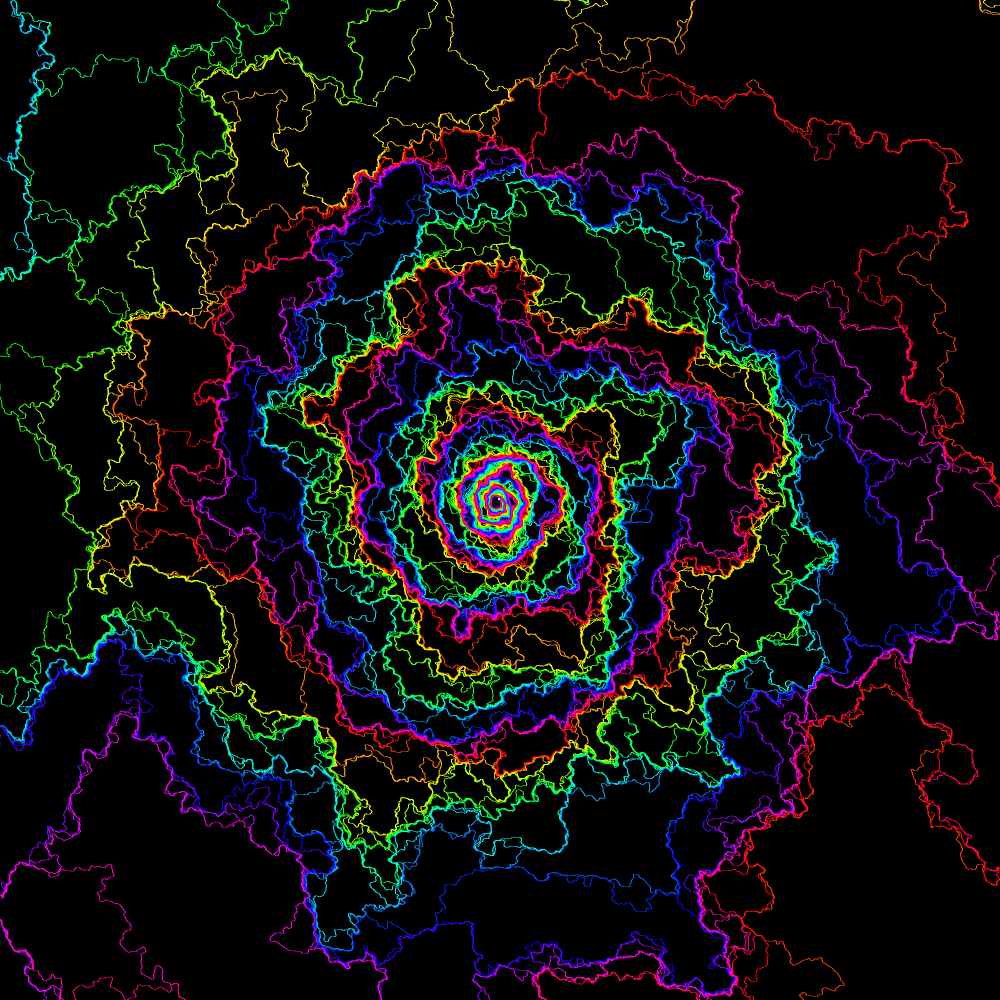}
\end{ficomment}
\end{center}
\caption{ \label{fig::spiral_fan}  Numerically generated flow lines, started at a common point, of $e^{i(h/\chi+\theta)}$ where $h$ is the sum of the projection of a GFF onto the space of functions piecewise linear on the triangles of a $800 \times 800$ grid and $\beta \log|\cdot|$; $\beta=-5$, $\kappa=4/3$ and $\chi = 2/\sqrt{\kappa} - \sqrt{\kappa}/2 = \sqrt{4/3}$.  Different colors indicate different values of $\theta \in [0,2\pi)$.  Paths tend to wind clockwise around the origin.  If we instead took $\beta > 0$, then the paths would wind counterclockwise around the origin.}
\end{figure}

\subsubsection{Flow line interaction}

\begin{figure}[ht!]
\begin{center}
\includegraphics[scale=0.80]{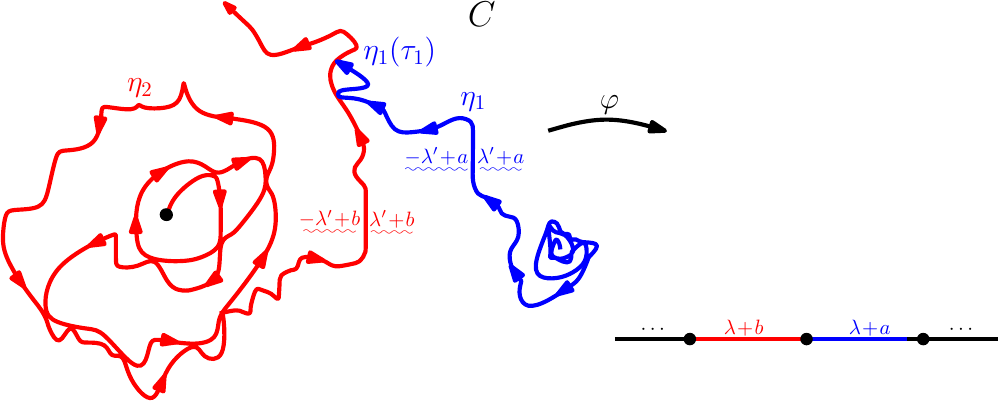}
\end{center}
\caption{\label{fig::angle_difference}  Let $h$ be a GFF on a domain $D \subseteq \C$; we view $h$ as a distribution defined up to a global multiple of $2\pi \chi$ if $D = \C$.  Suppose that $z_1,z_2 \in \ol{D}$ (in particular, we could have $z_i \in \partial D$) and $\theta_1,\theta_2 \in \R$ and, for $i=1,2$, we let $\eta_i$ be the flow line of $h$ starting at $z_i$ with angle $\theta_i$.  Fix $\eta_2$ and let $\tau_1$ be a stopping time for $\eta_1$ given $\eta_2$ and assume that we are working on the event that $\eta_1$ hits $\eta_2$ at time $\tau_1$ on its right side.  Let $C$ be the connected component of $\C \setminus (\eta_1([0,\tau_1]) \cup \eta_2)$ part of whose boundary is traced by the right side of $\eta_1|_{[\tau_1-\epsilon,\tau_1]}$ for some $\epsilon > 0$ and let $\varphi \colon C \to \h$ be a conformal map which takes $\eta_1(\tau_1)$ to $0$ and $\eta_1(\tau_1-\epsilon)$ to $1$.  Let $\wt{h} = h \circ \varphi^{-1} - \chi \arg (\varphi^{-1})'$ and let $\CD$ be the difference between the values of $h|_{\partial \h}$ immediately to the right and left of $0$ (the images of $\eta_1$ and $\eta_2$ near $\eta_1(\tau_1)$).  Although in some cases $h$ hence also $\wt{h}$ will be defined only up to an additive constant, $\CD$ is nevertheless a well-defined constant.  Then $\CD / \chi$ gives the {\bf angle difference} between $\eta_1$ and $\eta_2$ upon intersecting at $\eta_1(\tau_1)$ and $\CD$ gives the {\bf height difference}.  In the illustration, $\CD = a-b$.    In general, the height and angle difference (modulo $2\pi \chi$) can be easily read off using our notation for indicating boundary data of GFFs.  It is given (modulo $2\pi \chi$) by running backwards along both paths until finding a segment with the same orientation for both paths (typically, this will be north, as in the illustration) and then subtracting the height on the right side of $\eta_2$ from the height on the right side of $\eta_1$.  (In practice, we will in fact only indicate boundary data so that the height difference made be read off exactly --- and not just modulo $2\pi \chi$.)   The height and angle difference when $\eta_1$ hits $\eta_2$ on the left is defined analogously.  We can similar define the height and angle difference when a path hits a segment of the boundary.
}
\end{figure}

While proving Theorem~\ref{thm::existence}, Theorem~\ref{thm::uniqueness}, and Theorem~\ref{thm::alphabeta}, we also obtain information regarding the interaction between distinct paths with each other as well as with the boundary.  In \cite[Theorem~1.5]{MS_IMAG}, we described the interaction of boundary emanating flow lines in terms of their relative angle (this result is restated in Section~\ref{subsec::imaginary}).  When flow lines start from a point in the interior of a domain, their relative angle at a point where they intersect depends on how many times the two paths have wound around their initial point before reaching the point of intersection.  (This is an informal statement since paths started from interior points a.s.\ wind around their starting point an infinite number of times.)  Thus before we state our flow line interaction result in this setting, we need to describe what it means for two paths to intersect each other at a given height or angle; this is made precise in terms of conformal mapping in Figure~\ref{fig::angle_difference}.

\begin{theorem}
\label{thm::flow_line_interaction}
Assume that we have the same setup as described in the caption of Figure~\ref{fig::angle_difference}.  On the event that $\eta_1$ hits $\eta_2$ on its right side at the stopping time $\tau_1$ for $\eta_1$ given $\eta_2$ we have that the height difference $\CD$ between the paths upon intersecting is a constant in $(-\pi \chi, 2\lambda-\pi \chi)$.  Moreover,
\begin{enumerate}[(i)]
\item\label{it::cross} If $\CD \in (-\pi \chi,0)$, then $\eta_1$ crosses $\eta_2$ upon intersecting but does not subsequently cross back,
\item\label{it::merge} If $\CD = 0$, then $\eta_1$ merges with $\eta_2$ at time $\tau_1$ and does not subsequently separate from $\eta_2$, and
\item If $\CD \in (0,2\lambda-\pi \chi)$, then $\eta_1$ bounces off $\eta_2$ at time $\tau_1$ but does not cross $\eta_2$.
\end{enumerate}
The conditional law of $h$ given $\eta_1|_{[0,\tau_1]}$ and $\eta_2$ is that of a GFF on $\C \setminus (\eta_1([0,\tau_1]) \cup \eta_2)$ with flow line boundary conditions with angle $\theta_i$, for $i=1,2$, on $\eta_1([0,\tau_1])$ and $\eta_2$, respectively.  If, instead, $\eta_1$ hits $\eta_2$ on its left side then the same statement holds but with $-\CD$ in place of $\CD$ (in particular, the range of height differences in which paths can hit is $(\pi \chi-2\lambda,\pi \chi)$).  If the $\eta_i$ for $i=1,2$ are instead flow lines of $h_{\alpha \beta}$ then the same result holds except the conditional field given the paths has $\alpha$-flow line boundary conditions and a $2\pi \alpha$ jump on $(-\infty,0)$.  Finally, the same result applies if we replace $\eta_2$ with a segment of the domain boundary except that in the case that either~\eqref{it::cross} or~\eqref{it::merge} occurs, we have that $\eta_2$ terminates upon hitting the boundary.
\end{theorem}

By dividing $\CD$ by $\chi$, it is possible to rephrase Theorem~\ref{thm::flow_line_interaction} in terms of angle rather than height differences.  The angle $\theta_c = 2\lambda/\chi - \pi = 2\lambda'/\chi$ is called the {\bf critical angle}, and the set of allowed angle gaps described in the first part Theorem~\ref{thm::flow_line_interaction} is then $(-\pi, \theta_c)$, with the intervals $(-\pi, 0)$, $\{0 \}$ and $(0,\theta_c)$ corresponding to flow line pairs that respectively bounce off each other, merge with each other, and cross each other at their first intersection point.  We will discuss the critical angle further in Section~\ref{subsec::critical_angle}.

We emphasize that Theorem~\ref{thm::flow_line_interaction} describes the interaction of both flow lines starting from interior points and flow lines starting from the boundary, or even one of each type.  It also describes the types of boundary data that a flow line can hit.  We also emphasize that it is not important in Theorem~\ref{thm::flow_line_interaction} to condition on the entire realization of $\eta_2$ before drawing $\eta_1$.  Indeed, a similar result holds if first draw an initial segment of $\eta_2$, and then draw $\eta_1$ until it hits that segment.  This also generalizes further to the setting in which we have many paths.  One example of a statement of this form is the following.

\begin{theorem}
\label{thm::commutation}
Suppose that $h$ is a GFF on a domain $D \subseteq \C$, where $h$ is defined up to a global multiple of $2\pi \chi$ if $D = \C$.  Fix points $z_1,\ldots,z_n \in \ol{D}$ and angles $\theta_1,\ldots,\theta_n \in \R$.  For each $1 \leq i \leq n$, let $\eta_i$ be the flow line of $h$ with angle $\theta_i$ starting from $z_i$.  Fix $N \in \N$.  For each $1 \leq j \leq N$, suppose that $\xi_j \in \{1,\ldots,N\}$ are non-random and that $\tau_j$ is a stopping time for the filtration
\[ \CF_t^j = \sigma( \eta_{\xi_j}(s) : s \leq t,\ \eta_{\xi_1}|_{[0,\tau_1]},\ldots,\eta_{\xi_{j-1}}|_{[0,\tau_{j-1}]}).\]
Then the conditional law of $h$ given $\CF = \sigma(\eta_{\xi_1}|_{[0,\tau_1]},\ldots,\eta_{\xi_N}|_{[0,\tau_N]})$ is that of a GFF on $D_N = D \setminus \cup_{j=1}^N \eta_{\xi_j}([0,\tau_j])$ with flow line boundary conditions with angle $\theta_{\xi_j}$ on each of $\eta_{\xi_j}([0,\tau_j])$ for $1 \leq j \leq N$ and the same boundary conditions as $h$ on $\partial D$.  For each $1 \leq i \leq n$, let $t_i = \max_{ j : \xi_j = i} \tau_j$.  On the event that $\eta_i(t_i)$ is disjoint from $\partial D_N \setminus \eta_i([0,t_i))$, the continuation of $\eta_i$ stopped upon hitting $\partial D_N \setminus \eta_i([0,t_i))$ is almost surely equal to the flow line of the \emph{conditional} GFF $h$ given $\CF$ starting at $\eta_i(t_i)$ with angle $\theta_i$ stopped upon hitting $\partial D_N \setminus \eta_i([0,t_i])$.
\end{theorem}

In the context of the final part of Theorem~\ref{thm::commutation}, the manner in which $\eta_i$ interacts with the other paths or domain boundary after hitting $\partial D_N \setminus \eta_i([0,t_i])$ is as described in Theorem~\ref{thm::flow_line_interaction}.  Theorem~\ref{thm::commutation} (combined with Theorem~\ref{thm::flow_line_interaction}) says that it is possible to draw the flow lines $\eta_1,\ldots,\eta_n$ of $h$ starting from $z_1,\ldots,z_n$ in any order and the resulting path configuration is almost surely the same.  After drawing each of the paths as described in the statement, the conditional law of the continuation of any of the paths can be computed using conformal mapping and~\eqref{eqn::ac_eq_rel}.  One version of this that will be important for us is stated as Theorem~\ref{thm::conditional_law} below.

In \cite[Theorem~1.5]{MS_IMAG}, we showed that boundary emanating GFF flow lines can cross each other at most once.  The following result extends this to the setting of flow lines starting from interior points.  If we subtract a multiple $\alpha$ of the argument, then depending on its value, Theorem~\ref{thm::merge_cross} will imply that the GFF flow lines can cross each other and themselves more than once, but at most a finite, non-random constant number of times; the constant depends only on $\alpha$ and $\chi$. Moreover, a flow line starting from the location of the conical singularity cannot cross itself.  The maximal number of crossings does not change if we subtract a multiple of the $\log$.

\begin{theorem}
\label{thm::merge_cross}
Suppose that $D \subseteq \C$ is a domain, $z_1,z_2 \in \ol{D}$, and $\theta_1, \theta_2 \in \R$.  Let $h$ be a GFF on $D$, which is defined up to a global multiple of $2\pi \chi$ if $D = \C$.  For $i=1,2$, let $\eta_i$ be the flow line of $h$ with angle $\theta_i$, i.e.\ the flow line of $h + \theta_i \chi$, starting from $z_i$.  Then $\eta_1$ and $\eta_2$ cross at most once (but may bounce off each other after crossing).  If $D = \C$ and $\theta_1 = \theta_2$, then $\eta_1$ and $\eta_2$ almost surely merge.  More generally, suppose that $\alpha > -\chi$, $\beta \in \R$, $h$ is a GFF on $D$, and $h_{\alpha \beta} = h - \alpha \arg(\cdot) - \beta \log| \cdot|$, viewed as a distribution defined up to a global multiple of $2\pi (\chi + \alpha)$ if $D = \C$, and that $\eta_1,\eta_2$ are flow lines of $h_{\alpha \beta}$ starting from $z_1,z_2$ with $z_1=0$.  There exists a constant $C(\alpha) < \infty$ such that $\eta_1$ and $\eta_2$ cross each other at most $C(\alpha)$ times and $\eta_2$ can cross itself at most $C(\alpha)$ times ($\eta_1$ does not cross itself).
\end{theorem}

Flow lines emanating from an interior point are also able to intersect themselves even in the case that we do not subtract a multiple of the argument.  In~\eqref{eqn::flow_line_maximal_self_intersect},~\eqref{eqn::whole_plane_sle_maximal_self_intersect} of Proposition~\ref{prop::number_of_self_intersections} we compute the maximal number of times that such a path can hit any given point.

Theorem~\ref{thm::merge_cross} implies that if we pick a countable dense subset $(z_n)$ of $D$ then the collection of flow lines starting at these points with the same angle has the property that a.s.\ each pair of flow lines eventually merges (if $D = \C$) and no two flow lines ever cross each other.  We can view this collection of flow lines as a type of planar space filling tree (if $D = \C$) or a forest (if $D \not = \C$).  See Figure~\ref{fig::tree_dual} for simulations which show parts of the trees associated with flow lines of angle $\tfrac{\pi}{2}$ and $-\tfrac{\pi}{2}$.  By Theorem~\ref{thm::uniqueness}, we know that that this forest or tree is almost surely determined by the GFF.  Theorem~\ref{thm::field_determined_by_tree} states that the reverse is also true: the underlying GFF is a deterministic function of the realization of its flow lines started at a countable dense set.  When $\kappa=2$, this can be thought of as a continuum analog of the Temperley bijection that takes spanning trees to dimer configurations which in turn come with height functions (a similar observation was made in \cite{DUB_PART}) and this construction generalizes this to other $\kappa$ values. We note that the basic idea of using a continuum analog of Wilson's algorithm \cite{wilson1996lerw} to construct a planar tree of radial SLE curves (to be a scaling limit of the uniform spanning tree) appeared in Schramm's original paper on SLE \cite{S0}.

\begin{theorem}
\label{thm::field_determined_by_tree}
Suppose that $h$ is a GFF on a domain $D \subseteq \C$, viewed as a distribution defined up to a global multiple of $2\pi \chi$ if $D = \C$.  Fix $\theta \in [0,2\pi)$.  Suppose that $(z_n)$ is any countable dense set and, for each $n$, let $\eta_n$ be the flow line of $h$ starting at $z_n$ with angle $\theta$.
If $D = \C$, then $(\eta_n)$ almost surely forms a ``planar tree'' in the sense that a.s.\ each pair of paths merges eventually, and no two paths ever cross each other.  In both the case that $D = \C$ and $D \not = \C$, the collection $(\eta_n)$ almost surely determines $h$ and $h$ almost surely determines $(\eta_n)$.
\end{theorem}

 In our next result, we give the conditional law of one flow line given another if they are started at the same point.  In Section~\ref{sec::timereversal}, we will show that this resampling property essentially characterizes the joint law of the paths (which extends a similar result for boundary emanating flow lines established in \cite{MS_IMAG2}).

\begin{theorem}
\label{thm::conditional_law}
Suppose that $h$ is a whole-plane GFF, $\alpha > -\chi$, $\beta \in \R$, and $h_{\alpha \beta} = h - \alpha \arg(\cdot) - \beta \log|\cdot|$, viewed as a distribution defined up to a global multiple of $2\pi(\chi+\alpha)$.  Fix angles $\theta_1,\theta_2 \in [0,2\pi(1+\alpha/\chi))$ with $\theta_1 < \theta_2$ and, for $i \in \{1,2 \}$, let $\eta_i$ be the flow line of $h_{\alpha \beta}$ starting from $0$ with angle $\theta_i$.  Then the conditional law of $\eta_2$ given $\eta_1$ is that of a chordal $\SLE_\kappa(\rho^L;\rho^R)$ process with
\[ \rho^L = \frac{(2\pi +\theta_1 - \theta_2) \chi + 2\pi \alpha}{\lambda} - 2\quad\text{and}\quad\rho^R = \frac{(\theta_2-\theta_1) \chi}{\lambda} - 2.\]
independently in each of the connected components of $\C \setminus \eta_1$ starting from and with force points located immediately to the left and right of the first point on the boundary of such a component visited by $\eta_1$ and targeted at last.
\end{theorem}

A similar result also holds when one considers more than two paths starting from the same point; see Proposition~\ref{prop::conditional_law} as well as Figure~\ref{fig::conditional_law}.  A version of Theorem~\ref{thm::conditional_law} also holds in the case that $h$ is a GFF on a domain $D \subseteq \C$ with harmonically non-trivial boundary.  In this case, the conditional law of $\eta_2$ given $\eta_1$ is an $\SLE_\kappa(\rho^L;\rho^R)$ process with the same weights $\rho^L,\rho^R$ independently in each of the connected components of $D \setminus \eta_1$ whose boundary consists entirely of arcs of $\eta_1$.  In the connected components whose boundary consists of part of $\partial D$, the conditional law of $\eta_2$ is a chordal $\SLE_\kappa(\ul{\rho})$ process where the weights $\ul{\rho}$ depend on the boundary data of $h$ on $\partial D$ (but are nevertheless straightforward to read off from the boundary data).

A whole-plane $\SLE_\kappa(\rho)$ process $\eta$ for $\kappa > 0$ and $\rho > -2$ is almost surely unbounded by its construction (its capacity is unbounded), however it is not immediate from the definition of $\eta$ that it is transient: that is, $\lim_{t \to \infty} \eta(t)~=~\infty$ almost surely.  This was first proved by Lawler for ordinary whole-plane $\SLE_\kappa$ processes, i.e.\ $\rho=0$, in \cite{LAW_ENDPOINT}.  Using Theorem~\ref{thm::flow_line_interaction}, we are able to extend Lawler's result to the entire class of whole-plane $\SLE_\kappa(\rho)$ processes.

\begin{theorem}
\label{thm::transience}
Suppose that $\eta$ is a whole-plane $\SLE_\kappa(\rho)$ process for $\kappa > 0$ and $\rho > -2$.  Then $\lim_{t \to \infty} \eta(t)=\infty$ almost surely.  Likewise, if $\eta$ is a radial $\SLE_\kappa(\rho)$ process for $\kappa > 0$ and $\rho > -2$ in $\D$ targeted at $0$ then, almost surely, $\lim_{t \to \infty} \eta(t)=0$.
\end{theorem}

\subsubsection{Branching and space-filling $\SLE$ curves}
\label{subsubsec::branching_space_filling}

As mentioned in Section~\ref{ss::overview} (and illustrated in Figures~\ref{fig::space_filling},~\ref{fig::space_filling2},~\ref{fig::space_filling3}, and~\ref{fig::bigger_than_8_reversible})
there is a natural space-filling path that traces the flow line tree in a natural order (so that a generic point $y \in D$ is hit before a generic point $z \in D$ when the flow line with angle $\tfrac{\pi}{2}$ from $y$ merges into the right side of the flow from $z$ with angle $\tfrac{\pi}{2}$).  The full details of this construction (including rules for dealing with various boundary conditions, and the possibility of flow lines that merge into the boundary before hitting each other) appear in Section~\ref{sec::duality_space_filling}.  The result is a space-filling path that traces through a ``tree'' of flow lines, each of which is a form of $\SLE_\kappa$.

We will show that there is another way to construct the space-filling path and interpret it as a variant of $\SLE_{\kappa'}$, where $\kappa' = 16/\kappa > 4$.  Indeed, this is true even when $\kappa' \in (4,8)$, in which case ordinary $\SLE_{\kappa'}$ is not space-filling.



\begin{figure}[h!]
\begin{center}
\begin{ficomment}
\includegraphics[width=0.95\textwidth]{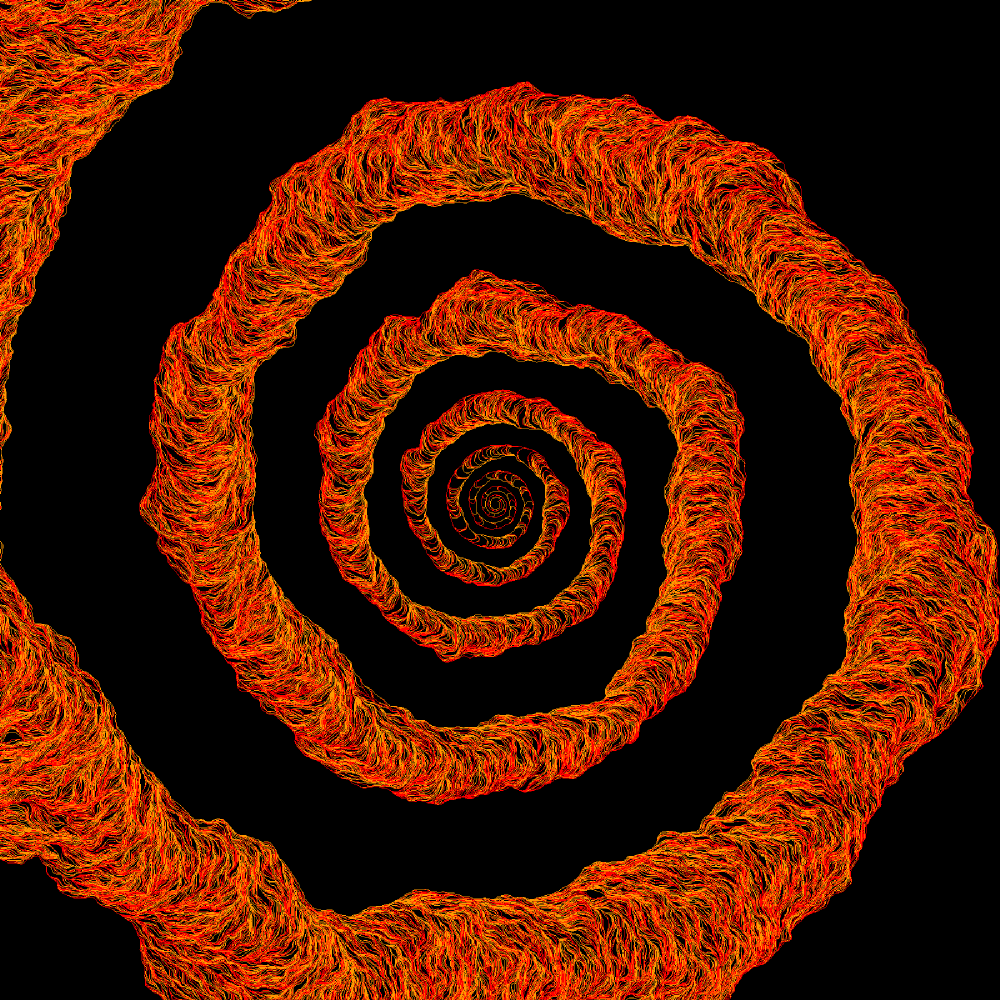}
\end{ficomment}
\end{center}
\caption{ \label{fig::spiral_lightcone}  Simulation of the light cone emanating from the origin of a whole-plane GFF plus $\beta \log |\cdot|$; $\beta = -5$.  The resulting counterflow line $\eta'$ from $\infty$ is a variant of whole-plane $\SLE_{64}$ targeted at $0$.  The $\log$ singularity causes the path to spiral around the origin.  Since $\beta < 0$, the angle-varying flow lines which generate the range of $\eta'$ spiral around $0$ in the clockwise direction, which corresponds to $\eta'$ spiraling around $0$ in the counterclockwise direction.  Changing the sign of $\beta$ would lead to the paths spiraling in the opposite direction.}
\end{figure}

\begin{figure}[h!]
\begin{center}
\begin{ficomment}
\subfigure[$\SLE_8$]{\includegraphics[width=0.48\textwidth]{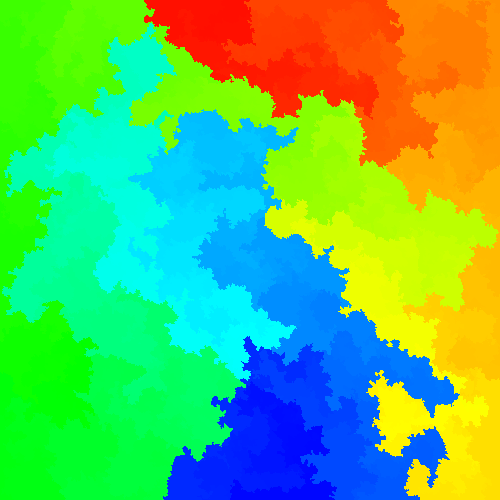}}
\hspace{0.02\textwidth}
\subfigure[$\SLE_{16}$]{\includegraphics[width=0.48\textwidth]{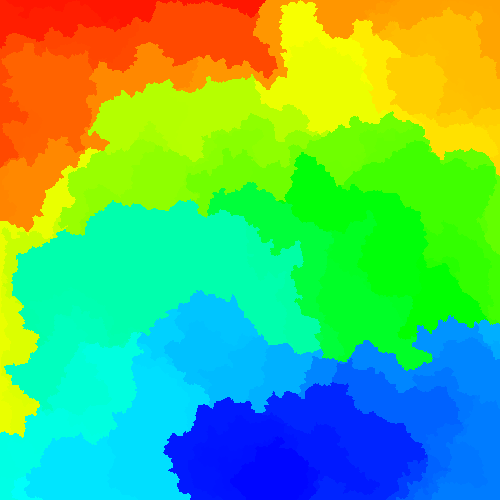}}

\subfigure{\includegraphics[scale=.9]{figures/hue_scale_time.pdf}}
\end{ficomment}
\end{center}
\caption{\label{fig::space_filling3} Simulations of $\SLE_{\kappa'}$ processes in $[-1,1]^2$ from $i$ to $-i$ for the indicated values of $\kappa'$.  These were generated from a discrete approximation of the GFF using the method described in Figure~\ref{fig::space_filling}.  The path on the left appears to be reversible while the path on the right appears not to be due to the asymmetry in its initial and terminal points.  This asymmetry is more apparent for larger values of $\kappa'$; see Figure~\ref{fig::bigger_than_8_reversible} for simulations with $\kappa'=128$.}
\end{figure}

Before we explain this, let us recall the principle of $\SLE$ {\em duality} (sometimes called {\em Duplantier duality}) which states that the outer boundary of an $\SLE_{\kappa'}$ process is a certain form of $\SLE_\kappa$.  This was first proved in various forms by Zhan \cite{ZHAN_DUALITY_1,ZHAN_DUALITY_2} and Dub\'edat \cite{DUB_DUAL}.  This duality naturally arises in the context of the GFF/$\SLE$ coupling and is explored in \cite{DUB_PART} and \cite{MS_IMAG}.  The simplest statement of this type is the following.  Suppose that $D \subseteq \C$ is a simply-connected domain with harmonically non-trivial boundary and fix $x,y \in \partial D$ distinct.  Let $\eta'$ be an $\SLE_{\kappa'}$ process coupled with a GFF $h$ on $D$ as a counterflow line from $y$ to $x$ (as defined just after the statement of \cite[Theorem~1.1]{MS_IMAG}).  Then the left (resp.\ right) side of the outer boundary of $\eta'$ is equal to the flow line of $h$ starting at $x$ with angle $\tfrac{\pi}{2}$ (resp.\ $-\tfrac{\pi}{2}$).  In the case that $\eta'$ is boundary filling, the flow lines starting from $x$ with these angles are taken to be equal to the segments of the domain boundary which connect $y$ to $x$ in the counterclockwise and clockwise directions.  More generally, the left (resp.\ right) side of the outer boundary of $\eta'$ stopped upon hitting any given boundary point $z$ is equal to the flow line of $h$ starting from $z$ with angle $\tfrac{\pi}{2}$ (resp.\ $-\tfrac{\pi}{2}$).  In \cite{MS_IMAG}, it is shown that it is possible to realize the entire trajectory of $\eta'$ stopped upon hitting $z$ as the closure of the union of a countable collection of angle varying flow lines starting from $z$ whose angle is restricted to be in $[-\tfrac{\pi}{2},\tfrac{\pi}{2}]$ and is allowed to change direction a finite number of times.  This path decomposition is the so-called {\bf $\SLE$ light cone}.

Our next theorem extends these results to describe the outer boundary and range of $\eta'$ when it is targeted at a given \emph{interior point} $z$ in terms of flow lines of $h$ starting from $z$.  A more explicit statement of the theorem hypotheses appears in Section~\ref{sec::duality_space_filling}, where the result is stated as Theorem~\ref{thm::lightcone}.

\begin{theorem}
\label{thm::duality_and_light_cones}
Suppose that $h$ is a GFF on a simply-connected domain $D \subseteq \C$ which is homeomorphic to $\D$.  Assume that the boundary data is such that it is possible to draw a counterflow line $\eta'$ from a fixed $y \in \partial D$ to a fixed point $z \in D$.  (See Section~\ref{sec::duality_space_filling}, Theorem~\ref{thm::lightcone}, for precise conditions.)  If we lift $\eta'$ to the universal cover of $D \setminus \{z\}$, then its left (resp.\ right) boundary is a.s.\ equivalent (when projected back to $D$) to the flow line $\eta^L$ (resp.\ $\eta^R$) of $h$ starting from $z$ with angle $\tfrac{\pi}{2}$ (resp.\ $-\tfrac{\pi}{2}$).  More generally, the range of $\eta'$ stopped upon hitting $z$ is almost surely equal to the closure of the union of the countable collection of all angle varying flow lines of $h$ starting at $z$ and targeted at $y$ which change angles at most a finite number of rational times and with rational angles contained in $[-\tfrac{\pi}{2},\tfrac{\pi}{2}]$.
\end{theorem}

As in the boundary emanating setting, we can describe the conditional law of a counterflow line given its outer boundary:

\begin{theorem}
\label{thm::duality_conditional}
Suppose that we are in the setting of Theorem~\ref{thm::duality_and_light_cones}.
If we are given $\eta^L$ and $\eta^R$, then the conditional law of $\eta'$ restricted to the interior connected components of $D \setminus (\eta^L \cup \eta^R)$ that it passes through (i.e., those components whose boundaries include the right side of a directed $\eta^L$ segment, the left side of a directed segment $\eta^R$ segment, and {\em no} arc of $\partial D$) is given by independent chordal $\SLE_{\kappa'}(\tfrac{\kappa'}{2}-4;\tfrac{\kappa'}{2}-4)$ processes (one process in each component, starting at the terminal point of the component's directed $\eta^L$ and $\eta^R$ boundary segments, ending at the initial point of these directed segments).
\end{theorem}

Various statements similar to Theorem~\ref{thm::duality_and_light_cones} and Theorem~\ref{thm::duality_conditional} hold if we start the counterflow line $\eta'$ from an interior point as in the setting of Theorem~\ref{thm::alphabeta_counterflow}.  One statement of this form which will be important for us is the following.

\begin{theorem}
\label{thm::whole_plane_duality}
Suppose that $h_{\alpha \beta} = h - \alpha \arg(\cdot) - \beta \log|\cdot|$ where $\alpha \geq -\tfrac{\chi}{2}$, $\beta \in \R$, $h$ is a whole-plane GFF, and $h_{\alpha \beta}$ is viewed as a distribution defined up to a global multiple of $2\pi(\chi + \alpha)$.  Let $\eta'$ be the counterflow line of $h_{\alpha \beta}$ starting from $\infty$ and targeted at $0$.  Then the left (resp.\ right) boundary of $\eta'$ is given by the flow line $\eta^L$ (resp.\ $\eta^R$) starting from $0$ with angle $\tfrac{\pi}{2}$ (resp.\ $-\tfrac{\pi}{2}$).  The conditional law of $\eta'$ given $\eta^L,\eta^R$ is independently that of a chordal $\SLE_{\kappa'}(\tfrac{\kappa'}{2}-4;\tfrac{\kappa'}{2}-4)$ process in each of the components of $\C \setminus (\eta^L \cup \eta^R)$ which are visited by $\eta'$.  The range of $\eta'$ is almost surely equal to the closure of the union of the set of points accessible by traveling along angle varying flow lines starting from $0$ which change direction a finite number of times and with rational angles contained in $[-\tfrac{\pi}{2},\tfrac{\pi}{2}]$.
\end{theorem}

Recall from after the statement of Theorem~\ref{thm::alphabeta_counterflow} that $-\tfrac{\chi}{2}$ is the critical value of $\alpha$ at or below which $\eta'$ fills its own outer boundary.  At the critical value $\alpha=-\tfrac{\chi}{2}$, the left and right boundaries of $\eta'$ are the same.

We will now explain briefly how one constructs the so-called space-filling $\SLE_{\kappa'}$ or $\SLE_{\kappa'}(\rho)$ processes as {\em extensions} of ordinary $\SLE_{\kappa'}$ and $\SLE_{\kappa'}(\rho)$ processes.  (A more detailed explanation appears in Section~\ref{sec::duality_space_filling}.)  First of all, if $\kappa' \geq 8$ and $\rho$ is such that ordinary $\SLE_{\kappa'}(\rho)$ is space-filling (i.e., $\rho \in (-2,\tfrac{\kappa'}{2}-4]$), then the space-filling  $\SLE_{\kappa'}(\rho)$ is the same as ordinary $\SLE_{\kappa'}(\rho)$. More interestingly, we will also define space-filling $\SLE_{\kappa'}(\rho)$ when $\rho$ and $\kappa'$ are in the range for which an ordinary $\SLE_{\kappa'}(\rho)$ path $\eta'$ is {\em not} space-filling and the complement $D \setminus \eta'$ a.s.\ consists of a countable set of components $C_i$, each of which is swallowed by $\eta'$ at a finite time $t_i$.

The space-filling extension of $\eta'$ hits the points in the range of $\eta'$ in the same order that $\eta'$ does; however, it is ``extended'' by splicing into $\eta'$, at each time $t_i$, a certain $\overline {C_i}$-filling loop that starts and ends at $\eta'(t_i)$.

In other words, the difference between the ordinary and the space-filling path is that while the former ``swallows'' entire regions $C_i$ at once, the latter fills up $C_i$ gradually (immediately after it is swallowed) with a continuous loop, starting and ending at $\eta'(t_i)$.  We parameterize the extended path so that the time represents the area it has traversed thus far (and the path traverses a unit of area in a unit of time).  Then $\eta'$ can be obtained from the extended path by restricting the extended path to the set of times when its tip lies on the outer boundary of the region traversed thus far.  In some sense, the difference between the two paths is that $\eta'$ is parameterized by capacity viewed from the target point (which means that entire regions $C_i$ are absorbed in zero time and never subsequently revisited) and the space-filling extension is parameterized by area (which means that these regions are filled in gradually).

It remains to explain how the continuous $\overline{C_i}$-filling loops are constructed.  More details about this will appear in Section~\ref{sec::duality_space_filling}, but we can give some explanation here.  Consider a countable dense set $(z_n)$ of points in $D$.  For each $n$ we can define a counterflow line $\eta_n'$, starting at the same position as $\eta'$ but targeted at $z_n$.  For $m \not = n$, the paths $\eta_m'$ and $\eta_n'$ agree (up to monotone reparameterization) until the first time they separate $z_m$ from $z_n$ (i.e., the first time at which $z_m$ and $z_n$ lie in different components of the complement of the path traversed thus far).  The space-filling curve will turn out to be an extension of {\em each} $\eta_n'$ curve in the sense that $\eta_n'$ is obtained by restricting the space-filling curve to the set of times at which the tip is harmonically exposed to $z_n$.

To construct the curve, suppose that $D \subseteq \C$ is a simply-connected domain with harmonically non-trivial boundary and fix $x,y \in \partial D$ distinct.  Assume that the boundary data for a GFF $h$ on $D$ has been chosen so that its counterflow line $\eta'$ from $y$ to $x$ is an $\SLE_{\kappa'}(\rho_1;\rho_2)$ process with $\rho_1,\rho_2 \in (-2,\tfrac{\kappa'}{2}-2)$; this is the range of $\rho$ values in which $\eta'$ almost surely hits both sides of $\partial D$.  Explicitly, in the case that $D = \h$, $y=0$, and $x=\infty$, this corresponds to taking $h$ to be a GFF whose boundary data is given by $\lambda'(1+\rho_1)$ (resp.\ $-\lambda'(1+\rho_2)$) on $\R_-$ (resp.\ $\R_+$).  Suppose that $(z_n)$ is any countable, dense set of points in $D$.  For each $n$, let $\eta_n^L$ (resp.\ $\eta_n^R$) be the flow line of $h$ starting from $z_n$ with angle $\tfrac{\pi}{2}$ (resp.\ $-\tfrac{\pi}{2}$). As explained in Theorem~\ref{thm::duality_and_light_cones}, these paths can be interpreted as the left and right boundaries of $\eta_n'$.  For some boundary data choices, it is possible for these paths to hit the same boundary point many times; they wind around (adding a multiple of $2\pi \chi$ to their heights) between subsequent hits.  As explained in Section~\ref{sec::duality_space_filling}, we will assume that $\eta_n^L$ and $\eta_n^R$ are stopped at the first time they hit the appropriate left or right arcs of $\partial D \setminus \{x, y \}$ at the ``correct'' angles (i.e., with heights described by the multiple of $2\pi \chi$ that corresponds to the outer boundary of $\eta'$ itself).

\begin{figure}[htb!]
\begin{center}
\includegraphics[scale=0.85]{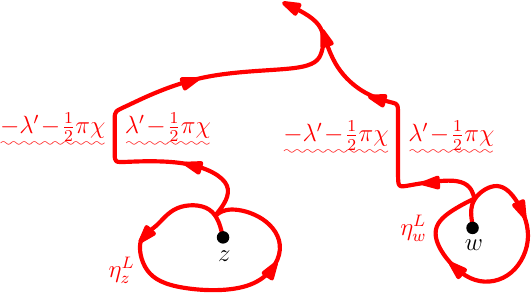}
\end{center}
\caption{\label{fig::space_filling_ordering}
Suppose that $h$ is a GFF on either a Jordan domain $D$ or $D = \C$; if $D = \C$ then $h$ is viewed as a distribution defined up to a global multiple of $2\pi \chi$.  Fix $z,w \in D$ distinct and let $\eta_z^L,\eta_w^L$ be the flow lines of $h$ starting from $z,w$, respectively, with angle $\tfrac{\pi}{2}$.  We order $z$ and $w$ by declaring that $w$ comes before $z$ if $\eta_w^L$ merges into $\eta_z^L$ on its right side, as shown.  Equivalently, if we let $\eta_z^R,\eta_w^R$ be the flow lines of $h$ starting from $z,w$, respectively, with angle $-\tfrac{\pi}{2}$, we declare that $w$ comes before $z$ if $\eta_w^R$ merges into the left side of $\eta_z^R$.  This is the ordering on points used to generate space-filling $\SLE_{\kappa'}$.
}
\end{figure}

For each $n$, $\eta_n^L \cup \eta_n^R$ divides $D$ into two parts:
\begin{enumerate}[(i)]
\item those points in complementary components whose boundary consists of a segment of either the right side of $\eta_n^L$ or the left side of $\eta_n^R$ (and possibly also an arc of $\partial D$) and
\item those points in complementary components whose boundary consists of a segment of either the left side of $\eta_n^L$ or the right side of $\eta_n^R$ (and possibly also an arc of $\partial D$).
\end{enumerate}
This in turn induces a total ordering on the $(z_n)$ where we say that $z_m$ comes before $z_n$ for $m \neq n$ if $z_m$ is on the side described by (i).  A {\bf space-filling $\SLE_{\kappa'}(\rho_1;\rho_2)$} is an almost surely non-self-crossing space-filling curve from $y$ to $x$ that visits the points $(z_n)$ according to this order and, when parameterized by $\log$ conformal radius (resp.\ capacity), as seen from any point $z \in D$ (resp.\ $z \in \partial D$) is almost surely equal to the counterflow line of $h$ starting from $y$ and targeted at $z$.

\begin{theorem}
\label{thm::space_filling_sle_existence}
Suppose that $h$ is a GFF on a Jordan domain $D \subseteq \C$ and $x,y \in \partial D$ are distinct.  Fix $\kappa' > 4$.  If the boundary data for $h$ is as described in the preceding paragraph, then space-filling $\SLE_{\kappa'}(\rho_1;\rho_2)$ from $y$ to $x$ exists and is well-defined: its law does not depend on the choice of countable dense set $(z_n)$.  Moreover, $h$ almost surely determines the path and the path almost surely determines $h$.
\end{theorem}

The final statement of Theorem~\ref{thm::space_filling_sle_existence} is really a corollary of Theorem~\ref{thm::field_determined_by_tree} because $\eta'$ determines the tree of flow lines with angle $\tfrac{\pi}{2}$ and with angle $-\tfrac{\pi}{2}$ starting from the points $(z_n)$ and vice-versa.  The ordering of the points $(z_n)$ in the definition of space-filling $\SLE_{\kappa'}$ can be equivalently constructed as follows (see Figure~\ref{fig::space_filling_ordering}).  From points $z_m$ and $z_n$ for $m \neq n$, we send the flow lines of $h$ with the same angle $\tfrac{\pi}{2}$, say $\eta_m^L,\eta_n^L$.  If $\eta_m^L$ hits $\eta_n^L$ on its right side, then $z_m$ comes before $z_n$ and vice-versa otherwise.  The ordering can similarly be constructed with the angle $\tfrac{\pi}{2}$ replaced by $-\tfrac{\pi}{2}$ and the roles of left and right swapped.  The time change used to get an ordinary $\SLE_{\kappa'}$ process targeted at a given point, parameterized either by $\log$ conformal radius or capacity depending on whether the target point is in the interior or boundary, from a space-filling $\SLE_{\kappa'}$, parameterized by area, is a monotone function which changes values at a set of times which almost surely has Lebesgue measure zero.  When $\kappa=2$ so that $\kappa'=8$, the final statement of Theorem~\ref{thm::space_filling_sle_existence} can be thought of as describing the limit of the coupling of the uniform spanning tree with its corresponding Peano curve \cite{LSW04}.

The conformal loop ensembles $\CLE_\kappa$ for $\kappa \in (8/3,8)$ are the loop version of $\SLE$ \cite{SHE_CLE,SHE_WER_CLE}.  As a consequence of Theorem~\ref{thm::space_filling_sle_existence}, we obtain the local finiteness of the $\CLE_{\kappa'}$ processes for $\kappa' \in (4,8)$.  (The corresponding result for $\kappa \in (8/3,4]$ is proved in \cite{SHE_WER_CLE} using the relationship between $\CLE$s and loop-soups.)

\begin{theorem}
\label{thm::cle_locally_finite}
Fix $\kappa' \in (4,8)$ and let $D$ be a (bounded) Jordan domain.  Let $\Gamma$ be a $\CLE_{\kappa'}$ process in $D$.  Then $\Gamma$ is almost surely locally finite.  That is, for every $\epsilon > 0$, the number of loops of $\Gamma$ which have diameter at least $\epsilon$ is finite almost surely.
\end{theorem}

As we will explain in more detail in Section~\ref{sec::duality_space_filling}, Theorem~\ref{thm::cle_locally_finite} follows from the almost sure continuity of space-filling $\SLE_{\kappa'}(\kappa'-6)$ because this process traces the boundary of all of the loops in a $\CLE_{\kappa'}$ process.

\subsubsection{Time-reversals of ordinary/space-filling $\SLE_{\kappa}(\rho_1; \rho_2)$}


\begin{figure}[h!]
\begin{center}
\begin{ficomment}
\subfigure[$\SLE_{128}$]{\includegraphics[width=0.48\textwidth]{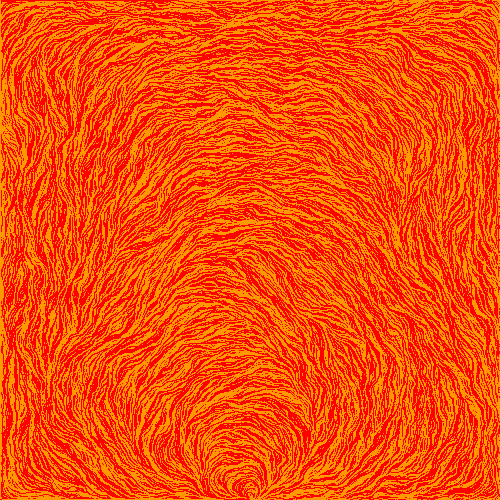}}
\hspace{0.02\textwidth}
\subfigure[$\SLE_{128}(30;30)$]{\includegraphics[width=0.48\textwidth]{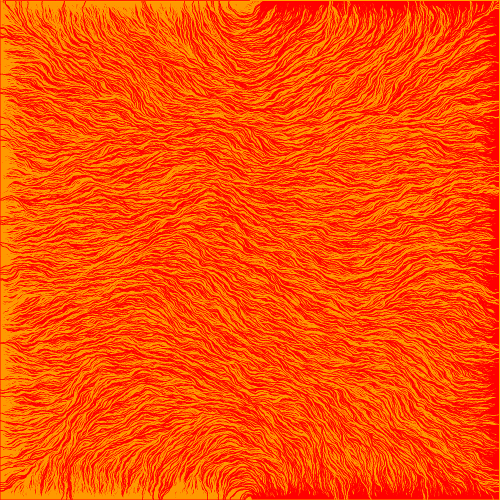}}
\subfigure[$\SLE_{128}$]{\includegraphics[width=0.48\textwidth]{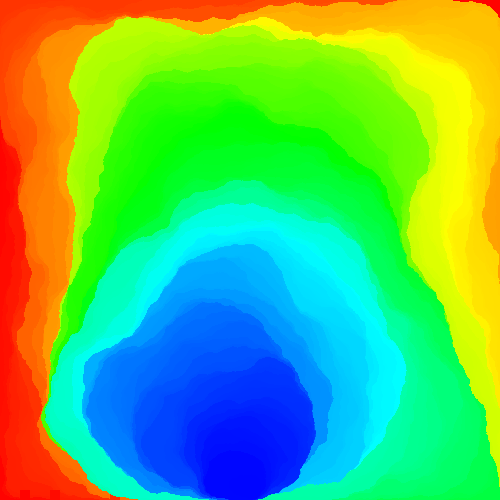}}
\hspace{0.02\textwidth}
\subfigure[$\SLE_{128}(30;30)$]{\includegraphics[width=0.48\textwidth]{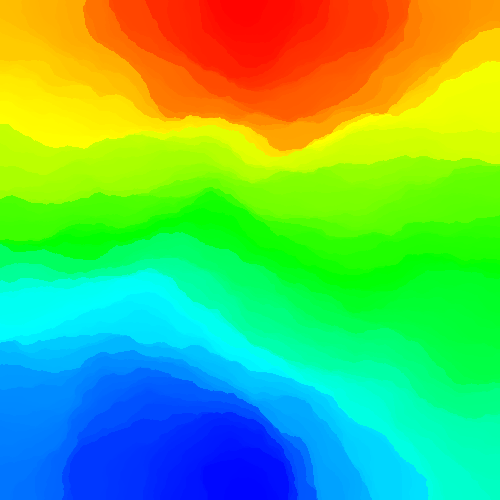}}
\end{ficomment}
\end{center}
\caption{\label{fig::bigger_than_8_reversible} \small{The simulations show the indicated $\SLE$ process in $[-1,1]^2$ from $i$ to $-i$.  The top two show the left and right boundaries in red and yellow of the process as it traverses $[-1,1]^2$ and the bottom two indicate the time parameterization of the path, as in Figure~\ref{fig::space_filling} and Figure~\ref{fig::space_filling2}.  The time-reversal of an $\SLE_\kappa$ process for $\kappa > 8$ is not an $\SLE_\kappa$ process \cite{RS_REVERSE}.  This can be seen in the simulation on the left side because the process is obviously asymmetric in its start and terminal points.  The process on the right side appears to have time-reversal symmetry and we prove this to be the case in Theorem~\ref{thm::bigger_than_8_reversibility}.}}
\end{figure}

\begin{figure}[h!]
\begin{center}
\includegraphics[scale=0.85]{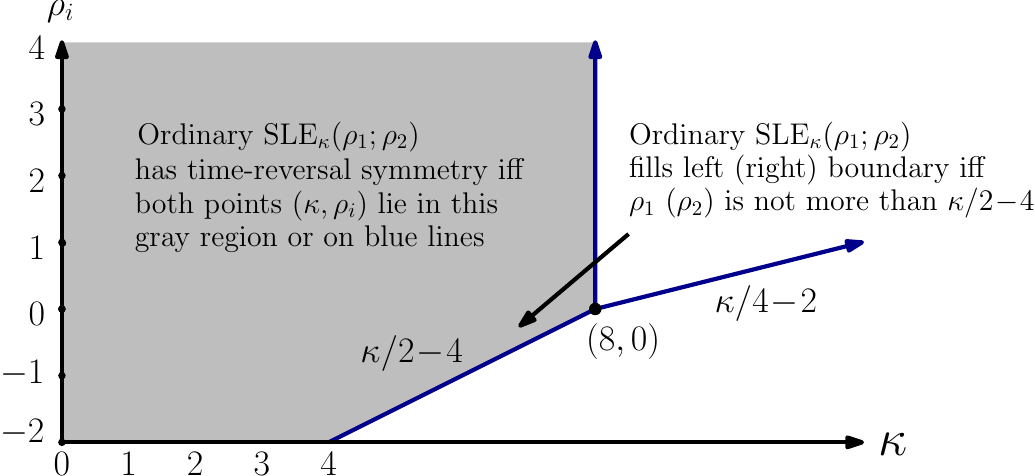}
\end{center}
\caption{ \label{fig::sfrd}  Ordinary chordal $\SLE_{\kappa}(\rho_1; \rho_2)$ is well-defined for all $\kappa \geq 0$ and $\rho_1, \rho_2 > -2$.  It has time-reversal symmetry if and only if both $\rho_1$ and $\rho_2$ belong to the region shaded in gray or lie along the blue lines.  That is, either $\kappa \leq 4$ and $\rho_i > -2$ or $\kappa \in (4,8]$ and $\rho_i \geq \tfrac{\kappa}{2}-4$ or $\kappa > 8$ and $\rho_i = \tfrac{\kappa}{4}-2$.  The threshold $\tfrac{\kappa}{2}-4$ is where the path becomes boundary filling: if $\rho_1 \leq \tfrac{\kappa}{2}-4$ (resp.\ $\rho_2 \leq \tfrac{\kappa}{2}-4$) then the path fills the left (resp.\ right) arc of the boundary which connects the initial and terminal points of the process.  If $\rho_1=\rho_2=\tfrac{\kappa}{4}-2$ then the law of the outer boundary of the path stopped upon hitting any fixed boundary point $w$ is invariant under the anti-conformal map of the domain which fixes $w$ and swaps the initial and terminal points of the path.  These reversibility results were shown for $\kappa \leq 8$ in \cite{MS_IMAG, MS_IMAG2, MS_IMAG3} (see also \cite{Z_R_KAPPA,Z_R_KAPPA_RHO,DUB_DUAL})  and will be established for $\kappa > 8$ here.  The time-reversal of an ordinary $\SLE_\kappa$ for $\kappa > 8$ is not an $\SLE_\kappa$ process, it is an $\SLE_\kappa(\tfrac{\kappa}{2}-4;\tfrac{\kappa}{2}-4)$ process.
}
\end{figure}

\begin{figure}[h!]
\begin{center}
\includegraphics[scale=0.85]{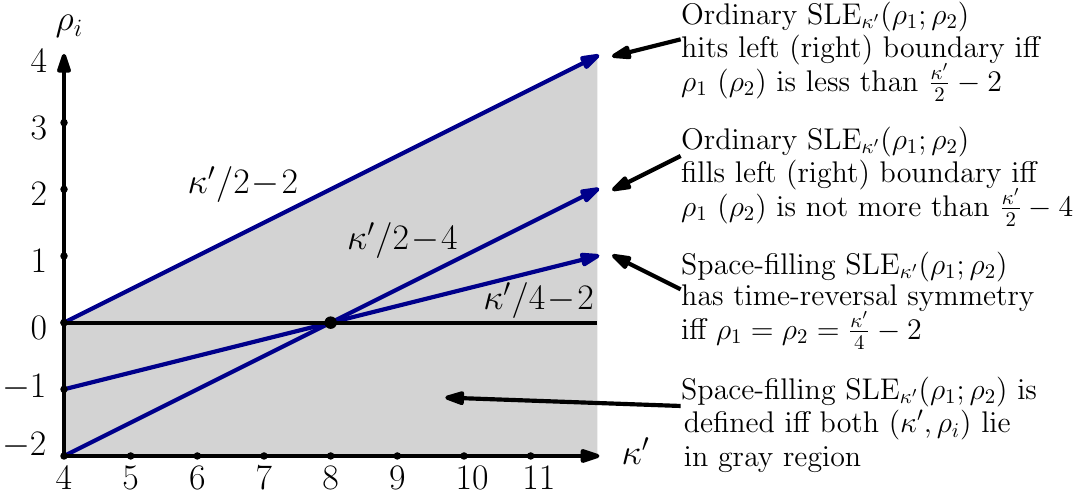}
\end{center}
\caption{\label{fig::sfrd2}  Fully space-filling $\SLE_{\kappa'}(\rho_1;\rho_2)$ can only be defined when $\kappa' > 4$ and $\rho_1, \rho_2 \in (-2, \tfrac{\kappa'}{2}-2)$ (gray region).  The upper boundary ($\tfrac{\kappa'}{2} - 2$ line) is necessary because ordinary $\SLE_{\kappa'}(\rho_1;\rho_2)$ only hits the left (resp.\ right) boundary when $\rho_1$ (resp.\ $\rho_2$) is strictly less than $\tfrac{\kappa'}{2}-2$.  Space-filling $\SLE_{\kappa'}(\rho_1; \rho_2)$ can be formed from an ordinary $\SLE_{\kappa'}(\rho_1; \rho_2)$ path $\eta$ by splicing in a space-filling curve that fills each component of $\h \setminus \eta$ after that component is swallowed by $\eta$; when $\rho_i \geq \tfrac{\kappa'}{2}-2$ some points are never swallowed by $\eta$ in finite time, so this extension does not reach these points before reaching $\infty$.  When $\rho_1 = \rho_2 = \tfrac{\kappa'}{4}-2$, the path has the same law as its time-reversal (up to parameterization).  Recall Figure~\ref{fig::sfrd} for the significance of these $\rho$ values.  In general, the time-reversal is a space-filling $\SLE_{\kappa'}(\wt \rho_2; \wt \rho_1)$ where $\wt \rho_i = \tfrac{\kappa'}{2}-4 - \rho_i$ is the reflection of $\rho_i$ about $\tfrac{\kappa'}{4}-2$.  The $\tfrac{\kappa'}{2} - 4$ line is the boundary filling threshold: ordinary $\SLE_{\kappa'}(\rho_1; \rho_2)$ fills the entire boundary if and only if $\rho_i \leq \tfrac{\kappa'}{2} - 4$ for $i \in \{1,2\}$.  When this is the case, space-filling  $\SLE_{\kappa'}(\rho_1; \rho_2)$ has the property that the {\em first} hitting times of boundary points occur in order; i.e., the path in $\h$ starting from $0$ never hits a point $x \in \R$ before filling the interval between $0$ and $x$.  The $\rho_i = 0$ line is the reflection of the $\tfrac{\kappa'}{2}-4$ line about the $\tfrac{\kappa'}{4}-2$ line: when $\rho_i \geq 0$ for $i \in \{1,2 \}$ the {\em last} hitting times of boundary points occur in order; i.e., the path never revisits a boundary point after visiting a later boundary point on the same axis.
}
\end{figure}

We say that a random path $\eta$ in a simply-connected Jordan domain $D$ from $x$ to $y$ for $x,y \in \partial D$ distinct has {\bf time-reversal symmetry} if its image under any anti-conformal map $D \to D$ which swaps $x$ and $y$ and run in the reverse direction has the same law as $\eta$ itself, up to reparameterization.  In this article, we complete the characterization of the chordal $\SLE_{\kappa}(\rho_1; \rho_2)$ processes that have time-reversal symmetry, as we explain in the caption of Figure~\ref{fig::sfrd}.  As explained earlier, throughout we generally use the symbol $\kappa$ for values less than $4$ and $\kappa' = 16/\kappa$ for values greater than $4$.  We violate this convention in a few places (such as Figure~\ref{fig::sfrd}) when we want to make a statement that applies to all $\kappa \geq 0$.  The portion of Figure~\ref{fig::sfrd} that is new to this article is the following:

\begin{theorem}
\label{thm::bigger_than_8_reversibility}
When $\kappa' > 8$, ordinary $\SLE_{\kappa'}(\rho_1;\rho_2)$ has time-reversal symmetry if and only if $\rho_1 = \rho_2 = \tfrac{\kappa'}{4}-2$.
\end{theorem}

The $\rho$ values from Theorem~\ref{thm::bigger_than_8_reversibility} are significant because for any fixed boundary point $w$, the law of the outer boundary of an $\SLE_{\kappa'}(\rho_1;\rho_2)$ stopped upon hitting $w$ is invariant under the anti-conformal map which swaps the initial and terminal points of the path and fixes $w$ if and only if $\rho_1=\rho_2=\tfrac{\kappa'}{2}-4$.  This is a necessary condition for time-reversal symmetry and Theorem~\ref{thm::bigger_than_8_reversibility} implies that it is also sufficient.  The time-reversal of an ordinary $\SLE_{\kappa'}$ process for $\kappa' > 8$ is not itself an $\SLE_{\kappa'}$ process, though the time-reversal is a certain $\SLE_{\kappa'}(\rho_1;\rho_2)$ process.  In particular, the result stated just below implies that it is an $\SLE_{\kappa'}(\tfrac{\kappa'}{2}-4;\tfrac{\kappa'}{2}-4)$ process.  The value $\tfrac{\kappa'}{2}-4$ is significant because it is the critical threshold at or below which an $\SLE_{\kappa'}$ process is boundary filling.  It was shown in \cite{MS_IMAG3} that if $\kappa' \in (4,8)$ and at least one of the $\rho_i$ is strictly below $\tfrac{\kappa'}{2}-4$, then the time-reversal of ordinary $\SLE_{\kappa'}(\rho_1 ; \rho_2)$ is not an $\SLE_{\kappa'}(\wt \rho_1; \wt \rho_2)$ process for any values of $\wt \rho_1, \wt \rho_2$.  The story is quite different if one considers space-filling $\SLE_{\kappa'}(\rho_1; \rho_2)$, as illustrated in Figure~\ref{fig::sfrd2}, and formally stated as Theorem~\ref{thm::space_filling_reversibility} below.  (The other thresholds mentioned in the figure caption are standard Bessel process observations; see e.g.\ \cite{MS_IMAG3} for more discussion.)  Theorem~\ref{thm::bigger_than_8_reversibility} is a special case of Theorem~\ref{thm::space_filling_reversibility}.

\begin{theorem}
\label{thm::space_filling_reversibility}
The time reversal of a space-filling $\SLE_{\kappa'}(\rho_1;\rho_2)$ for $\rho_1,\rho_2 \in (-2,\tfrac{\kappa'}{2}-2)$ is a space-filling $\SLE_{\kappa'}(\wt \rho_2; \wt \rho_1)$ where $\wt \rho_i = \tfrac{\kappa'}{2}-4 - \rho_i$ is the reflection of $\rho_i$ about $\tfrac{\kappa'}{4}-2$.
\end{theorem}

\subsubsection{Whole-plane time-reversal symmetry}

The final result we state concerns the time-reversal symmetry of whole-plane $\SLE_\kappa(\rho)$ processes for $\kappa \in (0,8]$.

\begin{theorem}
\label{thm::whole_plane_reversibility}
Suppose that $\eta$ is a whole-plane $\SLE_\kappa(\rho)$ process from $0$ to $\infty$ for $\kappa \in (0,4)$ and $\rho > -2$.  Then the time-reversal of $\eta$ is a whole-plane $\SLE_\kappa(\rho)$ process from $\infty$ to $0$.  If $\eta$ is a whole-plane $\SLE_{\kappa'}(\rho)$ process for $\kappa' \in (4,8]$ and $\rho \geq \tfrac{\kappa'}{2}-4$, then the time-reversal of $\eta$ is a whole-plane $\SLE_{\kappa'}(\rho)$ process from $\infty$ to $0$.
\end{theorem}

Our proof of Theorem~\ref{thm::whole_plane_reversibility} also implies the time-reversal symmetry of a flow line $\eta$ of $h_{\alpha \beta}$ as in Theorem~\ref{thm::alphabeta} with $D = \C$ for any choice of $\beta \in \R$.  These processes are variants of whole-plane $\SLE_\kappa(\rho)$ for $\kappa \in (0,4)$ and $\rho > -2$ in which one adds a constant drift to the driving function, which leads to spiraling, as illustrated in Figure~\ref{fig::spiral_fan}.  The proof also implies the reversibility of similar variants of whole-plane $\SLE_{\kappa'}(\rho)$ for $\rho \geq \tfrac{\kappa'}{2}-4$ as in Theorem~\ref{thm::alphabeta_counterflow} and illustrated in Figure~\ref{fig::spiral_lightcone}.  See Remark~\ref{rem::reversible_beta} in Section~\ref{sec::timereversal}.

\begin{remark}
\label{rem::bigger_than_8_whole_plane}
When $\kappa' > 8$, the time-reversal of a whole-plane $\SLE_{\kappa'}(\rho)$ process $\eta'$ for $\rho \geq \tfrac{\kappa'}{2}-4$ is not an $\SLE_{\kappa'}(\wt{\rho})$ process for any value of $\wt{\rho}$ (what follows will explain why this is the case).  We can nevertheless describe its time-reversal in terms of whole-plane GFF flow lines and chordal $\SLE_{\kappa'}(\rho_1;\rho_2)$ processes.  Indeed, by Theorem~\ref{thm::duality_and_light_cones}, the left and right boundaries of $\eta'$ are given by a pair of GFF flow lines $\eta^L,\eta^R$ with angles $\tfrac{\pi}{2}$ and $-\tfrac{\pi}{2}$, respectively.  By Theorem~\ref{thm::whole_plane_reversibility}, we know that $\eta^L$ has time-reversal symmetry.  Moreover, by Theorem~\ref{thm::conditional_law}, we know that the law of $\eta^R$ given $\eta^L$ is a chordal $\SLE_\kappa(\rho_1;\rho_2)$ process independently in the connected components of $\C \setminus \eta^L$.  Consequently, we know that the time-reversal of $\eta^R$ given the time-reversal of $\eta^L$ is independently an $\SLE_\kappa(\rho_2;\rho_1)$ process in each of the connected components of $\C \setminus \eta^L$ by the main result of \cite{MS_IMAG2}.  Conditionally on $\eta^L,\eta^R$, Theorem~\ref{thm::duality_and_light_cones} gives us that the law of $\eta'$ is that of a chordal $\SLE_{\kappa'}(\tfrac{\kappa'}{2}-4;\tfrac{\kappa'}{2}-4)$ in each of the components of $\C \setminus (\eta^L \cup \eta^R)$ part of whose boundary is traced by the right side of $\eta^L$ and the left side of $\eta^R$.  By Theorem~\ref{thm::space_filling_reversibility}, we thus know that the time-reversal of $\eta'$ given the time-reversals of $\eta^L$ and $\eta^R$ is independently that of an ordinary chordal $\SLE_{\kappa'}$ process in each of the components of $\C \setminus (\eta^L \cup \eta^R$).  This last statement proves our claim that the time-reversal is not a whole-plane $\SLE_{\kappa'}(\wt{\rho})$ process for any value of $\wt{\rho}$ because the conditional law of the time-reversal of $\eta'$ given its outer boundary is not that of an $\SLE_{\kappa'}(\tfrac{\kappa'}{2}-4;\tfrac{\kappa'}{2}-4)$.
\end{remark}

\begin{remark}
We will not treat the case that $\eta'$ is a whole-plane $\SLE_{\kappa'}(\rho)$ with $\kappa' > 4$ and $\rho < \tfrac{\kappa'}{2}-4$, because it does not fit into the framework of this paper as naturally.  It is not hard to show that in this case the lifting of $\eta'$ to the universal cover of $\C \setminus \{0\}$ has left and right boundaries that (when projected back to $\C$) actually coincide with each other, so that $\eta^L = \eta^R$.  However (in contrast to the $\rho \geq \tfrac{\kappa'}{2}-4$ case described above) the $\eta^L = \eta^R$ path is {\em not} an ordinary flow line in $\C \setminus \{0\}$ from $\infty$ to $0$.  Rather, it is an angle-varying flow line, alternating between two different angles.  (Using the notation of~\eqref{eqn::sle_radial_equation}, the points on $\eta^L = \eta^R$ are the points hit at times when $O_t = W_t$. Those hit when $W_t$ collides with $O_t$ from the right lie on flow lines of one angle. Those hit when $W_t$ is collides with $O_t$ from the left lie on flow lines of a different angle.)  This angle-varying flow line is not a local set, and the conditional law of $\eta'$ given $\eta^L = \eta^R$ is somewhat complicated.
\end{remark}

\section{Preliminaries}
\label{sec::preliminaries}

This section has three parts.  First, in Section~\ref{subsec::SLEoverview}, we will give an overview of the different variants of $\SLE$ and $\SLE_\kappa(\rho)$ (chordal, radial, and whole-plane) that will be important throughout this article.  We will in particular show how the continuity of whole-plane and radial $\SLE_\kappa(\rho)$ processes for all $\kappa >0$ and $\rho > -2$ can be extracted from the results of \cite{MS_IMAG}.  Next, in Section~\ref{subsec::gff}, we will give a brief overview of the whole-plane GFF.  Finally, in Section~\ref{subsec::imaginary}, we will review the aspects of the theory of boundary emanating GFF flow lines developed in \cite{MS_IMAG} which will be relevant for this article.

\subsection{$\SLE_\kappa(\rho)$ processes}
\label{subsec::SLEoverview}

$\SLE_\kappa$ is a one-parameter family of conformally invariant random growth processes introduced by Oded Schramm in \cite{S0} (which were proved to be generated by random curves by Rohde and Schramm in \cite{RS05}).  In this subsection, we will give a brief overview of three types of $\SLE$: chordal, radial, and whole-plane.  More detailed introductions to $\SLE$ can be found in many excellent surveys of the subject, e.g., \cite{W03, LAW05}.

\subsubsection{Chordal $\SLE_\kappa(\rho)$}

Chordal $\SLE_\kappa(\ul{\rho})$ in $\h$ targeted at $\infty$ is the growth process $(K_t)$ associated with the random family of conformal maps $(g_t)$ obtained by solving the Loewner ODE
\begin{equation}
\label{eqn::loewner_ode}
\partial_t g_t(z) = \frac{2}{g_t(z) - W_t},\ \ \ g_0(z) = z
\end{equation}
where $W$ is taken to be the solution to the SDE
\begin{equation}
\label{eqn::sle_kappa_rho_eqn}
\begin{split}
dW_t &= \sqrt{\kappa} dB_t + \sum_{i} \re \left( \frac{\rho^{i}}{W_t- V_t^i} \right) dt \\
dV_t^i &= \frac{2}{V_t^i - W_t} dt,\ \ \ V_0^i = z^i.
\end{split}
\end{equation}
The compact set $K_t$ is given by the closure of the complement of the domain of $g_t$ in $\h$ and $g_t$ is the unique conformal transformation $\h \setminus K_t \to \h$ satisfying $g_t(z) = z+o(1)$ as $z \to \infty$.  The points $z^i \in \ol{\h}$ are the \emph{force points} of $W$ and the $\rho^i \in \R$ are the \emph{weights}.  When $z^i \in \R$ (resp.\ $z^i \in \h$), $z^i$ is said to be a boundary (resp.\ interior) force point.  It is often convenient to organize the $z^i$ into groups $z^{i,L}, z^{i,R}, z^{i,I}$ where the superscripts $L,R,I$ indicate whether the point is to the left or right of $0$ in $\R$ or in $\h$, respectively, and we take the $z^{i,L}$ (resp.\ $z^{i,R}$) to be given in decreasing (resp.\ increasing) order.  We also group the weights $\rho^{i,L},\rho^{i,R},\rho^{i,I}$ and time evolution of the force points $V^{i,L}, V^{i,R}, V^{i,I}$ under the Loewner flow accordingly.  The existence and uniqueness of solutions to~\eqref{eqn::sle_kappa_rho_eqn} with only boundary force points is discussed in \cite[Section~2]{MS_IMAG}.  It is shown that there is a unique solution to~\eqref{eqn::sle_kappa_rho_eqn} until the first time $t$ that $W_t = V_t^{j,q}$ where $\sum_{i=1}^j \rho^{i,q} \leq -2$ for either $q = L$ or $q = R$.  We call this time the {\bf continuation threshold}.  In particular, if $\sum_{i=1}^j \rho^{i,q} > -2$ for all $1 \leq j \leq |\ul{\rho}^q|$ for $q \in \{L,R\}$ (where we use the notation $|\ul{\rho}^q|$ for the number of elements in the vector $\ul{\rho}^q$), then~\eqref{eqn::sle_kappa_rho_eqn} has a unique solution for all times $t$.  This even holds when one or both of $z^{1,L} = 0^-$ or $z^{1,R} = 0^+$ hold.  The almost sure continuity of the $\SLE_\kappa(\ul{\rho})$ trace with only boundary force points is proved in \cite[Theorem~1.3]{MS_IMAG}.  It thus follows from the Girsanov theorem \cite{KS98} that~\eqref{eqn::sle_kappa_rho_eqn} has a unique solution and the growth process associated with the Loewner evolution in~\eqref{eqn::loewner_ode} is almost surely generated by a continuous path, even in the presence of interior force points, up until either the continuation threshold or when an interior force point is swallowed.

\subsubsection{Radial $\SLE_\kappa^\mu(\rho)$}
\label{subsec::radial_sle}

A radial $\SLE_\kappa$ in $\D$ targeted at $0$ is the random growth process $(K_t)$ in $\D$ starting from a point on $\partial \D$ growing towards $0$ which is described by the random family of conformal maps $(g_t)$ which solve the radial Loewner equation:
\begin{equation}
\label{eqn::radial_loewner} \partial_t g_t(z) = g_t(z) \frac{W_t + g_t(z)}{W_t-g_t(z)},\quad g_0(z) = z.
\end{equation}
Here, $W_t = e^{i \sqrt{\kappa} B_t}$ where $B_t$ is a standard Brownian motion; $W$ is referred to as the driving function for the radial Loewner evolution.  The set $K_t$ is the complement of the domain of $g_t$ in $\D$ and $g_t$ is the unique conformal transformation $\D \setminus K_t \to \D$ fixing $0$ with $g_t'(0) > 0$.  Time is parameterized by the logarithmic conformal radius as viewed from $0$ so that $\log g_t'(0) = t$ for all $t \geq 0$.  As in the chordal setting, radial $\SLE_\kappa^\mu(\rho)$ is a generalization of radial $\SLE$ in which one keeps track of one extra marked point.  To describe it, following \cite{SCHRAMM_WILSON} we let
\[ \Psi(w,z) = -z \frac{z+w}{z-w} \quad\text{and}\quad \wt{\Psi}(z,w) = \frac{\Psi(z,w) + \Psi(1/\ol{z},w)}{2}\]
and
\[ \CG^\mu(W_t,dB_t,dt) = \left( i \kappa \mu - \frac{\kappa}{2} \right)W_t dt + i\sqrt{\kappa} W_t dB_t.\]
We say that a pair of processes $(W,O)$, each of which takes values in $\s^1$, solves the radial $\SLE_\kappa^\mu(\rho)$ equation for $\rho,\mu \in \R$ with a single boundary force point of weight $\rho$ provided that
\begin{equation}
\label{eqn::sle_radial_equation}
\begin{split}
dW_t &= \CG^\mu(W_t,dB_t,dt) + \frac{\rho}{2} \wt{\Psi}(O_t,W_t) dt\\
dO_t &= \Psi(W_t,O_t) dt.
\end{split}
\end{equation}
A radial $\SLE_\kappa^\mu(\rho)$ is the growth process corresponding to the solution $(g_t)$ of~\eqref{eqn::radial_loewner} when $W$ is taken to be as in~\eqref{eqn::sle_radial_equation}.  We will refer to a radial $\SLE_\kappa^0(\rho)$ process simply a radial $\SLE_\kappa(\rho)$ process.  In this section, we are going to collect several facts about radial $\SLE_\kappa^\mu(\rho)$ processes which will be useful for us in Section~\ref{subsec::existence} when we prove the existence of the flow lines of the GFF emanating from an interior point.

Throughout, it will often be useful to consider the SDE
\begin{equation}
\label{eqn::theta_equation}
 d\theta_t = \left(\frac{\rho+2}{2} \cot(\theta_t/2) + \kappa \mu\right) dt + \sqrt{\kappa} dB_t, \quad \theta_t \in [0,2\pi]
\end{equation}
where $B$ is a standard Brownian motion.  This SDE can be derived by taking a solution $(W,O)$ to~\eqref{eqn::sle_radial_equation} and then setting $\theta_t = \arg W_t - \arg O_t$ (see \cite[Equation~(4.1)]{SHE_CLE} for the case $\rho=\kappa-6$ and $\mu=0$).  One can easily see that~\eqref{eqn::theta_equation} has a unique solution which takes values in $[0,2\pi]$ if $\rho > -2$.  Indeed, we first suppose that $\mu = 0$.  A straightforward expansion implies that $1/x - \cot(x)$ is bounded for $x$ in a neighborhood of zero (and in fact tends to zero as $x \to 0$).  When $\theta_t$ is close to zero, this fact and the Girsanov theorem \cite{KS98} imply that its evolution is absolutely continuous with respect to $\sqrt{\kappa}$ times a Bessel process of dimension $d(\rho,\kappa) = 1+\tfrac{2(\rho+2)}{\kappa} > 1$ and, similarly, when $\theta_t$ is close to $2\pi$, the evolution of $2\pi-\theta_t$ is absolutely continuous with respect to $\sqrt{\kappa}$ times a Bessel process, also of dimension $d(\rho,\kappa)$.  The existence for $\mu \neq 0$ follows by noting that, by the Girsanov theorem \cite{KS98}, its law is equal to that of $\theta$ with $\mu = 0$ reweighted by $\exp(\mu \sqrt{\kappa} B_t - \mu^2 \kappa t/2)$ where $B$ is the Brownian motion driving $\theta$.

The existence and uniqueness of solutions to~\eqref{eqn::sle_radial_equation} can be derived from the existence and uniqueness of solutions to~\eqref{eqn::theta_equation}.  Another approach to this for $\mu=0$ is to use \cite[Theorem~3]{SCHRAMM_WILSON} to relate~\eqref{eqn::sle_radial_equation} to a chordal $\SLE_\kappa(\ul{\rho})$ driving process with a single interior force point and then to invoke the results of \cite[Section~2]{MS_IMAG} and the Girsanov theorem \cite{KS98}.  Pathwise uniqueness can easily be seen by considering two solutions $\theta^1,\theta^2$ coupled together to be driven by the same Brownian motion and then analyzing the process $\ol{\theta} = \theta^1 - \theta^2$.  In particular, if $\theta^1 < \theta^2$, then $\ol{\theta}$ moves (deterministically) upwards and if $\theta^1 > \theta^2$ then $\ol{\theta}$ moves (deterministically) downwards.  So, it must be the case that $\theta^1$ and $\theta^2$ eventually meet and do not subsequently separate.  We remark that one can consider radial $\SLE_\kappa^\mu(\ul{\rho})$ processes with many boundary force points $\ul{\rho}$ as in \cite{SCHRAMM_WILSON}, though we will not consider the more general case in this article.  We now turn to prove the existence of a unique stationary solution to~\eqref{eqn::sle_radial_equation}.

\begin{proposition}
\label{prop::sle_kappa_rho_stationary}
Suppose that $\rho > -2$ and $\mu \in \R$.  There exists a unique stationary solution $(\wt{W}_t,\wt{O}_t)$ for $t \in \R$ to~\eqref{eqn::sle_radial_equation}.  If $(W_t,O_t)$ is any solution to~\eqref{eqn::sle_radial_equation} and $\Theta_t$ is the shift operator $\Theta_t f(s) = f(s+t)$, then the law of $(\Theta_T W,\Theta_T O)$ converges to $(\wt{W},\wt{O})$ weakly with respect to the topology of local uniform convergence on continuous functions $\R \to \s^1 \times \s^1$ as $T \to \infty$.
\end{proposition}
\begin{proof}
Suppose that $(W^i,O^i)$ for $i=1,2$ are two solutions to~\eqref{eqn::sle_radial_equation} starting from $W_0^i,O_0^i$ and let $\theta_t^i = \arg W_t^i - \arg O_t^i$.  Fix $\epsilon > 0$.  We will show that there exists $T > 0$ and a coupling of the laws of $(W^i,O^i)$ for $i=1,2$ so that the probability that
\[ \sup_{t \geq T} \left( | W_t^1 - W_t^2| + |O_t^1 - O_t^2| \right) \geq \epsilon\]
is at most $\epsilon$.  The result will follow upon showing that this is the case.  In order to prove this, it suffices to show that there is a coupling of the laws of $(W^i,O^i)$ for $i=1,2$ and $T \geq 0$ so that the probability of the event that
\[ \{\theta_T^1 \neq \theta_T^2\} \cup \left\{ |O_T^1 - O_T^2| \geq \epsilon \right\}\]
is at most $\epsilon$.  Indeed, on the complement of this event we can couple $\theta^1,\theta^2$ so that $\theta_t^1 = \theta_t^2$ for all $t \geq T$ by the uniqueness of solutions to~\eqref{eqn::theta_equation} established just above and we have that $|W_t^1 - W_t^2| = |O_t^1 - O_t^2| = |O_T^1 - O_T^2|< \epsilon$ for all $t \geq T$.

To show that this is true, we take $(W^i,O^i)$ for $i=1,2$ to be independent and fix $M > 0$ large.  Then in each time interval of the form $(2kM+1,(2k+1) M+1]$ for $k \in \N$ there is a positive chance that $\theta^1$ stays in $[\pi,\pi+\tfrac{\epsilon}{4}]$ and $\theta^2$ stays in $[\pi+\tfrac{\epsilon}{2},\pi+\epsilon]$ for the entire interval uniformly in the realization of the $(W^i,O^i)$ in the previous intervals.  If~$M$ is chosen sufficiently large relative to $\epsilon$ and this event occurs, then it follows from the evolution equation for $O_t^1,O_t^2$ that $O_t^1$ and $O_t^2$ will meet in such an interval and, at this time, $\theta_t^1$ and $\theta_t^2$ will differ by at most $\epsilon$.  Conditional on this happening, it is then a positive probability event that $\theta_t^1$ and $\theta_t^2$ will coalesce in a time which goes to zero in law as $\epsilon \to 0$ and, by this time, the distance between $O_t^1$ and $O_t^2$ will be bounded by a quantity that tends to zero in law as $\epsilon \to 0$.  It therefore follows that if $\tau$ is the first time $t$ that $\theta_t^1 = \theta_t^2$ and $|O_t^1 - O_t^2| \leq \epsilon$, then $\p[ \tau \geq t]$ decays exponentially fast in $t$ at a rate which depends only on $\epsilon$.  From this, the result follows.
\end{proof}

The following conformal Markov property is immediate from the definition of radial $\SLE_\kappa^\mu(\rho)$:

\begin{proposition}
\label{prop::radial_conf_markov} Suppose that $K_t$ is a radial $\SLE_\kappa^\mu(\rho)$ process, let $(g_t)$ be the corresponding family of conformal maps, and let $(W,O)$ be the driving process.  Let $\tau$ be any almost surely finite stopping time for $K_t$.  Then $g_\tau(K_t \setminus K_\tau)$ is a radial $\SLE_\kappa^\mu(\rho)$ process whose driving function $(\wt{W},\wt{O})$ has initial condition $(\wt{W}_0,\wt{O}_0) = (W_\tau,O_\tau)$.
\end{proposition}

We will next show that radial $\SLE_\kappa^\mu(\rho)$ processes are almost surely generated by continuous curves by using \cite[Theorem~1.3]{MS_IMAG}, which gives the continuity of chordal $\SLE_\kappa(\rho)$ processes.

\begin{proposition}
\label{prop::radial_continuity}
Suppose that $K_t$ is a radial $\SLE_\kappa^\mu(\rho)$ process with $\rho > -2$ and $\mu \in \R$.  For each $T \in (0,\infty)$, we have that $K|_{[0,T]}$ is almost surely generated by a continuous curve.
\end{proposition}
\begin{proof}
It suffices to prove the result for $\mu = 0$ since, as we remarked just after~\eqref{eqn::theta_equation}, the law of the process for $\mu \neq 0$ up to any fixed and finite time is absolutely continuous with respect to the case when $\mu = 0$.  Let $(W,O)$ be the driving function of a radial $\SLE_\kappa(\rho)$ process with $\rho > -2$ and let $\theta_t = \arg W_t - \arg O_t$.  We assume without loss of generality that $\theta_0 = 0$.  Let $\tau = \inf\{t \geq 0 : \theta_t = 2\pi\}$.  By the conformal Markov property of radial $\SLE_\kappa(\rho)$ (Proposition~\ref{prop::radial_conf_markov}) and the symmetry of the setup, it suffices to prove that $K|_{[0,\tau]}$ is generated by a continuous curve.

For each $\epsilon > 0$, we let $\tau_\epsilon = \inf\{t \geq 0 : \theta_t \geq 2\pi-\epsilon\}$.  Observe from the form of~\eqref{eqn::theta_equation} that $\tau_\epsilon < \infty$ almost surely (when $\theta_t \in [0,2\pi-\epsilon]$, its evolution is absolutely continuous with respect to that of a positive multiple of a Bessel process of dimension larger than $1$).  It follows from \cite[Theorem~3]{SCHRAMM_WILSON} that the law of a radial $\SLE_\kappa(\rho)$ process in $\D$ is equal to that of a chordal $\SLE_\kappa(\rho,\kappa-6-\rho)$ process on $\D$ where the weight $\kappa-6-\rho$ corresponds to an interior force point located at $0$, stopped at time $\tau$.  Assume that $K_t$ is parameterized by logarithmic conformal radius as viewed from $0$.  The Girsanov theorem implies that the law of $K|_{[0,\tau_\epsilon \wedge r]}$, any fixed $\epsilon, r > 0$, is mutually absolutely continuous with respect to that of a chordal $\SLE_\kappa(\rho)$ process (without an interior force point).  We know from \cite[Theorem~1.3]{MS_IMAG} that such processes are almost surely continuous, which gives us the continuity up to time $\tau_\epsilon \wedge r$.  This completes the proof in the case that $\rho \geq \tfrac{\kappa}{2}-2$ since $\theta_t$ does not hit $\{0,2\pi\}$ at positive times because $\tau_\epsilon \to \infty$ as $\epsilon \to 0$.  For the rest of the proof, we shall assume that $\rho \in (-2,\tfrac{\kappa}{2}-2)$.  By sending $r \to \infty$ and using that $\tau_\epsilon < \infty$, we get the continuity up to time $\tau_\epsilon$.

To get the continuity in the time interval $[\tau_\epsilon,\tau]$, we fix $\delta > 0$ and assume that $\epsilon > 0$ is so small so the event $E$ that $\theta_t|_{[\tau_\epsilon,\tau]}$ hits $2\pi$ before hitting $0$ satisfies $\p[E] \geq 1-\delta$.  By applying the conformal transformation $g_{\tau_\epsilon}$, the symmetry of the setup (using that $(\theta_t : t \geq 0) \stackrel{d}{=} (2\pi-\theta_t : t \geq 0)$, recall~\eqref{eqn::theta_equation}) and the argument we have described just above implies that $g_{\tau_\epsilon}(K|_{[\tau_\epsilon,\tau]})$ is generated by a continuous path on $E$.  Sending $\delta \to 0$ implies the desired result.
\end{proof}

We will prove in Section~\ref{subsec::transience} and Section~\ref{sec::duality_space_filling} that if $\eta$ is a radial $\SLE_\kappa^\mu(\rho)$ process with $\rho > -2$ and $\mu \in \R$ then $\lim_{t \to \infty} \eta(t) = 0$ almost surely.  This is the so-called ``endpoint'' continuity of radial $\SLE_\kappa^\mu(\rho)$ (first established by Lawler for ordinary radial $\SLE_\kappa$ in \cite{LAW_ENDPOINT}).  We finish by recording the following fact, which follows from the discussion  after~\eqref{eqn::theta_equation}.

\begin{lemma}
\label{lem::radial_critical_for_hitting}
Suppose that $\eta$ is a radial $\SLE_\kappa^\mu(\rho)$ process with $\rho \geq \tfrac{\kappa}{2}-2$ and $\mu \in \R$.  Then $\eta$ almost surely does not intersect $\partial \D$ and is a simple path.
\end{lemma}
\begin{proof}
Let $\theta_t = \arg W_t - \arg O_t$ where $(W,O)$ is the driving pair for $\eta$.  By the discussion after~\eqref{eqn::theta_equation}, we know that the evolution of $\theta_t$ (resp.\ $2\pi-\theta_t$) is absolutely continuous with respect to that of $\sqrt{\kappa}$ times a Bessel process of dimension $d(\rho,\kappa) \geq 2$ when it is near the singularity at $0$ (resp.\ $2\pi$).  Consequently, $\theta_t$ almost surely does not hit $0$ or $2\pi$ except possibly at time $0$.  From this, the result follows.
\end{proof}

\subsubsection{Whole-plane $\SLE_\kappa^\mu(\rho)$}
\label{subsec::whole_plane_sle}

Whole-plane $\SLE$ is a variant of $\SLE$ which describes a random growth process $K_t$ where, for each $t \in \R$, $K_t \subseteq \C$ is compact with $\C_t = \C \setminus K_t$ simply connected (viewed as a subset of the Riemann sphere).  For each $t$, we let $g_t \colon \C_t \to \C \setminus \D$ be the unique conformal transformation with $g_t(\infty) = \infty$ and $g_t'(\infty) > 0$.  Then $g_t$ solves the whole-plane Loewner equation
\begin{equation}
\label{eqn::whole_plane_loewner}
 \partial_t g_t = g_t(z) \frac{W_t + g_t(z)}{W_t - g_t(z)}.
\end{equation}
Here, $W_t = e^{i\sqrt{\kappa} B_t}$ where $B_t$ is a two-sided standard Brownian motion.  Equivalently, $W$ is given by the time-stationary solution to~\eqref{eqn::sle_radial_equation} with $\rho=\mu=0$.  Note that~\eqref{eqn::whole_plane_loewner} is the same as~\eqref{eqn::radial_loewner}.  In fact, for any $s \in \R$, the growth process $1/g_s(K_t \setminus K_s)$ for $s \geq t$ from $\partial \D$ to $0$ is a radial $\SLE_\kappa$ process in $\D$.  Thus, whole-plane $\SLE$ can be thought of as a bi-infinite time version of radial $\SLE$.  Whole-plane $\SLE_\kappa^\mu(\rho)$ is the growth process associated with~\eqref{eqn::whole_plane_loewner} where $W_t$ is taken to be the time-stationary solution of~\eqref{eqn::radial_loewner} described in Proposition~\ref{prop::sle_kappa_rho_stationary}.  As before, we will refer to a whole-plane $\SLE_\kappa^0(\rho)$ process as simply a whole-plane $\SLE_\kappa(\rho)$ process.

We are now going to prove the continuity of whole-plane $\SLE_\kappa^\mu(\rho)$ processes.  The idea is to deduce the result from the continuity of radial $\SLE_\kappa^\mu(\rho)$ proved in Proposition~\ref{prop::radial_continuity} and the relationship between radial and whole-plane $\SLE$ described just above.  This gives us that for any fixed $T \in \R$, a whole-plane $\SLE$ restricted to $[T,\infty)$ is generated by the conformal image of a continuous curve.  The technical issue that one has to worry about is whether pathological behavior in the way that the process gets started causes this conformal map to be discontinuous at the boundary.

\begin{proposition}
\label{prop::whole_plane_continuity}
Suppose that $K_t$ is a whole-plane $\SLE_\kappa^\mu(\rho)$ process with $\rho > -2$ and $\mu \in \R$.  Then $K_t$ is almost surely generated by a continuous curve $\eta$.
\end{proposition}

\begin{figure}[ht!]
\begin{center}
\includegraphics[scale=0.6125]{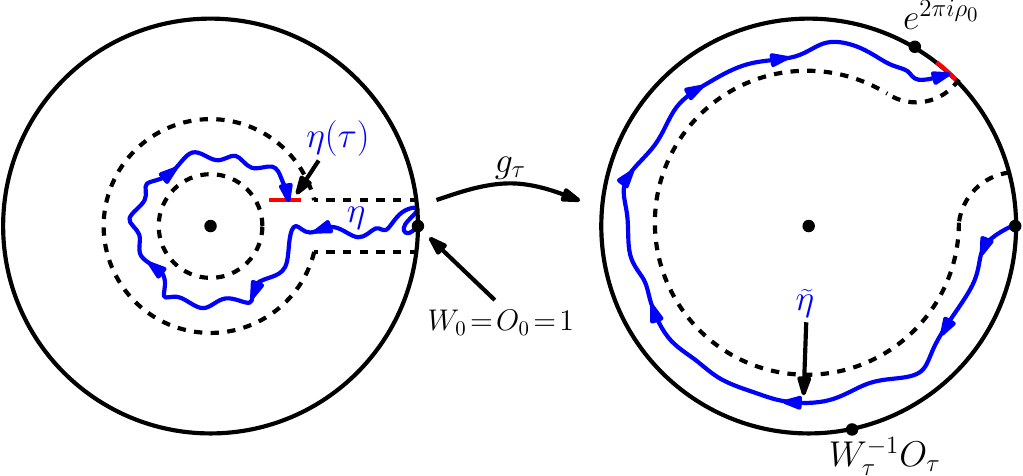}
\end{center}
\caption{\label{fig::radial_loop}
Fix $\kappa > 0$.  Suppose that $\eta$ is a radial $\SLE_\kappa(\rho)$ in $\D$ starting from $1$ with a single boundary force point of weight $\rho \in (-2 \vee (\tfrac{\kappa}{2}-4),\tfrac{\kappa}{2}-2)$, located at $1^+$ (immediately in the counterclockwise direction from $1$ on $\partial \D$).  For every $\delta \in (0,1)$, the first loop~$\eta$ makes around $0$ does not intersect $\partial \D$ and contains $\delta \D$ with positive probability.  To see this, we first note that the event~$E$ that $\eta$ wraps around the small disk with dashed boundary while staying inside of the dashed region and then hits the red line, say at time $\tau$, occurs with positive probability (see left side above).  This can be seen since a radial $\SLE_\kappa(\rho)$ process is equal in distribution to a chordal $\SLE_\kappa(\rho,\kappa-6-\rho)$ process up until the first time it swallows $0$ \cite[Theorem~3]{SCHRAMM_WILSON} and the latter is absolutely continuous with respect to a chordal $\SLE_\kappa(\rho)$ process up until just before swallowing~$0$ (see the proof of Proposition~\ref{prop::radial_continuity}).  The claim follows since~$E$ a positive probability event for the latter (see \cite[Section~4.7]{LAW05}; this can also be proved using the GFF).  Moreover, we note that conditional on~$E$, $\eta|_{[\tau,\infty)}$ surrounds~$0$ before hitting~$\partial \D$ with positive probability.  Indeed, the conformal invariance of Brownian motion implies that $\theta_\tau$ is equal to $2\pi$ times the probability that a Brownian motion starting at $0$ exits $\D \setminus \eta([0,\tau])$ in the right side of $\eta([0,\tau])$ and the Beurling estimate \cite{LAW05} implies that there exists $\rho_0 > 0$ such that the latter is at least $\rho_0$ on $E$.  Consequently, it suffices to show that $\wt{\eta} = W_\tau^{-1} g_\tau(\eta|_{[\tau,\infty)})$ wraps around $0$ and hits to the left of $W_\tau^{-1} O_\tau$ while staying inside of the dashed region indicated on the right side, with uniformly positive probability.  That this holds follows, as before, by using \cite[Theorem~3]{SCHRAMM_WILSON} and absolute continuity to compare to the case when $\wt{\eta}$ is a chordal $\SLE_\kappa(\rho)$ process and note that the latter stays inside of the dashed region and exits in the red segment with positive probability.}
\end{figure}

The main ingredient in the proof of Proposition~\ref{prop::whole_plane_continuity} is the following lemma (see Figure~\ref{fig::radial_loop}).

\begin{lemma}
\label{lem::radial_loop}
Suppose that $\eta$ is a radial $\SLE_\kappa^\mu(\rho)$ process with $\rho \in (-2,\tfrac{\kappa}{2}-2)$ and $\mu \in \R$.  For each $t \in [0,\infty)$, let $\D_t = \D \setminus K_t$.  Let
\[ \tau = \inf\{t \geq 0 : \partial \D_t \cap \partial \D = \emptyset\}.\]
Then $\p[\tau < \infty] = 1$.
\end{lemma}
\begin{proof}
We first suppose that $\rho \in  (-2 \vee (\tfrac{\kappa}{2}-4),\tfrac{\kappa}{2}-2)$ so that $\eta$ almost surely does not fill the boundary.  By Proposition~\ref{prop::radial_conf_markov}, the conformal Markov property for radial $\SLE_\kappa^\mu(\rho)$, it suffices to show that for each $\delta \in (0,1)$,
with positive probability, the first loop that $\eta$ makes about $0$ does not intersect $\partial \D$ and contains $\delta \D$.  By the absolute continuity of radial $\SLE_\kappa^\mu(\rho)$ and radial $\SLE_\kappa(\rho)$ processes up to any fixed time $t \in [0,\infty)$, as explained in Section~\ref{subsec::radial_sle}, it in turn suffices to show that this holds for ordinary radial $\SLE_\kappa(\rho)$ processes.  The proof of this is explained the caption of Figure~\ref{fig::radial_loop}.  The proof for $\rho \in (-2,\tfrac{\kappa}{2}-4]$ is similar.  The difference is that, in this range of $\rho$ values, the path is almost surely boundary filling.  In particular, it is not possible for the first loop that $\eta$ makes around the origin to be disjoint from the boundary.  Nevertheless, a small modification of the argument described in Figure ~\ref{fig::radial_loop} implies that the inner boundary of the \emph{second} loop made by $\eta$ around $0$ has a positive chance of being disjoint from $\partial \D$ and contain $\delta \D$.
\end{proof}

\begin{proof}[Proof of Proposition~\ref{prop::whole_plane_continuity}]
We first suppose that $\rho \geq \tfrac{\kappa}{2}-2$.  Fix $T \in \R$.  By the discussion in the beginning of this subsection, we know that we can express $K|_{[T,\infty)}$ as the conformal image of a radial $\SLE_\kappa^\mu(\rho)$ process.  By Proposition~\ref{prop::radial_continuity}, we know that the latter is generated by a continuous curve which, since $\rho \geq \tfrac{\kappa}{2}-2$, is almost surely non-boundary intersecting by Lemma~\ref{lem::radial_critical_for_hitting}.  This implies that $K|_{[T,\infty)}$ is generated by the conformal image of a continuous non-boundary-intersecting curve.  This implies that for any $S > T$, the restriction of this image to $[S,\infty)$ is a continuous path.  The result follows since $S$ can be taken arbitrarily small.

We now suppose that $\rho \in (-2,\tfrac{\kappa}{2}-2)$ and again fix $T \in \R$.  As before, we know that we can express $K|_{[T,\infty)}$ as the image under the conformal map $1/g_T^{-1}$ of a radial $\SLE_\kappa^\mu(\rho)$ process $\eta \colon [T,\infty) \to \ol{\D}$.  More precisely, $K|_{[T,\infty)}$ is the complement of the unbounded connected component of $\C \setminus ( (1/g_T^{-1})(\eta([T,t])) \cup K_T)$.  Let $\tau$ be the first time $t \geq T$ that $\partial K_t \cap \partial K_T = \emptyset$.  Lemma~\ref{lem::radial_loop} implies that $\p[\tau < \infty] = 1$.  By the almost sure continuity of the radial $\SLE_\kappa^\mu(\rho)$ processes, it follows that $K_\tau$ is locally connected which implies that $1/g_\tau^{-1}$ extends continuously to the boundary, which implies that $1/g_\tau^{-1}$ applied to $\eta|_{[\tau,\infty)}$ is almost surely a continuous curve.  The result follows since the distribution of $\tau - T$ does not depend on $T$ since the driving function of a whole-plane $\SLE_\kappa^\mu(\rho)$ process is time-stationary and we can take $T$ to be as small as we like.
\end{proof}

We finish this section by recording the following simple fact, that the trace of a whole-plane $\SLE_\kappa(\rho)$ process is almost surely unbounded.

\begin{lemma}
\label{lem::whole_plane_unbounded}
Suppose that $\eta$ is a whole-plane $\SLE_\kappa(\rho)$ process with $\rho > -2$.  Then $\limsup_{t \to \infty} |\eta(t)| = \infty$ almost surely.
\end{lemma}
\begin{proof}
This follows since whole-plane $\SLE_\kappa(\rho)$ is parameterized by capacity, exists for all time, and \cite[Proposition~3.27]{LAW05}.
\end{proof}

In Section~\ref{sec::interior_flowlines} and Section~\ref{sec::duality_space_filling}, we will establish the transience of whole-plane $\SLE_\kappa(\rho)$ processes: that $\lim_{t \to \infty} \eta(t) = \infty$ almost surely.

\subsection{Gaussian free fields}
\label{subsec::gff}

\newcommand{\slt}{H_s}
\newcommand{\sltc}{H}
\newcommand{\sltz}{H_{s,0}}

In this section, we will give an overview of the basic properties and construction of the whole-plane GFF.  For a domain $D \subseteq \C$, we let $\slt(D)$ denote the space of $C^\infty$ functions with compact support contained in $D$.  We will simply write $\slt$ for $\slt(\C)$.  We let $\sltz(D)$ consist of those $\phi \in \slt(D)$ with $\int \phi(z) dz = 0$ and write $\sltz$ for $\sltz(\C)$.

Any distribution (a.k.a.\ generalized function) $h$ describes a linear map $\phi \mapsto (h, \phi)$ from $\slt(\C)$ to $\R$.  The whole-plane GFF can be understood as a random distribution $h$ defined {\em modulo a global additive constant} in $\R$. One way to make this precise is to define an equivalence relation: two generalized functions $h_1$ and $h_2$ are equivalent {\em modulo global additive constant} if $h_1 - h_2 = a$ for some $a \in \R$ (i.e., $(h_1,\phi) - (h_2, \phi) = a \int \phi(z)dz$ for all test functions $\phi \in \slt$).  The whole-plane GFF modulo global additive constant is then a random equivalence class, which can be described by specifying a representative of the equivalence class.  Another way to say that $h$ is defined only modulo a global additive constant is to say that the quantities $(h,\phi)$ are defined only for test functions $\phi \in \sltz$.  It is not hard to see that this point of view is equivalent: the restriction of the map $\phi \mapsto (h, \phi)$ to $\sltz$ determines the equivalence class, and vice-versa.

If one fixes a constant $r > 0$, it is also possible to understand the whole-plane GFF as a random distribution defined modulo a global additive constant in $r \Z$.  This can also be understood as a random equivalence class of distributions (with $h_1$ and $h_2$ equivalent when $h_1 - h_2 = a \in r \Z$).  Another way to say that $h$ is defined only modulo a global additive constant in $r \Z$ is to say that \begin{enumerate}
 \item $(h, \phi)$ is well-defined for $\phi \in \sltz$ but
 \item for test functions $\phi \in \slt \setminus \sltz$, the value $(h, \phi)$ is defined only modulo an additive multiple of $r \int \phi(z)dz$.
\end{enumerate}
We can specify~$h$ modulo $r \Z$ by describing the map $\phi \mapsto (h, \phi)$ on $\sltz$ and also specifying the value $(h, \phi_0)$ modulo $r$ for {\em some fixed} test function~$\phi_0$ with $\int \phi_0(z)dz = 1$.  (A general test function is a linear combination of~$\phi_0$ and an element of~$\sltz$.)

In this subsection, we will explain how one can construct the whole-plane GFF (either modulo $\R$ or modulo $r \Z$) as an infinite volume limit of zero-boundary GFFs defined on an increasing sequence of bounded domains.  Finally, we will review the theory of local sets in the context of the whole-plane GFF.  We will assume that the reader is familiar with the ordinary GFF and the theory of local sets in this context \cite{SHE06,SchrammShe10,MS_IMAG}.  Since the whole-plane theory is parallel (statements and proofs are essentially the same), our treatment will be brief.

\subsubsection{Whole-plane GFF}
\label{subsec::gff_construction}

We begin by giving the definition of the whole-plane GFF defined modulo additive constant; see also \cite{SHE_WELD} for a similar (though somewhat more detailed) exposition.  We let~$\sltc$ denote the Hilbert space closure of $\slt$ modulo a global additive constant in~$\R$, equipped with the {\bf Dirichlet inner product}
\[ (f,g)_\nabla = \frac{1}{2\pi} \int \nabla f(z) \cdot \nabla g(z) dz.\]
(The normalization $(2\pi)^{-1}$ in the definition of the Dirichlet inner product is convenient because then, for example, the dominant term in the covariance function for the GFF is given by $-\log|x-y|$ rather than a multiple of $-\log|x-y|$.)  Let $(f_n)$ be an orthonormal basis of $\sltc$ and let $(\alpha_n)$ be a sequence of i.i.d.\ $N(0,1)$ random variables.
The {\bf whole-plane GFF} (modulo an additive constant in $\R$) is an equivalence class of distributions, a representative of which is given by the series expansion
\[ h = \sum_n \alpha_n f_n.\]
That is, for each $\phi \in \sltz$, we can write
\begin{equation}
\label{eqn::wholegfflimit} (h, \phi) := \lim_{n \to \infty} \left(\sum \alpha_n f_n, \phi\right),\end{equation}
where $(\cdot,\cdot)$ is the $L^2$ inner product.
As in the case of the GFF on a compact domain, the convergence almost surely holds for all such $\phi$,
and the law of $h$ turns out not to depend on the choice of $(f_n)$ \cite{SHE06}.  It is also possible to view $h$ as the standard Gaussian on $\sltc$, i.e.\ a collection of random variables $(h,f)_\nabla$ indexed by $f \in \sltc$ which respect linear combinations and have covariance
\[ \cov( (h,f)_\nabla, (h,g)_\nabla ) = (f,g)_\nabla \quad \text{for} \quad f,g \in \sltc.\]
Formally integrating by parts, we have that $(h,f)_\nabla = -2\pi (h,\Delta f)$ (where $(h,\Delta f)$ is, formally, the $L^2$ inner product of $h$ and $\Delta f$).  Thus, for any fixed function (or generalized function) $\phi$ with the property that $\Delta^{-1} \phi \in \sltc$, the limit defining $(h,\phi)$ in~\eqref{eqn::wholegfflimit} a.s.\ exists.  (Although the limit almost surely exists for any fixed $\phi$ with this property, it is a.s.\ not the case that the limit exists for {\em all} functions with this property.  However, the limit does exist a.s.\ for all $\phi \in \sltz$ simultaneously.)   We think of $h$ as being defined only up to an additive constant in $\R$ because
\[ (h+c,\phi) = (h,\phi) + (c,\phi) =  (h,\phi) \quad \text{for all}\quad  \phi \in \sltz \quad\text{and}\quad c \in \R.\]
We can fix the additive constant, for example, by setting $(h,\phi_0) = 0$ for some fixed $\phi_0 \in \slt$ with $\int \phi_0(z) dz =1$.

Suppose that $\phi_0,\phi_1 \in \slt$ are distinct with $\int \phi_j(z) dz = 1$ for $j = 0,1$.  Then we can use either $\phi_0$ or $\phi_1$ to fix the additive constant for $h$ by setting either $(h,\phi_0) = 0$ or $(h,\phi_1) = 0$.  Regardless of which choice we make, as $\phi_0 - \phi_1$ has mean zero, we have that $(h, \phi_0 - \phi_1)$ has the Gaussian distribution with mean zero.  Moreover, the variance of $(h,\phi_0 - \phi_1)$ does not depend on the choice of $\phi_0,\phi_1$ for fixing the additive constant.  Extending the definition of $h$ to all compactly supported test functions (not just those of mean zero) by requiring $(h, \phi_0) = 0$ is not exactly the same as extending the definition by requiring $(h,\phi_1)=0$.  Each of these two extension procedures can be understood as a different mapping from a space of equivalence classes (the space of distributions modulo additive constant) to the space of distributions (where this mapping sends an equivalence class to a representative element of itself).  However, note that for any {\em fixed} choice of $h$, these two extension procedures yield distributions $h_1$ and $h_2$ that differ from one another by an additive ``constant,'' namely the quantity $(h, \phi_1 - \phi_2)$.  This quantity is of course random (in the sense that it depends on $h$) but it is a constant in the sense that it does not depend on the spatial variable; that is, $(h_1 - h_2, \phi)  = (a,\phi)$, where the quantity $a$ does not depend on $\phi$.

In other words, while it is perfectly correct to say that $h$ is a random element of the space of ``distributions modulo additive constant,'' this does not mean that if you come up with {\em any} two different ways of {\em fixing} that additive constant (thereby producing two random distributions), the difference between the two distributions you produce will be deterministic.

Fix $r > 0$ and $\phi_0 \in \slt$ with $\int \phi_0(z) dz = 1$.  The {\bf whole-plane GFF modulo $r$} is a random equivalence class of distributions (where two distributions are equivalent when their difference is a constant in $r \Z$).  An instance can be generated by
\begin{enumerate}
\item sampling a whole-plane GFF $h$ modulo a global additive constant in $\R$, as described above, and then
\item choosing independently a uniform random variable $U \in [0,r)$ and fixing a global additive constant (modulo $r$) for $h$ by requiring that $(h,\phi_0) \in (U+r \Z)$.
\end{enumerate}

One reason that this object will be important for us is that if $h$ is a smooth function, then replacing $h$ with $h+a$ for $a \in 2\pi \chi \Z$ does not change the north-going flow lines of the vector field $e^{ih/\chi}$.  On the other hand, adding a constant in $(0,2\pi \chi)$ rotates all of the arrows by some non-trivial amount (so that the north-going flow lines become flow lines of angle $a/\chi$).  Thus, while it is not possible to determine the ``values'' of the whole-plane GFF in an absolute sense, we can construct the whole-plane GFF modulo an additive constant in $2\pi \chi \Z$, and this is enough to determine its flow lines.  These constructions will be further motivated in the next subsection in which we will show that they arise as various limits of the ordinary GFF subject to different normalizations.  Before proceeding to this, we will state the analog of the Markov property for the whole-plane GFF defined either modulo an additive constant in $\R$ or modulo an additive constant in $r \Z$.

\begin{proposition}
\label{prop::whole_plane_markov}
Suppose that $h$ is a whole-plane GFF viewed as a distribution defined up to a global additive constant in $\R$ and that $W \subseteq \C$ is open and bounded.  The conditional law of $h|_W$ given $h|_{\C \setminus W}$ is that of a zero-boundary GFF on $W$ plus the harmonic extension of its boundary values from $\partial W$ to $W$ (which is only defined up to a global additive constant in $\R$).  If $h$ is defined up to a global additive constant in $r \Z$, then the statement holds except the harmonic function is defined modulo a global additive constant in $r \Z$.
\end{proposition}

We note that it is somewhat informal to refer to the ``harmonic extension'' of the boundary values of $h$ from $\partial W$ to $W$ because $h$ is a distribution and not a function, so does not have ``boundary values'' on $\partial W$ in a traditional sense.  This ``harmonic extension'' is made sense of rigorously by orthogonal projection, as explained just below.  The proof of Proposition~\ref{prop::whole_plane_markov} is analogous to the proof of \cite[Proposition~3.1]{MS_IMAG}, however we include it here for completeness.

\begin{proof}[Proof of Proposition~\ref{prop::whole_plane_markov}]
Let $\sltc(W)$ be the closure of those functions in $\slt(W)$ with respect to $(\cdot,\cdot)_\nabla$, considered modulo additive constant.  Let $\sltc^\perp(W)$ consist of those functions in $\sltc$ which are harmonic in $W$ (we note that whether a function is harmonic in $W$ only depends on the values of $W$ modulo a global additive constant).  Then it is not difficult to see that $\sltc(W) \oplus \sltc^\perp(W)$ gives an orthogonal decomposition of $\sltc$.  Let $(f_n^W)$ (resp.\ $(f_n^{W^c})$) be a $(\cdot,\cdot)_\nabla$ orthonormal basis of $\sltc(W)$ (resp.\ $\sltc^\perp(W)$) and let $(\alpha_n^W)$, $(\alpha_n^{W^\perp})$ be i.i.d.\ $N(0,1)$ sequences.  Then we can write
\[ h = \sum_n \alpha_n^W f_n^W + \sum_n \alpha_n^{W^c} f_n^{W^c}.\]
The first summand has the law of a zero-boundary GFF on $W$, modulo additive constant, and the second summand is harmonic in $W$.  We note that we can view the first summand as a zero-boundary GFF on $W$ (i.e., with the additive constant fixed) because there is a unique distribution which represents the first summand which has the property that any test function whose support is disjoint from $W$ integrates to zero against it.

The second statement in the proposition is proved similarly.
\end{proof}

The following is immediate from the definitions:

\begin{proposition}
\label{prop::gff_radon_nikodym}
Suppose that $h$ is a whole-plane GFF defined up to a global additive constant in $\R$ and fix $g \in \sltc$.  Then the laws of $h+g$ and $h$ are mutually absolutely continuous.  The Radon-Nikodym derivative of the law of the former with respect to that of the latter is given by
\begin{equation}
\label{eqn::gff_radon_nikodym}
 \frac{1}{\CZ} e^{(h,g)_\nabla}
\end{equation}
where $\CZ$ is a normalization constant.  The same statement holds if $h$ is instead a whole-plane GFF defined up to a global additive constant in $r \Z$ for $r > 0$.  (Note: we can normalize so that $(h,\phi_0) \in [0,r)$ for some fixed $\phi_0 \in \slt$ whose support is disjoint from that of $g$.)
\end{proposition}
\begin{proof}
Recall that if $Z \sim N(0,1)$ then weighting the law of $Z$ by a normalizing constant times $e^{\mu x}$ yields the law of a $N(\mu,1)$.  We will deduce the result from this fact.   We write $h = \sum_n \alpha_n f_n$ where the $(\alpha_n)$ are i.i.d.\ $N(0,1)$ and $(f_n)$ is an orthonormal basis for $\sltc$.  Fix $g \in \sltc$.  Then we can write $g = \sum_n \beta_n f_n$.  Consequently, we have that
\[ \exp( (h,g)_\nabla) = \exp\left( \sum_n \alpha_n \beta_n \right),\]
which implies the result.
\end{proof}

\subsubsection{The whole-plane GFF as a limit}
\label{subsec::gff_convergence}

We are now going to show that infinite volume limits of the ordinary GFF converge to the whole-plane GFF modulo a global additive constant either in $\R$ or in $r \Z$ for $r > 0$ fixed, subject to appropriate normalization.  Suppose that $h$ is a zero-boundary GFF on a proper domain $D \subseteq \C$ with harmonically non-trivial boundary.  Then we can consider $h$ modulo additive constant in $\R$ by restricting $h$ to functions in $\sltz(D)$.  Fix $\phi_0 \in \slt(D)$ with $\int \phi_0(z) dz =~1$.  We can also consider the law of $h$ modulo additive constant in $r \Z$ for $r > 0$ fixed by replacing $h$ with $h-c$ where $c \in r \Z$ is chosen so that $(h,\phi_0)-c \in [0,r)$.

Let $\mu$ (resp.\ $\mu^r$) denote the law of the whole-plane GFF modulo additive constant in $\R$ (resp.\ $r \Z$ for $r >0$ fixed; the choice of $\phi_0$ will be clear from the context).
\begin{proposition}
\label{prop::gff_convergence}
Suppose that $D_n$ is any sequence of domains with harmonically non-trivial boundary containing $0$ such that $\dist(0,\partial D_n) \to \infty$ as $n \to \infty$ and, for each $n$, let $h_n$ be an instance of the GFF on $D_n$ with boundary conditions which are uniformly bounded in $n$.  Fix $R > 0$.  We have the following:
\begin{enumerate}[(i)]
\item\label{prop::gff_convergence_additive}
As $n \to \infty$, the laws of the distributions given by restricting the maps $\phi \mapsto (h_n,\phi)$ to $\phi\in \sltz \bigl( B(0,R) \bigr)$ (interpreted as distributions on $B(0,R)$ modulo additive constant) converge in total variation to the law of the distribution (modulo additive constant) obtained by restricting the whole-plane GFF to the same set of test functions.
\item\label{prop::gff_convergence_fixed} Fix $\phi_0 \in \slt$ with $\int \phi_0(z) dz = 1$ which is constant and positive on $B(0,R)$.  As $n \to \infty$, the laws of the distributions given by restricting the maps $\phi \mapsto (h_n,\phi)$ to $\phi\in \slt \bigl( B(0,R) \bigr)$ (interpreted as distributions modulo a global additive constant in $r \Z$) converge in total variation to the law of the analogous distribution (modulo $r \Z$) obtained from the whole-plane GFF (as defined modulo $r \Z$).
\end{enumerate}
\end{proposition}
\begin{proof}
Suppose that $h$ is a whole-plane GFF defined modulo additive constant in $\R$.  By Proposition~\ref{prop::whole_plane_markov}, the law of $h$ minus the harmonic extension $\Fh_n$ to $D_n$ of its boundary values on $\partial D_n$ (this difference does not depend on the arbitrary additive constant) is that of a zero-boundary GFF on $D_n$.

 If $\dist(0,\partial D_n)$ is large, $\Fh_n$ is likely to be nearly constant on $B(0,R)$.  Indeed, this can be seen as follows.  Let $p(z,y)$ be the density with respect to Lebesgue measure of harmonic measure in $B(0,2R)$ as seen from $z \in B(0,2R)$.  We also let $dy$ denote Lebesgue measure on $\partial B(0,2R)$.  Using in the second inequality that there exists a constant $C > 0$ such that $p(z,y) \leq C/R$ for all $z \in B(0,R)$ and $y \in \partial B(0,2R)$, we have that
\begin{align*}
      \sup_{z \in B(0,R)} | \Fh_n(z) - \Fh_n(0)|
&\leq \sup_{z \in B(0,R)} \int_{\partial B(0,2R)} p(z,y) | \Fh_n(y) - \Fh_n(0)| dy\\
&\leq \frac{C}{R} \int_{\partial B(0,2R)} |\Fh_n(y) - \Fh_n(0)| dy.
\end{align*}
The claim follows because it is not difficult to see that for each $\epsilon > 0$ and $R > 0$ there exists $n_0$ such that $n \geq n_0$ implies that $\var( \Fh_n(y) - \Fh_n(0)) \leq \epsilon^2$ for so that $\E| \Fh_n(y) - \Fh_n(0)| \leq \epsilon$.  We arrive at the first statement by combining this with~\eqref{eqn::gff_radon_nikodym}.

Now suppose that we are in the setting of part~\eqref{prop::gff_convergence_fixed}. Note that $(h_n,\phi_0)$ is a Gaussian random variable whose variance tends to $\infty$ with $n$.  Hence, $(h_n,\phi_0)$ modulo $r \Z$ becomes uniform in $[0,r)$ in the limit.  The second statement thus follows because, for each $n \in \N$ fixed, $(h_n,\phi_0)$ is independent of the family of random variables $(h_n,g)_\nabla$ where $g$ ranges over $\slt(B(0,R))$.
\end{proof}

\begin{proposition}
\label{prop::whole_plane_abs_continuity}
Suppose that $D \subseteq \C$ is a domain with harmonically non-trivial boundary and $h$ is a GFF on $D$ with given boundary conditions.  Fix $W \subseteq D$ bounded and open with $\dist(W, \partial D) > 0$.  The law of $h$ considered modulo a global additive constant restricted to $W$ is mutually absolutely continuous with respect to the law of the whole-plane GFF modulo a global additive constant restricted to $W$.  Likewise, the law of $h$ considered modulo additive constant in $r \Z$ restricted to $W$ is mutually absolutely continuous with respect to the whole-plane GFF modulo additive constant in $r \Z$ restricted to $W$.
\end{proposition}
\begin{proof}
The first statement just follows from the Markov properties of the ordinary and whole-plane GFFs and the fact that adding a smooth function affects the law of the field away from the boundary in an absolutely continuous manner (Proposition~\ref{prop::gff_radon_nikodym}).  For a fixed choice of normalizing function~$\phi_0$, the laws of the whole-plane GFF modulo an additive constant in~$r \Z$ and the ordinary GFF modulo an additive constant in~$r \Z$ each integrated against~$\phi_0$ are mutually absolutely continuous since the former is uniform on $[0,r)$ and the latter has the law of a Gaussian with positive variance taken modulo~$r$.  The second statement thus follows from the first by taking $\phi_0 \in \slt$ with $\int \phi_0(z) dz = 1$ and which is constant and positive on~$W$ and using that, with such a choice of~$\phi_0$, $(h,\phi_0)_\nabla$ is independent of the family of random variables $(h,g)_\nabla$ where~$g$ ranges over~$\slt(W)$.
\end{proof}

\subsubsection{Local sets}
\label{subsec::local_sets}

The notion of a local set, first introduced in \cite{SchrammShe10} for ordinary GFFs, serves to generalize the Markov property to the setting in which we condition the GFF on its values on a \emph{random} (rather than deterministic) closed set.  If $D$ is a planar domain and $h$ is a GFF on $D$ with some boundary conditions, then we say that a random closed set $A$ coupled with $h$ is a {\bf local set} if there exists a law on pairs $(A,h_1)$, where $h_1$ is a distribution with the property that $h_1|_{\C \setminus A}$ is a.s.\ a harmonic function, such that a sample with the law $(A,h)$ can be produced by
\begin{enumerate}
\item choosing the pair $(A,h_1)$,
\item then sampling an instance $h_2$ of the zero boundary GFF on $\C \setminus A$ and setting $h=h_1+h_2$.
\end{enumerate}
There are several equivalent definitions of local sets given in \cite[Lemma~3.9]{SchrammShe10}.

There is a completely analogous theory of local sets for the whole-plane GFF which we summarize here.  Suppose that $h$ is a whole-plane GFF defined modulo a global additive constant either in $\R$ or in $r \Z$ for $r > 0$ fixed, and suppose that $A \subseteq \C$ is a random closed subset which is coupled with $h$ such that $\C \setminus A$ has harmonically non-trivial boundary.  Then we say that $A$ is a {\bf local set} of $h$ if there exists a law on pairs $(A,h_1)$, where $h_1$ is a distribution with the property that $h_1|_{\C \setminus A}$ is a.s.\ a harmonic function, such that a sample with the law $(A,h)$ can be produced by
\begin{enumerate}
\item choosing the pair $(A,h_1)$,
\item then sampling an instance $h_2$ of the zero boundary GFF on $\C \setminus A$ and setting $h=h_1+h_2$,
\item then considering the equivalence class of distributions modulo additive constant (in $\R$ or $r\Z$) represented by $h$.
\end{enumerate}
The definition is equivalent if we consider $h_1$ as being defined only up to additive constant in $\R$ or $r \Z$.

Using this definition, Theorem~\ref{thm::existence} implies that the flow line $\eta$ of a whole-plane GFF drawn to any positive stopping time is a local set for the field modulo $2\pi \chi \Z$.  We will write $\CC_A$ for the $h_1$ described above (which is a harmonic function in the complement of $A$).  In this case, the set $A$ (the flow line) a.s.\ has Lebesgue measure zero, so $\CC_A$ can be interpreted as a random harmonic function a.s.\ defined a.e.\ in plane (and since this $\CC_A$ is a.s.\ locally a function in $L^1$, see the proof of \cite[Theorem~1.1]{SHE_WELD}, defined modulo additive constant, writing $(\CC_A, \phi) = \int \CC_A(z) \phi(z)dz$ allows us to interpret $\CC_A$ as a distribution modulo additive constant).  Here we consider $\CC_A$ to be defined only up to an additive constant in $\R$ (resp.\ $r \Z$) if $h$ is a whole-plane GFF modulo additive constant in $\R$ (resp.\ $r \Z$).  This function should be interpreted as the conditional mean of $h$ given $A$ and $h|_A$.

There is another way to think about what it means to be a local set.  Given a coupling of $h$ and $A$, we can let $\pi_h$ denote the conditional law of $A$ given $h$.  A local set determines the measurable map $h \mapsto \pi_h$, which is a regular version of the conditional probability of $A$ given $h$ \cite{SchrammShe10}.  (The map $h \mapsto \pi_h$ is uniquely defined up to re-definition on a set of measure zero.)  Let $B$ be a deterministic open subset of $\C$, and let $\pi^B_h$ be the measure for which $\pi^B_h(\mathcal A) = \pi_h (\mathcal A \cap \{A : A \subseteq B \})$.  (In other words, $\pi^B_h$ is obtained by restricting $\pi_h$ to the event $A \subseteq B$.)  Using this notation, the following is a restatement of part of \cite[Lemma~3.9]{SchrammShe10}:

\begin{proposition}
\label{prop::other_local_characterization}
Suppose $h$ is a GFF on a domain $D$ with harmonically non-trivial boundary and $A$ is a random closed subset of $D$ coupled with $h$.  Then $A$ is local for $h$ if and only if for each open set $B$ the map $h \mapsto \pi_h^B$ is (up to a set of measure zero) a measurable function of the restriction of $h$ to $B$.
\end{proposition}

A similar statement holds in the whole-plane setting:
\begin{proposition}
\label{prop::other_local_characterization}
Suppose $h$ is a whole-plane GFF defined modulo additive constant (in $\R$ or $r \Z$) coupled with a random closed subset $A$ of $\C$.  Then $A$ is local for $h$ if and only if the map $h \mapsto \pi_h^B$ is a measurable function of the restriction of $h$ to $B$ (a function that is invariant under the addition of a global additive constant in $\R$ or $r \Z$).
\end{proposition}

Suppose that $A_1,A_2$ are local sets coupled with a GFF $h$.  Then we will write $A_1 \wt{\cup} A_2$ for the random, closed set which is coupled with $h$ by first sampling $A_1,A_2$ conditionally independently given $h$ and then taking their union.  We refer to $A_1 \wt{\cup} A_2$ as the conditionally independent union of $A_1$ and $A_2$.  The following is a restatement of \cite[Lemma~3.10]{SchrammShe10} for the whole-plane GFF. 

\begin{proposition}
\label{prop::cond_union_local}
Suppose that $A_1$ and $A_2$ are local sets coupled with a whole-plane GFF $h$, defined either modulo additive constant in $\R$ or in $r \Z$ for $r > 0$ fixed.  Then the conditionally independent union $A_1 \wt{\cup} A_2$ of $A_1$ and $A_2$ given $h$ is a local set for $h$.
\end{proposition}
\begin{proof}
In \cite[Lemma~3.10]{SchrammShe10}, this statement was proved in the case that $h$ is a GFF on a domain $D$ with boundary, and the same proof works identically here.
\end{proof}

We note that in the case that $A_1$ and $A_2$ in Proposition~\ref{prop::cond_union_local} are $\sigma(h)$-measurable, the conditionally independent union of $A_1$ and $A_2$ is almost surely equal to the usual union.  By Theorem~\ref{thm::uniqueness}, we know that GFF flow lines are a.s.\ determined by the field.  Therefore, we will have \emph{a posteriori} that the conditionally independent union of flow lines is a.s.\ equal to the usual union, hence finite unions of flow lines are local.  The following is a restatement of \cite[Lemma~3.11]{SchrammShe10}; see also \cite[Proposition~3.8]{MS_IMAG}.

\begin{proposition}
\label{prop::cond_union_mean}
Let $A_1$ and $A_2$ be connected local sets of a whole-plane GFF $h$, defined either modulo additive constant in $\R$ or in $r \Z$ for $r > 0$ fixed.  Then $\CC_{A_1 \wt{\cup} A_2} - \CC_{A_2}$ is almost surely a harmonic function in $\C \setminus (A_1 \wt{\cup} A_2)$ that tends to zero along all sequences of points in $\C \setminus (A_1 \wt{\cup} A_2)$ that tend to a limit in a connected component of $A_2 \setminus A_1$ which consists of more than a single point.  The same is also true along all sequences of points that tend to a limit in a connected component of $A_1 \cap A_2$ which both consists of more than a single point and is at positive distance from either $A_1 \setminus A_2$ or $A_2 \setminus A_1$.
\end{proposition}
\begin{proof}
In \cite[Lemma~3.11]{SchrammShe10}, this statement was proved in the case that $h$ is a GFF on a domain $D$ with boundary, and the same proof works here.
\end{proof}

We emphasize that $\CC_{A_1 \wt{\cup} A_2} - \CC_{A_2}$ does not depend on the additive constant and hence is defined as a function on $\C \setminus (A_1 \wt{\cup} A_2)$.  We will prove our results regarding the existence, uniqueness, and interaction of GFF flow lines first in the context of whole-plane GFFs since the proofs are cleaner in this setting.  The following proposition provides the mechanism for converting these results into statements for the ordinary GFF.

\begin{proposition}
\label{prop::local_set_whole_plane_bounded_compare}
Suppose that $A$ is a local set for a whole-plane GFF $h$ defined up to a global additive constant in $\R$ or in $r \Z$ for $r > 0$ fixed.  Fix $W \subseteq \C$ open and bounded and assume that $A \subseteq W$ almost surely.  Let $D$ be a domain in $\C$ with harmonically non-trivial boundary such that $W \subseteq D$ with $\dist(W,\partial D) > 0$ and let $h_D$ be a GFF on $D$.  There exists a law on random closed sets $A_D$ which is mutually absolutely continuous with respect to the law of $A$ such that $A_D$ is a local set for $h_D$.  Let $\CC_{A_D}^\C$ be the function which is harmonic in $\C \setminus A$ which has the same boundary behavior as $\CC_{A_D}$ on $A_D$.  Then, moreover, the law of $\CC_{A_D}^\C$, up to the additive constant (taken in $\R$ or in $r \Z$, as above), and that of $\CC_A$ are mutually absolutely continuous.  Finally, if $A$ is almost surely determined by $h$ then $A_D$ is almost surely determined by $h_D$.
\end{proposition}
\begin{proof}
This follows from Proposition~\ref{prop::whole_plane_abs_continuity}.
\end{proof}

\subsection{Boundary emanating flow lines}
\label{subsec::imaginary}

\begin{figure}[ht!]
\begin{center}
\includegraphics[scale=0.85]{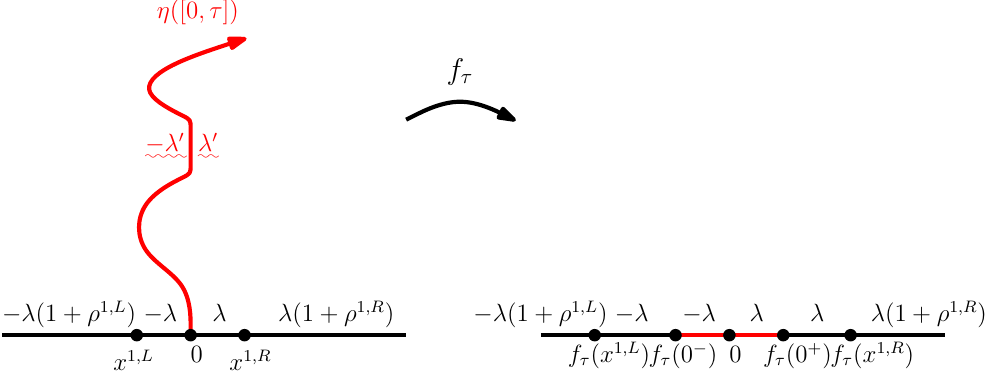}
\caption{\label{fig::conditional_boundary_data}  Suppose that $h$ is a GFF on $\h$ with the boundary data depicted above.  Then the flow line $\eta$ of $h$ starting from $0$ is an $\SLE_\kappa(\ul{\rho}^L;\ul{\rho}^R)$ curve in $\h$ where $|\ul{\rho}^L| = |\ul{\rho}^R| = 1$.  For any $\eta$-stopping time $\tau$, the law of $h$ given $\eta|_{[0,\tau]}$ is equal in distribution to a GFF on $\h \setminus \eta([0,\tau])$ with the boundary data depicted above (the notation $\uwave{a}$ is explained in Figure~\ref{fig::winding}).  It is also possible to couple $\eta' \sim\SLE_{\kappa'}(\ul{\rho}^L;\ul{\rho}^R)$ for $\kappa' > 4$ with $h$ and the boundary data takes on the same form (with $-\lambda' := \frac{\pi}{\sqrt \kappa'}$ in place of $\lambda := \frac{\pi}{\sqrt \kappa}$).  The difference is in the interpretation.  The (almost surely self-intersecting) path $\eta'$ is not a flow line of $h$, but for each $\eta'$-stopping time $\tau'$ the left and right {\em boundaries} of $\eta'([0,\tau'])$ are $\SLE_{\kappa}$ flow lines, where $\kappa=16/\kappa'$, angled in opposite directions.  The union of the left boundaries --- over a collection of $\tau'$ values --- is a tree of merging flow lines, while the union of the right boundaries is a corresponding dual tree whose branches do not cross those of the tree.}
\end{center}
\end{figure}

\begin{figure}[ht!]
\begin{center}
\subfigure[$\theta_1 \leq \theta_2-\tfrac{2\lambda}{\chi}+\pi$]{\includegraphics[scale=0.85]{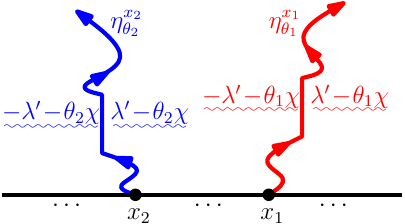}}
\hspace{0.05\textwidth}
\subfigure[$\theta_1 \in (\theta_2-\tfrac{2\lambda}{\chi}+\pi,\theta_2)$]{\includegraphics[scale=0.85]{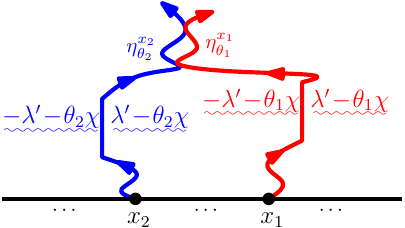}}
\subfigure[$\theta_1 = \theta_2$]{\includegraphics[scale=0.85]{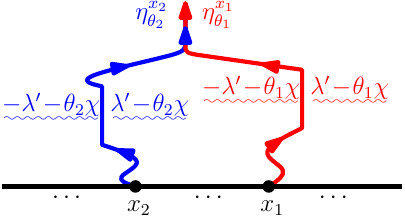}}
\end{center}
\caption{\label{fig::boundary_flowline_interaction1}  An illustration of the different types of flow line interaction, as proved in \cite[Theorem~1.5]{MS_IMAG} (continued in Figure~\ref{fig::boundary_flowline_interaction2}).  In each illustration, we suppose that $h$ is a GFF on $\h$ with piecewise constant boundary data with a finite number of changes and that $\eta_{\theta_i}^{x_i}$ is the flow line of $h$ starting at $x_i$ with angle $\theta_i$, i.e.\ a flow line of $h+\theta_i \chi$, for $i=1,2$, with $x_2 \leq x_1$.  If $\theta_1 \leq \theta_2-\tfrac{2\lambda}{\chi}+\pi$, then $\eta_{\theta_1}^{x_1}$ stays to the right of and does not intersect $\eta_{\theta_2}^{x_2}$.  If $\theta_1 \in (\theta_2-\tfrac{2\lambda}{\chi}+\pi,\theta_2)$, then $\eta_{\theta_1}^{x_1}$ stays to the right of but may bounce off $\eta_{\theta_2}^{x_2}$.  If $\theta_1 = \theta_2$, then $\eta_{\theta_1}^{x_1}$ merges with $\eta_{\theta_2}^{x_2}$ upon intersecting and the flow lines never separate thereafter.}
\end{figure}

\begin{figure}[ht!]
\begin{center}
\subfigure[$\theta_1 \in (\theta_2,\theta_2+\pi)$]{\includegraphics[scale=0.85]{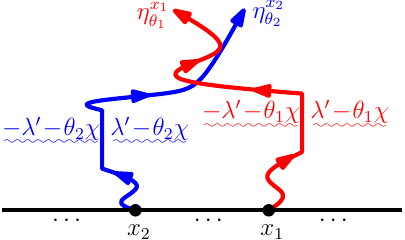}}
\hspace{0.05\textwidth}
\subfigure[$\theta_1 \geq \theta_2+\pi$]{\includegraphics[scale=0.85]{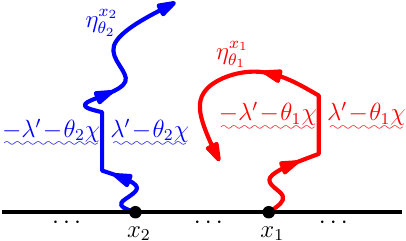}}
\end{center}
\caption{\label{fig::boundary_flowline_interaction2}  (Continuation of Figure~\ref{fig::boundary_flowline_interaction1}.)  If $\theta_1 \in (\theta_2,\theta_2+\pi)$, then $\eta_{\theta_1}^{x_1}$ crosses from the right to the left of $\eta_{\theta_2}^{x_2}$ upon intersecting.  After crossing, the flow lines can bounce off each other as illustrated but $\eta_{\theta_1}^{x_1}$ can never cross from the left back to the right.  Finally, if $\theta_1 \geq \theta_2+\pi$, then $\eta_{\theta_1}^{x_1}$ cannot hit the right side of $\eta_{\theta_2}^{x_2}$ except in $[x_2, x_1]$.
}
\end{figure}

We will now give a brief overview of the theory of boundary emanating GFF flow lines developed in \cite{MS_IMAG}.  We will explain just enough so that this article may be read and understood independently of \cite{MS_IMAG}, though we refer the reader to \cite{MS_IMAG} for proofs.  We assume throughout that $\kappa \in (0,4)$ so that $\kappa' := 16/\kappa \in (4,\infty)$.  We will often make use of the following definitions and identities:
\begin{align}
 \label{eqn::deflist} &\lambda := \frac{\pi}{\sqrt \kappa}, \,\,\,\,\,\,\,\,\lambda' := \frac{\pi}{\sqrt{16/\kappa}} = \frac{\pi \sqrt{\kappa}}{4} = \frac{\kappa}{4} \lambda < \lambda, \,\,\,\,\,\,\,\, \chi := \frac{2}{\sqrt \kappa} - \frac{\sqrt \kappa}{2}\\
 \label{eqn::fullrevolution} &\quad\quad\quad\quad\quad\quad\quad\quad 2 \pi \chi = 4(\lambda-\lambda'), \,\,\,\,\,\,\,\,\,\,\,\lambda' = \lambda - \frac{\pi}{2} \chi\\
\label{eqn::fullrevolutionrho} &\quad\quad\quad\quad\quad\quad\quad\quad\quad 2 \pi \chi = (4-\kappa)\lambda = (\kappa'-4)\lambda'.
\end{align}

\begin{figure}[ht!]
\begin{center}
\includegraphics[scale=0.85]{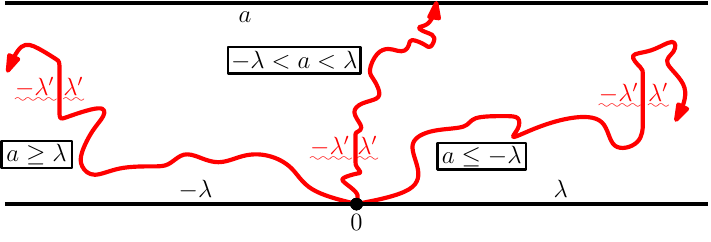}
\caption{\label{fig::hittingrange}  Suppose that $h$ is a GFF on the strip $\strip$ with the boundary data depicted above and let $\eta$ be the flow line of $h$ starting at $0$.  The interaction of $\eta$ with the upper boundary $\striptop$ of $\partial \strip$ depends on $a$, the boundary data of $h$ on $\striptop$.  Curves shown represent almost sure behaviors corresponding to the three different regimes of $a$ (indicated by the closed boxes).  The path hits $\striptop$ almost surely if and only if $a \in (-\lambda, \lambda)$.  When $a \geq \lambda$, it tends to $-\infty$ (left end of the strip) and when $a \leq - \lambda$ it tends to $+\infty$ (right end of the strip) without hitting $\striptop$.  If $\eta$ can hit the continuation threshold upon hitting some point on $\stripbot$, then $\eta$ only has the possibility of hitting $\striptop$ if $a \in (-\lambda,\lambda)$ (but does not necessarily do so); if $a \notin (-\lambda,\lambda)$ then $\eta$ almost surely does not hit $\striptop$.  By conformally mapping and applying~\eqref{eqn::ac_eq_rel}, we can similarly determine the ranges of boundary values that a flow line can hit for segments of the boundary with other orientations.}
\end{center}
\end{figure}

The boundary data one associates with the GFF on $\h$ so that its flow line $\eta$ from $0$ to $\infty$ is an $\SLE_\kappa(\ul{\rho}^L;\ul{\rho}^R)$ process with force points located at $\ul{x} = (\ul{x}^L;\ul{x}^R)$ for $\ul{x}^L = (x^{k,L} < \cdots < x^{1,L} \leq 0 ^-)$ and $\ul{x}^R = (0^+ \leq x^{1,R} < \cdots < x^{\ell,R})$ and with weights $(\ul{\rho}^L;\ul{\rho}^R)$ for $\ul{\rho}^L = (\rho^{1,L},\ldots,\rho^{k,L})$ and $\ul{\rho}^R = (\rho^{1,R},\ldots,\rho^{\ell,R})$ is
\begin{align}
 -&\lambda\left( 1 + \sum_{i=1}^j \rho^{i,L}\right) \quad\text{for}\quad x \in [x^{j+1,L},x^{j,L}) \quad\text{and}\\
 &\lambda\left( 1 + \sum_{i=1}^j \rho^{i,R}\right) \quad\text{for}\quad x \in [x^{j,R},x^{j+1,R})
\end{align}
This is depicted in Figure~\ref{fig::conditional_boundary_data} in the special case that $|\ul{\rho}^L| = |\ul{\rho}^R| = 1$.  For any $\eta$-stopping time $\tau$, the law of $h$ conditional on $\eta|_{[0,\tau]}$ is a GFF in $\h \setminus \eta([0,\tau])$.  The boundary data of the conditional field agrees with that of $h$ on $\partial \h$.  On the right side of $\eta([0,\tau])$, it is $\lambda' + \chi \cdot {\rm winding}$, where the terminology ``winding'' is explained in Figure~\ref{fig::winding}, and to the left it is $-\lambda' + \chi \cdot {\rm winding}$.  This is also depicted in Figure~\ref{fig::conditional_boundary_data}.

A complete description of the manner in which GFF flow lines interact with each other is given in \cite[Theorem~1.5]{MS_IMAG}.  In particular, we suppose that $h$ is a GFF on $\h$ with piecewise constant boundary data with a finite number of changes and $x_1,x_2 \in \R$, with $x_2 \leq x_1$.  Fix angles $\theta_1,\theta_2 \in \R$ and let $\eta_{\theta_i}^{x_i}$ be the flow line of $h$ with angle $\theta_i$ starting at $x_i$, i.e.\ a flow line of the field $h+ \theta_i \chi$, for $i=1,2$.  If $\theta_1 < \theta_2$, then $\eta_{\theta_1}^{x_1}$ almost surely stays to the right of $\eta_{\theta_2}^{x_2}$.  If, moreover, $\theta_1 \leq \theta_2 - \tfrac{2\lambda}{\chi} + \pi$ then $\eta_{\theta_1}^{x_1}$ almost surely does not hit $\eta_{\theta_2}^{x_2}$ and if $\theta_1 \in (\theta_2-\tfrac{2\lambda}{\chi} + \pi,\theta_2)$ then $\eta_{\theta_1}^{x_1}$ bounces off $\eta_{\theta_2}^{x_2}$ upon intersecting.  If $\theta_1 = \theta_2$, then $\eta_{\theta_1}^{x_1}$ merges with $\eta_{\theta_2}^{x_2}$ upon intersecting and the flow lines never separate thereafter.  If $\theta_1 \in (\theta_2,\theta_2+\pi)$, then $\eta_{\theta_1}^{x_1}$ crosses from the right to the left of $\eta_{\theta_2}^{x_2}$ upon intersecting.  After crossing, the flow lines may bounce off each other but $\eta_{\theta_1}^{x_1}$ can never cross back from the left to the right of $\eta_{\theta_2}^{x_2}$.  Finally, if $\theta_1 \geq \theta_2+\pi$, then $\eta_{\theta_1}^{x_1}$ cannot hit the right side of $\eta_{\theta_2}^{x_2}$ except in $\partial \h$.  Each of these possible behaviors is depicted in either Figure~\ref{fig::boundary_flowline_interaction1} or Figure~\ref{fig::boundary_flowline_interaction2}.

It is also possible to determine which segments of the boundary that a GFF flow line can hit and this is explained in the caption of Figure~\ref{fig::hittingrange}.  Using the transformation rule~\eqref{eqn::ac_eq_rel}, we can extract from Figure~\ref{fig::hittingrange} the values of the boundary data for the boundary segments that $\eta$ can hit with other orientations.  We can also rephrase this in terms of the weights $\ul{\rho}$: a chordal $\SLE_{\kappa}(\ul{\rho})$ process almost surely does not hit a boundary interval $(x^{i,R},x^{i+1,R})$ (resp.\ $(x^{i+1,L},x^i)$) if $\sum_{s=1}^i \rho^{s,R} \geq \tfrac{\kappa}{2}-2$ (resp.\ $\sum_{s=1}^i \rho^{s,L} \geq \tfrac{\kappa}{2}-2$).  See \cite[Lemma~5.2 and Remark~5.3]{MS_IMAG}.  These facts hold for all $\kappa > 0$.

\section{Interior flow lines}
\label{sec::interior_flowlines}

In this section, we will construct and develop the general theory of the flow lines of the GFF emanating from interior points.  We begin in Section~\ref{subsec::existence} by proving Theorem~\ref{thm::existence}, the existence of a unique coupling between a whole-plane $\SLE_\kappa(2-\kappa)$ process for $\kappa \in (0,4)$ and a whole-plane GFF $h$ so that $\eta$ may be thought of as the flow line of $h$ starting from $0$ and then use absolute continuity to extend this result to the case that $h$ is a GFF on a proper subdomain $D$ of $\C$.  Next, in Section~\ref{subsec::interaction}, we will give a description of the manner in which flow lines interact with each other and the domain boundary, thus proving Theorem~\ref{thm::flow_line_interaction} and Theorem~\ref{thm::merge_cross} for ordinary GFF flow lines.  In Section~\ref{subsec::conical}, we will explain how the proof of Theorem~\ref{thm::existence} can be extended in order to establish the existence component of Theorem~\ref{thm::alphabeta}, i.e.\ the existence of flow lines for the GFF plus a multiple of one or both of $\log| \cdot |$ and $\arg(\cdot)$.  We will also complete the proof of Theorem~\ref{thm::flow_line_interaction} and Theorem~\ref{thm::merge_cross} in their full generality.  Next, in Section~\ref{subsec::uniqueness}, we will prove that the flow lines are almost surely determined by the GFF, thus proving Theorem~\ref{thm::uniqueness} as well as completing the proof of Theorem~\ref{thm::alphabeta}.  This, in turn, will allow us to establish Theorem~\ref{thm::commutation}, Theorem~\ref{thm::field_determined_by_tree}, and Theorem~\ref{thm::conditional_law}.  In Section~\ref{subsec::transience} we will use the results of Section~\ref{subsec::interaction} and Section~\ref{subsec::conical} to prove the transience (resp.\ endpoint continuity) of whole-plane (resp.\ radial) $\SLE_\kappa^\mu(\rho)$ processes for $\kappa \in (0,4)$, $\rho > -2$, and $\mu \in \R$.  We finish in Section~\ref{subsec::critical_angle} with a discussion of the so-called critical angle as well as the self-intersections of GFF flow lines.  Throughout, we will make frequent use of the identities~\eqref{eqn::deflist}--\eqref{eqn::fullrevolutionrho}.

\subsection{Generating the coupling}
\label{subsec::existence}

\begin{figure}[ht!]
\begin{center}
\includegraphics[scale=0.85]{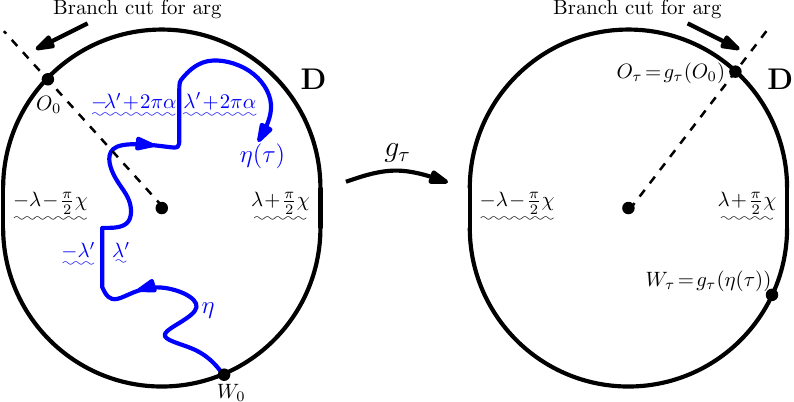}
\end{center}
\caption{\label{fig::radial_bd}  Fix $\alpha \in \R$, $\kappa \in (0,4)$, $W_0,O_0 \in \partial \D$, and suppose that $h$ is a GFF on $\D$ such that $h+\alpha \arg$ has the illustrated boundary values on $\partial \D$.   (The reason that $\D$ appears not to be perfectly round is to keep with our convention of labeling the boundary data only along vertical and horizontal segments.)  Let $\eta$ be the flow line of $h+\alpha \arg$ starting from $W_0$.  Then $\eta$ has the law of a radial $\SLE_\kappa(\rho)$ process with $\rho=\kappa-6+2\pi \alpha/\lambda$.  The conditional law of $h+\alpha \arg$ given $\eta|_{[0,\tau]}$, $\tau$ a stopping time for $\eta$, is that of a GFF on $\D \setminus \eta([0,\tau])$ plus $\alpha \arg$ so that the sum has the illustrated boundary data.  In particular, the boundary data is the same as that of $h+\alpha \arg$ on $\partial \D$ and is given by $-\alpha$-flow line boundary conditions on $\eta([0,\tau])$ (recall Figure~\ref{fig::interior_path_bd2}).  If one applies the change of coordinates $g_\tau$ as indicated and then moves the branch cut for $\arg(\cdot)$ so that it passes through $O_\tau$, then the boundary data for the field $h \circ g_\tau^{-1} + \alpha \arg(g_\tau^{-1}(0)) - \chi \arg (g_\tau^{-1})'$ is as illustrated on the right, up to an additive constant in $2\pi(\chi-\alpha) \Z$.}
\end{figure}

\begin{figure}[ht!]
\begin{center}
\includegraphics[scale=0.85]{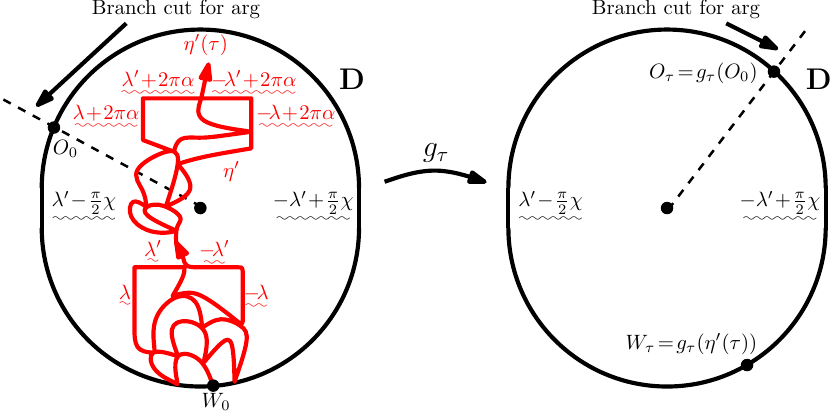}
\end{center}
\caption{\label{fig::radial_bd_cfl} Fix $\alpha \in \R$, $\kappa' \in (4,\infty)$, $W_0,O_0 \in \partial \D$, and suppose that $h$ is a GFF on $\D$ such that $h+\alpha \arg$ has the illustrated boundary values on $\partial \D$.  We take the branch cut for $\arg$ to be on the half-infinite line from $0$ through $O_0$.  Let $\eta'$ be the counterflow line of $h+\alpha \arg$ starting from $W_0$.  Then $\eta'$ has the law of a radial $\SLE_{\kappa'}(\rho)$ process with $\rho=\kappa'-6-2\pi \alpha / \lambda'$.  The conditional law of $h+\alpha \arg$ given $\eta'|_{[0,\tau]}$, $\tau$ a stopping time for $\eta'$, is that of a GFF on $\D \setminus \eta'([0,\tau])$ plus $\alpha \arg$ so that the sum has the illustrated boundary data.  In particular, the boundary data is the same as that of $h+\alpha \arg$ on $\partial \D$ and is given by $-\alpha$-flow line boundary conditions with angle $\tfrac{\pi}{2}$ (resp.\ $-\tfrac{\pi}{2}$) on the left (resp.\ right) side of $\eta'([0,\tau'])$ (recall Figure~\ref{fig::interior_path_bd2}).  If one applies the change of coordinates $g_\tau$ as indicated and then moves the branch cut for $\arg(\cdot)$ so that it passes through $O_\tau$, then the boundary data for the field $h \circ g_\tau^{-1} + \alpha \arg(g_\tau^{-1}(0)) - \chi \arg (g_\tau^{-1})'$ is as illustrated on the right, up to an additive constant in $2\pi(\chi-\alpha) \Z$.}
\end{figure}

\begin{figure}[h!]
\begin{center}
\includegraphics[scale=0.85]{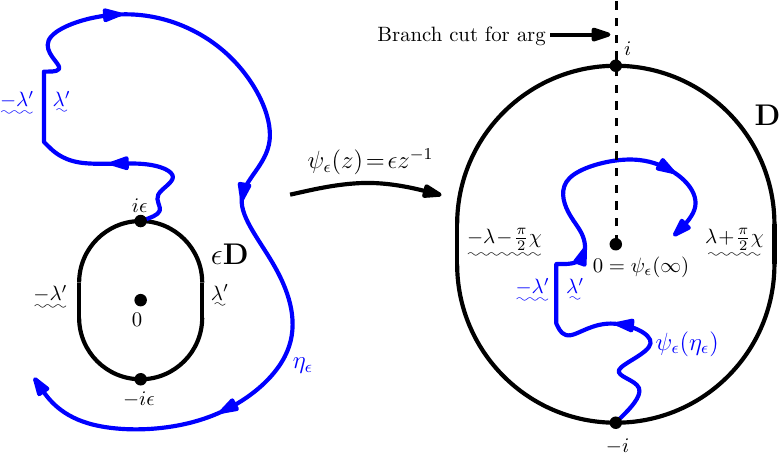}
\end{center}
\caption{\label{fig::diskstart} Suppose that $h_\epsilon$ is a GFF on $\C_\epsilon = \C \setminus (\epsilon \D)$ with the boundary data depicted on the left side.  Let $\eta_\epsilon$ be the flow line of $h_\epsilon$ starting at $i\epsilon$ and let $\psi_\epsilon \colon \C_\epsilon \to \D$ be the conformal map $\psi_\epsilon(z) = \epsilon/z$.  Then $\wt{h}_\epsilon = h_\epsilon \circ \psi_\epsilon^{-1} - \chi \arg(\psi_\epsilon^{-1})'$ is the sum of a GFF on $\D$ with the boundary data depicted on the right side plus $2\chi \arg(\cdot)$ minus the harmonic extension of its boundary values.  In particular, $\wt{h}_\epsilon$ has the boundary data as indicated on the right.  The branch cut for $\arg$ on the right is on the half-infinite vertical line from $0$ through $i$.  In particular, it follows from Proposition~\ref{prop::interior_force_point_coupling} (see also Figure~\ref{fig::radial_bd}) that $\psi_\epsilon(\eta_\epsilon)$ is a radial $\SLE_\kappa(\rho)$ process in $\D$ starting from $-i$ and targeted at $0$ with a single boundary force point of weight $\rho = \kappa-6+(2\pi)(2\chi)/\lambda = 2-\kappa$ located at~$i$.}
\end{figure}

In this section, we will establish the existence of a unique coupling between a whole-plane GFF $h$, defined modulo a global additive integer multiple of $2\pi \chi$, and a whole-plane $\SLE_\kappa(2-\kappa)$ process $\eta$ for $\kappa \in (0,4)$ emanating from $0$ which satisfies a certain Markov property.  Using absolute continuity (Proposition~\ref{prop::local_set_whole_plane_bounded_compare}), we will then deduce Theorem~\ref{thm::existence}.  The strategy of the proof is to consider, for each $\epsilon > 0$, the plane minus a small disk $\C_\epsilon \equiv \C \setminus (\epsilon \D)$ and then take $h_\epsilon$ to be a GFF on $\C_\epsilon$ with certain boundary conditions.  By the theory developed in \cite{MS_IMAG}, we know that there exists a flow line $\eta_\epsilon$ of $h_\epsilon$ emanating from $i\epsilon$ satisfying a certain Markov property.  Proposition~\ref{prop::interior_force_point_coupling}, stated and proved just below, will allow us to identify the law of this flow line as that of a radial $\SLE_\kappa(2-\kappa)$ process.  The results of Section~\ref{subsec::gff} imply that $h_\epsilon$ viewed as a distribution defined up to a global multiple of $2\pi \chi$ converges to a whole-plane GFF defined up to a global multiple of $2\pi \chi$ as $\epsilon \to 0$.   To complete the proof of the existence for the whole-plane coupling, we just need to show that $\eta_\epsilon$ converges to a whole-plane $\SLE_\kappa(2-\kappa)$ process as $\epsilon \to 0$ and that the pair $(h,\eta)$ satisfies the desired Markov property.  For proper subdomains $D$ in $\C$, the existence follows from the absolute continuity properties of the GFF (Proposition~\ref{prop::local_set_whole_plane_bounded_compare}).  We begin by recording the following proposition, which explains how to construct a coupling between radial $\SLE_\kappa(\rho)$ with a single boundary force point with the GFF.

\begin{proposition}
\label{prop::interior_force_point_coupling}
Fix $\alpha \in \R$ and $W_0,O_0 \in \partial \D$.  Suppose that $h$ is a GFF on $\D$ whose boundary conditions are chosen so that $h+\alpha \arg(\cdot)$ has the boundary data depicted in the left side of Figure~\ref{fig::radial_bd} if $\kappa \in (0,4)$ and in the left side of Figure~\ref{fig::radial_bd_cfl} if $\kappa' > 4$.  We take the branch cut for $\arg$ to be on the half-infinite line from $0$ through $O_0$.  That is, if $W_0 = -i$ and $\kappa \in (0,4)$ (resp.\ $\kappa' > 4$) then the boundary data for the field $h+\alpha \arg(\cdot)$ is equal to $-\lambda$ (resp.\ $\lambda'$) plus $\chi$ times the winding of $\partial \D$ on the clockwise segment of $\partial \D$ from $W_0=-i$ to $O_0$ and $\lambda$ (resp.\ $-\lambda'$) plus $\chi$ times the winding of $\partial \D$ on the counterclockwise segment of $\partial \D$ from $W_0=-i$ to $O_0$.  The boundary data is the same for other values of $W_0 \in \partial \D$ except it is shifted by the constant $\chi \big( \arg(W_0) - \tfrac{3 \pi}{2}\big)$ where here $\arg$ takes values in $[0,2\pi)$.

There exists a unique coupling between $h+\alpha \arg(\cdot)$ and a radial $\SLE_\kappa(\rho)$ process $\eta$ in $\D$ starting at $W_0$, targeted at $0$, and with a single boundary force point of weight $\rho = \kappa-6+2\pi \alpha/\lambda$ (resp.\ $\rho=\kappa'-6-2\pi\alpha/\lambda'$) if $\kappa \in (0,4)$ (resp.\ $\kappa' > 4$) located at $O_0$ satisfying the following Markov property.  For every $\eta$-stopping time $\tau$, the conditional law of $h+\alpha \arg(\cdot)$ given $\eta|_{[0,\tau]}$ is that of $\wt{h} + \alpha \arg(\cdot)$ where $\wt{h}$ is a GFF on $\D \setminus \eta([0,\tau])$ such that $\wt{h} + \alpha \arg(\cdot)$ has the same boundary conditions as $h+\alpha \arg(\cdot)$ on $\partial \D$.  If $\kappa \in (0,4)$, then $\wt{h}+\alpha \arg(\cdot)$ has $-\alpha$-flow line boundary conditions on $\eta([0,\tau])$.  If $\kappa' > 4$, then $\wt{h}+\alpha \arg(\cdot)$ has $-\alpha$-flow line boundary conditions with angle $\tfrac{\pi}{2}$ (resp.\ $-\tfrac{\pi}{2}$) on the left (resp.\ right) side of $\eta([0,\tau])$.

In this coupling, $\eta([0,\tau])$ is a local set for $h$.
\end{proposition}
In the $-\alpha$-flow line boundary conditions in the statement of Proposition~\ref{prop::interior_force_point_coupling}, the location of the branch cut in the argument function is the half-infinite line starting from $0$ through $O_0$.  This is in slight contrast to the flow line boundary conditions we introduced in Figure~\ref{fig::interior_path_bd2} in which the branch cut started from the initial point of the relevant path.  The reason that we take $-\alpha$ rather than $\alpha$-flow line boundary conditions along the path in Proposition~\ref{prop::interior_force_point_coupling} in contrast to Theorem~\ref{thm::alphabeta} is because in Proposition~\ref{prop::interior_force_point_coupling} we added $\alpha \arg(\cdot)$ to the GFF rather than subtracting it.  This sign difference arises because one transforms from the whole-plane setting to the radial setting by the inversion $z \mapsto 1/z$.

As in the case of \cite[Theorem~1.1]{MS_IMAG}, we note that Proposition~\ref{prop::interior_force_point_coupling} can be extracted from \cite{DUB_PART}.  See, in particular, \cite[Theorem~5.3 and Theorem~6.4]{DUB_PART}.  In order to have a proof which is independent of \cite{DUB_PART}, in what follows we will indicate the modifications that need to be made to the proof of \cite[Theorem~1.1]{MS_IMAG} in order to establish the result.

\begin{proof}[Proof of Proposition~\ref{prop::interior_force_point_coupling}]
Following the argument of the proof of \cite[Theorem~1.1]{MS_IMAG}, in order to prove the existence of the coupling it suffices to show that the analog of \cite[Lemma~3.11]{MS_IMAG} holds in the present setting.  By \cite[Theorem~3]{SCHRAMM_WILSON}, it suffices to prove that the analog of \cite[Theorem~1.1]{MS_IMAG} holds in the setting of chordal $\SLE_\kappa(\ul{\rho})$ with \emph{both} boundary and interior force points.  Indeed, this follows because \cite[Theorem~3]{SCHRAMM_WILSON} implies that a radial $\SLE_\kappa(\rho)$ process has the same law as a chordal $\SLE_\kappa(\ul{\rho})$ process with two force points: a boundary force point of weight $\rho$ with the same location as the force point of the radial process and an interior force point of weight $\kappa-6-\rho$ located at the target point of the radial process.  (We remark that it is possible to give a proof working purely in the radial setting, though the computations are simpler in the chordal setting with interior force points.  See \cite[Section~5]{qle2013} for a proof of the so-called reverse $\SLE$/GFF coupling in the radial setting, which contains similar computations.)

Recall from~\eqref{eqn::sle_kappa_rho_eqn} that the driving function for a chordal $\SLE_\kappa(\ul{\rho})$ process with force points starting from $z_1,\ldots,z_k \in \ol{\h}$ with weights $\ul{\rho} = (\rho_1,\ldots,\rho_k)$ is given by the solution to the SDE
\begin{align*}
   dW_t = \sqrt{\kappa} dB_t + \sum_{j=1}^k \re \left( \frac{\rho_j}{W_t - V_t^j} \right) dt \quad\text{and}\quad dV_t^j = \frac{2}{V_t^j - W_t} dt,\quad V_0^j = z_j,
\end{align*}
where $B$ is a standard Brownian motion.  We let $(g_t)$ be the chordal Loewner flow driven by $W$ and let $f_t = g_t - W_t$ be the associated centered chordal Loewner flow.  Let $\ol{\rho}$ be the sum of the weights of the force points contained in $\R_+$.  For each $t \geq 0$, we let
\begin{align*}
\Fh_t^*(z)
=& \frac{\pi i}{\sqrt{\kappa}} (\ol{\rho}+1) -\sum_{j=1}^k \frac{\rho_j}{2\sqrt{\kappa}} \left( \log( f_t(z) - f_t(z_j)) + \log( f_t(z) - \ol{f_t(z_j)} )\right)-\\
  & \frac{2}{\sqrt{\kappa}} \log f_t(z) - \chi \log f_t'(z)
\end{align*}
and we let
\begin{align*}
\Fh_t(z)
=& \im(\Fh_t^*(z))\\
=& \frac{\pi (\ol{\rho}+1)}{\sqrt{\kappa}} -\sum_{j=1}^k \frac{\rho_j}{2\sqrt{\kappa}} \left( \arg( f_t(z) - f_t(z_j)) + \arg( f_t(z) - \ol{f_t(z_j)} )\right)-\\
  & \frac{2}{\sqrt{\kappa}} \arg f_t(z) - \chi \arg f_t'(z).
\end{align*}
We note that if all of the~$z_j$ are in~$\R$ so that all of the~$V_t^j$ are in~$\R$ (i.e., we only have boundary force points) then~$\Fh_t^*$ and~$\Fh_t$ respectively agree with the corresponding definitions given in \cite[Equation~(2.12)]{MS_IMAG} and in the statement of \cite[Theorem~1.1]{MS_IMAG}.

If some of the~$z_j$ are in~$\h$, then $\Fh_t^*$ is a multi-valued function.  In order to make it single-valued (to justify our applications of It\^o's formula), we introduce branch cuts which start from each such $z_j \in \h$ given by a straight line to $\infty$.  We are now going to show that for each stopping time $\tau$ for $(W,V^j)$ which almost surely occurs before the continuation threshold is hit or one of the branch cuts is hit (in particular, before one of the interior force points is mapped to $\R$) we have that
\begin{equation}
\label{eqn::coupling_equation}
 \wh{h} \circ f_\tau + \Fh_\tau \stackrel{d}{=} \wh{h} + \Fh_0.
\end{equation}
This is the analog of \cite[Theorem~1.1]{MS_IMAG} in the setting of $\SLE_\kappa(\ul{\rho})$ with interior force points.  We will prove the result using It\^o calculus.

Suppose that $t < \tau$ (so that none of the branch cuts have been hit).  Applying It\^o's formula, we have that
\begin{align*}
     d f_t(z) &= \left( \frac{2}{f_t(z)} - \sum_{j=1}^k \re \left( \frac{\rho_j}{W_t - V_t^j} \right) \right) dt - \sqrt{\kappa} dB_t\\
          d \log f_t(z) &= \left( \frac{4-\kappa}{2 f_t^2(z)} - \sum_{j=1}^k \frac{1}{f_t(z)}\re\left( \frac{\rho_j}{W_t - V_t^j} \right) \right) dt - \frac{\sqrt{\kappa}}{f_t(z)} dB_t,\\
     d f_t'(z) &= -\frac{2 f_t'(z)}{f_t^2(z)} dt, \quad\text{and}\\
     d \log f_t'(z) &= \frac{-2}{f_t^2(z)} dt.
\end{align*}
Inserting these expressions into the explicit form of $\Fh_t^*$ and $\Fh_t$, we thus see that
\begin{equation}
\label{eqn::fh_equality}
d\Fh_t^*(z) = \frac{2}{f_t(z)} dB_t\quad\text{and}\quad d\Fh_t(z) = \im\left( \frac{2}{f_t(z)} \right) dB_t.	
\end{equation}
We recall that the Green's function for~$\Delta$ on~$\h$ with Dirichlet boundary conditions is given by
\[ G(z,w) = -\log|z-w| + \log|z-\ol{w}|\]
and that $G$ gives the covariance function for the GFF on $\h$ with Dirichlet boundary conditions.  We let $G_t(z,w) = G(f_t(z),f_t(w)) = G(g_t(z),g_t(w))$.  Then an elementary calculation implies that (see, e.g., \cite[Section~4.1]{SHE_WELD})
\begin{align}
\label{eqn::green_change}
 d G_t(z,w) &= -\im\left( \frac{2}{f_t(z)}\right) \im\left( \frac{2}{f_t(w)}\right) dt.
\end{align}
Combining~\eqref{eqn::fh_equality} and~\eqref{eqn::green_change}, we thus see that
\begin{align}
\label{eqn::fh_cov_eq}
 d \langle \Fh_t(z), \Fh_t(w) \rangle_t &= - dG_t(z,w).	
\end{align}
We recall from the proof of \cite[Lemma~3.11]{MS_IMAG} that~\eqref{eqn::fh_cov_eq} is the necessary equality to construct the coupling of $\SLE$ with the GFF as its flow or counterflow line.  The remainder of the existence of the coupling of $\SLE$ with the GFF thus follows from the same argument used to prove \cite[Theorem~1.1]{MS_IMAG}.

At this point, we proved the existence of the coupling of chordal $\SLE_\kappa(\ul{\rho})$ with the GFF with interior force points up until the first time that one of the branch cuts is hit or the continuation threshold is hit.  We note that, given the path up to a stopping time which occurs before this happens, the boundary conditions for the conditional law of the field along the path are $-\lambda'$ (resp.\ $\lambda'$) on the left (resp.\ right) side of the path plus $\chi$ times the winding.  That is, they are the same as in the usual chordal coupling \cite[Theorem~1.1]{MS_IMAG}.  Note that if we move one of the branch cuts so that it passes through the path, then there will be a discontinuity in the boundary data arising because of the discontinuity of the argument function along the branch cut.

We will now explain how to extend the coupling up until the first time that one of the interior force points is separated from $\infty$ or the continuation threshold is hit.  To this end, we let $\tau$ be the first time that one of the branch cuts is hit or the continuation threshold is hit.  Then we know that~\eqref{eqn::coupling_equation} holds up to time $\tau$.

We iterate this construction as follows.  We inductively define stopping times $(\tau_j)$ by taking $\tau_0 = \tau$.  For each $j \geq 1$, we take the branch cuts for the $\log$ singularities in $\Fh_t$ for $t = \tau_{j-1}$ to be given by vertical lines starting from each of the interior force points and through to $\infty$.  We then take $\tau_j$ to be the first time $t$ after $\tau_{j-1}$ that one of the branch cuts or the continuation threshold has been hit.  Iteratively applying~\eqref{eqn::coupling_equation} with these new branch cuts, we thus see that~\eqref{eqn::coupling_equation} holds for $t \leq T:= \sup_j \tau_j$.  We claim that $T$ is equal to the minimum of the first time that one of the force points is cut off from $\infty$ (equivalently, is mapped into $\R$) and when the continuation threshold is first hit.  To see this, we suppose that $T$ is strictly less than this time.  (In particular, $T < \infty$.)  It then follows that $\im(V_t^j)$ for each $j$ corresponding to an interior force point is bounded from below up to time $t \leq T$.    By the pigeon hole principle, there exists an index $j_0$ such that the number of times that the branch cut associated with $z_{j_0}$ is hit by time $T$ is infinite.  Elementary considerations for conformal mapping imply that there exists $c \in (0,1)$ such that $\im(V_{\tau_j}^{j_0}) \leq c \im(V_{\tau_{j-1}}^{j_0})$ for each $j \geq 1$  such that the path hits the branch cut associated with $z_{j_0}$ at time $\tau_j$.  This implies that $\im(V_t^{j_0})$ decreases to $0$ as $t \uparrow T$, which is a contradiction.

We have now proved the existence of the coupling of chordal $\SLE_\kappa(\ul{\rho})$ with the GFF, at least up until the first time that the process separates one of the interior force points from $\infty$ or the continuation threshold is hit.  Note that the path may pass through the branch cuts many times before this happens.  Due to the discontinuity of the argument function along each branch cut, the boundary data for the conditional law of the field given the path whenever it passes through such a branch cut jumps either up or down an amount which is equal to the corresponding jump discontinuity in the argument.

We will now explain how this implies the existence of the coupling of radial $\SLE_\kappa(\rho)$ with the GFF as in the statement of the proposition, at least up until the first time that the process separates its target point from a given marked boundary point.  By \cite[Theorem~3]{SCHRAMM_WILSON}, we know that a radial $\SLE_\kappa(\rho)$ on $\h$ with target point $i$ and force point located at $x \in \partial \h$ has the same law as a chordal $\SLE_\kappa(\rho,\kappa-6-\rho)$ process with a boundary force point of weight $\rho$ located at $x$ and an interior force point of weight $\kappa-6-\rho$ located at $i$.  Suppose that $x \geq 0$.  Then by the above argument, the latter is coupled with the field $h$ given by a GFF on $\h$ with boundary conditions $-\lambda$ on $\R_-$, $\lambda$ on $[0,x]$, and $\lambda(1+\rho)$ on $(x,\infty$), plus the function
\[ \alpha \left( \arg(z-i) + \arg(z+i) \right) \quad\text{where}\quad \alpha = -\frac{\kappa-6-\rho}{2\sqrt{\kappa}}.\]
(Note that this function vanishes on $\R$ and that $\alpha$ and $\rho$ are related as in the proposition statement.)  The boundary conditions are analogous in the case that $x \leq 0$.  Let $\varphi$ be the conformal map $\h \to \D$ with $\varphi(i) = 0$ and $\varphi(0) = -i$.  Consider the field $\wh{h} \circ \varphi^{-1} - \chi \arg (\varphi^{-1})'$ on $\D$.  Note that it has boundary conditions given by $-\lambda$ plus $\chi$ times the winding of the boundary on $\varphi((-\infty,0))$, $\lambda$ plus $\chi$ times the winding of the boundary on $\varphi((0,x))$, and $\lambda(1+\rho)$ plus $\chi$ times the winding of the boundary on $\varphi((x,\infty))$.  More generally, if we take $\varphi$ so that $\varphi(i) = 0$ and $\varphi(0) = W_0 \in \partial \D$, then the boundary data of the field is the same as in the case that $\varphi(0) = -i$ except it is shifted by the constant $\chi \left(\arg(W_0) - \tfrac{3\pi}{2}\right)$, where here the argument takes values in $[0,2\pi)$.  As
\[ 2\pi \alpha = \lambda (2+\rho) + 2\pi \chi,\]
we find that moving the branch cut so that it passes through $\varphi(x)$ yields the boundary data as indicated in the statement of the proposition.  (Note that the sign difference in the case that $\kappa' > 4$ is because counterflow lines are coupled with $-h$.)

This proves the existence of the coupling of radial $\SLE_\kappa(\rho)$ with the GFF as stated in the proposition, at least up until the first time that the process separates its target point from a given marked boundary point.  At this time, one can then ``continue'' the coupling by picking a new marked boundary point inside of the complementary component containing the target point and then repeating the above with this point as the target point.  Repeating this completes the proof of existence.

The uniqueness of the coupling follows from the same argument used to prove \cite[Theorem~2.4]{MS_IMAG}.  Namely, if $W,V^1,\ldots,V^k$ is a collection of continuous processes such that $\Fh_t(z)$ as defined just above evolves as a continuous local martingale then it follows that they form a solution to~\eqref{eqn::sle_kappa_rho_eqn}.
\end{proof}

We now have the ingredients to complete the proof of Theorem~\ref{thm::existence}.

\begin{figure}[h!]
\begin{center}
\includegraphics[scale=0.85]{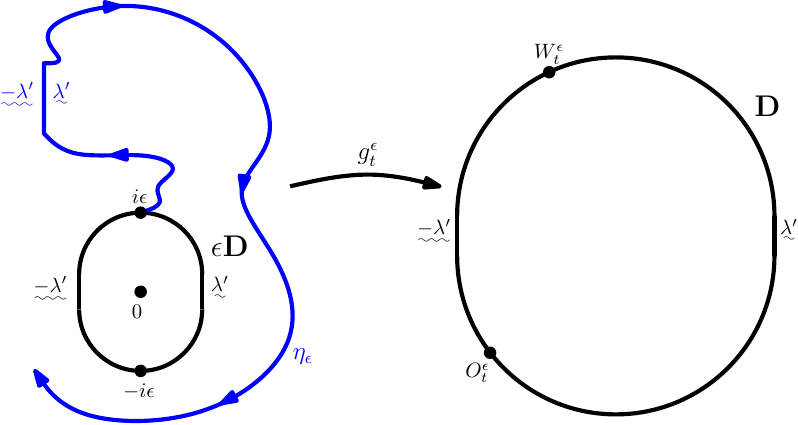}
\end{center}
\caption{\label{fig::disk_bd} Suppose that $h_\epsilon$ is a GFF on $\C_\epsilon = \C \setminus (\epsilon \D)$ with the boundary data depicted on the left side.  Let $\eta_\epsilon$ be the flow line of $h_\epsilon$ starting at $i\epsilon$.  Assume that $\eta_\epsilon \colon [T_\epsilon,\infty) \to \C$ is parameterized so that $t$ is the capacity of $(\epsilon \D) \cup \eta([T_\epsilon,t])$.  In particular, $T_\epsilon$ is the capacity of $\epsilon \D$.  For each $t$, let $g_t^\epsilon$ be the conformal map which takes the unbounded connected component $\C_{t,\epsilon}$ of $\C \setminus (\epsilon \D \cup \eta_\epsilon([T_\epsilon,t]))$ to $\C \setminus \D$ with $g_t^\epsilon(\infty) = \infty$ and $(g_t^\epsilon)'(\infty) > 0$.  Then the conditional law of $h_\epsilon$ given $\eta_\epsilon|_{[T_\epsilon,t]}$ in $\C_{t,\epsilon}$ is equal to the law of the sum $\wt{h} \circ g_t^\epsilon + F_t^{\epsilon} \circ g_t^\epsilon - \chi \arg (g_t^{\epsilon})'$ where $\wt{h}$ is a zero boundary GFF on $\C \setminus \D$ independent of $\eta|_{[T_\epsilon,t]}$ and $F_t^\epsilon$ is the harmonic function on $\C \setminus \D$ with the boundary data as indicated on the right side where $(W_t^\epsilon,O_t^\epsilon)$ is the whole-plane Loewner driving pair of $\eta_\epsilon$.
}
\end{figure}

\begin{proof}[Proof of Theorem~\ref{thm::existence}, whole-plane case]
We will first prove the existence of the coupling in the case that $D = \C$.  For each $\epsilon > 0$, let $h_\epsilon$ be a GFF on $\C_\epsilon = \C \setminus (\epsilon \D)$ whose boundary data is as depicted in the left side of Figure~\ref{fig::diskstart}.  By \cite[Theorem~1.1 and Proposition~3.4]{MS_IMAG}, it follows that we can uniquely generate the flow line $\eta_\epsilon$ of $h_\epsilon$ starting at $i \epsilon$.  In other words,~$\eta_\epsilon$ is an almost surely continuous path coupled with~$h_\epsilon$ such that for every $\eta_\epsilon$-stopping time~$\tau$, the conditional law of $h_\epsilon$ given $\eta_\epsilon|_{[0,\tau]}$ is that of a GFF on $\C_\epsilon \setminus \eta_\epsilon([0,\tau])$ whose boundary conditions agree with those of $h_\epsilon$ on $\partial \C_\epsilon$ and are given by flow line boundary conditions on $\eta_\epsilon([0,\tau])$.  Moreover, as explained in the caption of Figure~\ref{fig::diskstart}, we can read off the law of the path $\eta_\epsilon$: it is given by that of a radial $\SLE_\kappa(2-\kappa)$ process starting at $i \epsilon$ and targeted at $\infty$.  Let $T_\epsilon$ be the capacity of $\epsilon \D$ and assume that $\eta_\epsilon$ is defined on the time interval $[T_\epsilon,\infty)$; note that $T_\epsilon \to -\infty$ as $\epsilon \to 0$.  For each $t > T_\epsilon$, we let $g_t^\epsilon$ be the unique conformal map which takes the unbounded connected component $\C_{t,\epsilon}$ of $\C \setminus (\epsilon \D \cup \eta_\epsilon([T_\epsilon,t]))$ to $\C \setminus \D$ with $g_t^\epsilon(\infty) = \infty$ and $(g_t^\epsilon)'(\infty) > 0$.  We assume that $\eta_\epsilon \colon [T_\epsilon,\infty) \to \C$ is parameterized by capacity, i.e\ $-\log (g_t^\epsilon)'(\infty) = t$ for every $t \geq T_\epsilon$.  Let $W_t^\epsilon = g_t^\epsilon(\eta_\epsilon(t)) \in \partial \D$ be the image of the tip of $\eta_\epsilon$ in $\partial \D$.  This is the whole-plane Loewner driving function of $\eta_\epsilon$.  As explained in the caption of Figure~\ref{fig::diskstart}, we know that $W_t^\epsilon|_{[T_\epsilon,\infty)}$ solves the radial $\SLE_\kappa(\rho)$ SDE~\eqref{eqn::sle_radial_equation} with $\rho = 2-\kappa$; let $O_t^\epsilon$ denote the time evolution of the corresponding force point.  Proposition~\ref{prop::sle_kappa_rho_stationary} states that this SDE has a unique stationary solution $(W_t,O_t)$ for $t \in \R$ and that $(W^\epsilon,O^\epsilon)$ converges weakly to $(W,O)$ as $\epsilon \to 0$ with respect to the topology of local uniform convergence.  This implies that the family of conformal maps $(g_t^\epsilon)$ converge weakly to $(g_t)$, the whole-plane Loewner evolution driven by $W_t$, also with respect to the topology of local uniform convergence \cite[Section~4.7]{LAW05}.  The corresponding GFFs $h_\epsilon$ viewed as distributions defined up to a global multiple of $2\pi \chi$ converge to a whole-plane GFF $h$ which is also defined up to a global multiple of $2\pi\chi$ as $\epsilon \to 0$ by Proposition~\ref{prop::gff_convergence}.

For each $\epsilon > 0$ and each stopping time $\tau$ for $\eta_\epsilon$, we can write
\begin{equation}
\label{eqn::epsilon_markov_property}
 h_\epsilon|_{\C_{\tau,\epsilon}} \stackrel{d}{=} \wt{h} \circ g_\tau^\epsilon + F_\tau^\epsilon \circ g_\tau^\epsilon  - \chi \arg (g_\tau^\epsilon)'
\end{equation}
where $\wt{h}$ is a zero boundary GFF independent of $\eta_\epsilon|_{[T_\epsilon,\tau]}$ on $\C \setminus \D$ and $F_\tau^\epsilon$ is the function which is harmonic on $\C \setminus \D$ whose boundary values are $\lambda'+\chi \cdot {\rm winding}$ on the counterclockwise segment of $\partial \D$ from $O_\tau^\epsilon$ to $W_\tau^\epsilon$ and $-\lambda'+\chi \cdot {\rm winding}$ on the clockwise segment; see Figure~\ref{fig::disk_bd} for an illustration.  Fix an open set $U \subseteq \C$ and assume that~$\tau$ is a stopping time for~$\eta_\epsilon$ which almost surely occurs before the first time~$t$ that~$\eta_\epsilon$ hits~$U$.  Then~\eqref{eqn::epsilon_markov_property} implies that
\begin{equation}
\label{eqn::epsilon_markov_property_restricted}
h_\epsilon|_U \stackrel{d}{=} \big( \wt{h} \circ g_\tau^\epsilon + F_\tau^\epsilon \circ g_\tau^\epsilon  - \chi \arg (g_\tau^\epsilon)' \big)|_U.
\end{equation}
where the right hand side is viewed as distribution on $U$ with values modulo $2\pi\chi$.

The convergence of $(W^\epsilon,O^\epsilon)$ to $(W,O)$ implies that the functions $F_t^\epsilon$ converge locally uniformly to $F_t$, the harmonic function on $\C \setminus \D$ defined analogously to $F_t^\epsilon$ but with $(W_t^\epsilon,O_t^\epsilon)$ replaced by $(W_t,O_t)$.  Taking a limit as $\epsilon \to 0$ of both sides of~\eqref{eqn::epsilon_markov_property_restricted}, we see that if $\tau$ is any stopping time which almost surely occurs before $\eta$ hits $U$ then the law of
\[ \big( \wt{h} \circ g_\tau + F_\tau \circ g_\tau - \chi \arg g_\tau' \big)|_U,\]
viewed as a distribution on $U$ with values modulo $2\pi \chi$, is equal to that of a whole-plane GFF restricted to $U$.  The existence of the coupling then follows by applying the argument used to deduce \cite[Theorem~1.1]{MS_IMAG} from \cite[Lemma~3.11]{MS_IMAG}.

We will now prove the uniqueness of the coupling.  Suppose that $(h,\eta)$ is a coupling of a whole-plane GFF with values modulo $2\pi \chi$ and a path $\eta$ such that for each $\eta$-stopping time $\tau$ we have that the conditional law of $h$ given $\eta|_{(-\infty,\tau]}$ is that of a GFF on $\h \setminus \eta((-\infty,\tau])$ with flow line boundary conditions on $\eta((-\infty,\tau])$.  We will prove the uniqueness by showing that $\eta$ necessarily has the law of a whole-plane $\SLE_\kappa(2-\kappa)$ process from $0$ to $\infty$.  This then determines the joint law of the pair $(h,\eta)$ because the Markov property determines the conditional law of $h$ given $\eta$.  Fix $t \in \R$ and let $\varphi$ be the unique conformal map which takes the unbounded component of $\C \setminus \eta((-\infty,t])$ to $\D$ with $\infty$ sent to $0$ and with positive derivative at $\infty$.  Then we know that $\wt{h} = h \circ \varphi^{-1} - \chi \arg( (\varphi^{-1})')$ can be written as the sum of a GFF on $\D$ minus $2\chi \arg(\cdot)$ with boundary conditions on $\partial \D$ as in the statement of Proposition~\ref{prop::interior_force_point_coupling}.  Moreover, with $\wt{\eta} = \varphi(\eta|_{[t,\infty)})$, the Markov property for the pair $(h,\eta)$ implies that the pair $(\wt{h},\wt{\eta})$ satisfies the analogous Markov property (i.e., as described in the statement of Proposition~\ref{prop::interior_force_point_coupling}).  The uniqueness component of Proposition~\ref{prop::interior_force_point_coupling} then implies that $\wt{\eta}$ has the law of a radial $\SLE_{\kappa}(2-\kappa)$ process in $\D$ hence $\eta|_{[t,\infty)}$ has the law of a radial $\SLE_\kappa(2-\kappa)$ process in the unbounded component of $\C \setminus \eta((-\infty,t])$.  Sending $t \to -\infty$ implies that $\eta$ has the law of a whole-plane $\SLE_\kappa(2-\kappa)$ process from $0$ to $\infty$.
\end{proof}

\begin{proof}[Proof of Theorem~\ref{thm::existence}, existence for general domains]
We will now extend the existence of the coupling to general domains $D$; we will defer the proof of uniqueness of the coupling until we prove Theorem~\ref{thm::uniqueness} later on.  The key observation is that the law of~$h$ (modulo a global multiple of $2\pi \chi$) in a small neighborhood of zero changes in an absolutely continuous way when we replace $\C$ with $D$.  Hence, we can couple the path with the field {\em as if} the domain were $\C$, at least up until the first time the path exits this small neighborhood (see also the discussion just after \cite[Lemma~3.6]{MS_IMAG} regarding the Markov property for GFF flow lines when performing this type of change of measure in the context of boundary emanating flow lines).  The {\em actual} value of the field on the boundary of $\eta|_{[0,\tau]}$ (as opposed to just the value modulo a global multiple of $2 \pi \chi$) is then a Gaussian random variable restricted to a discrete set of possible values.  See Proposition~\ref{prop::local_set_whole_plane_bounded_compare}. Once the path has been drawn to a stopping time $\tau > - \infty$ (where the path is parameterized by capacity), the usual coupling rules \cite[Theorem~1.1 and Theorem~1.2]{MS_IMAG} allow us to extend it uniquely.
\end{proof}

So far we have shown that there is a unique coupling $(h,\eta)$ between a path $\eta$ and a whole-plane GFF $h$ with values modulo $2\pi \chi$ such that the conditional law of $h$ given~$\eta$ up to any $\eta$-stopping time $\tau$ is that of GFF in $\C \setminus \eta([0,\tau])$ with flow line boundary conditions on $\eta([0,\tau])$ and $\eta([0,\tau])$ is local for $h$.  We proved the existence of the coupling in the case of general domains $D$ by starting with the construction in the whole-plane case and then using absolute continuity.  One could worry that there are other possible laws.  This will be ruled out as a byproduct of our proof of Theorem~\ref{thm::uniqueness} below.  Unless explicitly stated otherwise, we shall assume throughout in what follows that a flow line on a bounded domain has the law as constructed just above (i.e., induced from the whole-plane coupling).

Suppose that $h$ is a GFF on $\C$ with values modulo $2\pi \chi$, $z_1,z_2 \in \C$ are distinct, and that $\eta_1,\eta_2$ are flow lines of $h$ starting from $z_1,z_2$, respectively, taken to be conditionally independent given $h$.  Suppose that $\tau_1$ is a stopping time for $\eta_1$.  Then we know that $\eta_2$ is a flow line for the GFF on $\C \setminus \eta_1((-\infty,\tau_1])$ given by $h$ given $\eta_1|_{(-\infty,\tau_1]}$.  Indeed, this follows because we know that for each $\eta_2$-stopping $\tau_2$ that the conditional law of $h$ given both $\eta_1|_{(-\infty,\tau_1]}$ and $\eta_2|_{(-\infty,\tau_2]}$ is that of a GFF on $\C \setminus (\eta_1((-\infty,\tau_1]) \cup \eta_2((-\infty,\tau_2]))$ with flow line boundary conditions on $\eta_1((-\infty,\tau_1])$ and $\eta_2((-\infty,\tau_2])$.  We claim that the coupling of $\eta_2$ with the GFF $h$ on $\C \setminus \eta_1((-\infty,\tau_1])$ given $\eta_1|_{(-\infty,\tau_1]}$ is the same as the one constructed in the proof of the existence component of Theorem~\ref{thm::existence} given just above.  (This will be important in some of our conditioning arguments in the proof of Theorem~\ref{thm::flow_line_interaction}, before we complete the proof of the uniqueness component of Theorem~\ref{thm::existence} for bounded domains given below.)  Suppose that $U \subseteq \C$ is an open set which contains $z_2$ but not $z_1$.  The locality of $\eta_2$ implies that the conditional law of $\eta_2$ up until exiting $U$ given $h$ is a function of the restriction of~$h$ to~$U$.  Moreover, our construction in the proof of the existence component of Theorem~\ref{thm::existence} given just above yields that the conditional law of the flow line given~$h$ up until exiting a subdomain~$U$ with boundary disjoint from $\partial D$ is given by the same function of the values of the restriction of~$h$ to~$U$.  This, in turn, implies the claim.

\begin{remark}
\label{rem::interior_rho_value}  The proof of Theorem~\ref{thm::existence} implies that in the special case that $D$ is a proper domain in $\C$ with harmonically non-trivial boundary, the law of $\eta$ stopped before hitting $\partial D$ is absolutely continuous with respect to that of a whole-plane $\SLE_\kappa(2-\kappa)$ process.  If $D = \C$, then $\eta$ is in fact a whole-plane $\SLE_\kappa(2-\kappa)$.  In particular, $\eta$ intersects itself if and only if $\kappa \in (8/3,4)$ by Lemma~\ref{lem::radial_critical_for_hitting}.
\end{remark}

\begin{remark}
\label{rem::alpha_coupling}
In the case that $D = \C$, if we replace the whole-plane GFF $h$ in the proof of Theorem~\ref{thm::existence} given just above with $h_\alpha = h- \alpha \arg(\cdot)$, viewed as a distribution defined modulo a global multiple of $2\pi(\chi+\alpha)$, the same argument yields the existence of a coupling $(h_\alpha,\eta)$ where $\eta$ is a whole-plane $\SLE_\kappa(\rho)$ for $\kappa \in (0,4)$ starting from $0$ with
\[ \rho = \rho(\alpha) = 2-\kappa+\frac{2\pi\alpha}{\lambda}\]
satisfying the analogous Markov property (the conditional law of the field given a segment of the path is a GFF off the path with $\alpha$-flow line boundary conditions; recall Figure~\ref{fig::interior_path_bd2}).  In order for this to make sense, we need to assume that $\alpha > -\chi$ so that $\rho > -2$.  The same proof also yields the existence of a coupling $(h_\alpha,\eta')$ where $\eta'$ is a whole-plane $\SLE_{\kappa'}(\rho)$ for $\kappa' > 4$ starting from $0$ with
\[ \rho = \rho(\alpha) = 2-\kappa'-\frac{2\pi\alpha}{\lambda'}\]
satisfying the analogous Markov property provided $\rho > -2$.

By applying the inversion $z \mapsto 1/z$, it is also possible to construct a coupling $(h_\alpha,\eta')$ where $\eta'$ is a whole-plane $\SLE_{\kappa'}(\rho)$ process starting from $\infty$ with
\[ \rho = \rho(\alpha) = \kappa'-6+\frac{2\pi\alpha}{\lambda'}.\]
This completes the proof of the existence component of Theorem~\ref{thm::alphabeta} for $\beta = 0$.  In Section~\ref{subsec::conical}, we will explain how to extend the existence statement to $\beta \neq 0$.
\end{remark}

\subsection{Interaction}
\label{subsec::interaction}

\begin{figure}[h!]
\begin{center}
\includegraphics[scale=0.85]{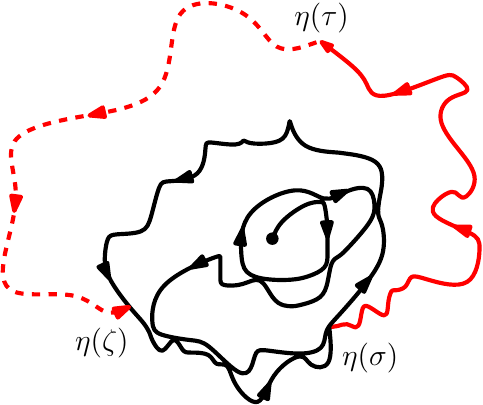}
\end{center}
\caption{\label{fig::tailsketch}
Suppose that $\eta$ is a flow line of a GFF and that $\tau$ is a finite stopping time for $\eta$ such that $\eta(\tau) \notin \eta((-\infty,\tau))$ almost surely.  The tail of $\eta$ associated with $\tau$, denoted by $\eta^\tau$, is $\eta|_{[\sigma,\zeta]}$ where $\sigma = \sup\{s \leq \tau : \eta(s) \in \eta((-\infty,s))\}$ and $\zeta = \inf\{s \geq \tau : \eta(s) \in \eta((-\infty,s))\}$.  The red segment above is a tail of the illustrated path.  In this subsection, we will describe the interaction of tails of flow lines using the boundary emanating theory from \cite[Theorem~1.5]{MS_IMAG} (see also Figure~\ref{fig::boundary_flowline_interaction1} and Figure~\ref{fig::boundary_flowline_interaction2}) by using absolute continuity.  We will then complete the proof of Theorem~\ref{thm::flow_line_interaction} by showing in Proposition~\ref{prop::tail_decomposition} that every flow line can be decomposed into a union of overlapping tails.}
\end{figure}

\begin{figure}[h!]
\begin{center}
\includegraphics[scale=0.7]{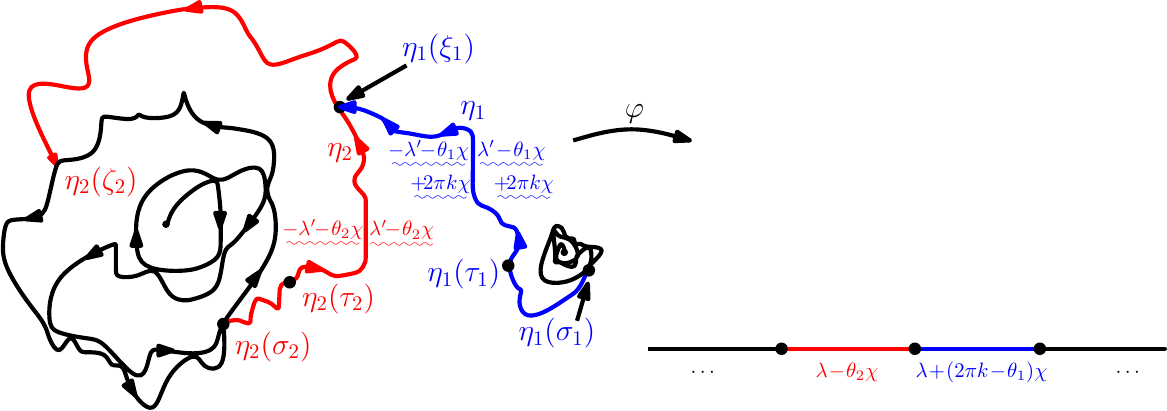}
\end{center}
\caption{\label{fig::doubletailsketch}  (Continuation of Figure~\ref{fig::tailsketch}.)  Let $h$ be a GFF on $\C$ defined up to a global multiple of $2\pi \chi$.  Fix $z_1,z_2$ distinct and $\theta_1,\theta_2 \in \R$.  For $i=1,2$, let $\eta_i$ be the flow line of $h$ starting at $z_i$ of angle $\theta_i$ and let $\tau_i$ be a stopping time for $\eta_i$ such that $\eta_i(\tau_i) \notin \eta_i((-\infty,\tau_i))$ almost surely.  Let $\sigma_i,\zeta_i$ be the start and end times, respectively, for the tail $\eta_i^{\tau_i}$ for $i=1,2$.  Let $\xi_1$ be the first time that $\eta_1$ hits $\eta_2((-\infty,\zeta_2])$.  Assume that we are working on the event that $\xi_1 \in (\tau_1,\zeta_1)$, $\eta_1(\xi_1) \in \eta_2((\tau_2,\zeta_2))$, and that $\eta_1^{\tau_1}$ hits $\eta_2^{\tau_2}$ on its right side at time $\xi_1$.  The boundary data for the conditional law of $h$ on $\eta_1((-\infty,\xi_1])$ and $\eta_2((-\infty,\zeta_2])$ is given by flow line boundary conditions with angle $\theta_i$, as in Figure~\ref{fig::interior_path_bd} where $k \in \Z$, up to an additive constant in $2\pi \chi \Z$.  Let $\CD = (2\pi k +\theta_2-\theta_1)\chi$ be the height difference of the tails upon intersecting, as described in Figure~\ref{fig::angle_difference}.  Proposition~\ref{prop::tail_interaction} states that $\CD \in (-\pi \chi,2\lambda-\pi \chi)$.  Moreover, if $\CD \in (-\pi \chi,0)$, then $\eta_1^{\tau_1}$ crosses $\eta_2^{\tau_2}$ upon intersecting but does not cross back.  If $\CD = 0$, then $\eta_1^{\tau_1}$ merges with $\eta_2^{\tau_2}$ upon intersecting and does not separate thereafter.  Finally, if $\CD \in (0,2\lambda-\pi \chi)$, then $\eta_1^{\tau_1}$ bounces off but does not cross $\eta_2^{\tau_2}$.  This describes the interaction of $\eta_1^{\tau_1}$ and $\eta_2^{\tau_2}$ up until any pair of times $\wt{\tau}_1$ and $\wt{\tau}_2$ such that $\eta_1((-\infty,\wt{\tau}_1]) \cup \eta_2((-\infty,\wt{\tau}_2])$ does not separate either $\eta_1(\sigma_1)$ or~$\eta_2(\sigma_2)$ from~$\infty$.  The idea is to use absolute continuity to reduce this to the interaction result for boundary emanating flow lines.  The case when~$\eta_1^{\tau_1}$ hits $\eta_2^{\tau_2}$ on its left side is analogous.  The same argument also shows that the conditional mean of~$h$ given the tails does not exhibit pathological behavior at the intersection points of~$\eta_1^{\tau_1}$ and $\eta_2^{\tau_2}$.
}
\end{figure}

\begin{figure}[ht!]
\begin{center}
\includegraphics[scale=0.7]{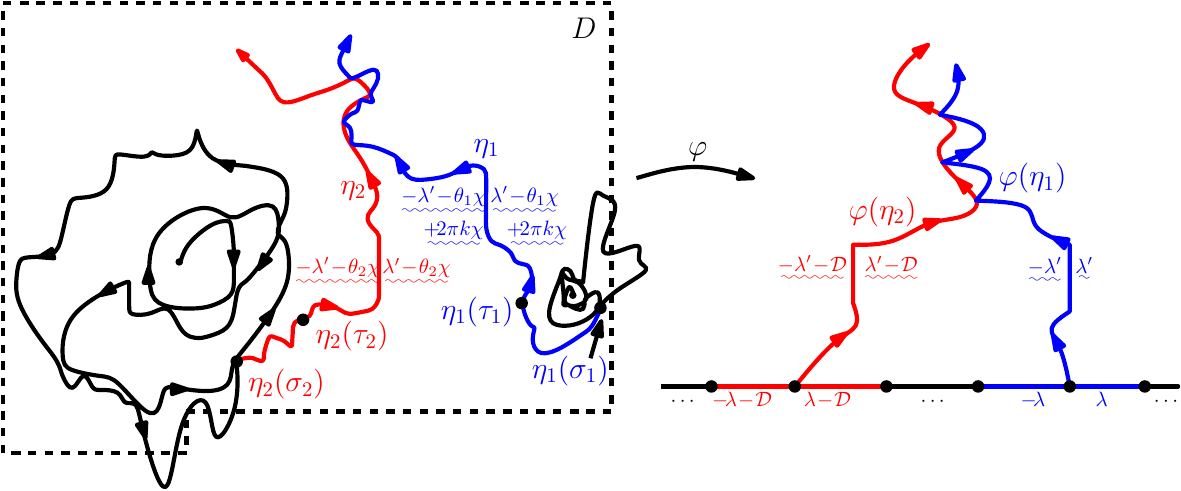}
\end{center}
\vspace{-0.035\textheight}
\caption{\label{fig::doubletailproof}  (Continuation of Figure~\ref{fig::doubletailsketch}.) To prove Proposition~\ref{prop::tail_interaction}, we let $D_0 \subseteq \C$ be a bounded domain with boundary which can be written as a finite union of linear segments which connect points with rational coordinates.  We assume that we are working on the event that $\eta_i(\sigma_i) \in \ol{D}_0$ for $i=1,2$ such that the intersection~$D$ of~$D_0$ and the unbounded complementary connected components of $\eta_1((-\infty,\tau_1])$ and $\eta_2((-\infty,\tau_2])$ is simply connected with $\eta_i(\sigma_i) \in \partial D$ for $i=1,2$.  In the illustration, $D_0$ is the region bounded by the dotted lines and $D$ is the region which is bounded by the dotted lines \emph{and} not disconnected from $\infty$ by $\eta_i|_{(-\infty,\tau_i]}$ for $i=1,2$.  Let $\varphi \colon D \to \h$ be a conformal map which takes $\eta_1(\tau_1)$ to~$1$ and $\eta_2(\tau_2)$ to $-1$.  Assume that the additive constant for $h$ in $2\pi \chi \Z$ has been chosen so that $\wt{h} = h \circ \varphi^{-1} - \chi \arg (\varphi^{-1})' +\theta_1 \chi$ is a GFF on $\h$ whose boundary data is $-\lambda$ (resp.\ $\lambda$) immediately to the left (resp.\ right) of~$1$.  Then $\varphi(\eta_1)$ (resp.\ $\varphi(\eta_2)$) is the zero (resp.\ $\CD/\chi$) angle  flow line of $\wt{h}$ starting at~$1$ (resp.\ $-1$).  Consequently, it follows from \cite[Theorem~1.5]{MS_IMAG} (see also Figure~\ref{fig::boundary_flowline_interaction1} and Figure~\ref{fig::boundary_flowline_interaction2}) and \cite[Proposition~3.4]{MS_IMAG} that $\varphi(\eta_1)$ and $\varphi(\eta_2)$ respect monotonicity if $\CD > 0$ but may bounce off each other if $\CD \in (0,2\lambda-\pi \chi)$, merge upon intersecting if $\CD = 0$, and cross exactly once upon intersecting if $\CD \in (-\pi \chi,0)$.  Moreover, $\varphi(\eta_1)$ can hit $\varphi(\eta_2)$ only if $\CD \in (-\pi \chi,2\lambda-\pi \chi)$.  These facts hold up until stopping times~$\wt{\tau}_i$ satisfying $\tau_i \leq \wt{\tau}_i \leq \zeta_i$ for $i=1,2$ on the event that~$\eta_i(\sigma_i)$ is not disconnected from~$\infty$ by $\eta_1((-\infty,\wt{\tau}_1]) \cap \eta_2((-\infty,\wt{\tau}_2])$ and $\eta_i([\tau_i,\wt{\tau}_i]) \subseteq D$ for $i=1,2$.  We note that the boundary data of the field $h \circ \varphi^{-1} - \chi \arg (\varphi^{-1})'$ is not piecewise constant, in particular on $\varphi(\partial D_0)$.  However, by using absolute continuity, the interaction of the paths up until hitting $\varphi(\partial D)$ can still be deduced from the piecewise constant case.  This describes the interaction of $\varphi(\eta_1)$ and $\varphi(\eta_2)$ up to any pair of stopping times before the paths hit~$\partial \h$ (which corresponds to describing the interaction of $\eta_1^{\tau_1}$ and~$\eta_2^{\tau_2}$ up until exiting~$D$).  Applying this result for all such domains~$D_0$ as described above completes the proof of Proposition~\ref{prop::tail_interaction}.
}
\end{figure}

In this subsection, we will study the manner in which flow lines interact with each other and the domain boundary in order to prove Theorem~\ref{thm::flow_line_interaction} and Theorem~\ref{thm::merge_cross} for ordinary GFF flow lines.  The strategy to establish the first result is to reduce it to the setting of boundary emanating flow lines, as described in Section~\ref{subsec::imaginary} (see Figure~\ref{fig::boundary_flowline_interaction1} and Figure~\ref{fig::boundary_flowline_interaction2} as well as \cite[Theorem~1.5]{MS_IMAG}).  This will require three steps.

\begin{enumerate}
\item In Section~\ref{subsubsec::tail_interaction}, we will show that the so-called \emph{tails} of flow lines --- path segments in between self-intersection times --- behave in the same manner as in the boundary emanating regime.  This is made precise in Proposition~\ref{prop::tail_interaction}; see Figure~\ref{fig::doubletailsketch} and Figure~\ref{fig::doubletailproof} for an illustration of the setup and proof of this result.
\item In Section~\ref{subsubsec::tail_decomposition}, we show that every flow line starting from an interior point can be decomposed into a union of overlapping tails (Proposition~\ref{prop::tail_decomposition}) and that it is possible to represent any segment of a tail with a further tail (Lemma~\ref{lem::tail_decomposition}).
\item In Section~\ref{subsubsec::flow_line_interaction}, we will explain how this completes the proof of our description of the manner in which flow lines interact.  At a high level, the result follows in the case of flow lines starting from interior points because in this case the interaction of any two flow lines reduces to the interaction of flow line tails.  The proof for the case in which a flow line starting from an interior point interacts with a flow line starting from the boundary follows from the same argument because, as it will not be hard to see from what follows, it is also possible to decompose a flow line starting from the boundary into a union of overlapping tails of flow lines starting from interior points.  This result is stated precisely as Proposition~\ref{prop::boundary_tail_decomposition}.
\end{enumerate}

\begin{remark}
\label{rem::multiple_paths_not_determined}
At this point in the article we have not proved Theorem~\ref{thm::uniqueness}, that flow lines of the GFF emanating from interior points are almost surely determined by the field, yet throughout this section we will work with more than one path coupled with the GFF.  We shall tacitly assume that the paths are conditionally independent given the field. That is, when we refer to {\em the} flow line of a given angle and starting point, and {\em the} flow line of another angle and starting point, we at this point assume only that the laws of these paths are conditionally independent given the field (since we have not yet shown that the flow line is a deterministic function of the field).  We also emphasize that, by the uniqueness theory for boundary emanating flow lines \cite[Theorem~1.2]{MS_IMAG} and absolute continuity \cite[Proposition~3.4]{MS_IMAG}, once we have drawn an infinitesimal segment of each path, the remainder of each of the paths \emph{is} almost surely determined by the field.  The only possible source of randomness is in how the path gets started.
\end{remark}

\subsubsection{Tail interaction}
\label{subsubsec::tail_interaction}

Suppose that $D \subseteq \C$ is a domain and let $h$ be a GFF on $D$ with given boundary conditions; if $D = \C$ then we take $h$ to be a whole-plane GFF defined up to a global multiple of $2\pi \chi$.  Fix $z \in D$, $\theta \in \R$, and let $\eta = \eta_\theta^z$ be the flow line of $h$ starting at $z$ with angle $\theta$.  This means that $\eta$ is the flow line of $h+\theta \chi$ starting at $z$.  Let $\tau$ be a stopping time for $\eta$ at which $\eta$ almost surely does not intersect its past, i.e.\ $\eta(\tau)$ is not contained in $\eta((-\infty,\tau))$.  We let $\sigma = \sup\{s \leq \tau : \eta(s) \in \eta((-\infty,s))\}$ be the largest time before $\tau$ at which $\eta$ intersects its past and $\zeta = \inf\{s \geq \tau : \eta(s) \in \eta((-\infty,s))\}$.  By Proposition~\ref{prop::whole_plane_continuity}, we know that $\eta$ is almost surely continuous, hence $\sigma \neq \tau \neq \zeta$.  We call the path segment $\eta|_{[\sigma,\zeta]}$ the {\bf tail} of~$\eta$ associated with the stopping time $\tau$ and will denote it by $\eta^\tau$ (see Figure~\ref{fig::tailsketch}).  The next proposition, which is illustrated in Figure~\ref{fig::doubletailsketch}, describes the manner in which the tails of flow lines interact with each other up until~$\infty$ is disconnected from the initial point of one of the tails.

\begin{proposition}
\label{prop::tail_interaction}
Let $h$ be a GFF on $\C$ defined up to a global multiple of $2\pi \chi$.  Suppose that $z_1,z_2 \in \C$ and $\theta_1,\theta_2 \in \R$.  For $i=1,2$, we let $\eta_i$ be the flow line of $h$ starting at $z_i$ with angle $\theta_i$ and let $\tau_i$ be a stopping time for $\eta_i$ such that $\eta_i(\tau_i) \notin \eta_i((-\infty,\tau_i))$ almost surely.  Let $\sigma_i,\zeta_i$ be the starting and ending times for the tail $\eta_i^{\tau_i}$ for $i=1,2$ as described above.  Let $\xi_1 = \inf\{t \in \R : \eta_1(t) \in \eta_2((-\infty,\zeta_2])\}$ and let $E_1$ be the event that $\xi_1 \in (\tau_1,\zeta_1)$, $\eta_1(\xi_1) \in \eta_2((\tau_2,\zeta_2))$, and that $\eta_1^{\tau_1}$ hits $\eta_2^{\tau_2}$ on its right side at time $\xi_1$.  Let $\CD$ be the height difference of the tails upon intersecting, as defined in Figure~\ref{fig::angle_difference}.  Then $\CD \in (-\pi \chi, 2\lambda-\pi \chi)$.

Let $\tau_i \leq \wt{\tau}_i \leq \zeta_i$ be a stopping time for $\eta_i$ for $i=1,2$ and let $E_2$ be the event that $\eta_i(\sigma_i)$ for $i=1,2$ is not disconnected from $\infty$ by $\eta_1((-\infty,\wt{\tau}_1]) \cup \eta_2((-\infty,\wt{\tau}_2])$.  Let $\wt{\eta}_i = \eta_i|_{(-\infty,\wt{\tau}_i]}$ for $i=1,2$.  On $E = E_1 \cap E_2$, we have that:
\begin{enumerate}[(i)]
\item If $\CD \in (-\pi \chi,0)$, then $\wt{\eta}_1$ crosses $\wt{\eta}_2$ upon intersecting and does not subsequently cross back,
\item If $\CD = 0$, then $\wt{\eta}_1$ merges with and does not subsequently separate from $\wt{\eta}_2$ upon intersecting, and
\item If $\CD \in (0,2\lambda-\pi \chi)$, then $\wt{\eta}_1$ bounces off but does not cross $\wt{\eta}_2$.
\end{enumerate}
Finally, the conditional law of $h$ given $\wt{\eta}_i|_{(-\infty,\wt{\tau}_i]}$ for $i=1,2$ on $E$ is that of a GFF on $\C \setminus (\wt{\eta}_1((-\infty,\wt{\tau}_1]) \cup \wt{\eta}_2((-\infty,\wt{\tau}_2]))$ whose boundary data on $\wt{\eta}_i((-\infty,\wt{\tau}_i])$ for $i=1,2$ is given by flow line boundary conditions with angle $\theta_i$.  If $\eta_1^{\tau_1}$ hits $\eta_2^{\tau_2}$ on its left rather than right side, the same result holds but with $-\CD$ in place of $\CD$ (so the range of values for $\CD$ where $\eta_1^{\tau_1}$ can hit $\eta_2^{\tau_2}$ is $(\pi \chi - 2\lambda,\pi \chi)$).
\end{proposition}
\begin{proof}
The proof is contained in the captions of Figure~\ref{fig::doubletailsketch} and Figure~\ref{fig::doubletailproof}, except for the following two points.  First, the reason that we know that $A = \eta_1((-\infty,\xi_1]) \cup \eta_2((-\infty,\zeta_2])$ is a local set for $h$ is that, if we draw each of the paths up until any fixed stopping time, then their union is local by Proposition~\ref{prop::cond_union_local} (recall that we took the paths to be conditionally independent given $h$).  Their continuations are local for and, moreover, almost surely determined by the conditional field given these initial segments (recall Remark~\ref{rem::multiple_paths_not_determined}).  Hence the claim follows from \cite[Lemma~6.2]{MS_IMAG}.  Second, the reason that the boundary data for $h$ given $A$ and $h|_A$ is given by flow line boundary conditions is that we were working on the event that $\eta_1((-\infty,\xi_1)) \cap \eta_2((-\infty,\zeta_2]) = \emptyset$.  Consequently, we can get the boundary data for $h$ given $A$ and $h|_A$ by using Proposition~\ref{prop::cond_union_mean} to compare to the conditional law of $h$ given $\eta_1$ and $h$ given $\eta_2$ separately.  The reason that the boundary data for the conditional law of $h$ given $\wt{\eta}_1$ and $\wt{\eta}_2$ has flow line boundary conditions (without singularities at intersection points) is that we can apply the boundary emanating theory from \cite{MS_IMAG}.
\end{proof}

\begin{figure}[h!]
\begin{center}
\includegraphics[scale=0.85]{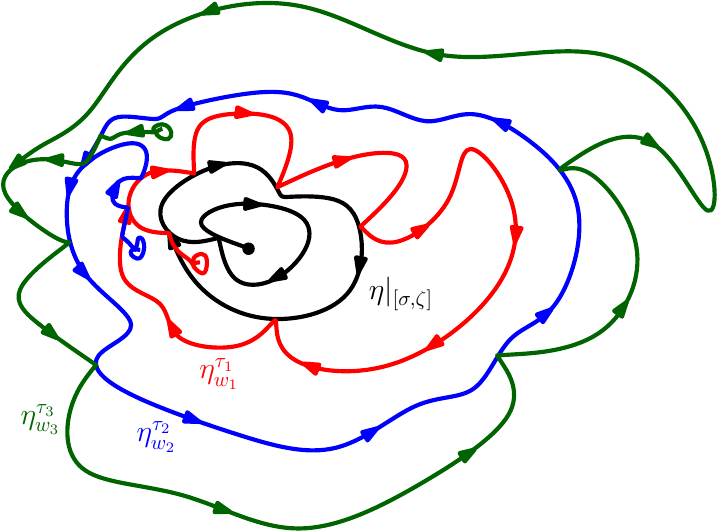}
\end{center}
\caption{\label{fig::overlappingtails} Suppose that $\kappa \in (8/3,4)$ --- this is the range of $\kappa$ values in which GFF flow lines started from interior points are self-intersecting --- and that $\eta$ is a flow line of a whole-plane GFF defined up to a global multiple of $2\pi \chi$ starting at $z \in \C$.  Let $\zeta$ be any stopping time for $\eta$ such that there exists $\sigma < \zeta$ so that $\eta|_{[\sigma,\zeta]}$ forms a clean loop around $z$, i.e.\ the loop does not intersect the past of the path except where the terminal point hits the initial point.  We prove in Proposition~\ref{prop::tail_decomposition} that it is possible to represent $\eta|_{[\zeta,\infty)}$ as a union of overlapping tails $(\eta_{w_i}^{\tau_i} : i \in \N)$.  These are depicted in the illustration above by different colors.  As shown, the decomposition has the property that the tails $\eta_{w_i}^{\tau_i}$ give the outer boundary of $\eta$ at successive times at which $\eta$ wraps around its starting point and intersects itself.  Moreover, for each $i$, the initial point $w_i$ has rational coordinates and is contained in a bounded complementary component of $\eta_{w_{i-1}}^{\tau_{i-1}}$.  Finally, conditional on $\eta|_{(-\infty,\zeta_i]}$, $\eta_{w_i}^{\tau_i}$ is independent of $h$ restricted to the unbounded complementary component of $\eta((-\infty,\zeta_i])$.  This allows us to reduce the interaction of flow lines to the interaction of tails, which we already described in Proposition~\ref{prop::tail_interaction}.}
\end{figure}

The reason that the statement of Proposition~\ref{prop::tail_interaction} is more complicated than the statement which describes the interaction of boundary emanating flow lines \cite[Theorem~1.5]{MS_IMAG} is that we needed a way to encode the height difference between $\eta_1$ and $\eta_2$ upon intersecting since the paths have the possibility of winding around their initial points many times after the stopping times $\tau_1$ and $\tau_2$.

\subsubsection{Decomposing flow lines into tails}
\label{subsubsec::tail_decomposition}

Now that we have described the interaction of tails of flow lines, we turn to show that it is possible to decompose a flow line into a union of overlapping tails.  This combined with Proposition~\ref{prop::tail_interaction} will lead to the proof of Theorem~\ref{thm::flow_line_interaction}.  Suppose that~$\eta$ is a non-crossing path starting at~$z$.  Then we say that~$\eta$ has made a {\bf clean loop} at time~$\zeta$ around~$w$ if the following is true.  With $\sigma = \sup\{t < \zeta : \eta(t) = \eta(\zeta)\}$, we have that $\eta|_{[\sigma,\zeta]}$ is a simple loop which surrounds $w$ and does not intersect $\eta((-\infty,\sigma))$.  Note that Lemma~\ref{lem::radial_loop} implies that GFF flow lines for $\kappa \in (8/3,4)$, which we recall are given by whole-plane $\SLE_\kappa(\rho)$ processes for $\rho = 2-\kappa \in (-2,\tfrac{\kappa}{2}-2)$ for $D = \C$ (and if $D \neq \C$, their law is absolutely continuous with respect to that of whole-plane $\SLE_\kappa(\rho)$ up until hitting $\partial D$), almost surely contain arbitrarily small clean loops.

\begin{proposition}[Tail Decomposition: Interior Regime]
\label{prop::tail_decomposition}
Assume that $\kappa \in (8/3,4)$.  Suppose that $h$ is a whole-plane GFF defined up to a global multiple of $2\pi \chi$.  Fix $z \in \C$, $\theta \in \R$, and let $\eta$ be the flow line of $h$ starting from $z$ with angle $\theta$.  Let $\zeta$ be any stopping time for $\eta$ such that $\eta$ has made a clean loop around $z$ at time $\zeta$, as described just above.  Then we can decompose $\eta|_{[\zeta,\infty)}$ into a union of overlapping tails $(\eta_{w_i}^{\tau_i} : i \in \N)$ of flow lines $(\eta_{w_i})$ where $\eta_{w_i}$ starts from $w_i$ and has angle $\theta$ with the following properties:
\begin{enumerate}[(i)]
\item For every $i \in \N$, the starting point $w_i$ of the flow line $\eta_{w_i}$ of $h$ with angle $\theta$ has rational coordinates and is contained in a bounded complementary component of $\eta_{w_{i-1}}^{\tau_{i-1}}$ (we take $\eta_{w_0}^{\tau_0} \equiv \eta|_{(-\infty,\zeta]}$),
\item There exists stopping times $\zeta_0 \equiv \zeta < \zeta_1 < \zeta_2 < \cdots$ for $\eta$ such that, for each $i \in \N$, the outer boundary of $\eta([\zeta_{i-1},\zeta_i])$ is almost surely equal to the outer boundary of $\eta_{w_i}^{\tau_i}$, and
\item For each $i \in \N$, conditional on $\eta|_{(-\infty,\zeta_i]}$, $\eta_{w_i}^{\tau_i}$ is independent of $h$ restricted to the unbounded complementary connected component of $\eta((-\infty,\zeta_i])$.
\end{enumerate}
\end{proposition}

We emphasize that in the statement of Proposition~\ref{prop::tail_decomposition}, the flow lines $(\eta_{w_i})$ are taken to be conditionally independent of $\eta$ given $h$.  The starting points $w_i$ of the flow lines $\eta_{w_i}$ are rational, as stated in the proposition statement, but are random and depend on $h$, $\eta$, and the collection of flow lines $\eta_w$ starting at $w \in \C$ with rational coordinates (as in the statement, all flow lines are taken to be conditionally independent given $h$).  We note that the tail of a flow line can either be a simple curve or a simple loop (i.e., homeomorphic to $\s^1$) depending on when the first intersection time of the path with itself occurs.  It is implicit in the second condition of Proposition~\ref{prop::tail_decomposition} that the tails $(\eta_{w_i}^{\tau_i} : i \in \N)$ are of the latter type.

The reason that we assume $\kappa \in (8/3,4)$ in the statement of Proposition~\ref{prop::tail_decomposition} is that if $\kappa \in (0,8/3]$ then $\eta$ is non-self-intersecting hence its entire trace is itself a tail.  Fix $T \in \R$ and let $\zeta$ be the first time after $T$ that $\zeta$ makes a clean loop around $z$.  It is a consequence of Lemma~\ref{lem::radial_loop} that $\zeta < \infty$ almost surely and it follows from Proposition~\ref{prop::sle_kappa_rho_stationary}, which gives the stationarity of the driving function of $\eta$, that the distribution of $\zeta - T$ does not depend on $T$.  Consequently, for every $\epsilon > 0$ and $S \in \R$ we can choose our stopping time $\zeta$ in Proposition~\ref{prop::tail_decomposition} such that $\p[\zeta > S] \leq \epsilon$.  The techniques we use to prove Proposition~\ref{prop::tail_decomposition} will also allow us to show that boundary emanating flow lines similarly admit a decomposition into a union of overlapping tails.  This will also us to describe the manner in which boundary emanating flow lines interact with flow lines starting from interior points using Proposition~\ref{prop::tail_interaction}.

\begin{proposition}[Tail Decomposition: Boundary Regime]
\label{prop::boundary_tail_decomposition}
Suppose that $h$ is a GFF on a domain $D \subseteq \C$ with harmonically non-trivial boundary and let $\eta$ be a flow line of $h$ starting from $x \in \partial D$ with angle $\theta \in \R$.  Then we can decompose $\eta$ into a union of overlapping tails $(\eta_{w_i}^{\tau_i} : i \in \N)$ with the following properties:
\begin{enumerate}[(i)]
\item For every $i \in \N$, the starting point $w_i$ of $\eta_{w_i}$ has rational coordinates and
\item The range of $\eta$ is contained in $\cup_{i} \eta_{w_i}^{\tau_i}$ almost surely.
\end{enumerate}
If, in addition, $\eta$ is self-intersecting then properties (ii) and (iii) as described in Proposition~\ref{prop::tail_decomposition} hold with $\zeta_0 = 0$.
\end{proposition}

In order to establish Proposition~\ref{prop::tail_decomposition} and Proposition~\ref{prop::boundary_tail_decomposition}, we need to collect first the following three lemmas.  The third, illustrated in Figure~\ref{fig::merge_into_boundary}, implies that if we start a flow line close to the tail of another flow line with the same angle, then with positive probability the former merges into the latter at their first intersection time and, moreover, this occurs without the path leaving a ball of fixed radius.

\begin{lemma}
\label{lem::path_close}
Suppose that $h$ is a whole-plane GFF defined up to a global multiple of $2\pi \chi$.  Let $\eta$ be the flow line of $h$ starting from $0$ and $\tau$ any almost surely finite stopping time for $\eta$ such that $\eta(\tau) \notin \eta((-\infty,\tau))$ almost surely.  Given $\eta|_{(-\infty,\tau]}$, let $\gamma \colon [0,1] \to \C$ be any simple path starting from $\eta(\tau)$ such that $\gamma((0,1])$ is contained in the unbounded connected component of $\C \setminus \eta((-\infty,\tau])$.  Fix $\epsilon > 0$ and let $A(\epsilon)$ be the $\epsilon$-neighborhood of $\gamma([0,1])$.  Finally, let
\[ \sigma_1 = \inf\{ t \geq \tau : \eta(t) \notin A(\epsilon)\}\quad\text{and}\quad
   \sigma_2 = \inf\{ t \geq \tau : |\eta(t) - \gamma(1)| \leq \epsilon\}.\]
Then $\p[ \sigma_2 < \sigma_1 \giv \eta|_{[0,\tau]}] > 0$.
\end{lemma}
\begin{proof}
Given $\eta|_{(-\infty,\tau]}$, we let $U \subseteq \C$ be a simply connected domain which contains $\gamma((-\infty,1])$, is contained in $A(\tfrac{\epsilon}{2})$, and such that there exists $\delta > 0$ with $\eta((\tau-\delta,\tau]) \subseteq \partial U$ and $\eta((-\infty,\tau-\delta]) \cap \partial U = \emptyset$.

We can construct $U$ explicitly as follows.  Let $\varphi$ be the unique conformal map from the unbounded component of $\C \setminus \eta([0,\tau])$ to $\C \setminus \ol{\D}$ with $\varphi(z)-z \to 0$ as $z \to \infty$.  As $\eta|_{[0,\tau]}$ is almost surely continuous, it follows that $\varphi$ extends to be a homeomorphism from the boundary of its domain (viewed as prime ends) to $\C \setminus \D$.  Therefore $\varphi(\gamma)$ is a continuous path in $\C \setminus \D$ and $\varphi(A(\epsilon))$ is a relatively open neighborhood of $\varphi(\gamma)$.  Moreover, as $\tau$ is almost surely a non-intersection time for $\eta$, we have that $\varphi(\eta(\tau)) \in \partial \D$ has positive distance from the image under $\varphi$ of an intersection point of $\eta$.  Therefore we can find $\wt{U} \subseteq \C \setminus \D$ which is relatively open which is contained in $\varphi(A(\epsilon))$, contains~$\varphi(\gamma)$, and which contains a neighborhood in $\partial \D$ of $\varphi(\eta(\tau))$ which is disjoint from images of self-intersection points of $\eta$.  By choosing $\wt{U}$ appropriately, we may further assume that $\partial \wt{U} \cap \partial \D = \varphi(\eta((\tau-\delta,\tau]))$ for some $\delta > 0$ small.  We can also take $\wt{U}$ to be the intersection with $\C \setminus \D$ of a domain with polygonal boundary with vertices with rational coordinates.  We then take $U = \varphi^{-1}(\wt{U})$ and note that $U$ satisfies the desired properties.

Fix $x_0 \in \partial U$ with $|x_0 - \gamma(1)| \leq \tfrac{\epsilon}{2}$.  Let $\wt{h}$ be a GFF on $U$ whose boundary data along $\eta((-\infty,\tau])$ agrees with that of $h$ (up to an additive constant in $2\pi \chi \Z$) and whose boundary data on $\partial U \setminus \eta((-\infty,\tau])$ is such that the flow line $\wt{\eta}$ of $\wt{h}$ starting from $\eta(\tau)$ is an ordinary chordal $\SLE_\kappa$ process in $U$ targeted at $x_0$.  Let $\wt{\sigma}_2 = \inf\{t   \geq 0 : |\wt{\eta}(t) - \gamma(1)| \leq \epsilon\}$ and note that $\wt{\sigma}_2 < \infty$ almost surely since $\wt{\eta}$ terminates at $x_0$.  Since $\wt{X} = \dist(\wt{\eta}((-\infty,\wt{\sigma}_2]),\partial U \setminus \eta((-\infty,\tau])) > 0$ almost surely, it follows that we can pick $\zeta > 0$ sufficiently small so that $\p[ \wt{X} \geq \zeta \giv\eta|_{[0,\tau]}] \geq \tfrac{1}{2}$.  The result follows since by \cite[Proposition~3.4]{MS_IMAG} the law of $\wt{h}$ restricted to $\{x \in U : \dist(x,\partial U \setminus \eta((-\infty,\tau])) \geq \zeta\}$ is absolutely continuous with respect to the law of $h$ given $\eta|_{(-\infty,\tau]}$ restricted to the same set, up to an additive constant in $2\pi \chi \Z$, and that $\wt{\eta}([0,\wt{\sigma}_2]) \subseteq A(\epsilon)$ almost surely.
\end{proof}

\begin{figure}[h!]
\begin{center}
\includegraphics[scale=0.85]{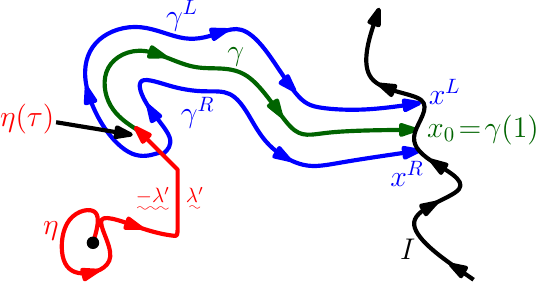}
\end{center}
\caption{\label{fig::path_close_hit}  An illustration of the proof of Lemma~\ref{lem::path_close_hit}.  We suppose that $h$ is a GFF on a proper domain $D \subseteq \C$ whose boundary consists of a finite, disjoint union of continuous paths each of which has flow line boundary conditions where each path has a given angle (which may vary from path to path).  We let $\tau$ be a stopping time for $\eta$ such that $\eta(\tau) \notin \eta((-\infty,\tau))$ and $\eta((-\infty,\tau]) \cap \partial D = \emptyset$ almost surely.  We assume that $\gamma \colon [0,1] \to D \setminus \eta((-\infty,\tau])$ is a simple curve connecting $\eta(\tau)$ to a boundary segment, say $I$, such that if we continued the boundary data of $h$ given $\eta|_{(-\infty,\tau]}$ along $\gamma$ as if it were a flow line then the height difference of $\gamma$ and $I$ upon intersecting is in the range for hitting (see Proposition~\ref{prop::tail_interaction}).  We let $\gamma^L$ (resp.\ $\gamma^R$) be a simple path contained in the $\epsilon$ neighborhood of $\gamma([0,1])$ which does not intersect $\gamma$, starts from the left (resp.\ right) side of $\eta((-\infty,\tau))$, and terminates at a point $x^L$ (resp.\ $x^R$) in $I$.  Assume, moreover, that $\gamma^L \cap \gamma^R = \emptyset$.  We take $U$ to be the region of $D \setminus \eta((-\infty,\tau])$ surrounded by $\gamma^L$ and $\gamma^R$ and let $\wt{h}$ be a GFF on $U$ whose boundary data agrees with $h$ on $I$ and $\eta((-\infty,\tau])$ and is given by flow line boundary conditions on $\gamma^L$ and $\gamma^R$.  We choose the angles on $\gamma^L,\gamma^R$ so that the flow line $\wt{\eta}$ of $\wt{h}$ starting from $\eta(\tau)$ almost surely hits $I$ and does not hit $\gamma^L$ and $\gamma^R$.  The result follows since the law of $\wt{\eta}$ is absolutely continuous with respect to the conditional law of $\eta$ given $\eta|_{(-\infty,\tau]}$ (since the law of $\wt{h}$ is absolutely continuous with respect to the law of $h$ given $\eta|_{(-\infty,\tau]}$ restricted to a subdomain of $U$ which stays away from $\gamma^L$ and $\gamma^R$).
}
\end{figure}

\begin{lemma}
\label{lem::path_close_hit} Suppose that $h$ is a GFF on a proper subdomain $D \subseteq \C$ whose boundary consists of a finite, disjoint union of continuous paths, each with flow line boundary conditions of a given angle (which can change from path to path), $z \in D$, and $\eta$ is the flow line of $h$ starting from $z$.  Fix any almost surely positive and finite stopping time $\tau$ for $\eta$ such that $\eta((-\infty,\tau]) \cap \partial D = \emptyset$ and $\eta(\tau) \notin \eta((-\infty,\tau))$ almost surely.  Given $\eta|_{(-\infty,\tau]}$, let $\gamma \colon [0,1] \to \ol{D}$ be any simple path in $\ol{D}$ starting from $\eta(\tau)$ such that $\gamma((0,1])$ is contained in the unbounded connected component of $\C \setminus \eta((-\infty,\tau])$, $\gamma([0,1)) \cap \partial D = \emptyset$, and $\gamma(1) \in \partial D$.  Moreover, assume that if we extended the boundary data of the conditional law of $h$ given $\eta|_{(-\infty,\tau]}$ along $\gamma$ as if it were a flow line then the height difference of $\gamma$ and $\partial D$ upon intersecting at time $1$ is in the admissible range for hitting.  Fix $\epsilon > 0$, let $A(\epsilon)$ be the $\epsilon$-neighborhood of $\gamma([0,1])$ in $D$, and let
\[ \tau_1 = \inf\{t \geq \tau : \eta(t) \notin A(\epsilon)\} \quad\text{and}\quad
   \tau_2 = \inf\{t \geq \tau : \eta(t) \in \partial D\}.\]
Then $\p[ \tau_2 < \tau_1 \giv \eta|_{[0,\tau]}] > 0$.
\end{lemma}

We recall that the admissible range of height differences for hitting is $(-\pi \chi,2\lambda-\pi \chi)$ if $\gamma$ is hitting on the right side and is $(\pi \chi - 2\lambda,\pi \chi)$ on the left side.  See Figure~\ref{fig::path_close_hit} for an illustration of the proof.

\begin{proof}[Proof of Lemma~\ref{lem::path_close_hit}]
Let $I$ be the connected component of $\partial D$ which contains $x_0 = \gamma(1)$.  Let $\gamma^L$ (resp.\ $\gamma^R$) be a simple path in $A(\epsilon) \setminus (\eta((-\infty,\tau]) \cup \gamma)$ which connects a point on the left (resp.\ right) side of $\eta((-\infty,\tau)) \cap A(\epsilon)$ to a point on the same side of $I$ hit by $\gamma$ at time $1$, say $x^L$ (resp.\ $x^R$), and does not intersect $\gamma$.  Assume that $\gamma^L \cap \gamma^R = \emptyset$.  Let $U$ be the region of $D \setminus \eta((-\infty,\tau])$ which is surrounded by $\gamma^L$ and $\gamma^R$.  Let $\wt{h}$ be a GFF on $U$ whose boundary data agrees with that of $h$ on $\eta((-\infty,\tau])$ and on $I$ and is otherwise given by flow line boundary conditions.  We choose the angles of the boundary data on $\gamma^L,\gamma^R$ so that the flow line $\wt{\eta}$ of $\wt{h}$ starting from $\eta(\tau)$ is an $\SLE_\kappa(\rho^L;\rho^R)$ process targeted at $x_0$ and the force points are located at $x^L$ and $x^R$.  Moreover, $\rho^L, \rho^R \in (\tfrac{\kappa}{2}-4,\tfrac{\kappa}{2}-2)$ since we assumed that if we continued the boundary data for $h$ given $\eta|_{(-\infty,\tau]}$ along $\gamma$ as if it were a flow line then it is in the admissible range for hitting (and by our construction, the same is true for both $\gamma^L$ and $\gamma^R$).  Let $\wt{\tau}_2 = \inf\{t \geq 0 : \wt{\eta}(t) \in \partial D\}$.  Since $\wt{X} = \dist(\wt{\eta}([0,\wt{\tau}_2]),\partial U \setminus (\eta((-\infty,\tau]) \cup I) ) > 0$ almost surely, it follows that we can pick $\zeta > 0$ sufficiently small so that $\p[ \wt{X} \geq \zeta\giv\eta|_{[0,\tau]}] \geq \tfrac{1}{2}$.  The result follows since the law of $\wt{h}$ restricted to $\{x \in U : \dist(x,\partial U \setminus (\eta((-\infty,\tau])) \cup I) \geq \zeta\}$ is absolutely continuous with respect to the law of $h$ given $\eta|_{(-\infty,\tau]}$ restricted to the same set and $\wt{\eta}((-\infty,\wt{\tau}_2]) \subseteq A(\epsilon)$ almost surely.
\end{proof}

Our first application of Lemma~\ref{lem::path_close_hit} is the following, which is an important ingredient in the proof of Proposition~\ref{prop::tail_decomposition}.

\begin{figure}[h!]
\begin{center}
\includegraphics[scale=0.85]{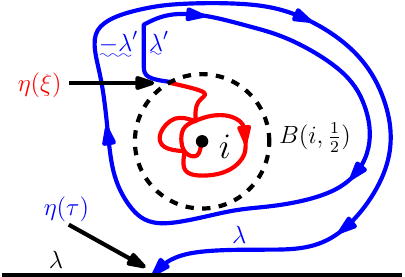}
\end{center}
\caption{\label{fig::merge_into_boundary} Suppose that $h$ is a GFF on $\h$ with constant boundary data $\lambda$ as depicted above and let $\eta$ be the flow line of $h$ starting at $i$.  We prove in Lemma~\ref{lem::clean_merge} that the following is true.  Let $\tau$ be the first time that $\eta$ hits $\partial \h$.  With positive probability, the height difference of $\eta$ and $\partial \h$ upon hitting at time $\tau$ is zero and $\eta((-\infty,\tau]) \subseteq B(0,2)$.  If $\partial \h$ were a flow line (rather than the boundary), then this can be rephrased as saying that the first time that $\eta$ hits $\partial \h$ is with positive probability the same as the first time that $\eta$ ``merges into'' $\partial \h$ and, moreover, this happens without $\eta$ leaving $B(0,2)$.  We call this a ``clean merge'' because the interaction of $\eta$ and $\partial \h$ upon hitting only involves a tail of $\eta$.  An analogous statement holds if $\lambda$ is replaced with $-\lambda$.}
\end{figure}

\begin{lemma}
\label{lem::clean_merge}
Suppose that $h$ is a GFF on $\h$ with constant boundary data $\lambda$ as depicted in Figure~\ref{fig::merge_into_boundary}.  Let $\eta$ be the flow line of $h$ starting at $i$ and let $\tau = \inf\{t \in \R : \eta(t) \in \partial \h\}$.  Let $E_1$ be the event that $\eta$ hits $\partial \h$ with a height difference of $0$ and let $E_2 = \{ \eta((-\infty,\tau]) \subseteq B(0,2)\}$.  There exists $\rho_0 > 0$ such that
\begin{equation}
\label{eqn::clean_merge}
\p[E_1 \cap E_2] \geq \rho_0.
\end{equation}
The same likewise holds when $\lambda$ is replaced with $-\lambda$.
\end{lemma}
\begin{proof}
See Figure~\ref{fig::merge_into_boundary} for an illustration of the setup.  We let $\xi$ be the first time that $\eta$ hits $\partial B(i,\tfrac{1}{2})$.  We let $\gamma$ be a simple path in $(\ol{\h} \cap B(0,2)) \setminus B(i,\tfrac{1}{2})$ starting from $\eta(\xi)$ which winds around $i$ precisely $k \in \Z$ times until hitting $\partial \h$.  We choose $k$ so that if we continued the boundary data of $\eta$ along $\gamma$, the height difference of $\gamma$ upon intersecting $\partial \h$ is zero.  The lemma then follows from Lemma~\ref{lem::path_close_hit}.
\end{proof}

\begin{figure}[ht!]
\begin{center}
\includegraphics[scale=0.85]{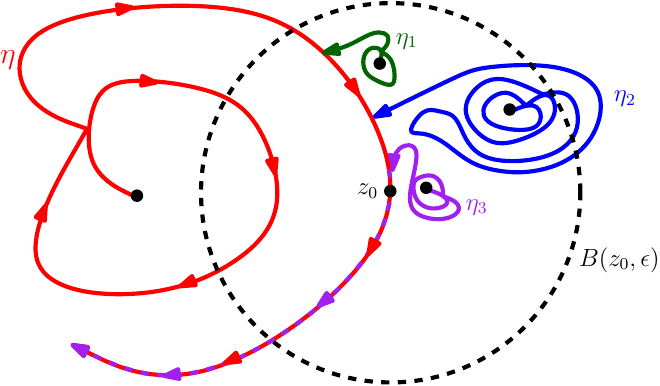}
\end{center}
\caption{\label{fig::clean_tail_merge}  Suppose that $\eta$ is a flow line of a whole-plane GFF defined up to a global multiple of $2\pi \chi$ starting at $z$ and that $\tau$ is a stopping time for $\eta$ such that $\eta(\tau) \notin \eta((-\infty,\tau))$ almost surely.  Let $\sigma$ (resp.\ $\zeta$) be the start (resp.\ end) time for the tail $\eta^\tau$.  Fix $\epsilon > 0$.  We prove in Lemma~\ref{lem::clean_tail_merge} that if $(w_k)$ is any sequence in $\C \setminus \eta((-\infty,\zeta])$ converging to a point $z_0 \in \eta((\sigma,\zeta))$ which is not a self-intersection point of $\eta|_{(-\infty,\zeta]}$, for each $k \in \N$, we let $\eta_k = \eta_{w_k}$ be the flow line starting at $w_k$ and let $N = N(\epsilon)$ be the smallest integer $k \geq 1$ such that $\eta_k$ merges cleanly (recall Figure~\ref{fig::merge_into_boundary}) into $\eta^\tau$ without leaving $B(z_0,\epsilon)$ then $N < \infty$ almost surely.  In the illustration, $N = 3$ since $\eta_1$ hits $\eta^\tau$ for the first time at the wrong height and $\eta_2$ leaves $B(z_0,\epsilon)$ before hitting $\eta^\tau$.  By applying this result to a time which occurs before $\tau$, we see that $\eta|_{[\tau,\zeta]}$ is almost surely represented as a tail.  We will use this fact in the proof of Theorem~\ref{thm::flow_line_interaction} in order to relax the restriction of Proposition~\ref{prop::tail_interaction} that the tails do not disconnect either of their initial points from $\infty$.}
\end{figure}

By repeated applications of Lemma~\ref{lem::clean_merge}, we are now going to prove that the restriction of the tail $\eta^\tau$ of a flow line $\eta$ associated with the stopping time $\tau$ and ending at $\zeta$ to the time interval $[\tau,\zeta]$ almost surely is represented as the tail of a flow line whose initial point is close to $\eta(\tau)$ and which merges into $\eta^\tau$ without leaving a small ball centered at $\eta(\tau)$.  See Figure~\ref{fig::clean_tail_merge} for an illustration of the setup of this result.

\begin{lemma}
\label{lem::clean_tail_merge}
Suppose that~$h$ is a GFF on~$\C$ defined up to a global multiple of $2\pi \chi$, let~$\eta$ be the flow line of~$h$ starting from~$0$, and let~$\tau$ be a stopping time for~$\eta$ such that $\eta(\tau) \notin \eta((-\infty,\tau))$ almost surely.  Let~$\sigma$ (resp.\ $\zeta$) be the start (resp.\ end) time of the tail~$\eta^\tau$.  Fix $\epsilon > 0$ and, given $\eta|_{(-\infty,\zeta]}$, a point $z_0 \in \eta((\sigma,\zeta))$ which is not a self-intersection point of $\eta|_{(-\infty,\zeta]}$.  Given $\eta|_{(-\infty,\zeta]}$, let $(w_k)$ be a sequence of points $\C \setminus \eta((-\infty,\zeta])$ with rational coordinates and which converges to $z_0$.  For each~$k$, let $\eta_k = \eta_{w_k}$ be the flow line of~$h$ starting from~$w_k$.  Finally, let $N = N(\epsilon)$ be the first index $k \in \N$ such that~$\eta_k$ cleanly merges, in the sense of Figure~\ref{fig::merge_into_boundary}, into~$\eta$ before leaving $B(z_0,\epsilon)$.  Then $\p[N < \infty] = 1$.
\end{lemma}
\begin{proof}
Let $\sigma,\zeta$ be the start and end times of the tail $\eta^\tau$.  By passing to a subsequence, we may assume without loss of generality that all of the elements of the sequence $(w_k)$ are contained in the same complementary connected component $U$ of $\eta((-\infty,\zeta])$.  Let $\varphi_k \colon U \to \h$ be the conformal transformation which takes $z_0$ to $0$ and $w_k$ to $i$.  Let $\wt{h}_k = h \circ \varphi_k^{-1} - \chi \arg (\varphi_k^{-1})'$ where we have chosen the additive constant for $h$ in $2\pi \chi \Z$ so that the boundary data for $\wt{h}_k$ in a neighborhood of $0$ is equal to either $\lambda$ or $-\lambda$ if $\partial U$ near $z_0$ is traced by the right or left, respectively, side of $\eta$.  In particular, the additive constant does not depend on $k$ and whether the boundary data of $\wt{h}_k$ near zero is $\lambda$ or $-\lambda$ also does not depend on $k$; we shall assume without loss of generality that we are in the former situation.

For each $k$, we let $\tau_k$ be the first time that $\eta_k$ exits $U$.  We also let
\[ A_k = \eta((-\infty,\zeta]) \cup \bigcup_{j=1}^k \eta_j((-\infty,\tau_j]).\]
We claim that $A_k$ is a local set for $h$.  We will prove this using the first characterization of local sets from \cite[Lemma~3.6]{MS_IMAG}.  Fix $W \subseteq \C$ open.  Then the event that $A_k \cap W \neq \emptyset$ is determined by the collection of all flow lines of $h$ starting from points with rational coordinates stopped upon hitting $W$.  Since the conditional law of the projection of $h$ onto those functions which are supported in $W$ given this collection and the projection of $h$ onto those functions which are harmonic in $W$ is a measurable function of the latter, we conclude that~$A_k$ is in fact local.  Let~$\CA_k$ be the $\sigma$-algebra generated by the values of~$h$ in an infinitesimal neighborhood of~$A_k$ and let~$\CF_k$ be the $\sigma$-algebra generated by $\CA_k$, $\eta|_{(-\infty,\zeta]}$, and $\eta_j|_{(-\infty,\tau_j]}$ for $1 \leq j \leq k$.

We claim that the conditional law of $h$ given~$\CF_k$ is that of a GFF on $\C \setminus A_k$.  We will explain this in the case that $k=1$.  The proof of this for general values of $k$ follows from the same argument.  For each $t > 0$, we let $\CF_{1,t}$ be the $\sigma$-algebra generated by $\eta((-\infty,\zeta])$, $\eta_1((-\infty,t])$, and the values of $h$ in an infinitesimal neighborhood of $\eta((-\infty,\zeta]) \cup \eta_1((-\infty,t])$.  For each $w \in \C$ with rational coordinates, we also let $\CG_{w,t}$ be the $\sigma$-algebra generated by $\eta((-\infty,\zeta])$, $\eta_w((-\infty,t])$, and the values of $h$ in an infinitesimal neighborhood of $\eta((-\infty,\zeta]) \cup \eta_w((-\infty,t])$.  For any event $A \in \sigma(h)$ we have that
\begin{align*}
	\p[ A \giv \CF_{1,t}]
&= \sum_{w \in \Q^2} \p[ A \giv \CF_{1,t}] \one_{\{w = w_1\}}
 = \sum_{w \in \Q^2} \p[ A \giv \CG_{w,t}] \one_{\{w = w_1\}}.
\end{align*}
\cite[Proposition~3.7]{MS_IMAG} implies that the conditional law of $h$ given $\CG_{w,t}$ is that of a GFF on $\C \setminus (\eta((-\infty,\zeta]) \cup \eta_w((-\infty,t]))$.  Therefore the conditional law of~$h$ given~$\CF_{1,t}$ is that of a GFF on $\C \setminus (\eta((-\infty,\zeta]) \cup \eta_1((-\infty,t]))$.  Note that~$\tau_1$ is a stopping time for the filtration $(\CF_{1,t})$.  Consequently, the martingale convergence theorem implies that the conditional law of~$h$ given~$\CF_1$ is a GFF on $\C \setminus A_1$, as desired.

We also observe that if $N$ is an almost surely finite stopping time for $(\CF_k)$, then the conditional law of $h$ given $\CF_N$ is that of a GFF in $\C \setminus A_N$.

We will now construct a further subsequence $(w_{j_k})$ where, for each $k$, $w_{j_k}$ is measurable with respect to $\CF_{j_{k-1}}$ and $j_k$ is a stopping time for $(\CF_k)$.  We take $j_1 = 1$ and inductively define $j_{k+1}$ for $k \geq 1$ as follows.  First, we note that the probability that $\eta_j|_{(-\infty,\tau_j]}$ hits $\partial U$ at a particular, fixed point near $z_0$ is zero for any $j \in \N$ \cite[Lemma~7.16]{MS_IMAG}.  Consequently, it follows from the Beurling estimate \cite[Theorem~3.69]{LAW05} that for every $\delta > 0$, there exists $j_{k+1} \geq j_k+1$ such that the probability that a Brownian motion starting from $w_{j_{k+1}}$ hits $\partial U$ before hitting $\cup_{i=1}^k \eta_{j_i}((-\infty,\tau_{j_i}])$ is at least $1-\delta$ given $\CF_{j_k}$ almost surely.  By taking $\delta > 0$ sufficiently small, the conformal invariance of Brownian motion then implies that we can arrange so that
\begin{enumerate}
\item $\varphi_{j_{k+1}}(\eta_{j_i}((-\infty,\tau_{j_i}])) \cap B(i,100) = \emptyset$ for every $1 \leq i \leq k$,
\item $\varphi_{j_{k+1}}(B^c(z_0,\epsilon)) \cap B(i,100) = \emptyset$, and
\item The total variation distance between the law of $\wt{h}_{j_{k+1}}|_{B(i,100)}$ given $\CF_{j_k}$ and that of a GFF on $\h$ with constant boundary data $\lambda$ restricted to $B(i,100)$ is almost surely at most $\rho_0/2$ where $\rho_0$ is the constant from Lemma~\ref{lem::clean_merge} (recall~\eqref{eqn::clean_merge}).
\end{enumerate}
Consequently, it follows from Lemma~\ref{lem::clean_merge} that the probability that $\wt{\eta}_{j_{k+1}}$ makes a clean merge into $\partial \h$ given $\CF_{j_k}$ without leaving $B(0,2)$ is almost surely at least $\rho_0/2$.  Combining this with Proposition~\ref{prop::tail_interaction} implies the assertion of the lemma.
\end{proof}

Recall that Proposition~\ref{prop::tail_interaction} only describes the interaction of tails of flow lines up until the base of one of the tails is disconnected from~$\infty$.  Lemma~\ref{lem::clean_tail_merge} allows us to strengthen this statement to give a complete description of the manner in which tails of flow lines interact.  We will not give a precise statement of this here since it will be part of our proof of Theorem~\ref{thm::flow_line_interaction} which we will give shortly.  Informally, this is the case because Lemma~\ref{lem::clean_tail_merge} implies that if the tails of flow lines $\eta_1,\eta_2$ are interacting \emph{after} one of their base points has been disconnected from $\infty$, we can represent each of the flow line tails in a neighborhood of where they are interacting by tails of another pair of flow lines whose base points have not yet been disconnected from~$\infty$.  Proposition~\ref{prop::tail_interaction} then applies to this second set of tails which, in turn, tells us how the first set of tails are interacting with each other.

\begin{figure}[h!]
\begin{center}
\subfigure[]{\includegraphics[scale=0.85]{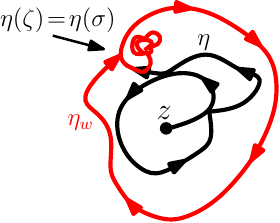}}
\hspace{0.1\textwidth}
\subfigure[]{\includegraphics[scale=0.85]{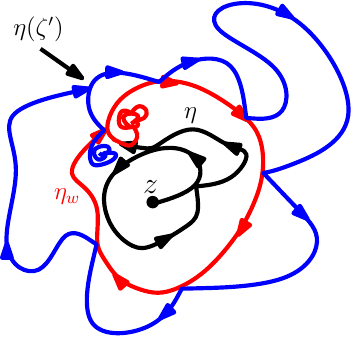}}
\end{center}
\caption{
\label{fig::outer_boundary_tail} In the left panel, $\eta$ is a GFF flow line and $\zeta$ is a stopping time so that $\eta(\zeta) \in \eta((-\infty,\zeta))$ almost surely and that with $\sigma$ the largest time before $\zeta$ with $\eta(\sigma) = \eta(\zeta)$, we have that $\eta|_{[\sigma,\zeta]}$ is given by the tail of a flow line $\eta_w$ starting at $w$.  Moreover, the starting point $w$ of $\eta_w$ has rational coordinates and is contained in one of the bounded complementary connected components of $\eta((-\infty,\zeta])$.  When this holds, we say that the outer boundary of $\eta((-\infty,\zeta])$ is represented by a tail.  Let $\zeta'$ be the first time $t$ after $\zeta$ that $\eta(t) \in  \eta([\zeta,t))$.  We prove in Lemma~\ref{lem::tail_decomposition} that if the outer boundary of $\eta((-\infty,\zeta])$ is represented by a tail, then the outer boundary of $\eta((-\infty,\zeta'])$ is almost surely also represented by a tail (right panel).}
\end{figure}

\begin{figure}[h!]
\begin{center}
\includegraphics[scale=0.85]{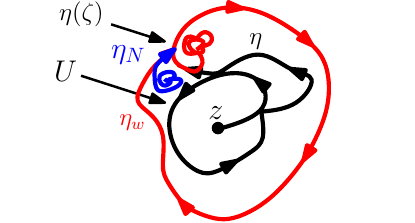}
\end{center}
\caption{\label{fig::outer_boundary_tail_proof} (Continuation of Figure~\ref{fig::outer_boundary_tail}.)  The idea to prove Lemma~\ref{lem::tail_decomposition} is to start flow lines of $h$ in a bounded complementary connected component $U$ of $\eta((-\infty,\zeta])$ whose boundary contains $\eta(\zeta-\epsilon)$ and such that $\eta|_{[\zeta-2\epsilon,\zeta]}$ traces part of $\partial U$ for some very small $\epsilon > 0$.  We choose the starting points $(w_k)$ of these flow lines to have rational coordinates and get progressively closer to the outer boundary of $\eta((-\infty,\zeta])$ near $\eta(\zeta-\epsilon)$.  By using Lemma~\ref{lem::clean_tail_merge}, we see that $\p[N< \infty] =1$ where $N$ is the first index $k$ such that $\eta_{w_k}$ merges into the tail of $\eta_w$ (hence also $\eta$) which generates the outer boundary of $\eta((-\infty,\zeta])$.  Let $\eta_N = \eta_{w_N}$ be the flow line which is the first to have a ``clean merge'' with $\eta_w$ (hence also $\eta$), in the sense that the merging time is the same as the first hitting time of $\eta_N$ and $\eta_w$ near $\eta(\zeta-\epsilon)$.  Proposition~\ref{prop::tail_interaction} implies that $\eta_w$ and $\eta_N$ (hence also $\eta_N$ and $\eta$) merge with each other and stay together at least until they hit $\eta(\zeta)$.}
\end{figure}

\begin{figure}[ht!]
\begin{center}
\includegraphics[scale=0.85]{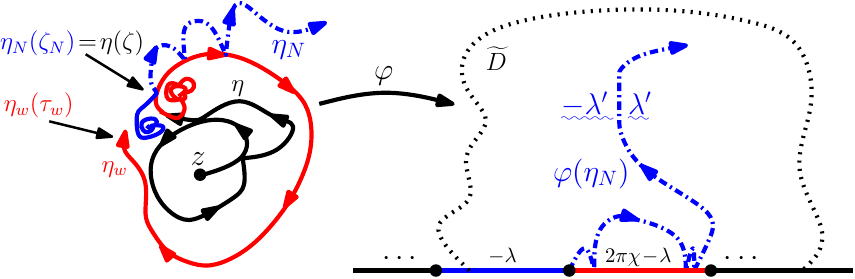}
\end{center}
\caption{\label{fig::outer_boundary_tail_proof2} (Continuation of Figure~\ref{fig::outer_boundary_tail_proof}) We then need to show that $\eta$ and $\eta_N$ continue to agree with each other after hitting $\eta(\zeta)$; call this time $\zeta_N$ for the latter.  The reason that this holds is that another application of Proposition~\ref{prop::tail_interaction} implies that $\eta_N$ may be able to bounce off but cannot cross the tail of $\eta_w$.  Indeed, this is accomplished by applying the proposition to $\eta_w$ stopped at a time $\tau_{w}$ which is slightly before the time it completes generating the outer boundary of $\eta((-\infty,\zeta])$ and does not merge with $\eta_N$.  Since $\eta_N$ cannot cross the tail $\eta_w^{\tau_w}$ stopped at time $\tau_w$, it consequently follows that $\eta_N|_{[\zeta_N,\infty)}$ is contained in the unbounded complementary connected component of $\eta((-\infty,\zeta])$.  The result follows since the conditional field of $h$ given $\eta|_{(-\infty,\zeta]}$ coupled with $\eta|_{[\zeta,\infty)}$ satisfies the same Markov property as when coupled with $\eta_N|_{[\zeta_N,\infty)}$ (see the right panel for this after applying a conformal mapping).  Thus by absolute continuity \cite[Proposition~3.4]{MS_IMAG}, the boundary emanating uniqueness theory of flow lines given in \cite[Theorem~1.2]{MS_IMAG} implies that the paths have to agree, at least until wrapping around $z$ and then intersecting themselves.}
\end{figure}

Suppose that $\eta$ is a GFF flow line and that $\zeta$ is an almost surely finite stopping time for $\eta$ such that $\eta(\zeta) \in \eta((-\infty,\zeta))$.  We say that the outer boundary of $\eta((-\infty,\zeta))$ can be represented as the tail of a flow line if the following is true.  Let $\sigma$ be the largest time before $\zeta$ that $\eta(\sigma) = \eta(\zeta)$ (note $\sigma \neq \zeta$).  Then there exists $w \in \C$ with rational coordinates which is contained in a bounded complementary component of $\eta((-\infty,\zeta])$ such that $\eta|_{[\sigma,\zeta]}$ is contained in a tail of the flow line $\eta_w$ starting at $w$.  See Figure~\ref{fig::outer_boundary_tail} for an illustration.  The main ingredient in our proof of the existence of a decomposition of a flow line into overlapping tails is the following lemma, which says that if the outer boundary of $\eta((-\infty,\zeta])$ is represented as a tail, then the outer boundary of $\eta((-\infty,\zeta'])$ is almost surely represented by a tail where $\zeta'$ is the first time after $\zeta$ that $\eta$ wraps around its starting point and intersects itself.  One example of such a stopping time $\zeta$ is the first time after a fixed time $T \in \R$ that $\eta$ makes a clean loop about $z$.  We will establish this through repeated applications of Lemma~\ref{lem::clean_tail_merge}.

\begin{lemma}
\label{lem::tail_decomposition}
Let $h$ be a whole-plane GFF defined up to a global multiple of $2\pi \chi$ and let $\eta$ be the flow line of $h$ starting at $z \in \C$.  Let $\zeta$ be any almost surely finite stopping time for $\eta$ such that the outer boundary of $\eta((-\infty,\zeta])$ can be represented as a tail of a flow line, as described just above.  Let $\zeta'$ be the first time $t \geq \zeta$ that $\eta(t) \in \eta([\zeta,t))$.  Then the outer boundary of $\eta((-\infty,\zeta'])$ is almost surely represented as a tail of a flow line.
\end{lemma}
\begin{proof}
Let $w,\sigma$ be as described just before the statement of the lemma and let $\eta_w$ be the flow line of $h$ starting at $w$ whose tail covers the outer boundary of $\eta|_{(-\infty,\zeta]}$ (see Figure~\ref{fig::outer_boundary_tail}).  Let $U$ be a bounded connected component of $\C \setminus  \eta((-\infty,\zeta])$ whose boundary contains $\eta(\zeta-\epsilon)$ and is traced by $\eta|_{[\zeta-2\epsilon,\zeta]}$ for some $\epsilon > 0$ small with $\zeta-2\epsilon > \sigma$ (see Figure~\ref{fig::outer_boundary_tail_proof} for an illustration).  Let $(w_k)$ be a sequence in $U$ with rational coordinates which converges to $\eta(\zeta-\epsilon)$ and, for each $k$, let $\eta_k =\eta_{w_k}$ be the flow line of $h$ starting from $w_k$.  Let $N$ be first integer $k \geq 1$ such that $\eta_k$ cleanly merges into $\eta_w$.  Then Lemma~\ref{lem::clean_tail_merge} implies that $N < \infty$ almost surely.  It follows from Proposition~\ref{prop::tail_interaction} that $\eta_N$ merges into and does not separate from $\eta_w$, at least up until both paths reach $\eta(\zeta)$ (this is when the starting point of at least one of the two tails is separated from $\infty$).  Consequently, $\eta_N$ also merges with $\eta$ and the two paths agree with each other, at least up until they both hit $\eta(\zeta)$.

Let $\zeta_N$ be the first time that $\eta_N$ hits $\eta(\zeta)$ and let $\zeta_N'$ be the first time after $\zeta_N$ that $\eta_N$ wraps around $z$ and hits itself.  We are now going to argue that $\eta_N|_{[\zeta_N,\zeta_N']}$ agrees with $\eta|_{[\zeta,\zeta']}$ up to reparameterization (see Figure~\ref{fig::outer_boundary_tail_proof2} for an illustration of the argument).  To see this, we will first argue that $\eta_N|_{[\zeta_N,\zeta_N']}$ is almost surely contained in the unbounded connected component of $\eta((-\infty,\zeta])$.  Let $\sigma_w,\zeta_w$ be the start and end times of the tail of $\eta_w$ which represents the outer boundary of $\eta((-\infty,\zeta])$.  Since $\eta_N$ cleanly merges into $\eta_w$, we know that the height difference of $h$ at $\eta_N(\zeta_N)$ and $\eta_w(\sigma_w)$, as made precise in Figure~\ref{fig::angle_difference}, is either $-2\pi \chi$ (if $\eta|_{[\sigma,\zeta]}$ is a clockwise loop, as in Figure~\ref{fig::outer_boundary_tail_proof2}, heights go down by $2\pi \chi$) or $2\pi \chi$ (if $\eta|_{[\sigma,\zeta]}$ is a counterclockwise loop, heights go up by $2\pi \chi$).  Consequently, it follows by applying Proposition~\ref{prop::tail_interaction} to $\eta_N$ and $\eta_w$ stopped at a time before $\eta_w$ merges into $\eta_N((-\infty,\zeta_N])$ that $\eta_N|_{[\zeta_N,\zeta_N']}$ does not cross $\eta_w([\sigma_w,\zeta_w])$ (see Remark~\ref{rem::tail_meas} below).  This proves our claim.  The proof that $\eta_N|_{[\zeta_N,\zeta_N']}$ is equal to $\eta|_{[\zeta,\zeta']}$ follows since boundary emanating flow lines are almost surely determined by the field \cite[Theorem~1.2]{MS_IMAG} and absolute continuity \cite[Proposition~3.4]{MS_IMAG}.  This is explained in more detail in the caption of Figure~\ref{fig::outer_boundary_tail_proof2}.
\end{proof}

\begin{remark}
\label{rem::tail_meas}
Although the starting point of $\eta_N$ is random and depends on the realization of $\eta$ (or $\eta_w$) in the proof of Lemma~\ref{lem::tail_decomposition} above, we can still apply Proposition~\ref{prop::tail_interaction}.  The reason is that the starting points of $\eta_N$ and $\eta_w$ were assumed to be rational and Proposition~\ref{prop::tail_interaction} describes almost surely the interaction of all tails of flow lines started at rational points simultaneously.  Consequently, we do not have to worry about the dependencies in the definitions of the flow lines in order to determine the manner in which their tails interact upon hitting.
\end{remark}

\begin{proof}[Proof of Proposition~\ref{prop::tail_decomposition}]
The result follows by repeatedly applying the previous lemma to get that each time $\eta$ wraps around $z$ and hits itself, it can be represented by a tail.
\end{proof}

\begin{proof}[Proof of Proposition~\ref{prop::boundary_tail_decomposition}]
It is easy to see from the proof of Proposition~\ref{prop::tail_interaction} that it generalizes to describe the interaction of a tail of a flow line of $h$ starting from an interior point of $D$ and a tail of $\eta$ (which we recall starts from $\partial D$).  (Depending on the boundary data of $h$, it may be $\eta$ can intersect itself.  This, for example, is the case in the setting of Figure~\ref{fig::diskstart}.)  Let $\tau$ be the first time that $\eta$ intersects itself.  Consequently, it is easy to see from Lemma~\ref{lem::path_close_hit} as well as the argument used to prove Lemma~\ref{lem::clean_tail_merge} that $\eta|_{(-\infty,\tau]}$ can be decomposed into a union of overlapping tails starting at points with rational coordinates.  That the same holds for $\eta|_{[\tau,\infty)}$ follows from the argument used to prove Proposition~\ref{prop::tail_decomposition}.
\end{proof}

\subsubsection{Flow line interaction}
\label{subsubsec::flow_line_interaction}

\begin{figure}[ht!]
\begin{center}
\includegraphics[scale=0.85]{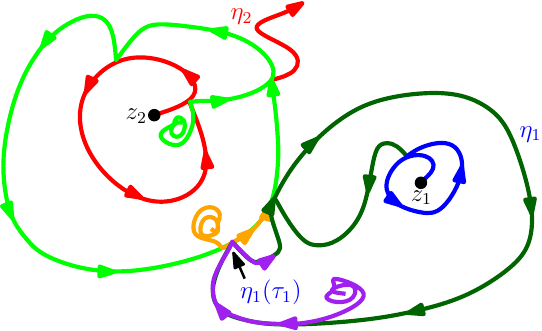}
\end{center}
\caption{\label{fig::interaction_proof} An illustration of the completion of the proof of Theorem~\ref{thm::flow_line_interaction}.  Suppose that $h$ is a whole-plane GFF defined up to a global multiple of $2\pi \chi$, $z_1,z_2 \in \C$ are distinct, $\theta_1,\theta_2 \in \R$, and let $\eta_i$ be the flow line of $h$ starting at $z_i$ with angle $\theta_i$ for $i=1,2$.  Let $\tau_1$ be a stopping time for the filtration $\CF_t = \sigma(\eta_1(s) : s \leq t,\ \ \eta_2)$ and assume throughout that we are working on the event $\eta_1(\tau_1) \in \eta_2$.  Let $\sigma_2$ be such that $\eta_1(\tau_1) = \eta_2(\sigma_2)$.  Then we can describe the interaction of $\eta_1$ near time $\tau_1$ and $\eta_2$ near time $\sigma_2$ in terms of tails.  Indeed, it follows from Proposition~\ref{prop::tail_decomposition} that there exists stopping times $\zeta_j < \zeta_{j+1}$ for $\eta_2$ such that $\sigma_2 \in [\zeta_j,\zeta_{j+1})$ and such that $\eta_2|_{[\zeta_j,\zeta_{j+1}]}$ is covered by a tail of a flow line (shown in light green in the illustration).  The same is likewise true for $\eta_1$ near $\eta_1(\tau_1)$ (shown in dark green in the illustration).  The reason that we can apply this result is that, as we remarked earlier, we can find arbitrarily small stopping times at which $\eta_i$ for $i=1,2$ makes a clean loop around $z_i$ before exiting the ball of radius $\tfrac{1}{2}|z_1-z_2|$.  Therefore we can apply Proposition~\ref{prop::tail_interaction} to describe the interaction of $\eta_1$ and $\eta_2$ near times $\tau_1$ and $\sigma_2$, respectively.  Note that it might be that one of the tail base points is separated from $\infty$ when $\eta_1$ and $\eta_2$ are interacting near these times (in the illustration, this is true for the tail corresponding to $\eta_1$).  We can circumvent this issue by applying Lemma~\ref{lem::clean_tail_merge} to further represent the tails of $\eta_1,\eta_2$ near their interaction times as tails of flow lines starting from points with rational coordinates in which the base point of the tails are not separated from $\infty$ when interacting with each other (this is shown in purple for $\eta_1$ and in orange for $\eta_2$ in the illustration).}
\end{figure}

Proposition~\ref{prop::tail_decomposition} allows us to extend the observations from Proposition~\ref{prop::tail_interaction} from tails to entire flow lines, since wherever two flow lines intersect (or one flow line intersects its past), locally it can be described by two flow line tails starting from points with rational coordinates intersecting.

\begin{proof}[Proof of Theorem~\ref{thm::flow_line_interaction} for ordinary GFF flow lines]
In the case that $\kappa \in (0,8/3]$, we know from Section~\ref{subsec::whole_plane_sle} and Lemma~\ref{lem::radial_critical_for_hitting} that $\SLE_\kappa$ flow lines of the GFF are non-self intersecting hence are themselves tails.  Consequently, in this case the result follows from Proposition~\ref{prop::tail_interaction} as well as absolute continuity if $D$ is a proper subdomain in $\C$ (Proposition~\ref{prop::local_set_whole_plane_bounded_compare}).  This leaves us to handle the case that $\kappa \in (8/3,4)$ which is in turn explained in the caption of Figure~\ref{fig::interaction_proof}.  We emphasize that it is important that all of the tails considered in the proof have initial points with rational coordinates so that we can use Proposition~\ref{prop::tail_interaction} to describe the interaction of all of these tails simultaneously almost surely.  The result for boundary emanating flow lines follows by a similar argument and Proposition~\ref{prop::boundary_tail_decomposition}.
\end{proof}

We will explain in Section~\ref{subsec::conical} after completing the proof of the existence component of Theorem~\ref{thm::alphabeta} how using absolute continuity, the version of Theorem~\ref{thm::flow_line_interaction} for ordinary GFF flow lines implies Theorem~\ref{thm::flow_line_interaction} as stated, in particular in the presence of a conical singularity.

\begin{figure}[ht!]
\begin{center}
\includegraphics[scale=0.85]{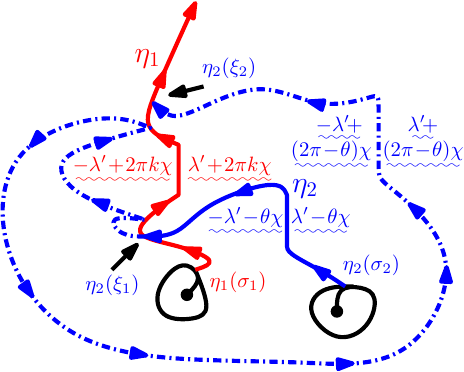}
\end{center}
\caption{\label{fig::tails_cross_once_intersect}
The proof that tails of ordinary GFF flow lines can cross each other at most once almost surely.  Suppose that $h$ is a GFF on $\C$ defined up to a global multiple of $2\pi \chi$ and that $z_1,z_2 \in \C$ are distinct.  Let $\eta_1, \eta_2$ be flow lines of $h$ starting at $z_1,z_2$, respectively.  Assume that $\eta_1$ has zero angle while $\eta_2$ has angle $\theta \in [0,2\pi)$.  For $i=1,2$, let $\tau_i$ be a stopping time for $\eta_i$ such that $\eta_i(\tau_i) \notin \eta_i((-\infty,\tau_i))$ and let $\eta_i^{\tau_i}$ be the corresponding tail.  The boundary data for the conditional law of $h$ given the segments of $\eta_1$ and $\eta_2$ is as illustrated where $k \in \Z$, modulo a global additive constant in $2\pi \chi \Z$.  Assume that $\eta_2^{\tau_2}$ hits and crosses $\eta_1^{\tau_1}$ from the right to the left side, say at time $\xi_1$.  Then the height difference $\CD = -(2\pi k + \theta) \chi$ of the paths upon intersecting is in $(-\pi \chi,0)$ by Theorem~\ref{thm::flow_line_interaction}.  In particular, $k = 0$.  Let $\sigma_1$ be the start time of the tail of $\eta_1^{\tau_1}$.  Let $\sigma_2$ be the start time of $\eta_2^{\tau_2}$.  Then Proposition~\ref{prop::tail_interaction} implies that $\eta_2^{\tau_2}$ cannot cross $\eta_1^{\tau_1}$ again before the paths separate one of $\eta_i(\sigma_i)$ from $\infty$ for $i=1,2$.  This can only happen if, after time $\xi_1$, $\eta_2$ wraps around and hits $\eta_1^{\tau_1}$ on its right side a second time, as illustrated above, say at time $\xi_2$.  The height difference of the paths upon intersecting this time is $\CD + 2\pi \chi \geq 0$, thus Theorem~\ref{thm::flow_line_interaction} implies that the paths cannot cross again.
}
\end{figure}

\begin{figure}[ht!!]
\begin{center}
\includegraphics[scale=0.85]{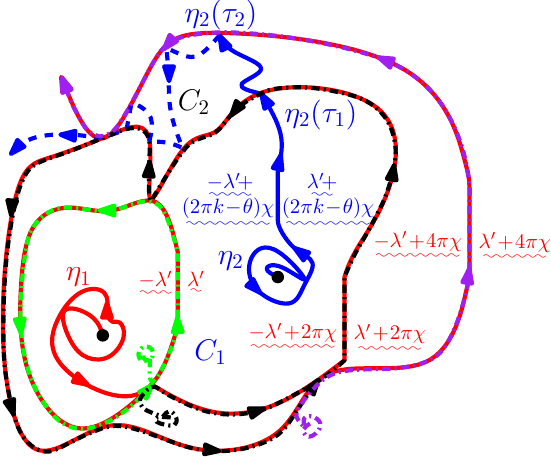}
\end{center}
\caption{\label{fig::cross_once_intersect}
{\small The proof that ordinary GFF flow lines can cross each other at most once for $\kappa \in (8/3,4)$, i.e.\ the regime in which the paths are self-intersecting.  Suppose that $h$ is a GFF on $\C$ defined up to a global multiple of $2\pi \chi$ and that $z_1,z_2 \in \C$ are distinct.  Let $\eta_1,\eta_2$ be the flow lines of $h$ starting at $z_1,z_2$, respectively.  Assume that $\eta_1$ has zero angle while $\eta_2$ has angle $\theta \in [0,2\pi)$.  Shown in the illustration are three flow line tails colored green, black, purple which represent the outer boundary of $\eta_1$ at successive times as it wraps around $z_1$.  Let $C_1$ be the component of $\C \setminus \eta_1$ containing $z_2$.  Then we can decompose $\partial C_1$ into an inner and outer part, each of which are represented by segments of tails (green and black in the illustration).  Let $\tau_1$ be the first time that $\eta_2$ exits $C_1$; assume that $\eta_2$ hits $\eta_1$ at time $\tau_1$ on its left side and in the black tail (the other cases are analogous)  The boundary data for the conditional law of $h$ given $\eta_1$ and $\eta_2|_{(-\infty,\tau_1]}$ is as depicted with $k \in \Z$, up to an additive constant in $2\pi \chi \Z$.  Consequently, the height difference $\CD$ of the paths upon intersecting is $(2\pi (k-1) - \theta) \chi$.  If $\eta_2$ crosses $\eta_1$ at time $\tau_1$, then $\CD \in (0,\pi \chi)$ hence $k=2$.  Thus if $\eta_2$ subsequently hits the purple tail, it does so with a height difference of $(2\pi (k-2) - \theta) \chi < 0$, hence does not cross.  Therefore, after crossing the black tail, $\eta_2$ follows the pockets of $\eta_1$ which lie between the black and purple tails in their natural order.  From this, we see that $\eta_2$ does not subsequently cross $\eta_1$.  This handles the case that the purple tail winds around $z_1$ with the same orientation as the black tail.  If the path switches direction, a similar analysis implies that $\eta_2$ cannot cross again.}
}
\end{figure}

\begin{figure}[ht!]
\begin{center}
\includegraphics[scale=0.85]{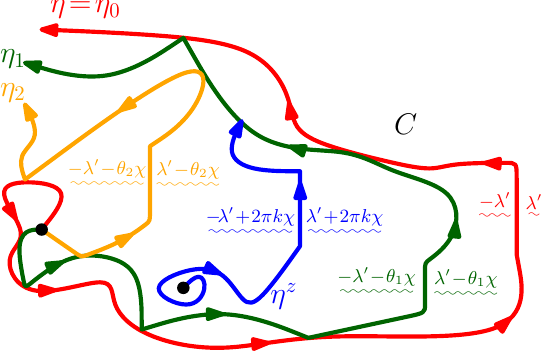}
\end{center}
\caption{{\label{fig::non_intersect_almost_surely_merge}  Fix $\kappa \in (0,8/3]$ and let $h$ be a whole-plane GFF  defined up to a global multiple of $2\pi \chi$.  Let $z \in \C \setminus \{0\}$ and $\eta,\eta^z$ be the flow lines of $h$ starting at $0,z$, respectively, both with zero angle.  Fix evenly spaced angles $0 \equiv \theta_0 < \theta_1 < \theta_2 < \cdots < \theta_n < 2\pi$ such that with $\eta_j$ the flow line of the \emph{conditional field} $h$ given $\eta$ on $\C \setminus \eta$ with angle $\theta_j$ starting at $0$, we have that $\eta_j$ almost surely intersects the left side of $\eta_{j-1}$ and the right side of $\eta_{j+1}$ for all $0 \leq j \leq n$; the indices are taken mod $n$ (\cite[Theorem~1.5]{MS_IMAG} implies that we can fix such angles).  By \cite[Theorem~1.5]{MS_IMAG}, we know that $\eta_j$ stays to the left of $\eta_{j-1}$ for all $0 \leq j \leq n$.  Moreover, since we can represent each of the $\eta_j$ using flow lines tails, Theorem~\ref{thm::flow_line_interaction} implies that $\eta^z$ and the $\eta_j$ obey the same flow line interaction rules.  Let $C$ be the connected component of $\C \setminus \cup_{j=0}^n \eta_j$ which contains $z$; then $C$ is almost surely bounded.  Since $\eta^z$ is unbounded, it follows that $\eta^z$ exits $C$ almost surely.  Theorem~\ref{thm::flow_line_interaction} implies that $\eta^z$ has to cross or merge into one of the two flow lines which generate $\partial C$ upon exiting $C$.   Indeed, illustrated above is the boundary data for $h$ given $\eta_0,\ldots,\eta_n$ and $\eta^z$ in the case that $\eta^z$ hits $\eta_1$.  If $k \geq 0$, then $2\pi k + \theta_1 > 0$ so that $\eta^z$ crosses $\eta_1$ upon hitting.  If $k < 0$, then $\eta^z$ bounces off $\eta_1$ upon hitting hence has to hit $\eta_2$.  Since $2\pi k + \theta_2 < 0$, it follows that $\eta^z$ crosses $\eta_2$ upon up hitting.  A similar analysis implies that $\eta^z$ crosses out or merges into the boundary of any pocket whose interior it intersects.  Proposition~\ref{prop::cross_merge} thus implies that $\eta^z$ can enter into the interior of at most a finite collection of components of $\C \setminus \cup_{j=0}^n \eta_j$, hence must eventually merge with one of the $\eta_j$.  Since $\theta_j \notin 2\pi \Z$ for all $1 \leq j \leq n$, it follows that $\eta^z$ cannot merge with $\eta_j$ for $1 \leq j \leq n$, hence must merge with $\eta$.}}
\end{figure}

\begin{figure}[ht!]
\begin{center}
\includegraphics[scale=0.85]{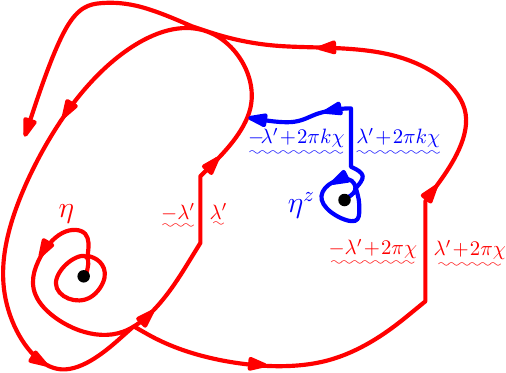}
\end{center}
\caption{\label{fig::inthepocket}
Fix $\kappa \in (8/3,4)$.  Suppose that $h$ is a whole-plane GFF defined up to a global multiple of $2\pi \chi$.  Let $\eta,\eta^z$ be the flow lines of $h$ starting at $0,z$, respectively, both with zero angle where $z \in \C \setminus \{0\}$.  Lemma~\ref{lem::radial_loop} implies that $z$ is almost surely contained in a bounded connected complementary component of $\eta$.  Upon hitting $\eta$, Theorem~\ref{thm::flow_line_interaction} implies that $\eta^z$ will either merge or bounce off $\eta$.  The reason that $\eta^z$ cannot cross $\eta$ is that the height difference of the paths upon intersecting takes the form $2\pi k \chi$ for $k \in \Z$, in particular cannot lie in $(-\pi \chi,0)$ (to cross from right to left) or $(0,\pi \chi)$ (to cross from left to right).
}
\end{figure}

\begin{figure}[ht!!]
\begin{center}
\includegraphics[scale=0.85]{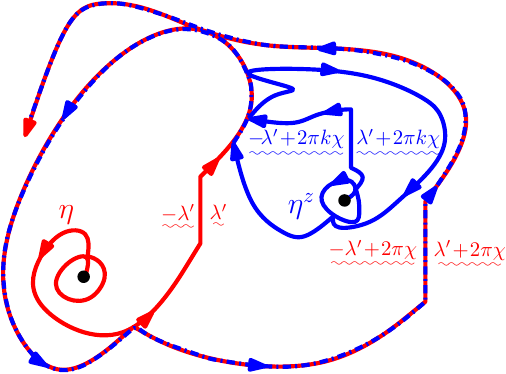}
\end{center}
\caption{\label{fig::inthepocket2}
(Continuation of Figure~\ref{fig::inthepocket}) After intersecting $\eta$, it may be that $\eta^z$ has to wind around $z$ several times before it reaches the correct height in order to merge with $\eta$ (in the illustration, $\eta^z$ winds around $z$ once after hitting $\eta$ before merging).  It is not possible for $\eta^z$ to bounce off the boundary of the complementary component of $\eta$ which contains $z$ and then exit without merging.
}
\end{figure}

Now that we have proved Theorem~\ref{thm::flow_line_interaction} for ordinary GFF flow lines, we turn to prove Theorem~\ref{thm::merge_cross} for ordinary GFF flow lines ($\alpha = 0$).

\begin{proposition}
\label{prop::cross_merge}
Suppose that $D \subseteq \C$ is a domain, $h$ is a GFF on $D$ defined up to a global multiple of $2\pi \chi$ if $D = \C$, $z_1,z_2 \in \ol{D}$, and $\theta_1,\theta_2 \in \R$.  For $i=1,2$, let $\eta_i$ be the flow line of $h$ with angle $\theta_i$, i.e.\ the flow line of $h+\theta_i \chi$, starting from $z_i$.  Then $\eta_1$ and $\eta_2$ almost surely cross each other at most once.  If $z_1=z_2$, then $\eta_1,\eta_2$ almost surely do not cross.
\end{proposition}
\begin{proof}
We first suppose that $z_1,z_2$ are distinct.  The first step in the proof is to show that tails of ordinary GFF flow lines can cross each other at most once.  This is explained in the caption of Figure~\ref{fig::tails_cross_once_intersect}.  Since ordinary GFF flow lines are themselves tails when $\kappa \in (0,8/3]$, to complete the proof we just need to handle the case that $\kappa \in (8/3,4)$.  The proof of this is explained in the caption of Figure~\ref{fig::cross_once_intersect}.  This completes the proof in the case that $z_1,z_2$ are distinct.  Now suppose that $z_1=z_2$.  We can run, say $\eta_1$, up until an almost surely positive stopping time $\tau$ so that $\eta_1(\tau) \neq z_1$.  Then the same arguments imply that $\eta_1|_{[\tau,\infty)}$ almost surely crosses $\eta_2$ at most once.  Since the stopping time $\tau$ was arbitrary, it follows that $\eta_1$ and $\eta_2$ almost surely cross at most once.  In the case that $D=\C$, a trivial scaling argument implies that $\eta_1$ and $\eta_2$ almost surely do not cross.  That $\eta_1,\eta_2$ almost surely do not cross in the case that $D \neq \C$ follows from the $D = \C$ case and absolute continuity.
\end{proof}

\begin{proposition}
\label{prop::merge_almost_surely}
Suppose that $h$ is a whole-plane GFF defined up to a global multiple of $2\pi \chi$ and that $\eta_1,\eta_2$ are flow lines of $h$ starting at $z_1,z_2 \in \C$ distinct with the same angle.  Then $\eta_1$ and $\eta_2$ almost surely merge.
\end{proposition}
\begin{proof}
We first consider the case that $\kappa \in (8/3,4)$.  Since whole-plane $\SLE_\kappa(\rho)$ processes are almost surely unbounded,  it follows from Lemma~\ref{lem::radial_loop} that $\eta_1$ almost surely surrounds and separates $z_2$ from $\infty$.  Let $U_1$ be the (necessarily bounded) connected component of $\C \setminus \eta_1$ which contains $z_2$.  Theorem~\ref{thm::flow_line_interaction} implies that $\eta_2$ cannot cross $\partial U_1$, hence can exit $\ol{U}_1$ only upon merging with $\eta_1$ (see also Figure~\ref{fig::inthepocket} and Figure~\ref{fig::inthepocket2}).  Since $\eta_2$ is also almost surely unbounded, $\eta_2$ exits $\ol{U}_1$ almost surely, from which the result in this case follows.  When $\kappa \in (0,8/3]$, $\eta_1$ does not intersect itself so the argument we gave for $\kappa \in (8/3,4]$ does not apply directly.  The appropriate modification is explained in the caption of Figure~\ref{fig::non_intersect_almost_surely_merge}.  One aspect of the proof which is not explained in the caption is why the connected components of $\C \setminus \cup_{j=0}^n \eta_j$ are almost surely bounded.  The reason is that the conditional law of $\eta_j$ given $\eta_{j-1}$ and $\eta_{j+1}$ for $1 \leq j \leq n$ and indices taken mod $n$ is that of an $\SLE_\kappa(\rho^L;\rho^R)$ process with $\rho^L,\rho^R \in (-2,\tfrac{\kappa}{2}-2)$ and such processes almost surely swallow any fixed point in finite time.
\end{proof}

\begin{remark}
\label{rem::not_nec_merge}
Suppose that we are in the setting of Proposition~\ref{prop::merge_almost_surely} with $h$ replaced by a GFF on a proper subdomain $D$ in $\C$.  Then it is not necessarily true that $\eta_1$ and $\eta_2$ merge with probability one.  The reason is that, depending on the boundary data, there is the possibility that $\eta_1$ and $\eta_2$ get stuck in the boundary (i.e., intersect the boundary with a height difference which does not allow the curve to bounce off the boundary) before intersecting each other with the appropriate height difference.
\end{remark}

\begin{remark}
\label{rem::inthepocket}
Assume that we are in the setting of Proposition~\ref{prop::merge_almost_surely} with $\kappa \in (8/3,4)$ so that $\eta_1$ almost surely surrounds $z_2$, except that $\eta_i$ has angle $\theta_i \in [0,2\pi)$ for $i=1,2$ and $\theta_1 \neq \theta_2$.  By Theorem~\ref{thm::flow_line_interaction}, once $\eta_2$ hits $\eta_1$, it either crosses $\eta_1$ immediately or bounces off $\eta_1$.  Just as in the case that $\theta_1 = \theta_2$, as in Figure~\ref{fig::inthepocket} and Figure~\ref{fig::inthepocket2}, it might be that $\eta_2$ has to wind around $z_2$ several times before ultimately leaving the complementary connected component (``pocket'') of $\eta_1$ which contains $z_2$.  Note that the pockets of $\eta_1$ are ordered according to the order in which $\eta_1$ traces their boundary and, after leaving the pocket which contains $z_2$, $\eta_2$ passes through the pockets of $\eta_1$ according to this ordering.
\end{remark}

\subsection{Conical singularities}
\label{subsec::conical}

\begin{figure}[ht!]
\begin{center}
\includegraphics[scale=0.85]{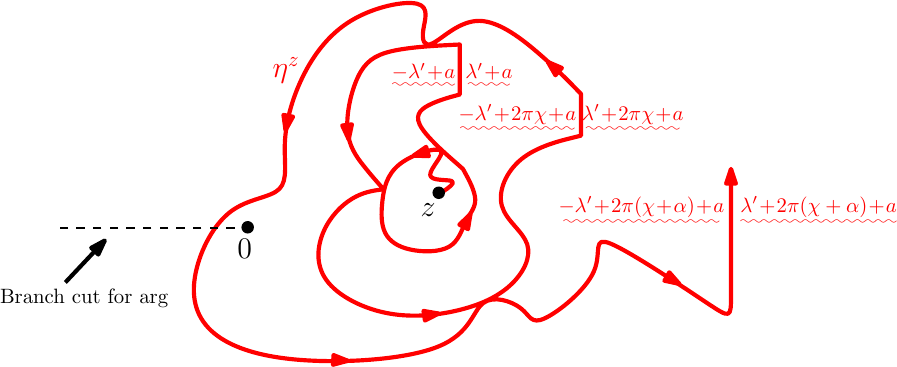}
\end{center}
\caption{\label{fig::conical_boundary_data}
Suppose that $h$ is a whole-plane GFF and fix constants $\alpha > -\chi$ and $\beta \in \R$.  Let $\eta^z$ be a flow line of $h_{\alpha \beta} = h-\alpha \arg(\cdot)-\beta\log|\cdot|$, defined up to a global multiple of $2\pi(\chi+\alpha)$, starting from $z \neq 0$.  If, in addition, we know $h_{\alpha \beta}$ up to a global multiple of $2\pi \chi$ then we can fix an angle for $\eta^z$ (see Remark~\ref{rem::flow_line_alpha_non_zero}).  Otherwise, the angle of $\eta^z$ is random (see Remark~\ref{rem::flow_line_random_angle}).  The boundary data for $h$ given $\eta$ is as shown, up to an additive constant in $2\pi(\chi+\alpha) \Z$; $a \in \R$.  Each time $\eta^z$ wraps around $z$ before wrapping around $0$, the boundary data of $h$ given $\eta^z$ along a vertical segment of $\eta^z$ increases (resp.\ decreases) by $2\pi\chi$ if the loop has a counterclockwise (resp.\ clockwise) orientation.  When $\eta^z$ wraps around $0$, the height along a vertical segment increases (resp.\ decreases) by $2\pi(\chi+\alpha)$ if the loop has a counterclockwise (resp.\ clockwise) orientation.  Since for every $\epsilon > 0$, the path winds around $z$ an infinite number of times by time $\epsilon$, it follows that observing the boundary data for the conditional field along $\eta$ only allows us to determine the angle of $\eta$ up to an additive constant in the additive subgroup $A$ of $\R$ generated by $2\pi \alpha$ and $2\pi \chi$ (recall Remark~\ref{rem::alpha_beta_determined}).  This is in contrast with the case in which $\eta$ starts from $0$ (or $\alpha$ is a non-negative integer multiple of $\chi$), in which case the boundary data of the conditional field is enough to determine the angle of the path.  Finally, this also implies that the height difference between two different flow lines starting at points different from $0$ with the same angle (up to an additive constant in $A$) is contained in $A$.}
\end{figure}

Let $h$ be a whole-plane GFF.  In this section, we are going to construct a coupling between $h_{\alpha\beta} = h-\alpha \arg(\cdot) - \beta \log|\cdot|$, viewed as a distribution up to a global multiple of $2\pi(\chi+\alpha)$, and a whole-plane $\SLE_\kappa^\beta(\rho)$ process where $\rho=2-\kappa+2\pi \alpha/\lambda$, provided $\alpha > -\chi$ (so that $\rho > -2$).  This will complete the proof of the existence part of Theorem~\ref{thm::alphabeta}.  In Section~\ref{subsec::uniqueness}, we will establish the uniqueness component of Theorem~\ref{thm::alphabeta}.  Recall from Remark~\ref{rem::alpha_coupling} that the proof of Theorem~\ref{thm::existence} goes through without modification in order to prove the existence of the coupling when $\beta = 0$ and $\alpha > -\chi$.  Thus we just need to extend the proof of Theorem~\ref{thm::existence} in order to get the result for $\beta \neq 0$, which is stated and proved just below.

\begin{proposition}
\label{prop::alphabeta_existence}
Fix constants $\alpha > -\chi$, $\beta \in \R$, and suppose that $h_{\alpha \beta} = h - \alpha \arg(\cdot) - \beta\log|\cdot|$ viewed as a distribution defined up to a global multiple of $2\pi(\chi+\alpha)$ where $h$ is a whole-plane GFF.  There exists a unique coupling $(h_{\alpha \beta},\eta)$ where $\eta$ is an $\SLE_\kappa^\beta(\rho)$ process with
\begin{align}
 \rho = 2-\kappa + \frac{2\pi \alpha}{\lambda} \label{eqn::rho_alpha_beta}
\end{align}
such that the following is true.  For every $\eta$-stopping time $\tau$, the conditional law of $h_{\alpha \beta}$ given $\eta|_{(-\infty,\tau]}$ is given by a GFF on $\C \setminus \eta((-\infty,\tau])$ with $\alpha$-flow line boundary conditions (as in Figure~\ref{fig::interior_path_bd2}) on $\eta((-\infty,\tau])$, a $2\pi \alpha$ gap along $(-\infty,0)$ viewed as a distribution defined up to a global multiple of $2\pi(\chi+\alpha)$, and the same boundary behavior at $\infty$ as $h_{\alpha \beta}$.
\end{proposition}
\begin{proof}
We are going to give the proof when $\alpha = 0$ for simplicity since, just as in the proof of Theorem~\ref{thm::existence} as explained in Remark~\ref{rem::alpha_coupling}, the case for general $\alpha > -\chi$ follows from the same argument.  We consider the same setup used to prove Theorem~\ref{thm::existence} described in Section~\ref{subsec::existence}: we let $\C_\epsilon = \C \setminus (\epsilon \D)$, $h_\epsilon$ be a GFF on $\C_\epsilon$ such that $h_\epsilon-\beta\log|\cdot|$ has the boundary data as indicated in the left side of Figure~\ref{fig::diskstart}.  More concretely, this means that the boundary data for $h_\epsilon$ is $-\lambda'+\beta \log \epsilon$ near $-\epsilon$ on the left side of $\partial (\epsilon \D)$, $\lambda'+\beta\log\epsilon$ near $+\epsilon$ on the right side of $\partial (\epsilon \D)$, and changes by $\chi$ times the winding of the boundary otherwise.

By \cite[Theorem~1.1, Theorem~1.2, and Proposition~3.4]{MS_IMAG}, we know that there exists a well-defined flow line $\eta_\epsilon$ of $h_\epsilon - \beta \log|\cdot|$ starting from $i\epsilon$ (i.e., a coupling $(h_\epsilon - \beta \log|\cdot|,\eta_\epsilon)$ which satisfies the desired Markov property).  Let $\wt{h}_\epsilon = h_\epsilon \circ \psi_\epsilon^{-1} - \chi \arg (\psi_\epsilon^{-1})'$ where, as in the proof of Theorem~\ref{thm::existence}, $\psi_\epsilon(z) = \epsilon/z$.  Then $\wt{h}_\epsilon - \beta \log|\psi_\epsilon^{-1}(\cdot)|$ is equal in law to $\wh{h}^\beta \equiv \wt{h} + 2\chi \arg(\cdot) -F + \beta \log|\cdot|$ where $\wt{h}$ is a GFF on $\D$ with the boundary data as indicated in the right side of Figure~\ref{fig::diskstart} and $F$ is equal to the harmonic extension of $2 \chi \arg(\cdot)$ from $\partial \D$ to $\D$.  As before, the branch cut of $\arg$ is on the half-infinite line from $0$ through $i$.   It follows from \cite[Proposition~3.4]{MS_IMAG} and Proposition~\ref{prop::interior_force_point_coupling} that there exists a unique coupling of $\wh{h}^\beta$ with a continuous process $\wh{\eta}^\beta$ (equal in law to $\psi_\epsilon(\eta_\epsilon)$) whose law is mutually absolutely continuous with respect to that of a radial $\SLE_\kappa(2-\kappa)$ process in $\D$ targeted at $0$ with a single boundary force point at $i$, up until any finite time when parameterized by $\log$-conformal radius as viewed from $0$, such that the coupling $(\wh{h}^\beta,\wh{\eta}^\beta)$ satisfies the Markov property described in the statement of Proposition~\ref{prop::interior_force_point_coupling}.  To complete the proof of the proposition, we just need to determine the law of $\wh{\eta}^\beta$.  We are going to accomplish this by computing the Radon-Nikodym derivative $\wh{\rho}_t^\beta$ of the law of $\wh{\eta}^\beta|_{[0,t]}$ with respect to the law of $\wh{\eta}|_{[0,t]}$ where $\wh{\eta} \equiv \wh{\eta}^0$ for each $t > 0$.  Note that $\wh{\eta}$ is a radial $\SLE_\kappa(2-\kappa)$ process in $\D$ targeted at $0$ with a single force point located at $i$.

We begin by computing the Radon-Nikodym derivative of the law of $\wh{h}^\beta$ with respect to the law of $\wh{h} \equiv \wh{h}^0$ away from $0$.  For each $\delta > 0$ and $z \in \D$ such that $B(z,\delta) \subseteq \D$, we let $\wh{h}_\delta(z)$ denote the average of $\wh{h}$ about $\partial B(z,\delta)$ (see \cite[Section~3]{DS08} for a discussion of the construction and properties of the circle average process).  Let
\[ \xi_\delta(z) = \log \max(|z|,\delta)\]
and $\D_\delta = \D \setminus (\delta \D)$ be the annulus centered at $0$ with in-radius $\delta$ and out-radius $1$.  Note that $\big(\wh{h}+\beta \xi_\delta\big) |_{\D_\delta} \stackrel{d}{=} \wh{h}^\beta|_{\D_\delta}$.  The Radon-Nikodym derivative of the law of $\wh{h} + \beta \xi_\delta$ with respect to $\wh{h}$ is proportional to
\begin{equation}
\label{eqn::log_radon}
 \exp(\beta (\wh{h},\xi_\delta)_\nabla).
\end{equation}
This is in turn proportional to $\exp(-\beta \wh{h}_\delta(0))$ since $(\wh{h},\xi_\delta)_\nabla = -\wh{h}_\delta(0)$ (see the end of the proof of \cite[Proposition~3.1]{DS08}; the reason for the difference in sign is that the function $\xi$ in \cite{DS08} is $-1$ times the function $\xi$ used here).

Let $\wh{\CF}_t = \sigma(\wh{\eta}(s) : s \leq t)$.  Since $\wh{\eta}$ is parameterized by $\log$-conformal radius as viewed from $0$, by the Koebe-$1/4$ theorem \cite[Corollary 3.18]{LAW05} we know that $\wh{\eta}([0,t]) \subseteq \ol{\D}_\delta$ for all $t \leq \log \tfrac{4}{\delta}$.  By \cite[Proposition~3.2]{DS08} and the Markov property for the GFF, we know that the law of $\wh{h}_\delta(0)$ given $\wh{\CF}_t$ for $t \leq \log \tfrac{4}{\delta}$ is equal to the sum of $m_t = \E[\wh{h}(0)|\wh{\CF}_t]$ and a mean-zero Gaussian random variable $Z$ with variance $\log \tfrac{1}{\delta} - t$.  Moreover, $m_t$ and $Z$ are independent.  By combining this with~\eqref{eqn::log_radon}, it thus follows that
\begin{equation}
\label{eqn::log_radon_curve}
\wh{\rho}_t^\beta = \exp\left( -\beta m_t - \frac{\beta^2}{2} t \right)
\end{equation}
(note that this makes sense as a Radon-Nikodym derivative between laws on paths because $m_t$ is determined by $\wh{\eta}$).  The reason that we know that we have the correct normalization constant $\exp(-\beta^2 t /2)$ is that it follows from \cite[Proposition~6.5]{MS_IMAG} that $m_t$ evolves as a standard Brownian motion in $t$.

Let $(W,O)$ be the driving pair for $\wh{\eta}$.  In this case, $O_0 = i$.  By the Girsanov theorem \cite{KS98,RY04}, to complete the proof we just need to calculate the cross-variation of $m_t$ and the Brownian motion which drives $(W,O)$.  By conformally mapping and applying~\eqref{eqn::ac_eq_rel} (see Figure~\ref{fig::radial_bd}), we see that we can represent $m_t$ explicitly in terms of $W$ and $O$.  Let $\theta_t = \arg W_t - \arg O_t$.  Note that $\theta_t/2\pi$ (resp.\ $1-\theta_t/2\pi$) is equal to the harmonic measure of the counterclockwise segment of $\partial \D$ from $O_t$ to $W_t$ (resp.\ $W_t$ to $O_t$).  We have that
\begin{align*}
m_t
&= \lambda \left(1 - \frac{\theta_t}{2\pi}\right) -\lambda \frac{\theta_t}{2\pi} + \frac{\chi}{2\pi} \left(\int_{\theta_t}^{2\pi-\theta_t} s ds \right) + \left( \arg W_t + \frac{\pi}{2}\right) \chi \\
&= \lambda \left(1-\frac{\theta_t}{\pi}\right) + \chi \arg O_t + \frac{3}{2}\pi \chi\\
&= -\frac{\theta_t}{\sqrt{\kappa}} + \lambda + \chi \arg O_t + \frac{3}{2}\pi \chi.
\end{align*}
The integral in the first equality represents the contribution to the conditional mean from the winding terms in the boundary data.  This corresponds to integrating $\chi$ times the harmonic extension of the winding of $\partial \D$ to $\D$.  Since we are evaluating the harmonic extension at $0$, this is the same as computing the mean of $\chi$ times the winding of $\partial \D$.  Finally, the reason that the term $\big(\arg W_t+\tfrac{\pi}{2}\big)\chi$ appears is due to the coordinate change formula~\eqref{eqn::ac_eq_rel}. (Note that if $W_t = -i$ then the boundary data is $-\lambda$ immediately to the left of $-i$ and $\lambda$ immediately to the right.  Here, when we write $\arg W_t$ we are taking the branch of $\arg$ with values in $(-\pi,\pi)$ where the branch cut is on $(-\infty,0)$.  In particular, $\arg(-i) = -\tfrac{\pi}{2}$.)

Recall the form~\eqref{eqn::theta_equation} of the SDE for $\theta_t$.  We thus see that the cross-variation $\langle m, \theta \rangle_t$ of $m$ and $\theta$ is given by $-\sqrt{\kappa} t$.  Since the Brownian motion which drives $\theta$ is the same as that which drives $(W,O)$, we therefore have that the driving pair $(W^\beta, O^\beta)$ of $\eta^\beta$ solves the SDE~\eqref{eqn::sle_radial_equation} with $\mu = 0$ where the driving Brownian motion $B_t$ is replaced with $B_t + \beta \sqrt{\kappa} t$.  In other words, $(W^\beta,O^\beta)$ solves~\eqref{eqn::sle_radial_equation} with $\mu=\beta$, as desired.
\end{proof}

Now that we have established the existence component of Theorem~\ref{thm::alphabeta}, we can complete the proof of Theorem~\ref{thm::flow_line_interaction}.

\begin{proof}[Proof of Theorem~\ref{thm::flow_line_interaction}]
In Section~\ref{subsec::interaction}, we established Theorem~\ref{thm::flow_line_interaction} for ordinary GFF flow lines.  Since the law of $h_{\alpha \beta}$ away from $0$ is mutually absolutely continuous with respect to the law of the ordinary field on the same domain and GFF flow lines are local (i.e., a flow line started at a point in an open set $U$ stopped upon exiting $U$ depends only on the field in $U$), the interaction result for flow lines of $h_{\alpha \beta}$ follows from the interaction result for the ordinary GFF.
\end{proof}

\begin{remark}
\label{rem::flow_line_alpha_non_zero}
If we are given $h_{\alpha \beta}= h - \alpha \arg(\cdot) - \beta \log|\cdot|$ where $h$ is a GFF on a domain $D$ and we want to speak of ``the flow line started at $z \in D \setminus \{0\}$ with angle $\theta \in [0, 2\pi)$'' then, in order to decide how to get the flow line started, we have to know $h_{\alpha \beta}$ (or at least its restriction to a neighborhood of $z$) modulo an additive multiple of $2 \pi \chi$.  We also have to choose a branch of the argument function defined on a neighborhood of $z$ (or, alternatively, to define the argument on the universal cover of $\C \setminus \{0 \}$ and to let $z$ represent a fixed element of that universal cover).  However, once these two things are done, there is no problem in uniquely defining the flow line of angle $\theta$ beginning at $z$.

In order to define a flow line started from the origin of a fixed angle $\theta$ with $\theta \in [0,2\pi(1+\alpha/\chi))$, it is necessary to know $h_{\alpha \beta}$ modulo a global additive multiple of $2\pi(\chi+\alpha)$, at least in a neighborhood of $0$.  We can simultaneously draw a flow line from {\em both} an interior point $z \neq 0$ and from $0$ if we know the distribution both modulo a global additive multiple $2\pi(\chi+\alpha)$ {\em and} modulo a global additive multiple of $2\pi \chi$.  This is possible, for example, if $D$ is a bounded domain, and we know $h_{\alpha \beta}$ precisely (not up to additive constant).  It is also possible if $D = \C$ and the field is known modulo a global additive multiple of some constant which is an integer multiple of both $2\pi(\chi+\alpha)$ and $2\pi \chi$.
\end{remark}

\begin{remark}
\label{rem::flow_line_random_angle}
In the setting of Remark~\ref{rem::flow_line_alpha_non_zero}, it is still possible to draw a flow line $\eta$ starting from a point $z \neq 0$ even if we do not know $h_{\alpha\beta}$ in a neighborhood of $z$ up to a multiple of $2\pi \chi$.  This is accomplished by taking the ``angle'' of $\eta$ to be chosen at random.  This is explained in more detail in Remark~\ref{rem::alpha_beta_determined}.  
\end{remark}

\medskip
If we say that $\eta$ is a flow line of $h_{\alpha \beta}$ starting from $z \neq 0$, then we mean it is either generated in the sense of Remark~\ref{rem::flow_line_alpha_non_zero} if we know the the field up to a global multiple of $2\pi \chi$ or in the sense of Remark~\ref{rem::flow_line_random_angle} otherwise.  We end this subsection with the following, which combined with Proposition~\ref{prop::cross_merge} completes the proof of Theorem~\ref{thm::merge_cross}.

\begin{figure}[ht!]
\begin{center}
\includegraphics[scale=0.85]{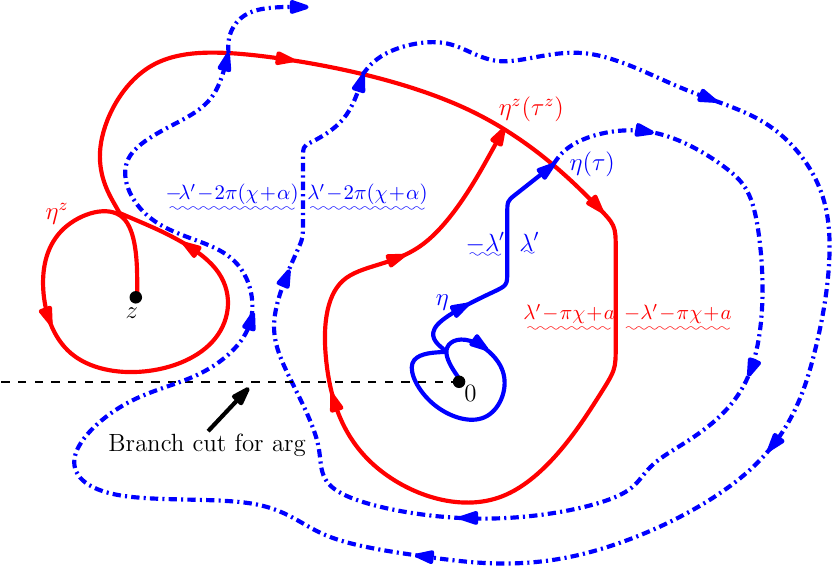}
\end{center}
\caption{\label{fig::alphacrossing_proof} 
Fix constants $\alpha > - \chi$, $\beta \in \R$, and suppose that $h_{\alpha \beta} = h - \alpha \arg(\cdot) - \beta \log|\cdot|$ where $h$ is a GFF on a domain $D \subseteq \C$, viewed as a distribution defined up to a global multiple of $2\pi(\chi+\alpha)$ if $D = \C$.  Let $\eta$ be the flow line of $h_{\alpha \beta}$ starting from $0$ and let $\eta^z$ be a flow line starting from $z \in D \setminus \{0\}$ (recall Remark~\ref{rem::flow_line_alpha_non_zero} and Remark~\ref{rem::flow_line_random_angle}).  Let $\tau^z$ be the first time that $\eta^z$ surrounds $0$ and let $\tau$ be the first time that $\eta$ hits $\eta^z((-\infty,\tau^z])$.  In between each time $\eta|_{(\tau,\infty)}$ crosses $\eta^z((-\infty,\tau^z])$, it must wind around $0$.  Indeed, its interaction with $\eta^z((-\infty,\tau^z])$ is absolutely continuous with respect to the case in which $\alpha = 0$ in between the times that it winds around $0$. In this case, Proposition~\ref{prop::cross_merge} implies that the paths can cross at most once.  Two such revolutions are shown.  Consequently, the height difference of the intersection at each crossing changes by $\pm 2\pi(\chi+\alpha)$, hence by Theorem~\ref{thm::flow_line_interaction} the number of times that $\eta|_{(\tau,\infty)}$ can cross $\eta^z((-\infty,\tau^z])$ is at most $\lfloor \pi \chi / (2\pi(\chi+\alpha)) \rfloor = \lfloor 1/(2(1+\alpha/\chi)) \rfloor$.
}
\end{figure}

\begin{proposition}
\label{prop::cross_finite}
Fix constants $\alpha > -\chi$, $\beta \in \R$, suppose that $D \subseteq \C$ is a domain, $h$ is a GFF on $D$, and $h_{\alpha \beta} = h - \alpha \arg(\cdot) - \beta \log|\cdot|$.  When $D = \C$, assume that $h$ is defined modulo a global multiple of $2\pi (\chi+\alpha)$.  Fix $z \in D$.  Let $\eta,\eta^z$ be flow lines of $h_{\alpha \beta}$ starting from $0$ (resp.\ $z$).  Assume that $\eta$ has zero angle.  In the case that $z \neq 0$, we assume that we either know $h$ up to a global multiple of $2\pi \chi$ or that $\eta^z$ has a random angle; see Remark~\ref{rem::flow_line_alpha_non_zero} and Remark~\ref{rem::flow_line_random_angle}.  There exists a constant $C(\alpha) < \infty$ such that $\eta$ and $\eta^z$ almost surely cross each other at most $C(\alpha)$ times.
Moreover, $\eta^z$ almost surely can cross itself at most the same constant $C(\alpha)$ times ($\eta$ does not cross itself).
\end{proposition}

Before we prove Proposition~\ref{prop::cross_finite}, we first restate \cite[Lemma~7.16]{MS_IMAG}:

\begin{lemma}
\label{lem::sle_kappa_rho_boundary_intersection}
Suppose that $\kappa \in (0,4)$ and $\rho^L,\rho^R > -2$.  Let $\vartheta$ be an $\SLE_\kappa(\rho^L;\rho^R)$ process in $\h$ from $0$ to $\infty$ with force points located at $x^L \leq 0 \leq x^R$.  Then the Lebesgue measure of $\vartheta \cap \partial \h$ is almost surely zero.
\end{lemma}

\begin{lemma}
\label{lem::sle_kappa_rho_harmonic}
Suppose that $\eta$ is an $\SLE_\kappa(\rho)$ process with $\kappa \in (0,4)$ and $\rho \in (-2,\tfrac{\kappa}{2}-2)$ in $\h$ from $0$ to $\infty$ with a single boundary force point located at $1$.  Let $\tau$ be the first time that $\eta$ hits $[1,\infty)$.  Then the law of $\eta(\tau)$ has a density with respect to Lebesgue measure on $[1,\infty)$.
\end{lemma}
\begin{proof}
For each $t > 0$, let $E_t$ be the event that $\eta$ has not swallowed $1$ by time $t$.  On $E_t$, we let $\psi_t$ be the unique conformal map $\h \setminus \eta([0,t]) \to \h$ which sends $\eta(t)$ to $0$ and fixes both $1$ and $\infty$.  Fix $x > 1$.  It is clear from the form of the driving function for chordal $\SLE_\kappa(\rho)$ that, on $E_t$, $\psi_t(x)$ has a density with respect to Lebesgue measure.  Consequently, the same is likewise true for $\psi_t^{-1}(x)$ (since $y \mapsto \psi_t(y)$ for $y \in [1,\infty)$ is smooth on $E_t$).  Let $\tau$ be the time that $\eta$ first hits $[1,\infty)$.  Let $\wt{\eta}$ be an independent copy of $\eta$ and let $\wt{\tau}$ be the first time that it hits $[1,\infty)$.  For a Borel set $A \subseteq [1,\infty)$, we have that
\begin{align*}
 \p[ \eta(\tau) \in A | E_t] &= \p[ \psi_t(\eta(\tau)) \in \psi_t(A) | E_t] = \p[ \wt{\eta}(\wt{\tau}) \in \psi_t(A) | E_t]\\
 &= \p[ \psi_t^{-1}(\wt{\eta}(\wt{\tau})) \in A | E_t].
 \end{align*}
This implies that, on $E_t$, the law of $\eta(\tau)$ has a density with respect to Lebesgue measure.  Since $\p[E_t] \to 1$ as $t \to 0$, it follows that $\eta(\tau)$ in fact has a density with respect to Lebesgue measure. 
\end{proof}

\begin{lemma}
\label{lem::exit_harmonic_measure}
Suppose that $h$ is a GFF on $\h$ with piecewise constant boundary data.  Let $\eta$ be the flow line of $h$ starting from $i$ and let $\tau$ be the first time that $\eta$ hits $\partial \h$.  Fix any open interval $I = (a,b) \subseteq \R$ on which the boundary data for $h$ is constant.  Assume that the probability of the event $E$ that $\eta$ exits $\h$ in $I$ is positive.  Conditional $E$, the law of $\eta(\tau)$ has a density with respect to Lebesgue measure on $I$. 
\end{lemma}
\begin{proof}
This follows from Lemma~\ref{lem::sle_kappa_rho_harmonic} and absolute continuity.
\end{proof}

\begin{lemma}
\label{lem::whole_plane_self_intersection_harmonic_measure}
Suppose that we have the same setup as Proposition~\ref{prop::cross_finite}.  Fix $w \in D \setminus \{z\}$ and let $P_w$ be the component of $D \setminus \eta^z$ which contains $w$.  Then the harmonic measure of the self-intersection points of $\eta^z$ which are contained in $\partial P_w$ as seen from $w$ is almost surely zero.
\end{lemma}
\begin{proof}
This follows from absolute continuity and Lemma~\ref{lem::sle_kappa_rho_boundary_intersection}.  In particular, the self-intersection set of $\eta^z$ which is contained in $\partial P_w$ can be described in terms of intersections of tails, which, by the proof of Proposition~\ref{prop::tail_interaction} can in turn be compared to the intersection of boundary emanating flow lines.
\end{proof}

\begin{proof}[Proof of Proposition~\ref{prop::cross_finite}]
The beginning of the proof in the case that $z \neq 0$ is explained in the caption of Figure~\ref{fig::alphacrossing_proof}.  We are left to bound the number of times that $\eta|_{[\tau,\infty)}$ can cross $\eta^z|_{[\tau^z,\infty)}$ (using the notation of the figure).  Let $\tau_0 = \tau$ and let $\tau_1 < \tau_2 < \cdots$ be the times at which $\eta|_{[\tau,\infty)}$ crosses $\eta^z|_{[\tau^z,\infty)}$.  For each $j$, let $\CD_j$ be the height difference of the paths when $\eta$ crosses at time $\tau_j$.  Assume (as is illustrated in Figure~\ref{fig::alphacrossing_proof}) that $\eta$ hits $\eta^z$ at time $\tau$ on its right side.  We are going to show by induction that, for each $j \geq 0$ for which $\tau_j < \infty$, that
\begin{enumerate}[(i)]
\item $\eta$ crosses from the right to the left of $\eta^z$ at time $\tau_j$,
\item $\eta$ almost surely does not hit a self-intersection point of $\eta^z$ at time $\tau_j$, and
\item the height difference $\CD_j$ of the paths upon crossing at this time satisfies
\begin{equation}
\label{eqn::height_difference_increase}
 \CD_j \geq \CD_{j-1} + 2\pi(\chi+\alpha).
\end{equation}
\end{enumerate}
Here, we take $\CD_{-1} = -\infty$ so that (iii) holds automatically for $j=0$.  That (ii) holds for $j=0$ follows from Lemma~\ref{lem::exit_harmonic_measure} and Lemma~\ref{lem::whole_plane_self_intersection_harmonic_measure}.  Finally, (i) holds by assumption. Upon completing the proof of the induction, Theorem~\ref{thm::flow_line_interaction} implies that $\CD_j \in (-\pi \chi,0)$ for each $j$, so~\eqref{eqn::height_difference_increase} implies that the largest $k$ for which $\tau_k < \infty$ is at most $\lfloor \pi / (2\pi(\chi+\alpha)) \rfloor = \lfloor 1/(2(\chi+\alpha)) \rfloor$.

Suppose that $k \geq 0$ and that (i)--(iii) hold for all $0 \leq j \leq k$.  We will show that (i)--(iii) hold for $j=k+1$.  Let $\sigma^z$ be the first time that $\eta^z$ hits $\eta^z(\tau^z)$ so that $\eta^z|_{[\sigma^z,\tau^z]}$ forms a loop around $0$.  That (ii) holds for $k$ implies that there exists a complementary pocket $P_k$ of $\eta^z|_{[\sigma_z,\infty)}$ and $\epsilon > 0$ such that $\eta([\tau_k,\tau_k+\epsilon]) \subseteq  \ol{P}_k$.  Let $S_1$ (resp.\ $S_2$) be the first (resp.\ second) segment of $\partial P_k$ which is traced by $\eta^z|_{[\sigma^z,\infty)}$.  Theorem~\ref{thm::flow_line_interaction} implies that $\eta|_{[\tau_k,\tau_{k+1}]}$ cannot cross $S_1$ (this is the side of $P_k$ that $\eta$ crossed into at time $\tau_k$).  Therefore $\eta|_{[\tau_k,\tau_{k+1}]}$ can exit $P_{k}$ only through $S_2$ or through the terminal point of $P_k$, i.e., the last point on $\partial P_k$ drawn by $\eta^z|_{[\sigma^z,\infty)}$.

Note that $\eta|_{[\tau_k,\tau_{k+1}]}$ is coupled with the conditional GFF $h|_{P_k}$ given $\eta^z$ starting from $\eta(\tau_k)$ as a flow line.  (Theorem~\ref{thm::flow_line_interaction} implies that the conditional mean of $h|_{P_k}$ given $\eta$ drawn up to any stopping time before exiting $\ol{P}_k$ and $\eta^z$ has flow line boundary conditions on $\eta$ and the arguments of \cite[Section~6]{MS_IMAG} imply that $\eta$ has a continuous Loewner driving function viewed as a path in $P_k$.)  In particular, the conditional law of $\eta|_{[\tau_k,\tau_{k+1}]}$ given $\eta^z$ is that of an $\SLE_\kappa(\ul{\rho})$ process.  Thus, in the case that $\eta|_{[\tau_k,\tau_{k+1}]}$ crosses out of $P_k$, that (ii) holds for $j=k+1$ follows from Lemma~\ref{lem::sle_kappa_rho_boundary_intersection}, Lemma~\ref{lem::whole_plane_self_intersection_harmonic_measure}, and absolute continuity.  In this case, it is also immediate from the setup that (i) and (iii) hold for $j=k+1$ (it is topologically impossible for the path to cross out of $P_k$ through $S_2$ from the left to the right).  It is not difficult to see that if $\eta$ does not cross out of $P_k$, i.e., exits $P_k$ at the last point on $\partial P_k$ drawn by $\eta^z$, then it does not subsequently cross $\eta^z$.
\end{proof}

\subsection{Flow lines are determined by the field}
\label{subsec::uniqueness}

We are now going to prove Theorem~\ref{thm::uniqueness}, that in the coupling $(h,\eta)$ of Theorem~\ref{thm::existence}, $\eta$ is almost surely determined by $h$, as well as the corresponding statement from Theorem~\ref{thm::alphabeta}.  (We will also complete our proof of the uniqueness of the law of the coupling in the case of flow lines started from interior points associated with a GFF on a proper subdomain $D$ of $\C$.)  This is the interior flow line analog of the uniqueness statement for boundary emanating flow lines from \cite[Theorem~1.2]{MS_IMAG}.  As we will explain just below, the result is a consequence of the following proposition.  We remind the reader of Remark~\ref{rem::multiple_paths_not_determined}.

\begin{figure}[ht!]
\begin{center}
\includegraphics[scale=0.85]{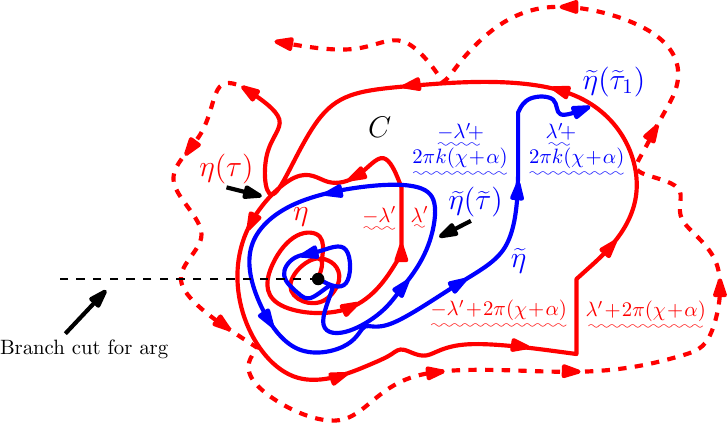}
\end{center}
\caption{\label{fig::interior_determined}  { Proof that paths are determined by the field in the self-intersecting regime.  Suppose $h$ is a whole-plane GFF, $\alpha \in (-\chi,\tfrac{3\sqrt{\kappa}}{4}-\tfrac{2}{\sqrt{\kappa}})$, $\beta \in \R$, and $h_{\alpha \beta} = h - \alpha \arg(\cdot) - \beta \log|\cdot|$, viewed as a distribution defined up to a global multiple of $2\pi (\chi + \alpha)$.  Let $\eta$ (red) and $\wt{\eta}$ (blue) be coupled with $h_{\alpha \beta}$ as in Theorem~\ref{thm::alphabeta}, conditionally independent given $h_{\alpha \beta}$, so that each is a flow line of $h_{\alpha \beta}$ starting from $0$.  Let $\wt{\tau}$ be a stopping time for $\CF_t = \sigma(\wt{\eta}(s) : s \leq t,\ \ \eta)$ such that either $\wt{\eta}(\wt{\tau}) \notin \eta$ or $\wt{\tau} = \infty$.  Lemma~\ref{lem::radial_loop} implies that all of the connected components of $\C \setminus \eta$ are bounded; on $\{\wt{\tau} < \infty\}$, let $C$ be the one which contains $\wt{\eta}(\wt{\tau})$.  We are going to argue that $\wt{\eta}|_{[\wt{\tau},\infty)}$ merges with $\eta$ upon exiting $C$; this completes the proof by scale invariance.  Moreover, by Lemma~\ref{lem::do_not_cross} it suffices to show that the only possibilities are for the paths eventually to cross or merge.  Let $\tau$ be the time that $\eta$ closes the pocket containing $\wt{\eta}(\wt{\tau})$.  Since $\wt{\eta}$ is almost surely unbounded, it must eventually exit $C$, hence intersect $\eta$ after time $\wt{\tau}$, say at time $\wt{\tau}_1$.  The boundary data of the conditional field $h_{\alpha \beta}$ given $\eta$ and $\wt{\eta}|_{(-\infty,\wt{\tau}_1]}$ is as illustrated where $k \in \Z$.  If $k < 1$, then Theorem~\ref{thm::flow_line_interaction} implies that $\wt{\eta}$ bounces off $\eta$ upon hitting at time $\wt{\tau}_1$.  Thus, in this case, $\wt{\eta}$ can exit $C$ only through $\eta(\tau)$, hence must hit the other boundary of $C$, which upon intersecting Theorem~\ref{thm::flow_line_interaction} implies that $\wt{\eta}$ either crosses $\eta$ to get out of $C$ or merges with $\eta$.  If $k =1$, then $\wt{\eta}$ merges with $\eta$ at time $\wt{\tau}_1$.  If $k > 1$, then $\wt{\eta}$ crosses $\eta$ at time $\wt{\tau}_1$.  The analysis when $\wt{\eta}$ hits the other part of the boundary of $C$ first is similar.}}
\end{figure}

\begin{figure}[ht!]
\begin{center}
\includegraphics[scale=0.85]{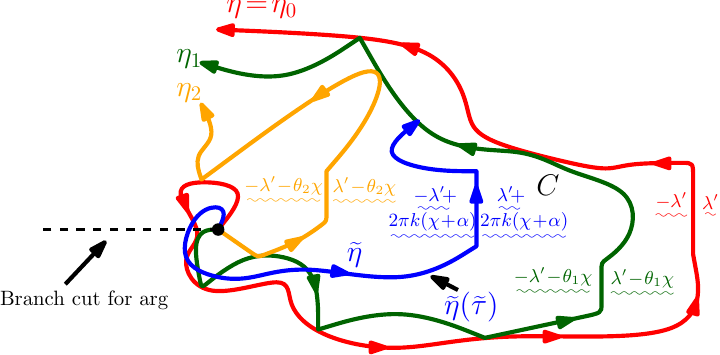}
\end{center}
\caption{\label{fig::interior_determined_cannot_intersect} {The proof that flow lines are determined by the field in the case that they do not intersect themselves.  Throughout, we use the notation of Figure~\ref{fig::interior_determined}; $\alpha \geq \tfrac{3\sqrt{\kappa}}{4}-\tfrac{2}{\sqrt{\kappa}}$.  As in Figure~\ref{fig::interior_determined}, we are going to show that $\wt{\eta}$ almost surely merges with $\eta$; this suffices by scale invariance.  The argument is similar to that of Figure~\ref{fig::non_intersect_almost_surely_merge}. Fix evenly spaced angles $0 < \theta_1 < \cdots < \theta_n$ and, conditionally on $\eta$, let $\eta_i$ be the flow line of the \emph{conditional field} $h_{\alpha \beta}$ on $\C \setminus \eta$ with angle $\theta_i$ starting from $0$.  The arguments of Section~\ref{subsec::interaction} imply that $\eta$ and the $\eta_i$ obey the interaction rules of Theorem~\ref{thm::flow_line_interaction}.  We assume that $n$ is large enough so that $\eta_i$ almost surely intersects both $\eta_{i-1}$ and $\eta_{i+1}$ where $\eta_0 = \eta$ and the indices are taken mod $n$ (that we can do this follows from \cite[Theorem~1.5]{MS_IMAG}).  Illustrated is the boundary data for $h_{\alpha \beta}$ given $\eta,\eta_1,\ldots,\eta_n$ modulo an additive constant in $2\pi(\chi+\alpha) \Z$.  Suppose that $\wt{\tau}$ is a stopping time for $\CF_t = \sigma(\wt{\eta}(s) : s \leq t,\ \eta_0,\ldots,\eta_n)$ such that either $\wt{\eta}(\wt{\tau}) \notin \cup_{j=0}^n \eta_j$ or $\wt{\tau}=\infty$.  Each of the connected components of $\C \setminus \cup_{i=0}^n \eta_i$ is bounded.  On $\{\wt{\tau} < \infty\}$, let $C$ be the one containing $\wt{\eta}(\wt{\tau})$; since $\wt{\eta}$ is unbounded, it must exit $C$.  Theorem~\ref{thm::flow_line_interaction} implies that $\wt{\eta}$ has to cross or merge into one of the two flow lines whose boundary generates $C$; this follows from an analysis which is analogous to that given in Figure~\ref{fig::non_intersect_almost_surely_merge}.  Therefore $\wt{\eta}$ crosses out of, hence by Lemma~\ref{lem::do_not_cross} cannot intersect the interior of, any pocket whose boundary does not contain a segment of $\eta$. If $\wt{\eta}$ is in a pocket part of whose boundary is given by a segment of $\eta$, again by Lemma~\ref{lem::do_not_cross}, since $\wt{\eta}$ cannot cross out, it must hit and merge with $\eta$ (if $\wt{\eta}$ is in such a pocket and does not hit $\eta$ with the correct height difference to merge, then it either crosses $\eta$ or crosses the flow line which forms the other boundary of the pocket). }
}
\end{figure}

\begin{proposition}
\label{prop::uniquely_determined}
Fix constants $\alpha > -\chi$, $\beta \in \R$, and let $h_{\alpha \beta} = h - \alpha \arg(\cdot) - \beta \log|\cdot|$, viewed as a distribution defined up to a global multiple of $2\pi(\chi+\alpha)$, where $h$ is a whole-plane GFF.  Suppose that $\eta$, $\wt{\eta}$ are coupled with $h_{\alpha \beta}$ as flow lines starting from the origin as in the statement of Theorem~\ref{thm::existence} ($\alpha = 0$) or Theorem~\ref{thm::alphabeta} ($\alpha \neq 0$) such that given $h_{\alpha\beta}$, $\eta$ and $\wt{\eta}$ are conditionally independent.  Then $\eta = \wt{\eta}$ almost surely.
\end{proposition}

Before we prove the proposition, we record the following simple lemma.

\begin{lemma}
\label{lem::do_not_cross}
Fix constants $\alpha > -\chi$, $\beta \in \R$, and let $h_{\alpha \beta} = h - \alpha \arg(\cdot) - \beta \log|\cdot|$, viewed as a distribution defined up to a global multiple of $2\pi(\chi+\alpha)$, where $h$ is a whole-plane GFF.  Fix angles $\theta,\wt{\theta} \in [0,2\pi(\chi+\alpha))$ and let $\eta,\wt{\eta}$ be the flow lines of $h_{\alpha \beta}$ with angles $\theta,\wt{\theta}$ starting from $0$ taken to be conditionally independent given $h_{\alpha \beta}$.  Then $\eta$ and $\wt{\eta}$ almost surely do not cross.  Likewise, if $\wh{\theta} \in \R$ and $\wh{\eta}$ is the flow line of the conditioned field $h_{\alpha \beta}$ given $\eta$ with angle $\wh{\theta}$, then $\wh{\eta}$ and $\wt{\eta}$ almost surely do not cross.
\end{lemma}
\begin{proof}
We know from Proposition~\ref{prop::cross_finite} that $\eta$ and $\wt{\eta}$ cross each other at most a finite number of times.  Let $R$ be the distance to $0$ of the last crossing of $\eta$ and $\wt{\eta}$.  We take $R = 0$ if $\eta$ does not cross $\wt{\eta}$.  By the scale invariance of the coupling $(h_{\alpha \beta}, \eta,\wt{\eta})$, it follows that $R = 0$ almost surely, which proves the first assertion of the lemma.  The second assertion is proved similarly.  (In particular, the argument of Proposition~\ref{prop::cross_finite} implies that $\wt{\eta}$ and $\wh{\eta}$ can cross each other at most the same constant $C(\alpha)$ times.)
\end{proof}

\begin{proof}[Proof of Proposition~\ref{prop::uniquely_determined}]
Recall that if $\beta = 0$ then $\eta$ and $\wt{\eta}$ are distributed as whole-plane $\SLE_\kappa(\rho)$ processes with $\rho = 2-\kappa+2\pi \alpha / \lambda$.  Thus the paths are self-intersecting in this case when $\alpha \in (-\chi,3\sqrt{\kappa}/4-2/\sqrt{\kappa})$ (so that $\rho \in (-2,\kappa/2-2)$) and are simple if $\alpha \geq 3\sqrt{\kappa}/4 - 2/\sqrt{\kappa}$ (so that $\rho \geq \kappa/2-2$).  Moreover, for $\beta \neq 0$ these are the same ranges of $\alpha$ in which the paths are either self-intersecting or simple, respectively.  Indeed, these assertions follow from Lemma~\ref{lem::radial_critical_for_hitting}.  We will handle the two cases separately.

We first suppose that we are in the self-intersecting regime.
Suppose that $\wt{\tau}$ is a stopping time for the filtration $\CF_t = \sigma(\wt{\eta}(s) : s \leq t,\ \ \eta)$ such that $\wt{\eta}(\wt{\tau})$ is not contained in the range of $\eta$; on the event $\{\wt{\eta} = \eta\}$, we take $\wt{\tau} = \infty$.  It is explained in Figure~\ref{fig::interior_determined} that, on $\{\wt{\tau} < \infty\}$, $\wt{\eta}|_{[\wt{\tau},\infty)}$  almost surely merges into and does not separate from $\eta$, say at time $\wt{\sigma}$.  Let $R$ be the modulus of the point where $\wt{\eta}$ first merges into $\eta$.  Then what we have shown implies that $R$ is finite almost surely.  The scale invariance of the coupling $(h_{\alpha \beta}, \eta, \wt{\eta})$ implies that the law of $R$ is also scale invariant, hence $R = 0$ almost surely.  Therefore $\p[ \wt{\eta} \neq \eta] = 0$.

We now suppose that we are in the regime of $\alpha$ in which the paths are simple.  We condition on $\eta$ and fix evenly spaced angles $0 < \theta_1 < \cdots < \theta_n$ such that the following is true.  Let $\eta_i$ be the flow line of the \emph{conditional} GFF $h_{\alpha \beta}$ on $\C \setminus \eta$ given $\eta$ starting at $0$ with angle $\theta_i$.  We assume that the angles have been chosen so that $\eta_i$ almost surely intersects both $\eta_{i+1}$ and $\eta_{i-1}$ where the indices are taken mod $n$ and $\eta_0 = \eta$ (by \cite[Theorem~1.5]{MS_IMAG} we know that we can arrange this to be the case --- see also Figure~\ref{fig::non_intersect_almost_surely_merge} as well as Section~\ref{subsec::imaginary}).  Again by the scale invariance of the coupling $(h_{\alpha \beta}, \eta, \wt{\eta})$, it suffices as in the case that paths are self-intersecting to show that $\wt{\eta}|_{[\wt{\tau},\infty)}$ almost surely merges with $\eta$.  The argument is given in the caption of Figure~\ref{fig::interior_determined_cannot_intersect}.  We remark that the reason that we took the flow lines $\eta_1,\ldots,\eta_n$ in Figure~\ref{fig::interior_determined_cannot_intersect} to be for the conditional field $h_{\alpha \beta}$ given $\eta$ is that we have not yet shown that the rays of the field started at the same point with varying angle are monotonic.
\end{proof}

\begin{proof}[Proof of Theorem~\ref{thm::uniqueness} and Theorem~\ref{thm::alphabeta}]
Proposition~\ref{prop::uniquely_determined} implies Theorem~\ref{thm::uniqueness} and the uniqueness statement of Theorem~\ref{thm::alphabeta} in the case that the domain $D$ of the GFF is given by the whole-plane $\C$.  Thus to complete the proofs of these results, we just need to handle the setting that $D$ is a proper subdomain of $\C$.  This, however, follows from absolute continuity (Proposition~\ref{prop::local_set_whole_plane_bounded_compare}).
\end{proof}

\begin{proof}[Proof of Theorem~\ref{thm::existence} and Theorem~\ref{thm::alphabeta}, uniqueness for general domains]
We will first give the proof in the case that $\kappa \in (8/3,4)$ and $\alpha=0$.  Suppose that $h$ is a GFF on a general domain $D$ and that $\eta$ is a path coupled with $h$ with the property that for each $\eta$-stopping time $\tau$ we have that $\eta([0,\tau])$ is a local set for $h$ and the conditional law of $h$ given $\eta|_{[0,\tau]}$ is that of a GFF on $D \setminus \eta([0,\tau])$ with flow line boundary conditions on $\eta([0,\tau])$.  We assume without loss of generality that the starting point of $\eta$ is equal to $0$.  Suppose that $\wt{\eta}$ is a flow line of $h$ starting from $0$ whose law is induced from the whole-plane measure on flow lines as in the proof of Theorem~\ref{thm::existence}.  Then we know that the law of $\wt{\eta}$ is, by construction, absolutely continuous with respect to that of a whole-plane $\SLE_\kappa(2-\kappa)$, at least up until the first time that it hits the ball of radius $r = \min(1,\dist(0,\partial D))/2$.  As $\kappa \in (8/3,4)$, it therefore follows that for each $s \in (0,r)$ the number of components that $\wt{\eta}$ separates from $\partial D$ before hitting $\partial B(0,s)$ is infinite almost surely.  We know that the flow line interaction result applies to $\eta$, $\wt{\eta}$.  In particular, $\eta$ can cross $\wt{\eta}$ at most once and if $\eta$ intersects $\wt{\eta}$ with a height difference of $0$ then the two paths merge and do not subsequently separate.  If $\eta$ stopped upon hitting $\partial B(0,s)$ has not yet merged with  $\wt{\eta}$ then it would be forced to cross $\wt{\eta}$ an infinite number of times before hitting $\partial B(0,s)$.  This is a contradiction and therefore $\eta$ must merge with $\wt{\eta}$ before hitting $\partial B(0,s)$.  Since this holds for all $s \in (0,r)$, we conclude that $\eta$ and $\wt{\eta}$ coincide up until hitting $\partial B(0,r)$.  The flow line interaction result implies that the two paths cannot separate after hitting $\partial B(0,r)$ and therefore agree for all time.

The case that $\kappa \in (0,8/3]$ and $\alpha=0$ is proved in a similar manner except one uses a collection of flow lines $\wt{\eta}_1,\ldots,\wt{\eta}_n$ starting from $0$ with equally spaced angles, all induced from the whole-plane measure, in place of $\wt{\eta}$ as in the proof of Proposition~\ref{prop::uniquely_determined} given above.

The case for general values of $\alpha$ is handled similarly.
\end{proof}

The proof that the boundary emanating flow lines of the GFF are uniquely determined by the field established in \cite{DUB_PART} and extended in \cite{MS_IMAG} is based on $\SLE$ duality (a flow line can be realized as the outer boundary of a certain counterflow line).  There is also an $\SLE$ duality based approach to establishing the uniqueness of flow lines started at interior points which will be a consequence of the material in Section~\ref{subsec::duality}.  The duality approach to establishing these uniqueness results also gives an alternative proof of the merging phenomenon and the boundary emanating version of this is discussed in \cite[Section~7]{MS_IMAG}.

Now that we have proved Theorem~\ref{thm::uniqueness} , we can prove Theorem~\ref{thm::commutation}.

\begin{proof}[Proof of Theorem~\ref{thm::commutation}]
The statement that $\cup_{j=1}^N \eta_{\xi_j}((-\infty,\tau_j])$ is a local set for $h$ is a consequence of \cite[Proposition~6.2]{MS_IMAG} and Theorem~\ref{thm::uniqueness}.  Consequently, the statement regarding the conditional law of $h$ given $\eta_{\xi_1}|_{(-\infty,\tau_1]},\ldots,\eta_{\xi_N}|_{(-\infty,\tau_N]}$ follows since we know the conditional law of $h$ given $\eta_1,\ldots,\eta_N$ and Proposition~\ref{prop::cond_union_mean}.  The final statement of the theorem follows \cite[Theorem~1.2]{MS_IMAG}.
\end{proof}

Theorem~\ref{thm::alphabeta} implies that we can simultaneously construct the entire family of rays of the GFF starting from a countable, dense set of points each with the same angle as a deterministic function of the underlying field.  We will now prove Theorem~\ref{thm::field_determined_by_tree}, that the family of rays in fact determine the field.

\begin{proof}[Proof of Theorem~\ref{thm::field_determined_by_tree}]
Let $(z_n)$ be a countable, dense set of $D$ and fix $\theta \in [0,2\pi)$.  For each $n \in \N$, let $\eta_n$ be the flow line of $h$ starting from $z_n$ with angle $\theta$.  For each $N \in \N$, let $\CF_N$ be the $\sigma$-algebra generated by $\eta_1,\ldots,\eta_N$.  Fix $f \in C_0^\infty(D)$ and recall that $(h,f)$ denotes the $L^2$ inner product of $h$ and $f$.  We are going to complete the proof of the theorem by showing that
\begin{equation}
\label{eqn::field_determined_by_tree}
\lim_{N \to \infty} \E\big[ (h,f) \big| \CF_N \big] = (h,f) \quad\text{almost surely}.
\end{equation}
By the martingale convergence theorem, the limit on the left side of~\eqref{eqn::field_determined_by_tree} exists almost surely.  Consequently, to prove~\eqref{eqn::field_determined_by_tree}, it in turn suffices to show that
\begin{equation}
\label{eqn::conditional_variance_to_zero}
\lim_{N \to \infty} \var[(h,f)|\CF_N] = 0 \quad\text{almost surely}.
\end{equation}
For each $N$, let $G_N$ denote the Green's function of $D_N = D \setminus \cup_{j=1}^N \eta_j$.  We note that $G_N$ makes sense for each $N \in \N$ since $D_N$ almost surely has  harmonically non-trivial boundary.  Moreover, we have that $G_{N+1}(x,y) \leq G_N(x,y)$ for each fixed $x,y \in D_N$ distinct and $N \in \N$ (this can be seen from the stochastic representation of $G_N$).  Let $K$ be a compact set which contains the support of $f$ and let $K_N = K \cap D_N$.  Since
\begin{align*}
 \var[(h,f)|\CF_N]
 =& \int_{D_N} \int_{D_N} f(x) G_N(x,y) f(y) dx dy\\
 \leq& \| f \|_\infty^2 \int_{K_N} \int_{K_N}  G_N(x,y) dx dy
\end{align*}
and $G_1$ is integrable on $K_1 \supseteq K_N$, by the dominated convergence theorem it suffices to show that $\lim_{N \to \infty} G_N(x,y) = 0$ for each fixed $x,y \in D$ distinct.  This follows from the Beurling estimate \cite[Theorem~3.69]{LAW05} (by the stochastic representation of $G_N$) since there exists a subsequence $(z_{n_k})$ of $(z_n)$ with $\lim_{k \to \infty} z_{n_k} = x$.  This proves~\eqref{eqn::conditional_variance_to_zero}, hence the theorem.
\end{proof}

Theorem~\ref{thm::alphabeta} also implies that we can simultaneously construct the entire family of rays of the GFF starting from any single point (resp.\ the location of the conical singularity) if $\alpha = 0$ (resp.\ $\alpha \neq 0$) with varying angles as a deterministic functional of the underlying field, that is, the family of flow lines $\eta_\theta$ of $h_{\alpha \beta} + \theta \chi$ for $\theta \in [0,2\pi(\chi+\alpha))$.  We are now going to show that the rays are monotonic in their angle as well as compute the conditional law of one such path given another.  This is in analogy with \cite[Theorem~1.5]{MS_IMAG} for flow lines emanating from interior points.

\begin{figure}[ht!]
\begin{center}
\includegraphics[width=\textwidth]{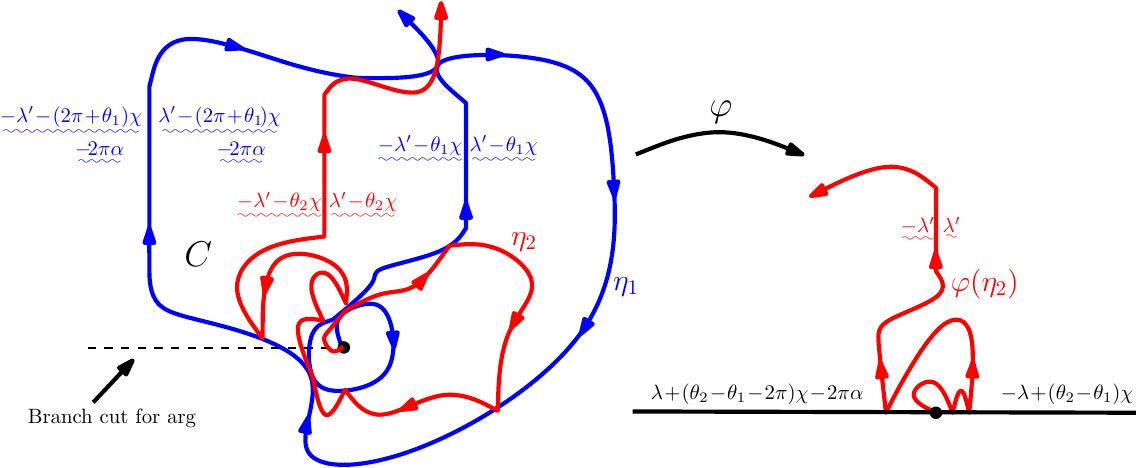}
\end{center}
\caption{\label{fig::conditional_law}
Suppose that $h$ is a whole-plane GFF, $\alpha > -\chi$, and $\beta \in \R$.  Let $h_{\alpha \beta} = h - \alpha \arg(\cdot) - \beta \log|\cdot|$, viewed as a distribution defined up to a global multiple of $2\pi(\chi+\alpha)$.  Fix $\theta_1,\theta_2 \in [0,2\pi(1+\alpha/\chi))$ with $\theta_1 < \theta_2$ and, for $i=1,2$, let $\eta_i$ be the flow line of $h_{\alpha \beta}$ starting at $0$ with angle $\theta_i$.  Then the conditional law of $\eta_2$ given $\eta_1$ is that of an $\SLE_\kappa(\rho^L;\rho^R)$ process independently in each of the connected components of $\C \setminus \eta_1$ where $\rho^L = (2\pi+\theta_1-\theta_2)\chi /\lambda + 2\pi \alpha/\lambda -2$ and $\rho^R = (\theta_2-\theta_1)\chi/\lambda -2$.  This can be seen by fixing a connected component $C$ of $\C \setminus \eta_1$ and letting $\varphi \colon C \to \h$ be a conformal map which takes the first point on $\partial C$ traced by $\eta_1$ to $0$ and the last to $\infty$.  Then the GFF $h_{\alpha \beta} \circ \varphi^{-1} - \chi (\arg \varphi^{-1})' + \theta_2 \chi$ on $\h$ has the boundary data depicted on the right side above, up to an additive constant in $2\pi (\chi+\alpha) \Z$, from which we can read off the conditional law of $\eta_2$ in $C$.}
\end{figure}

\begin{proposition}
\label{prop::conditional_law}
Fix constants $\alpha > -\chi$, $\beta \in \R$, and angles $\theta_1, \theta_2 \in [0,2\pi(1+\alpha/\chi))$ with $\theta_1 < \theta_2$.  Let $h$ be a whole-plane GFF and $h_{\alpha\beta} = h-\alpha \arg(\cdot) - \beta\log|\cdot|$, viewed as a distribution defined up to a global multiple of $2\pi(\chi+\alpha)$.  For $i=1,2$, let $\eta_i$ be the flow line of $h_{\alpha \beta}$ starting from $0$ with angle $\theta_i$.  Then $\eta_1$ almost surely does not cross (though may bounce off) $\eta_2$.  In particular, if $\eta_1$ is self-intersecting, then $\eta_2$ visits the components of $\C \setminus \eta_1$ according to their natural ordering (the order that their boundaries are drawn by $\eta_1$).  Moreover, the conditional law of $\eta_2$ given $\eta_1$ is independently that of a chordal $\SLE_\kappa(\rho^L;\rho^R)$ process in each of the components of $\C \setminus \eta_1$ with
\[ \rho^L = \frac{(2\pi +\theta_1 - \theta_2) \chi + 2\pi \alpha}{\lambda} - 2\quad\text{and}\quad\rho^R = \frac{(\theta_2-\theta_1) \chi}{\lambda} - 2.\]
If $\theta_3 \in [0,2\pi(1+\alpha/\chi))$ with $\theta_3 > \theta_2$ and $\eta_3$ is the flow line of $h_{\alpha \beta}$ starting from $0$ with angle $\theta_3$, then the conditional law of $\eta_2$ given $\eta_1$ and $\eta_3$ is independently that of a chordal $\SLE_\kappa(\rho^L;\rho^R)$ process in each of the components of $\C \setminus (\eta_1 \cup \eta_3)$ with
\[ \rho^L = \frac{(\theta_3 - \theta_2) \chi}{\lambda} - 2\quad\text{and}\quad\rho^R = \frac{(\theta_2 - \theta_1) \chi}{\lambda} - 2.\]
\end{proposition}
\begin{proof}
The first assertion of the proposition, that $\eta_1$ and $\eta_2$ almost surely do not cross, is a consequence of Lemma~\ref{lem::do_not_cross}.  The remainder of the proof is explained in the caption of Figure~\ref{fig::conditional_law}, except for three points.  First, we know that the conditional mean of $h$ given $\eta_1$ and $\eta_2$ up to any fixed pair of stopping times does not exhibit singularities at their intersection points by Theorem~\ref{thm::flow_line_interaction}.  Secondly, that $\eta_2$ viewed as a path in complementary connected component of $\C \setminus \eta_1$ has a continuous Loewner driving function.  This follows from the arguments of \cite[Section~6.2]{MS_IMAG}.  Thirdly, the reason that $\eta_1$ intersects $\eta_2$ on its right side with a height difference of $(\theta_1-\theta_2) \chi$ (as illustrated in Figure~\ref{fig::conditional_law}) rather than $(\theta_1-\theta_2) \chi + 2\pi (\chi + \alpha) k$ for some $k \in \Z \setminus \{0\}$ is that this would lead to the paths crossing, contradicting Lemma~\ref{lem::do_not_cross}.  The result then follows by invoking \cite[Theorem~2.4 and Proposition~6.5]{MS_IMAG}.  The proof of the final assertion follows from an analogous argument.
\end{proof}

Proposition~\ref{prop::conditional_law} gives the conditional law of one GFF flow line given another, both started at a common point.  We will study in Section~\ref{sec::timereversal} the extent to which this resampling property characterizes their joint law.  We also remark that it is possible to state and prove (using the same argument) an analog of Proposition~\ref{prop::conditional_law} which holds for GFF flow lines on a bounded domain, though it is slightly more complicated to describe the conditional law of one path given the other because one may have to deal with force points on the outer boundary.  (Moreover, when a path hits the boundary, it is sometimes possible to ``branch'' the path in two different directions and the way that this is to be done has to be specified.) We recall that some additional discussion about what can happen when single path hits the boundary appears in Section~\ref{subsubsec::flow_line_interaction} (which includes the proof of Theorem~\ref{thm::flow_line_interaction}).

\subsection{Transience and endpoint continuity}
\label{subsec::transience}

Suppose that $h$ is a whole-plane GFF, $\alpha > -\chi$, and $\beta \in \R$.  Let $\eta$ be the flow line of $h_{\alpha \beta} = h-\alpha \arg(\cdot) - \beta\log|\cdot|$, viewed as a distribution defined up to a global multiple of $2\pi(\chi + \alpha)$, starting from $0$.  In this section, we are going to prove that $\eta$ is transient, i.e.\ $\lim_{t \to \infty} \eta(t) = \infty$ almost surely.  This is equivalent to proving the transience of whole-plane $\SLE_\kappa^\mu(\rho)$ processes for $\kappa \in (0,4)$, $\rho > -2$, and $\mu \in \R$.  This in turn completes the proof of Theorem~\ref{thm::transience} for $\kappa \in (0,4)$; we will give the proof for $\kappa' > 4$ in Section~\ref{sec::duality_space_filling}.  Using the relationship between whole-plane and radial $\SLE_\kappa^\mu(\rho)$ processes described in Section~\ref{subsec::SLEoverview}, this is equivalent to proving the so-called ``endpoint continuity'' of radial $\SLE_\kappa^\mu(\rho)$ for $\kappa \in (0,4)$, $\rho > -2$, and $\mu \in \R$.  This means that if $\eta$ is a radial $\SLE_\kappa^\mu(\rho)$ process in $\D$ and targeted at $0$ then $\lim_{t \to \infty} \eta(t) = 0$ almost surely.  This was first proved by Lawler in \cite{LAW_ENDPOINT} for ordinary radial $\SLE_\kappa$ processes and the result we prove here extends this to general $\SLE_\kappa^\mu(\rho)$ processes, though the method of proof is different.   The main ingredient in the proof of this is the following lemma.

\begin{figure}[ht!]
\begin{center}
\includegraphics[scale=0.80]{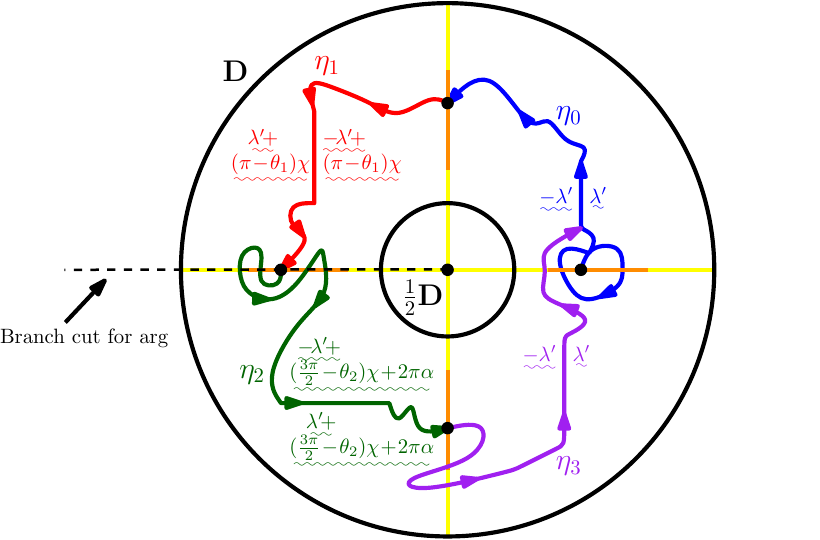}
\end{center}
\caption{\label{fig::endpoint_continuity}  {\small The proof of Lemma~\ref{lem::flow_line_ring}, the main input in the proof of the endpoint continuity (resp.\ transience) of radial (resp.\ whole-plane) $\SLE_\kappa^\mu(\rho)$ processes for $\kappa \in (0,4)$, $\rho > -2$, and $\mu \in \R$.  Suppose that $h$ is a whole-plane GFF, $\alpha > - \chi$, $\beta \in \R$, and $h_{\alpha \beta} = h - \alpha \arg(\cdot) - \beta \log|\cdot|$, viewed as a distribution defined up to a global multiple of $2\pi(\chi+\alpha)$.  Let $n = 2\lceil |1+\alpha/\chi| \rceil$; in the illustration, $n=3$.  We construct $\eta_0,\eta_1,\ldots,\eta_n$ as follows.  For $0 \leq j \leq n$, we let $\Delta = 2\pi (1+\alpha/\chi) / n \in [-\pi,\pi]$.  Let $\eta_0$ be a flow line of $h_{\alpha \beta}$ starting from $3/4$.  Let $E_0$ be the event that $\eta_0$ wraps around the origin with a counterclockwise orientation without leaving $\D \setminus (\D/2)$ before hitting the straight line with angle $2\pi/(n+1)$ starting from the origin, say at time $\tau_0$, and that $|\eta_0(\tau_0)| \in [2/3,3/4]$.  Then $\p[E_0] > 0$ by Lemma~\ref{lem::path_close} (extended to the case $\alpha,\beta \neq 0$).  We inductively define events $E_j$ for $1 \leq j \leq n-1$ as follows.  Given $E_{j-1}$, we let $\eta_j$ be the flow line of $h_{\alpha \beta}$ starting from $\eta_{j-1}(\tau_{j-1})$ with angle $\theta_j = j \Delta$ (relative to the angle of $\eta_0$) and let $E_j$ be the event that $\eta_j$ wraps counterclockwise around the origin without leaving $\D \setminus (\D/2)$ until hitting the line through the origin with angle $2 \pi(j+1)/(n+1)$, say at time $\tau_j$, with $|\eta_j(\tau_j)| \in [2/3,3/4]$.  By Lemma~\ref{lem::path_close}, $\p[E_j] > 0$.  On $E_{n-1}$, we let $F=E_n$ be the event that $\eta_n$, the flow line of $h_{\alpha \beta}$ starting at $\eta_{n-1}(\tau_{n-1})$ with angle $n \Delta$ hits and merges with $\eta_0$ without leaving $\D \setminus (\D/2)$.  Then $\p[F] > 0$ by Lemma~\ref{lem::path_close_hit} (extended to the case $\alpha,\beta \neq 0$).  The boundary data for $h_{\alpha \beta}$ given $\eta_0,\ldots,\eta_n$ is as illustrated up to an additive constant in $\R$.}
}
\end{figure}

\begin{lemma}
\label{lem::flow_line_ring}
Let $h$ be a GFF on $\C$, $\alpha > - \chi$, $\beta \in \R$, and let $h_{\alpha \beta} = h - \alpha \arg(\cdot) - \beta \log|\cdot|$, viewed as a distribution defined up to a global multiple of $2\pi(\chi+\alpha)$.  Let $n = 2\lceil |1+\alpha/\chi| \rceil$.  There exists $p > 0$ depending only on $\alpha$, $\beta$, and $\kappa$ such that the following is true.  Let $F_k$ be the event that there exists an angle varying flow line $\eta$ of $h_{\alpha \beta}$ which
\begin{enumerate}[(i)]
\item Starts from a point with rational coordinates,
\item Changes angles only upon hitting the straight lines with angles $\tfrac{2\pi(j+1)}{n+1}$ starting from $0$ for $j=0,\ldots,n-1$,
\item Travels with angles contained in $\tfrac{2\pi}{n} (1+\tfrac{\alpha}{\chi}) \Z$,
\item Stops upon exiting the annulus $A_k = 2^k \D \setminus (2^{k-1}\D )$, $k \in \Z$, and
\item Disconnects $0$ from $\infty$.
\end{enumerate}
Then $\p[F_k] \geq p$.
\end{lemma}
\begin{proof}
It is proved in the caption of Figure~\ref{fig::endpoint_continuity} that the event $F$ that there exists an angle varying flow line starting from $3/4$ with the properties described in the lemma statement for $k=0$ satisfies $\p[F] > 0$.  Consequently, the result follows by scale invariance.  We explain further a few points here.  First, the reason for our choice of $n$ is that it implies that the angle changes $\Delta$ as defined in the figure caption are contained in $[-\pi,\pi]$.  This is important because there is only a $2\pi$ range of angles at which we can draw flow lines from a given point in $\C \setminus \{0\}$.  Second, the reason that we have to construct an angle varying flow line is that, for some values of $\alpha$, it is not possible for a flow line to hit itself after wrapping around $0$.  Moreover, the reason that it is necessary to have multiple angle changes (as opposed to possibly just one), depending on the value of $\alpha$, is that with each angle change we can only change the boundary height adjacent to the path by at most $\pi \chi$ in absolute value while the height of a single path changes by $2\pi|\chi + \alpha|$ in absolute value upon winding once around $0$.
\end{proof}

\begin{proposition}
\label{prop::endpoint_continuity}
Suppose that $\eta$ is a whole-plane $\SLE_\kappa^\mu(\rho)$ process for $\kappa \in (0,4)$, $\rho > -2$, and $\mu \in \R$ starting at $0$.  Then $\lim_{t \to \infty} \eta(t) = \infty$ almost surely.  Moreover, if $\eta$ is a radial $\SLE_\kappa^\mu(\rho)$ process in $\D$ and targeted at $0$ for $\kappa \in (0,4)$, $\rho > -2$, and $\mu \in \R$, then $\lim_{t \to \infty} \eta(t) = 0$ almost surely.
\end{proposition}
\begin{proof}
From the relationship between whole-plane $\SLE_\kappa^\mu(\rho)$ and radial $\SLE_\kappa^\mu(\rho)$ described in Section~\ref{subsec::whole_plane_sle}, the two assertions of the proposition are equivalent.  Fix $\alpha > -\chi$, $\beta \in \R$, $h$ a whole-plane GFF, and let $h_{\alpha \beta} = h - \alpha \arg(\cdot) - \beta \log|\cdot|$, viewed as a distribution defined up to a global multiple of $2\pi(\chi+\alpha)$.  Let $\eta$ be the flow line of $h_{\alpha \beta}$ starting from $0$.  Then it suffices to show that $\lim_{t \to \infty} \eta(t) = \infty$ almost surely.  Let $A_k = 2^k \D \setminus (2^{k-1} \D)$ and let $F_k$ be the event as described in the statement of Lemma~\ref{lem::flow_line_ring} for the annulus $A_k$ and $\eta^k$ the corresponding angle-varying flow.  On $F_k$, it follows that $\liminf_{t \to \infty} |\eta(t)| \geq 2^k$ because Proposition~\ref{prop::cross_finite} implies that $\eta$ can cross each of the angle-varying flow lines $\eta^k$ hence $A_k$ a finite number of times and we know that $\limsup_{t \to \infty} |\eta(t)| = \infty$ almost surely (the capacity of the hull of $\eta((-\infty,t])$ is unbounded as $t \to \infty$).  Thus it suffices to show that, almost surely, infinitely many of the events $F_k$ occur.  This follows from the following two observations.  First, the Markov property implies that $h_{\alpha \beta}|_{\C \setminus (2^k \D)}$ is independent of $h_{\alpha \beta}|_{2^{k-1} \D}$ conditional on $h_{\alpha \beta}|_{A_k}$.  Moreover, the total variation distance between the law of $h_{\alpha \beta} |_{A_K}$ conditional on $h_{\alpha \beta}|_{A_k}$ and the law of $h_{\alpha \beta}|_{A_K}$ (unconditional) converges to zero almost surely as $K \to \infty$ (see Section~\ref{subsec::gff_convergence}).  Therefore the claim follows from Lemma~\ref{lem::flow_line_ring}.
\end{proof}

\begin{proof}[Proof of Theorem~\ref{thm::transience} for $\kappa \in (0,4)$]
This is a special case of Proposition~\ref{prop::endpoint_continuity}.
\end{proof}

\subsection{Critical angle and self-intersections}
\label{subsec::critical_angle}

The height gap $2\lambda'=2\lambda-\pi\chi$ divided by $\chi$ is called the {\bf critical angle} $\theta_c$.  Recalling the identities $\lambda' = \frac{\kappa}{4} \lambda$ and $2\pi \chi = (4 - \kappa)\lambda$ (see~\eqref{eqn::deflist}--\eqref{eqn::fullrevolutionrho}), we see that the critical angle can be written
\begin{equation}
\label{eqn::critical_angle}
\theta_c = \theta_c(\kappa) := \frac{2\lambda'}{\chi} = \frac{2 \frac{\kappa}{4} \lambda}{\chi} = \frac{\kappa}{2} \cdot \frac{\lambda}{\chi} = \frac{\tfrac{\kappa}{2} \cdot 2\pi}{4-\kappa} = \frac{\pi \kappa}{4-\kappa}.
\end{equation}
Note that $\theta_c(2) = \pi$, $\theta_c(8/3) = 2\pi$, $\theta_c(3) = 3\pi$, $\theta_c (16/5) = 4\pi$, and more generally
\begin{equation}
\label{eqn::critical_angle_integer_pi}
 \theta_c\left(\frac{4n}{n+1}\right) = n \pi.
\end{equation}
Recall from Figure~\ref{fig::boundary_flowline_interaction1} that in the boundary emanating setting, $\theta_c$ is the critical angle at which flow lines can intersect each other.  In the setting of interior flow lines, $\theta_c$ has the same interpretation as a consequence of Theorem~\ref{thm::flow_line_interaction} and Theorem~\ref{thm::conditional_law}.  Moreover, $\lfloor 2\pi/\theta_c \rfloor$ gives the maximum number of distinct ordinary GFF flow lines emanating from a single interior point that one can start which are non-intersecting.  In particular, $\kappa=8/3$ is the critical value for which an interior flow line does not intersect itself (recall also that for $\kappa=8/3$, we have $2-\kappa=\tfrac{\kappa}{2}-2$ as well as Lemma~\ref{lem::radial_critical_for_hitting}).  The value of $\kappa$ which solves $\theta_c(\kappa) = 2\pi/n$ is critical for being able to fit $n$ distinct non-intersecting interior GFF flow lines starting from a single point.  Explicitly, this value of $\kappa$ is given by
\begin{equation}
\label{eqn::critical_n}
 \kappa = \frac{8}{n+2}.
\end{equation}
If we draw $n$ distinct non-intersecting interior flow lines, starting at a common point, with angles evenly spaced on the circle $[0,2\pi)$, then these flow lines will intersect each other if and only if $\kappa > \frac{8}{2 + n}$.  The value of $\kappa$ which solves $\theta_c(\kappa) = 2\pi n$ is critical for a single flow line being able to visit the same point $n+1$ times (wrapping around the starting point of the path once between each visit --- so that the angle gap between the first and last visit is $2 \pi n$).  This allows us to determine whether a flow line almost surely has triple points, quadruple points, etc.  We record this formally in Proposition~\ref{prop::number_of_self_intersections} below.

If we consider flow lines of $h-\alpha \arg(\cdot) - \beta \log|\cdot|$, viewed as a distribution defined up to a global multiple of $2\pi(\chi+\alpha)$, where $h$ is a whole-plane GFF, $\alpha > -\chi$, and $\beta \in \R$ as in Theorem~\ref{thm::alphabeta}, the value of $\kappa
$ which is critical for a flow line starting at $0$ to be able to intersect itself is the non-negative solution to $\theta_c(\kappa) = 2\pi(1+\alpha/\chi)$.
Note that this in particular depends on $\alpha$ but not $\beta$ and that as $\alpha$ decreases to $-\chi$, the critical value of $\kappa$ decreases to $0$.  Similarly, the value of $\kappa$ which solves $\theta_c(\kappa) = 2\pi(1+\alpha/\chi)/n$ is critical for the intersection of flow lines whose angles differ by $2\pi/n$.


\begin{proposition}
\label{prop::number_of_self_intersections}
Suppose that $h_{\alpha \beta} = h - \alpha \arg(\cdot) - \beta \log|\cdot|$ where $h$ is a whole-plane GFF, $\alpha > -\chi$, and $\beta \in \R$, viewed as a distribution defined up to a global multiple of $2\pi(\chi+\alpha)$ and let $\eta$ be the flow line of $h_{\alpha \beta}$ starting from $0$.  Almost surely, the maximal number of times that $\eta$ hits any single point is (i.e., the maximal multiplicity)
\begin{equation}
\label{eqn::flow_line_maximal_self_intersect}
    \left\lceil \frac{8+4\alpha\sqrt{\kappa} - \kappa}{8+4\alpha\sqrt{\kappa}-2\kappa} \right\rceil -1.
\end{equation}
Equivalently, almost surely, the maximal number of times that a whole-plane or radial $\SLE_\kappa^\mu(\rho)$ process for $\kappa \in (0,4)$, $\rho > -2$, and $\mu \in \R$ hits any single point is given by
\begin{equation}
\label{eqn::whole_plane_sle_maximal_self_intersect}
 \left\lceil \frac{\kappa}{2(2+\rho)} \right\rceil.
\end{equation}
Finally, almost surely, the maximal number of times that a radial $\SLE_\kappa^\mu(\rho)$ process for $\kappa \in (0,4)$, $\rho > -2$, and $\mu \in \R$ can hit any point on the domain boundary (other than its starting point) is given by
\begin{equation}
\label{eqn::radial_sle_domain_boundary}
 \left\lceil \frac{\kappa}{2(2+\rho)} \right\rceil-1.
\end{equation}
\end{proposition}
\begin{proof}
In between each successive time that $\eta$ hits any given point, it must wrap around its starting point.  Consequently, when $\eta$ hits a given point for the $j$th time, it intersects itself with a height difference of $2\pi(\chi+\alpha)(j-1)$.  That is, the intersection can be represented as the intersection of two flow line tails which intersect each other with the aforementioned height difference.  Theorem~\ref{thm::flow_line_interaction} implies that $2\pi(\chi+\alpha)(j-1) \in (0,2\lambda-\pi\chi)$ from which~\eqref{eqn::flow_line_maximal_self_intersect} and~\eqref{eqn::whole_plane_sle_maximal_self_intersect} follow.  A similar argument gives~\eqref{eqn::radial_sle_domain_boundary}.
\end{proof}

\begin{figure}[ht!]
\begin{center}
\includegraphics[scale=0.85]{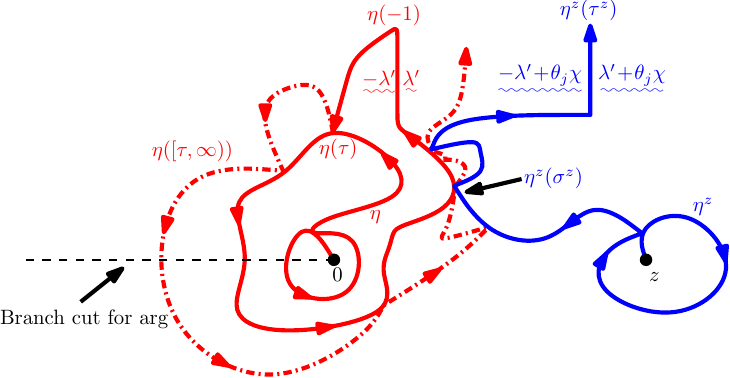}
\end{center}
\caption{\label{fig::intersection_dimension}
Setup for the proof of the lower bound of Proposition~\ref{prop::intersection_dimension}.  Suppose that $h_{\alpha \beta} = h - \alpha \arg(\cdot) - \beta\log|\cdot|$, $h$ a whole-plane GFF, viewed as a distribution defined up to a global multiple of $2\pi(\chi+\alpha)$.  Let $z=5$ and let $\eta,\eta^z$ be flow lines of $h_{\alpha\beta}$ starting from $0,z$, respectively.  We assume that $\eta$ (resp.\ $\eta^z$) has angle $0$ (resp.\ $\theta_j$ as in~\eqref{eqn::dimension_angle}) where the angle for $\eta^z$ is defined by first fixing an infinitesimal segment of $\eta$ (this allows us to determine the values of the remainder of the field up to a global multiple of $2\pi \chi$; recall Remark~\ref{rem::flow_line_alpha_non_zero}).  Let $\tau$ be the first time after capacity time $-1$ that $\eta$ makes a clean loop.  By Lemma~\ref{lem::radial_loop}, the event $E_1 = \{\tau < 0\}$ has positive probability.  Note that $z$ is contained in the unbounded connected component of $\C \setminus \eta((-\infty,\tau])$ on $E_1$ by \cite[Proposition~3.27]{LAW05}.  Let $\sigma^z$ be the first time that $\eta^z$ hits $\eta((-\infty,\tau])$.  By Lemma~\ref{lem::path_close_hit}, the event $E_2$ that $\eta^z$ hits $\eta((-\infty,\tau])$ at time $\sigma^z$  with a height difference of $\theta_j \chi$ occurs with positive conditional probability given $E_1$.  Let $\tau^z$ be a stopping time for $\eta^z$ given $\eta|_{(-\infty,\tau]}$ with $\tau^z > \sigma^z$ such that on $E_1 \cap E_2$, $\eta^z|_{[\sigma^z,\tau^z]}$ intersects $\eta|_{(-\infty,\tau]}$ only with a height difference of $\theta_j \chi$.  Given this, in each of the first $j-1$ times that $\eta|_{[\tau,\infty)}$ wraps around $0$ it has a positive chance of hitting $\eta^z((-\infty,\sigma^z])$ (hence itself) before capacity time $0$ by Lemma~\ref{lem::path_close_hit}.  After wrapping around $j-1$ times, $\eta|_{[\tau,\infty)}$ has a positive chance of making a clean a loop before time $0$ or intersecting itself in $\eta((-\infty,\tau])$.  On this event, the set of points that $\eta|_{(-\infty,0]}$ hits $j$ times contains $\eta((-\infty,\tau]) \cap \eta^z([\sigma^z,\tau^z])$.}
\end{figure}


The following proposition will be used in \cite{MW_INTERSECTION} to compute the almost sure Hausdorff dimension of the set of points that $\eta$ hits exactly $j$ times.  Let $\dim_\CH(A)$ denote the Hausdorff dimension of a set $A$.

\begin{proposition}
\label{prop::intersection_dimension}
Suppose that we have the same setup as Proposition~\ref{prop::number_of_self_intersections}.  For each $j \in \N$, let $\CI_j$ be the set of points that $\eta$ hits exactly $j$ times and
\begin{equation}
\label{eqn::dimension_angle}
 \theta_j = 2\pi (j-1) \left(1+\frac{\alpha}{\chi}\right) = 2\pi (j-1) \left(\frac{2+\rho}{4-\kappa} \right).
\end{equation}
Assume that, for each $\theta$, there exists a constant $d(\theta) \geq 0$ such that
\begin{equation}
\label{eqn::boundary_intersection_dimension}
 \p[ \dim_\CH(\eta_1 \cap \eta_2 \cap \h) = d(\theta) \ | \ \eta_1 \cap \eta_2 \cap \h \neq \emptyset] = 1
\end{equation}
where $\eta_1,\eta_2$ are flow lines of a GFF on $\h$ starting from $\partial \h$ with an angle difference of $\theta$ such that the event conditioned on in~\eqref{eqn::boundary_intersection_dimension} occurs with positive probability.  Then
\[ \p[ \dim_\CH(\CI_j) = d(\theta_j)] = 1.\]
\end{proposition}

We emphasize that the assumption~\eqref{eqn::boundary_intersection_dimension} in the statement of Proposition~\ref{prop::intersection_dimension} applies to any choice of boundary data for the GFF on $\h$ and starting points for $\eta_1,\eta_2$.  That is, the angle difference determines the almost sure dimension of the intersection points which are contained in (the interior of) $\h$ and nothing else.  (The dimension of $\eta_i \cap \partial \h$ for $i=1,2$, however, does depend on the boundary data for $h$.)  The value of $d(\theta_j)$ for $j \geq 2$ is computed in \cite{MW_INTERSECTION}.  The value of $d(\theta_1) = d(0)$ is the dimension of ordinary chordal $\SLE_\kappa$: $1+\tfrac{\kappa}{8}$ \cite{BEF_DIM}.  Note that when $\rho=\tfrac{\kappa}{2}-2$, $\theta_2$ is equal to the critical angle~\eqref{eqn::critical_angle} and $\theta_j$ exceeds the critical angle for larger values of $j$.

\begin{proof}[Proof of Proposition~\ref{prop::intersection_dimension}]
The set of points that $\eta$ hits exactly $j$ times can be covered by the set of intersections of pairs of flow line tails starting from points with rational coordinates which intersect with an angle gap of $\theta_j$.  By the proofs of Proposition~\ref{prop::tail_interaction} and Proposition~\ref{prop::tail_decomposition}, the law of the dimension of each of these intersections is absolutely continuous with respect to the intersection of two boundary emanating GFF flow lines with an angle gap of $\theta_j$.  This proves the upper bound.

We are now going to give the proof of the lower bound.  See Figure~\ref{fig::intersection_dimension} for an illustration.  Assume that $j$ is between $2$ and the (common) values in~\eqref{eqn::flow_line_maximal_self_intersect},~\eqref{eqn::whole_plane_sle_maximal_self_intersect} (for $j=1$ or values of $j$ which exceed~\eqref{eqn::flow_line_maximal_self_intersect},~\eqref{eqn::whole_plane_sle_maximal_self_intersect}, there is nothing to prove).  Assume that $\eta$ is parameterized by capacity.  For each $t \in \R$, let $\CI_j(t)$ be the set of points that $\eta$ hits exactly $j$ times by time $t$ and which are not contained in the boundary of the unbounded connected component of $\C \setminus \eta((-\infty,t])$.  By the conformal Markov property of whole-plane $\SLE_\kappa^\mu(\rho)$, it suffices to show that there exists $p_0 > 0$ such that
\begin{equation}
\label{eqn::self_intersect_hd_lbd}
\p[ \dim_\CH(\CI_j(0)) \geq d(\theta_j)] \geq p_0.
\end{equation}
To see that this is the case, we let $\tau$ be the first time after time $-1$ that $\eta$ makes a clean loop around $0$ and condition on $\eta|_{(-\infty,\tau]}$.  Lemma~\ref{lem::radial_loop} implies that $E_1 = \{\tau < 0\}$ occurs with positive probability.  Let $z=5$.  By \cite[Proposition~3.27]{LAW05}, $z$ is contained in the unbounded component of $\C \setminus \eta((-\infty,\tau])$ on $E_1$.  On $E_1$, let $\eta^z$ be the flow line of $h_{\alpha \beta}$ starting from $z$ with angle $\theta_j$.  (The reason we are able to set the angle for $\eta^z$ is that, after fixing an infinitesimal segment of $\eta$, we know the remainder of the field up to a multiple of $2\pi \chi$; see Remark~\ref{rem::flow_line_alpha_non_zero}.)  Let $\sigma^z$ be the first time $\eta^z$ hits $\eta((-\infty,\tau])$.  Lemma~\ref{lem::path_close_hit} implies that the event $E_2$ that $\eta^z$ hits $\eta((-\infty,\tau])$ at time $\sigma^z$ with an angle difference of $\theta_j$ occurs with positive conditional probability given $E_1$.  Let $\tau^z$ be a stopping time for $\eta^z$ given $\eta|_{(-\infty,\tau]}$ which is strictly larger than $\sigma^z$ such that, on $E_1 \cap E_2$, we have that $\eta|_{(-\infty,\tau]}$ and $\eta^z|_{[\sigma^z,\tau^z]}$ only intersect with an angle difference of $\theta_j$.  To finish the proof of~\eqref{eqn::self_intersect_hd_lbd}, it suffices to show that $\CI_j(0)$ contains $\eta((-\infty,\tau]) \cap \eta^z([\sigma^z,\tau^z])$ given $E_1 \cap E_2$ with positive conditional probability.  Iteratively applying Lemma~\ref{lem::path_close_hit} implies that, each of the first $j-1$ times that $\eta|_{[\tau,\infty)}$ wraps around $0$, it has a positive chance of hitting $\eta^z((-\infty,\sigma^z])$ (hence itself) and, moreover, this occurs before (capacity) time $0$ for $\eta$.  After wrapping around $j-1$ times, by Lemma~\ref{lem::radial_loop}, $\eta$ has a positive chance of making a clean a loop before intersecting itself again or time $0$.  Theorem~\ref{thm::flow_line_interaction} then implies that, on these events, the set of points that $\eta|_{[\tau,\infty)}$ hits in $\eta|_{(-\infty,\tau]}$ by the time it has wrapped around $j-1$ times contains $\eta((-\infty,\tau]) \cap \eta^z([\sigma^z,\tau^z])$.  This implies the desired result.
\end{proof}

\begin{proposition}
\label{prop::boundary_intersection_dimension}
Fix $\alpha > - \chi$, $\beta \in \R$, and let $h_{\alpha \beta} = h+\alpha \arg(\cdot) + \beta \log|\cdot|$ where $h$ is a GFF on $\D$ such that $h_{\alpha\beta}$ has the boundary data as illustrated in Figure~\ref{fig::radial_bd}.  Let $\eta$ be the flow line of $h_{\alpha \beta}$ starting from $W_0$ and assume that either $O_0 = W_0^+$ or $O_0 = W_0^-$.  For each $j \in \N$, let $\CJ_j$ be the set of points in $\partial \D$ that $\eta$ hits exactly $j$ times.  Assume that, for each $\theta$, there exists a constant $b(\theta) \geq 0$ such that
\begin{equation}
\label{eqn::boundary_intersection_dimension2}
 \p[ \dim_\CH(\eta_0 \cap \partial \h) = b(\theta)] = 1
\end{equation}
where $\eta_0$ is the flow line of a GFF on $\h$ starting from $0$ whose boundary data is such that $\eta_0$ almost surely intersects $\partial \h$ with an angle difference of $\theta$.
Then
\[ \p[ \dim_\CH(\CJ_j) = b(\theta_{j+1})\giv\CJ_j \neq \emptyset] = 1\]
provided $\p[\CJ_j \neq \emptyset] > 0$.
\end{proposition}
Note that Proposition~\ref{prop::boundary_intersection_dimension} gives the dimension of the set of points that radial $\SLE_\kappa(\rho)$, $\kappa \in (0,4)$ and $\rho \in (-2,\tfrac{\kappa}{2}-2)$, hits the boundary $j$ times.
\begin{proof}[Proof of Proposition~\ref{prop::boundary_intersection_dimension}]
This is proved in a very similar manner to Proposition~\ref{prop::intersection_dimension}.
\end{proof}

\section{Light cone duality and space-filling SLE}
\label{sec::duality_space_filling}

\subsection{Defining branching $\SLE_\kappa(\rho)$ processes}
\label{subsec::branching_definitions}

\begin{figure}[ht!]
\begin{center}
\includegraphics[scale=0.85]{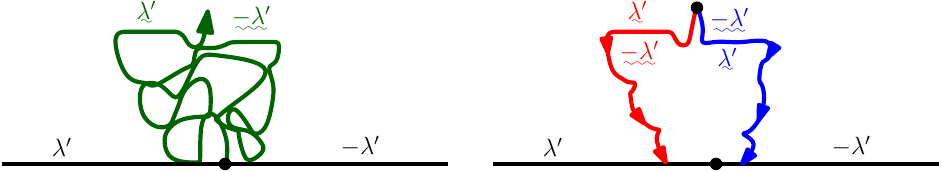}
\end{center}
\caption{\label{fig::counterflow}
Suppose that $h$ is a GFF on $\h$ with the boundary data shown.  Then the counterflow line $\eta'$ of $h$ starting from $0$ is a chordal $\SLE_{\kappa'}$ process from $0$ to $\infty$.  Let $\tau'$ be a stopping time for $\eta'$.  In Theorem~\ref{thm::lightcone}, we show that the outer boundary of $\eta'([0,\tau'])$ is, in a certain sense, given by the union of the flow lines of angle $\tfrac{\pi}{2}$ and $-\tfrac{\pi}{2}$ starting from $\eta'(\tau')$ and stopped upon hitting $\partial \h$.  The same result holds when the boundary data of $h$ is piecewise constant and changing only a finite number of times in which case $\eta'$ is an $\SLE_{\kappa'}(\ul{\rho})$ process.}
\end{figure}

As usual, we fix $\kappa \in (0,4)$ and $\kappa' =16/\kappa > 4$.  Consider a GFF on $\h$ with piecewise constant boundary conditions (and only finitely many pieces).  Recall that if the boundary conditions are constant and equal to $\lambda'$ on the negative real axis and constant and equal to $-\lambda'$ on the positive real axis, by \cite[Theorem~1.1]{MS_IMAG} one can draw a counterflow line from $0$ to $\infty$ whose law is that of ordinary $\SLE_{\kappa'}$, as in Figure~\ref{fig::counterflow}.  For each $z \in \ol{\h}$, let $\eta_z^L$ (resp.\ $\eta_z^R$) be the flow line of $h$ starting from $z$ with angle $\tfrac{\pi}{2}$ (resp.\ $-\tfrac{\pi}{2}$).  In this section, we will argue that the left and right boundaries of an initial segment of the counterflow line are in some sense described by $\eta_z^L$ and $\eta_z^R$ where $z$ is the tip of that segment, as illustrated in Figure~\ref{fig::counterflow}.  When $\kappa'\geq 8$ and the counterflow line is space-filling, we can describe $\eta_z^L$ and $\eta_z^R$ beginning at any point $z$ as the boundaries of a counterflow line stopped at the first time it hits $z$.  We will show in this section that when $\kappa' \in (4,8)$ it is still possible to construct a ``branching'' variant of $\SLE_{\kappa'}$ that has a branch that terminates at $z$ (and the law of this branch is the same as that of a certain radial $\SLE_{\kappa'}(\ul{\rho})$ process targeted at $z$).  The flow lines $\eta_z^L$ and $\eta_z^R$ will then be the left and right boundaries of this branch, just as in the case $\kappa' \geq 8$.  We will make sense of this construction for both boundary and interior points~$z$.  (When $z$ is an interior point, we will have to lift the $\SLE_{\kappa'}$ branch to the universal cover of $\C \setminus \{z \}$ in order to define its left and right boundaries.)

\begin{figure}[ht!]
\begin{center}
\includegraphics[scale=0.85]{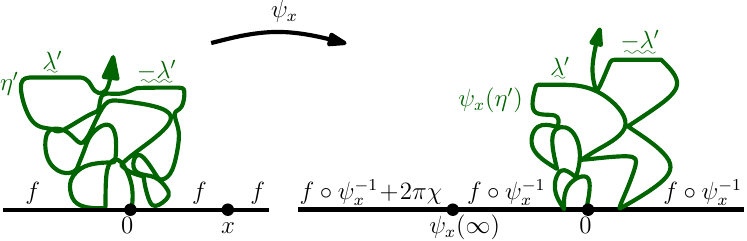}
\end{center}
\caption{\label{fig::branchable}
Suppose that $h$ is a GFF on $\h$ with the boundary data shown where $f$ is a piecewise constant function which changes values finitely many times.  Then the counterflow line $\eta'$ of $h$ starting from $0$ is a chordal $\SLE_{\kappa'}(\ul{\rho})$ process from $0$ to $\infty$.  If $f|_{\R_{-}} > -\lambda'$ and $f|_{\R_+} < \lambda'$, then $\eta'$ will reach $\infty$ almost surely \cite[Section~7]{MS_IMAG}.  Fix $x \in (0,\infty)$ and let $\psi_x \colon \h \to \h$ be a conformal map which takes $x$ to $\infty$ and fixes $0$.  Then $h \circ \psi_x^{-1} - \chi \arg (\psi_x^{-1})'$ has the boundary data shown on the right side.  In particular, $\psi_x(\eta')$ will reach $\infty$ almost surely if the boundary values on the right side are strictly larger than $-\lambda'$ on $\R_-$ and are strictly less than $\lambda'$ on $\R_+$.  These are the conditions which determine whether it is possible to draw a branch of $\eta'$ which is targeted at and almost surely reaches $x$.}
\end{figure}

We begin by constructing (for a certain range of boundary conditions) a counterflow line of $h$ that ``branches'' at boundary points.  Recall that we can draw a counterflow line from $0$ to $\infty$ provided that the boundary data is piecewise constant (with only finitely many pieces) and strictly greater than $-\lambda'$ on the negative real axis and strictly less than $\lambda'$ on the positive real axis \cite[Section~7]{MS_IMAG}.  (This is a generalization of the construction in Figure~\ref{fig::counterflow}.) Given this type of boundary data, the continuation threshold almost surely will not be reached before the path reaches $\infty$.  Indeed, it was shown in \cite[Section~7]{MS_IMAG} that if these constraints are satisfied then the counterflow line is almost surely continuous and almost surely tends to $\infty$ as (capacity) time tends to~$\infty$.

Now suppose we consider a real point $x >0$ and try to draw a path targeted at $x$.  By conformally mapping $x$ to $\infty$, we find that we avoid reaching the continuation threshold at a point on $[x, \infty)$ if each boundary value on that interval {\em plus} $2 \pi \chi = (4-\kappa)\lambda = (\kappa'-4)\lambda'$ is strictly greater than $-\lambda'$ (recall~\eqref{eqn::ac_eq_rel} and see Figure~\ref{fig::branchable}).  That is, the values on $[x, \infty)$ are strictly greater than $(3-\kappa')\lambda'$.  We can draw a counterflow line from $0$ to {\em every} point $x > 0$ if the boundary values on the positive real axis are in the interval $\bigl( (3-\kappa')\lambda', \lambda'\bigr)$.  Similarly, we can draw a counterflow line from $0$ to {\em every} point $x < 0$ if the boundary values on the negative real axis are in the interval $\bigl( -\lambda', (\kappa'-3)\lambda'\bigr)$.  We can draw the counterflow line to each fixed point in $\R \setminus \{0 \}$ if the boundary function is given by $f \colon \R \to \R$ which is piecewise constant, takes on only a finite number of values, and satisfies
\begin{equation}
\label{eqn::fullybranchable}
f(x) \in \begin{cases} \bigl( -\lambda', (\kappa'-3)\lambda'\bigr) &\quad\text{for}\quad\quad x < 0 \\ \bigl( (3-\kappa')\lambda', \lambda'\bigr) &\quad\text{for}\quad\quad x > 0 \end{cases}.
\end{equation}
(In the case of an $\SLE_{\kappa'}(\rho_1;\rho_2)$ process, this corresponds to each $\rho_i$ being in the interval $(-2, \kappa'-4)$.)   When these conditions hold, we say that the boundary data is {\bf fully branchable}, and we extend this definition to general domains via the coordinate change~\eqref{eqn::ac_eq_rel} (see also Figure~\ref{fig::coordinatechange}).  The reason we use this term is as follows.

If we consider distinct $x, y \in \R \setminus \{0\}$, then the path targeted at $x$ will agree with the path targeted at $y$ (up to time parameterization) until $\tau$, the first time~$t$ that either one of the two points is hit or the two points are ``separated,'' i.e., lie on the boundaries of different components of $\h \setminus \eta'([0,t])$.  Indeed, this follows because both paths (up until separating $x$ and $y$) are coupled with the field as described in \cite[Theorem~1.1]{MS_IMAG} and \cite[Theorem~1.2]{MS_IMAG} implies that there is a unique such path coupled with the field, so they must agree.  The two paths evolve independently after time $\tau$, so as explained in Figure~\ref{fig::boundarybranching} we can interpret the pair of paths as a single path that ``branches'' at time $\tau$ (with one of the branches being degenerate if a point is hit {\em at} time $\tau$).  We can apply the same interpretation when $x$ and $y$ are replaced by a set $\{x_1, x_2, \ldots, x_n \}$.  In that case, a branching occurs whenever two points are separated from each other for the first time.  At such a time, the number of distinct components containing at least one of the $x_i$ increases by one, so there will be $(n-1)$ branching times altogether.  If the boundary conditions are fully branchable, then we may fix a countable dense collection of $\R$, and consider the collection of all counterflow lines targeted at all of these points.  We refer to the entire collection as a {\bf boundary-branching counterflow line}.  Note that if we have constant boundary conditions on the left and right boundaries, so that the counterflow line targeted at $\infty$ is an $\SLE_{\kappa'}(\rho_1; \rho_2)$ process, then we can extend the path toward every $x \in \R$ provided that $\rho_i \in (-2, \kappa'-4)$ for $i \in \{1,2 \}$.

Similarly, we may consider a point $z$ in the interior of $\h$, and the (fully branchable) boundary conditions ensure that we may draw a radial counterflow line from $0$ targeted at $z$ which almost surely reaches $z$ before hitting the continuation threshold.  The branching construction can be extended to this setting, as explained in Figure~\ref{fig::interiorbranching}.  By considering a collection of counterflow lines targeted at a countable dense collection of such $z$, we obtain what we call an {\bf interior-branching counterflow line}.  This is a sort of space-filling tree whose branches are counterflow lines.  When the values of $h$ on $\partial \h$ are given by either
\begin{enumerate}[(i)]
\item\label{item::cle_left}  $-\lambda'+2\pi \chi$ on $\R_-$ and $-\lambda'$ on $\R_+$ or
\item\label{item::cle_right} $\lambda'$ on $\R_-$ and $\lambda'-2\pi \chi$ on $\R_+$,
\end{enumerate}
\noindent then this corresponds to the $\SLE_{\kappa'}(\kappa'-6)$ {\em exploration tree} rooted at the origin introduced in \cite{SHE_CLE}.  The boundary conditions in~\eqref{item::cle_left} correspond to having the $\kappa'-6$ force point lie to the left of $0$ and the boundary conditions in~\eqref{item::cle_right} correspond to having the $\kappa'-6$ force point lie to the right of $0$.

\begin{figure}[ht!]
\begin{center}
\includegraphics[scale=0.85]{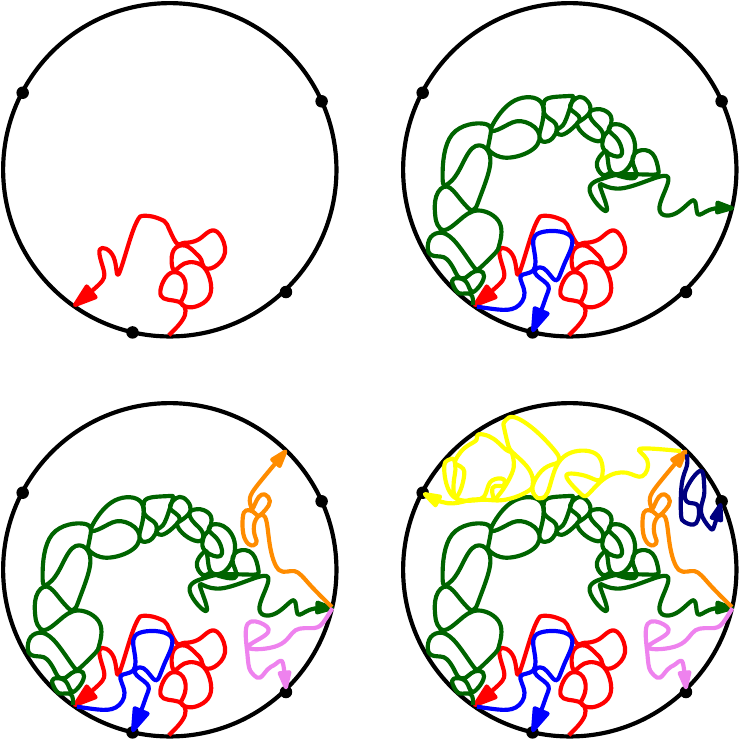}
\end{center}
\caption{\label{fig::boundarybranching}
Suppose that $h$ is a GFF on a Jordan domain $D$ and $x_1,\ldots,x_4 \in \partial D$.  If for each $i=1,\ldots,4$, $\eta_i'$ is the counterflow line of $h$ starting from a fixed boundary point and targeted at $x_i$, then any two $\eta_i'$ will almost surely agree up until the first time that their target points are separated (i.e., cease to lie in the same component of the complement of the path traced thus far).  We may therefore understand the union of the $\eta_i'$ as a single counterflow line that ``branches'' whenever any pair of points is disconnected, continuing in two distinct directions after that time, as shown.  (Whenever a curve branches a new color is assigned to each of the two branches.)
}
\end{figure}

\begin{figure}[ht!]
\begin{center}
\includegraphics[scale=0.85]{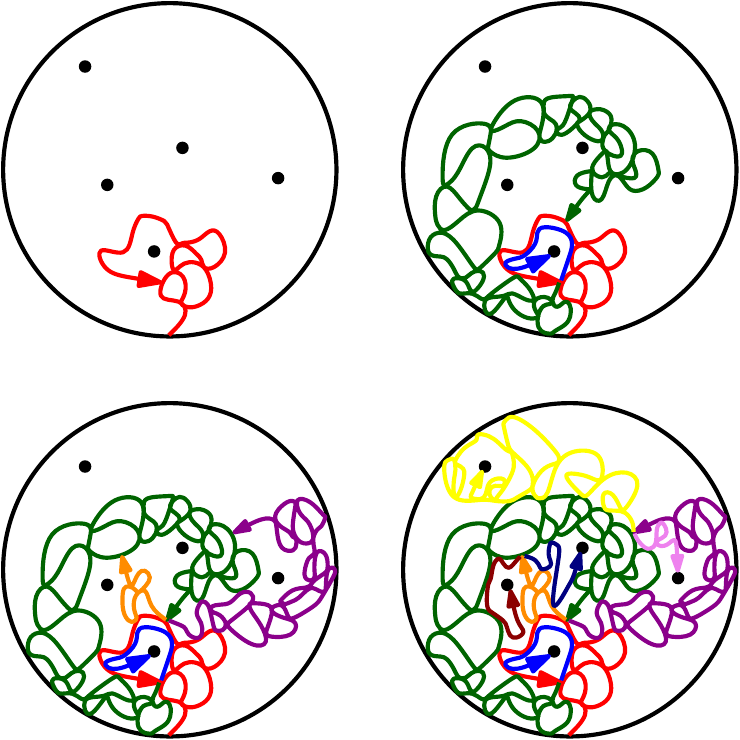}
\end{center}
\caption{\label{fig::interiorbranching}
Suppose that $h$ is a GFF on a Jordan domain $D$.  Fix interior points points $z_1, \ldots, z_5$ in $D$.  For each $i=1,\ldots,5$, let $\eta_i'$ be the counterflow line from a fixed boundary point to $z_i$.  Then any two $\eta_i'$ will almost surely agree up until the first time that their target points are separated (i.e., cease to lie in the same component of the complement of the path traced thus far).  We may therefore understand the union of the $\eta_i'$ as a single counterflow line that ``branches'' whenever any pair of points is disconnected, continuing in two distinct directions after that time, as shown.  (Whenever a curve branches a new color is assigned to each of the two branches.)
}
\end{figure}

\subsection{Duality and light cones}
\label{subsec::duality}

Figure~\ref{fig::lightcone_boundary_filling} is lifted from \cite[Section~7]{MS_IMAG}, which contains a general theorem about the boundaries of chordal $\SLE_{\kappa'}(\ul{\rho})$ processes.  Suppose that $D \subseteq \C$ is a Jordan domain.  The theorem states that if the boundary conditions on $\partial D$ are such that a counterflow line $\eta'$ can almost surely be drawn from $y \in \partial D$ to a point $x \in \partial D$ without hitting the continuation threshold, then the left and right boundaries of the counterflow line are respectively the flow lines of angle $\tfrac{\pi}{2}$ and $-\tfrac{\pi}{2}$ drawn from $x$ to $y$ (with the caveat that these flow lines may trace part of the boundary of the domain toward $y$ if they reach the continuation threshold before reaching $y$, as Figure~\ref{fig::lightcone_boundary_filling} illustrates --- this corresponds to the scenario in which the counterflow line fills an entire boundary arc, as explained in \cite[Section~7]{MS_IMAG}).  We will not repeat the full discussion of \cite[Section~7]{MS_IMAG} here, but we will explain how it can be extended to the setting where the boundary target  is replaced with an interior point.  (In this setting $\eta'$ can be understood as a branch of the interior-branching counterflow line, as discussed in the previous section.)  This is the content of Theorem~\ref{thm::lightcone}.

\begin{theorem}
\label{thm::lightcone}
Suppose that $D \subseteq \C$ is a Jordan domain, $x,y \in \partial D$ are distinct, and that $h$ is a GFF on $D$ whose boundary conditions are such that its counterflow line from $y$ to $x$ is fully branchable.  Fix $z \in D$ and let $\eta'$ be the counterflow line of $h$ starting from $y$ and targeted at $z$.  If we lift $\eta'$ to a path in the universal cover of $\C \setminus \{z\}$, then its left and right boundaries $\eta_z^L$ and $\eta_z^R$ are the flow lines of $h$ started at $z$ and targeted at $y$ with angles $\tfrac{\pi}{2}$ and $-\tfrac{\pi}{2}$, respectively (with the same caveat as described above, and in \cite[Section~7]{MS_IMAG}, in the chordal case: that these paths may trace boundary segments toward $y$ if they hit the continuation threshold before reaching $y$).  The conditional law of $\eta'$ given $\eta_z^L$ and $\eta_z^R$ in each of the connected components of $D \setminus (\eta_z^L \cup \eta_z^R)$ which are between $\eta_z^L, \eta_z^R$, consist of boundary segments of $\eta_z^L,\eta_z^R$ that do not trace $\partial D$, and are hit by $\eta'$ is independently that of a chordal $\SLE_{\kappa'}(\tfrac{\kappa'}{2}-4;\tfrac{\kappa'}{2}-4)$ process starting from the last point on the component boundary traced by $\eta_z^L$ and targeted at the first.
\end{theorem}

\begin{remark}
\label{rem::boundary_tracing_conditional_law}
Suppose that $C$ is a connected component of $D \setminus (\eta_z^L \cup \eta_z^R)$ which lies between $\eta_z^L$ and $\eta_z^R$ part of whose boundary \emph{is} drawn by a segment of either $\eta_z^L$ or $\eta_z^R$ which \emph{does} trace part of $\partial D$.  Then it is also possible to compute the conditional law of $\eta'$ given $\eta_z^L$ and $\eta_z^R$ inside of $C$.  It is that of an $\SLE_{\kappa'}(\ul{\rho})$ process where the weights $\ul{\rho}$ depend on the boundary data of $h$ on the segments of $\partial C$ which trace $\partial D$.
\end{remark}

\begin{remark}
\label{rem::interior_paths_duality}
Since $\eta'$ is almost surely determined by the GFF \cite[Theorem~1.2]{MS_IMAG} (see also \cite{DUB_PART}), this gives us a different way to construct $\eta_z^L$ and $\eta_z^R$.  In fact, taking a branching counterflow line gives us a way to construct simultaneously all of the flow lines beginning at points in a fixed countable dense subset of $D$.  It follows from \cite[Theorem~1.2]{MS_IMAG} that the branching counterflow line is almost surely determined by the GFF, so this also gives an alternative approach to proving Theorem~\ref{thm::uniqueness}, that GFF flow lines starting from interior points are almost surely determined by the field.
\end{remark}

\begin{remark}
\label{rem::interior_light_cone}
Both $\eta_z^L$ and $\eta_z^R$ can be projected to $D$ itself (from the universal cover) and interpreted as random subsets of $D$.  Once we condition on $\eta_z^L$ and $\eta_z^R$, the conditional law of $\eta'$ within each component of $D \setminus (\eta_z^L \cup \eta_z^R)$ (that does not contain an interval of $\partial D$ on its boundary) is given by an independent (boundary-filling) chordal $\SLE_{\kappa'}(\frac{\kappa'}{2}-4; \frac{\kappa'}{2}-4)$ curve from the first to the last endpoint of $\eta'$ within that component.   This is explained in more detail in the beginning of \cite{MS_IMAG3} and the end of \cite{MS_IMAG}, albeit in a slightly different context.  It follows from this and the arguments in \cite{MS_IMAG} that we can interpret the set of points in the range of $\eta'$ as a {\bf light cone} beginning at $z$.  Roughly speaking, this means that the range of $\eta'$ is equal to the set of points reachable by starting at $z$ and following angle-varying flow lines whose angles all belong to the width-$\pi$ interval of angles that lie between the angles of $\eta_z^L$ and $\eta_z^R$.  The analogous statement that applies when $z$ is contained in the boundary is explained in detail in \cite{MS_IMAG}.  We refrain from a more detailed discussion here, because the present context is only slightly different.
\end{remark}

\begin{remark}
\label{rem::radial}
In the case that $D = \h$ and the boundary conditions for $h$ are given by $\lambda'$ (resp.\ $-\lambda'$) on $\R_-$ (resp.\ $\R_+$), as shown in Figure~\ref{fig::counterflow}, the branch of an interior branching counterflow line targeted at a given interior point $z$ is distributed as a radial $\SLE_{\kappa'}(\kappa'-6)$ process.  Recall Figure~\ref{fig::radial_bd_cfl}.
\end{remark}

\begin{figure}[ht!]
\begin{center}
\subfigure[]{\includegraphics[height=0.31\textheight]{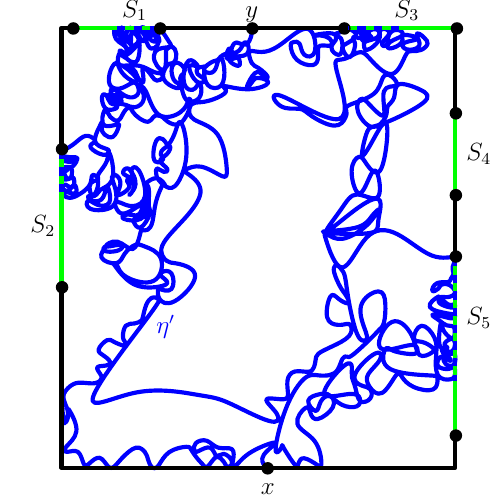}}
\subfigure[]{\includegraphics[height=0.31\textheight]{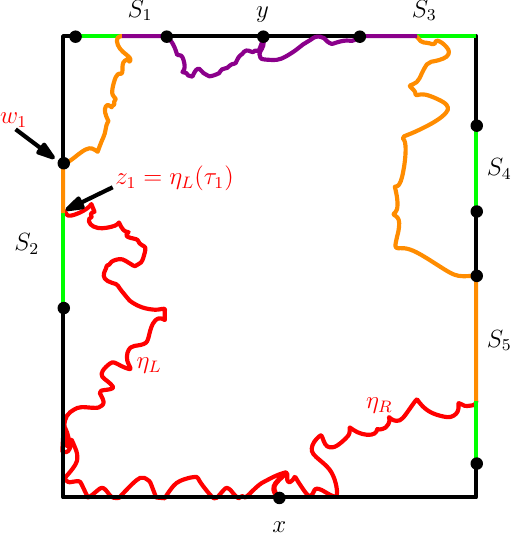}}
\end{center}
\caption{\label{fig::lightcone_boundary_filling} {\small Suppose that $h$ is a GFF on a Jordan domain $D$ and $x,y \in \partial D$ are distinct.  Let $\eta'$ be the counterflow line of $h$ starting at $y$ aimed at $x$.  Let $K = K^L \cup K^R$ be the outer boundary of $\eta'$, $K^L$ and $K^R$ its left and right sides, respectively, and let $I$ be the interior of $K^L  \cap \partial D$.  We suppose that the event $E = \{I \neq \emptyset\}$ that $\eta'$ fills a segment of the left side of $\partial D$ has positive probability, though we emphasize that this does not mean that $\eta'$ {\em traces} a segment of $\partial D$---which would yield a discontinuous Loewner driving function---with positive probability.  In the illustrations above, $\eta'$ fills parts of $S_1,\ldots,S_5$ with positive probability (but with positive probability does not hit any of $S_1,\ldots,S_5$).  The connected component of $K^L \setminus I$ which contains $x$ is given by the flow line $\eta^L$ of $h$ with angle $\tfrac{\pi}{2}$ starting at $x$ (left panel).  On $E$, $\eta^L$ hits the continuation threshold before hitting $y$ (in the illustration above, this happens when $\eta^L$ hits $S_2$).  On $E$ it is possible to describe $K^L$ completely in terms of flow lines using the following algorithm.  First, we flow along $\eta^L$ starting at $x$ until the continuation threshold is reached, say at time $\tau_1$, and let $z_1 = \eta^L(\tau_1)$.  Second, we trace along $\partial D$ in the clockwise direction until the first point $w_1$ where it is possible to flow starting at $w_1$ with angle $\tfrac{\pi}{2}$ without immediately hitting the continuation threshold.  Third, we flow from $w_1$ until the continuation threshold is hit again.  We then repeat this until $y$ is eventually hit.  This is depicted in the right panel above, where three iterations of this algorithm are needed to reach $y$ and are indicated by the colors red, orange, and purple, respectively.}
}
\end{figure}

\begin{figure}[ht!]
\begin{center}
\includegraphics[scale=0.85]{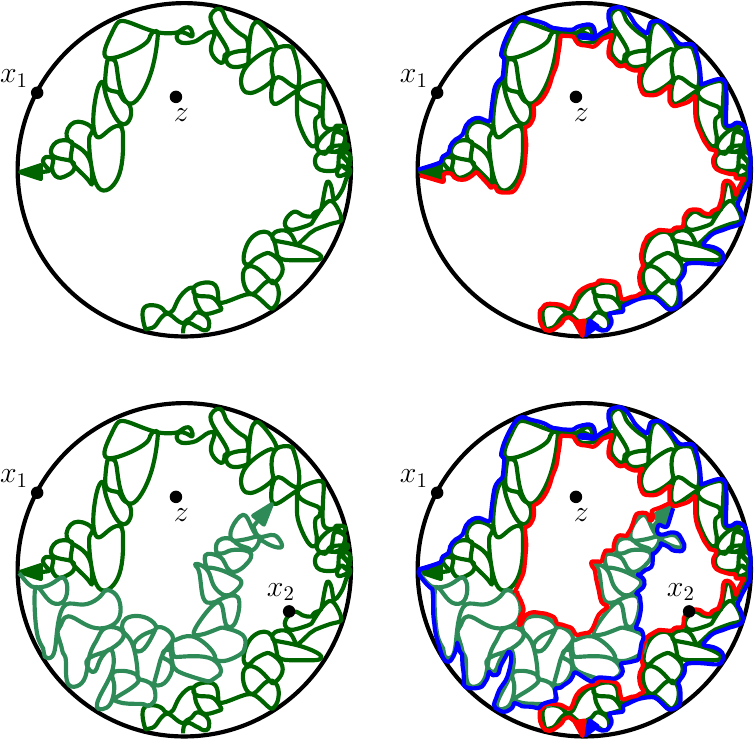}
\end{center}
\caption{\label{fig::interior_duality} {\small An interior point $z$ and boundary point $x_1$ are fixed and a counterflow line $\eta'$ drawn until the first time $T_1$ that $x_1$ is separated from $z$ (upper left).  By applying \cite[Section~7.4.3]{MS_IMAG} to a countable set of possible choices (including points arbitrarily close to $\eta'(T_1)$) we obtain that the left and right boundaries of $\eta'([0,T_1])$ are given by flow lines with angles $\tfrac{\pi}{2}$ and $-\tfrac{\pi}{2}$ (shown upper right).  We can then pick another point $x_2$ on the boundary of $D \setminus \eta'([0,T_1])$ and grow $\eta'$ until the first time $T_2$ that $x_2$ is disconnected from $z$, and the same argument implies that the left and right boundaries of $\eta'([0,T_2])$ (lifted to the universal cover of $D \setminus \{z \}$) are (when projected back to $D \setminus \{z\}$ itself) given by the flow lines of angles $\tfrac{\pi}{2}$ and $-\tfrac{\pi}{2}$ shown.  Note that it is possible for the blue and red paths to hit each other (as in the figure) but (depending on $\kappa$) it may also be possible for these flow lines to hit themselves (although their liftings to the universal cover of $D \setminus \{z \}$ are necessarily simple paths).}
}
\end{figure}

\begin{figure}[ht!]
\begin{center}
\includegraphics[scale=0.85]{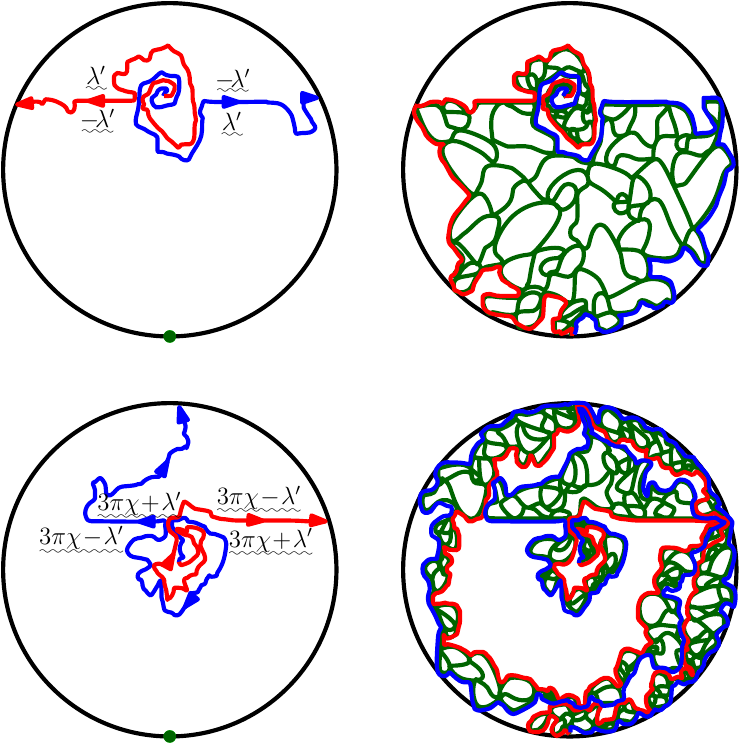}
\end{center}
\caption{\label{fig::first_hit_winding} {\small Consider the counterflow line $\eta'$ toward $z$ from Figure~\ref{fig::interior_duality}, whose left and right boundaries $\eta^L$ and $\eta^R$ are understood as flow lines with angles $\tfrac{\pi}{2}$ and $-\tfrac{\pi}{2}$ from $z$ to $\eta'(0)$.  The left figures above show two possible instances of initial segments of $\eta^L$ and $\eta^R$: segments started at $z$ and stopped at the first time they reach $\partial D$.  The GFF heights along these segments determine the number of times (and the direction) that $\eta^L$ and $\eta^R$ will wind around their initial segments before terminating at $\eta'(0)$.  In the lower figure, the counterflow line spirals counterclockwise and inward toward $z$, while the flow lines of angles $\tfrac{\pi}{2}$ and $-\tfrac{\pi}{2}$ spiral counterclockwise and outward.  Although their liftings to the universal cover of $D \setminus \{z \}$ are simple, the paths $\eta^L$ and $\eta^R$ themselves are self-intersecting at some points, including points on $\partial D$ where multiple strands coincide.}
}
\end{figure}

\begin{figure}[ht!]
\begin{center}
\subfigure[]{\includegraphics[scale=0.85]{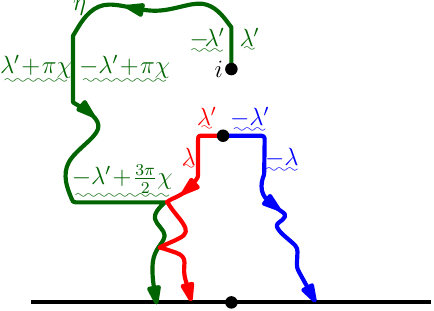}}
\subfigure[]{\includegraphics[scale=0.85]{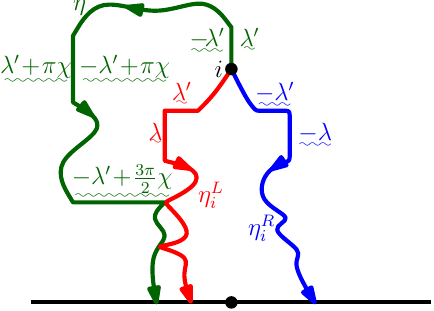}}
\end{center}
\caption{\label{fig::duality_no_loops}
Suppose that $h$ is a GFF on $\h$ with boundary data such that its counterflow line $\eta'$ starting from $0$ is fully branchable.  Assume that $\kappa' \geq 8$ hence $\kappa \in (0,2]$ so that flow lines of $h$ with an angle gap of $\pi$ starting from the same point do not intersect each other (apply~\eqref{eqn::critical_n} with $n=2$).  Let $\tau'$ be a stopping time for $\eta'$ such that $\eta'|_{[0,\tau']}$ almost surely does not swallow $i$.  Then we know that the left and right boundaries of $\eta'([0,\tau'])$ are contained in the union of flow lines of $h$ with angles $\tfrac{\pi}{2}$ and $-\tfrac{\pi}{2}$, respectively, starting from points with rational coordinates (first part of the proof of Theorem~\ref{thm::lightcone}).  This is depicted in the left panel above and allows us to use Theorem~\ref{thm::flow_line_interaction} to determine how the left and right boundaries of $\eta'([0,\tau'])$ interact with the flow lines of $h$.   In particular, the flow line $\eta$ of $h$ starting from $i$ almost surely cannot cross either the left or the right boundary of $\eta'([0,\tau'])$.  Indeed, as illustrated, if $\eta$ intersects the left boundary of $\eta'([0,\tau'])$ then it does so with a height difference contained in $\tfrac{\pi}{2} \chi + 2\pi \chi \Z$ and this set does not intersect the range $(-\pi \chi,0)$ for crossing.  The same is true if $\eta$ hits the right boundary of $\eta'([0,\tau'])$.  Since this holds almost surely for all stopping times $\tau'$ for $\eta'$, it follows that $\eta'$ is almost surely equal to the counterflow line of the conditional GFF $h$ given $\eta$.  Indeed, both processes satisfy the same conformal Markov property when coupled with $h$ given $\eta$, hence the claim follows from \cite[Theorem~1.2]{MS_IMAG}.  Consequently, it follows from chordal duality \cite[Section~7.4.3]{MS_IMAG} that the left and right boundaries of $\eta'$ stopped upon hitting $i$ are given by the flow lines $\eta_i^L$ and $\eta_i^R$ of $h$ starting from $i$ with angles $\tfrac{\pi}{2}$~and~$-\tfrac{\pi}{2}$.  This is depicted in the right panel.
}
\end{figure}

\begin{proof}[Proof of Theorem~\ref{thm::lightcone}]
The result essentially follows from the chordal duality in \cite[Section~7.4.3]{MS_IMAG} (which the reader may wish to consult before reading the argument here) together with the merging arguments of Section~\ref{subsec::uniqueness} (such as the proof of Proposition~\ref{prop::merge_almost_surely}).  The argument is sketched in Figure~\ref{fig::interior_duality}, which extends Figure~\ref{fig::lightcone_boundary_filling} to the case of an interior point.  It is somewhat cumbersome to repeat all of the arguments of \cite[Section~7.4.3]{MS_IMAG} here, so we instead give a brief sketch of the necessary modifications.  We inductively define points $x_k$ as follows.  We let $\psi_1 \colon D \to \D$ be the unique conformal map with $\psi_1(z) = 0$ and $\psi_1(x) = 1$ and then take $x_1 = \psi_1^{-1}(-1)$.  We then let $\tau_1$ be the first time $t$ for which $x_1$ fails to lie on the boundary of the connected component of $D \setminus \eta'([0,t])$ containing $z$.  Given that the stopping time $\tau_k$ has been defined for $k \in \N$, let $\psi_{k+1}$ be the unique conformal map which takes the connected component $U_k$ of $D \setminus \eta'([0,\tau_k])$ which contains $z$ to $\D$ with $\psi_{k+1}(z) = 0$ and $\psi_{k+1}(\eta'(\tau_k)) = 1$.  We then take $x_{k+1} = \psi_{k+1}^{-1}(-1)$ and let $\tau_{k+1}$ be the first time $t \geq \tau_k$ that $x_{k+1}$ fails to lie on the boundary of the component of $D \setminus \eta'([0,t])$ containing $z$.  It is easy to see that there exists constants $p_0 > 0$ and $a \in (0,1)$ such that the probability that the conformal radius of $U_{k+1}$ (as viewed from $z$) is at most $a$ times the conformal radius of $U_k$ (as viewed from $z$) is at least $p_0$.  By basic distortion estimates (e.g., \cite[Theorem~3.21]{LAW05}), one can make the analogous statement for the ordinary radius (e.g., $\dist(z,\partial U_k)$).  This implies that the concatenated sequence of counterflow line path segments illustrated in Figure~\ref{fig::interior_duality} contains $z$ as a limit point.  

For each $k$, it follows from \cite[Theorem~1.4]{MS_IMAG} that the left and right boundaries of $\eta'([\tau_{k-1},\tau_k])$ (where we take $\tau_0 = 0$) are contained in the flow lines $\eta_k^L$ and $\eta_k^R$ with angles $\tfrac{\pi}{2}$ and $-\tfrac{\pi}{2}$, respectively, of the \emph{conditional} GFF $h$ given $\eta'|_{[\tau_{j-1},\tau_j]}$ for $j=1,\ldots,k-1$ starting from $x_k$.  That is, we a priori have to observe $\eta'|_{[\tau_{j-1},\tau_j]}$ for $j=1,\ldots,k-1$ in order to draw the aforementioned flow lines.  We are going to show by induction on $k$ that the left and right boundaries of $\eta'([0,\tau_k])$ are contained in the union of all of the flow lines of $h$ itself with angles $\tfrac{\pi}{2}$ and $-\tfrac{\pi}{2}$, respectively, starting at points with rational coordinates.  (In other words, it is not necessary to observe $\eta'([\tau_{j-1},\tau_j])$ for $j=1,\ldots,k-1$ to draw these paths.)  
This, in turn, will allow us to use Theorem~\ref{thm::flow_line_interaction} to determine how the left and right boundaries of $\eta'([0,\tau_k])$ interact with the flow lines of $h$.  The claim for $k=1$ follows from \cite[Theorem~1.4]{MS_IMAG}.

Suppose that the claim holds for some fixed $k \in \N$.  Fix $\theta \in \R$, $w \in U_k$ with rational coordinates, and let $\eta$ be the flow line starting from $w$ with angle $\theta$.  Theorem~\ref{thm::flow_line_interaction} determines the manner in which $\eta$ interacts with the flow lines of both $h$ itself as well as those of $h$ given $\eta'|_{[\tau_{j-1},\tau_j]}$ for $j=1,\ldots,k$, at least up until the paths hit $\partial U_k$.  Therefore it follows from the arguments of Lemma~\ref{lem::clean_tail_merge} that the segments of $\eta_{k+1}^L \setminus \partial U_k$ are contained in the union of flow lines of $h$ starting from rational coordinates with angle $\tfrac{\pi}{2}$.  If $\theta=\tfrac{\pi}{2}$ and $\eta$ merges with $\eta_{k+1}^L$, then after hitting and bouncing off $\partial U_k$, both paths will continue to agree.  The reason is that both paths reflect off $\partial U_k$ in the same direction and satisfy the same conformal Markov property viewed as flow lines of the GFF $h$ given $\eta'|_{[\tau_{j-1},\tau_j]}$ for $j=1,\ldots,k$ as well as the segments of $\eta_{k+1}^L$ and $\eta$ drawn up until first hitting $\partial U_k$.  See the proof of Lemma~\ref{lem::tail_decomposition} for a similar argument. Consequently, the entire range of $\eta_{k+1}^L$ is contained in the union of flow lines of $h$ with angle $\tfrac{\pi}{2}$ starting from points with rational coordinates.  The same likewise holds with $\eta_{k+1}^R$ and $-\tfrac{\pi}{2}$ in place of $\eta_{k+1}^L$ and $\tfrac{\pi}{2}$.  This proves the claim.

We will now show that $\eta_z^L$ and $\eta_z^R$ give the left and right boundaries of $\eta'$.
Suppose that $\kappa' \in (4,8)$.  In this case, $\eta'$ makes loops around $z$ since the path is not space-filling.
Suppose that $\tau'$ is a stopping time for $\eta'$ at which it has made a clockwise loop around $z$.  Theorem~\ref{thm::flow_line_interaction} (recall Figure~\ref{fig::inthepocket} and Figure~\ref{fig::inthepocket2}) then implies that $\eta_z^L$ almost surely merges into the left boundary of $\eta'([0,\tau'])$ before leaving the closure of the component of $D \setminus \eta'([0,\tau'])$ which contains $z$.  The same likewise holds when we replace $\eta_z^L$ with $\eta_z^R$ and $\tau'$ is a stopping time at which $\eta'$ has made a counterclockwise loop.  The result follows because, as we mentioned earlier, $\eta'$ almost surely tends to $z$ and, by (the argument of) Lemma~\ref{lem::radial_loop}, almost surely makes arbitrarily small loops around $z$ with both orientations.  The argument for $\kappa' \geq 8$ so that $\eta'$ does not make loops around $z$ is explained in Figure~\ref{fig::duality_no_loops}.

It is left to prove the statement regarding the conditional law of $\eta'$ given $\eta_z^L$ and $\eta_z^R$.  Since the left and right boundaries of $\eta'$ are given by $\eta_z^L$ and $\eta_z^R$, respectively, it follows that $\eta'$ is almost surely equal to the concatenation of the counterflow lines of the conditional GFF $h$ given $\eta_z^L$ and $\eta_z^R$  in each of the components of $D \setminus (\eta_z^L \cup \eta_z^R)$ which are visited by $\eta'$.  Indeed, the restriction of $\eta'$ to each of these components satisfies the same conformal Markov property as each of the corresponding counterflow lines of the conditional GFF.  Thus, the claim follows from \cite[Theorem~1.2]{MS_IMAG}.  By reading off the boundary data of $h$ in each of these components, we see that if such a component does not intersect $\partial D$ in an interval, then the conditional law of $\eta'$ is independently that of an $\SLE_{\kappa'}(\tfrac{\kappa'}{2}-4;\tfrac{\kappa'}{2}-4)$, as desired.
\end{proof}

We now extend Theorem~\ref{thm::lightcone} to the setting in which we add a conical singularity to the GFF.


\begin{theorem}
\label{thm::lightcone_alpha}
The statements of Theorem~\ref{thm::lightcone} hold if we replace the GFF with $h_{\alpha \beta} = h+\alpha \arg(\cdot-z) + \beta \log|\cdot-z|$ and $\alpha \leq \tfrac{\chi}{2}$ where $h$ is a GFF on $D$ such that the counterflow line $\eta'$ of $h_{\alpha \beta}$ starting from $y \in \partial D$ is fully branchable.
\end{theorem}

\begin{remark}
\label{rem::boundary_filling_value}
The value $\tfrac{\chi}{2}$ is significant because it is the critical value for $\alpha$ at which $\eta'$ fills its own outer boundary.  That is, if $\alpha \geq \tfrac{\chi}{2}$ then $\eta'$ fills its own outer boundary and if $\alpha < \tfrac{\chi}{2}$ then $\eta'$ does not fill its own outer boundary.  This corresponds to a $\rho$ value of $\tfrac{\kappa'}{2}-4$ (recall Figure~\ref{fig::radial_bd_cfl}).
\end{remark}

\begin{remark}
\label{rem::interior_alpha_boundary_branching}
If $D = \h$ and the boundary data of $h_{\alpha \beta}$ is given by $\lambda'$ (resp.\ $-\lambda'$) on $\R_-$ (resp.\ $\R_+$), then $\eta'$ is a radial $\SLE_{\kappa'}^\beta(\rho)$ process where $\rho = \kappa'-6-2\pi \alpha/\lambda'$ (recall Figure~\ref{fig::radial_bd_cfl}).
\end{remark}

\begin{proof}[Proof of Theorem~\ref{thm::lightcone_alpha}]
The argument is the same as the proof of Theorem~\ref{thm::lightcone_alpha}.  In particular, we break the proof into the two cases where
\begin{enumerate}
\item $\alpha \in (\tfrac{\chi}{2}-\tfrac{1}{\sqrt{\kappa'}},\tfrac{\chi}{2}]$ so that $\eta'$ can hit its own outer boundary as it wraps around $z$ and
\item $\alpha \leq \tfrac{\chi}{2}-\tfrac{1}{\sqrt{\kappa'}}$ so that $\eta'$ cannot hit its own outer boundary as it wraps around $z$.
\end{enumerate}
The proof in the first case is analogous to the proof of Theorem~\ref{thm::lightcone} for $\kappa' \in (4,8)$ and the second is analogous to Theorem~\ref{thm::lightcone} with $\kappa' \geq 8$.  Note that for $\alpha=\tfrac{\chi}{2}$, $\eta_z^L = \eta_z^R$ almost surely.
\end{proof}

We can now complete the proofs of Theorem~\ref{thm::alphabeta_counterflow} and Theorem~\ref{thm::whole_plane_duality}.

\begin{proof}[Proof of Theorem~\ref{thm::alphabeta_counterflow} and Theorem~\ref{thm::whole_plane_duality}]
Assume that $h_{\alpha \beta}$, viewed as a distribution modulo a global multiple of $2\pi(\chi+\alpha)$, is defined on all of $\C$.  The construction of the coupling as described in Theorem~\ref{thm::alphabeta_counterflow} follows from the same argument used to construct the coupling in Theorem~\ref{thm::existence} and Theorem~\ref{thm::alphabeta}; see Remark~\ref{rem::alpha_coupling}.  The statement of Theorem~\ref{thm::whole_plane_duality} follows due to the way that the coupling is generated as well as Theorem~\ref{thm::lightcone} and Theorem~\ref{thm::lightcone_alpha}.  Finally, that the path $\eta'$ is almost surely determined by $h_{\alpha \beta}$ follows because Theorem~\ref{thm::uniqueness} and Theorem~\ref{thm::alphabeta} tell us that its left and right boundaries $\eta^L$ and $\eta^R$, respectively, are almost surely determined by $h_{\alpha \beta}$.  Moreover, once we condition on $\eta^L$ and $\eta^R$, we know that $\eta'$ is coupled with $h_{\alpha \beta}$ independently as a chordal counterflow line in each of the components of $\C \setminus (\eta^L \cup \eta^R)$ which are visited by $\eta'$.  This allows us to invoke the boundary emanating uniqueness theory \cite[Theorem~1.2]{MS_IMAG}.  The case that $D \neq \C$ follows by absolute continuity (Proposition~\ref{prop::local_set_whole_plane_bounded_compare}).  The other assertions of Theorem~\ref{thm::whole_plane_duality} follow similarly.
\end{proof}

\begin{proof}[Proof of Theorem~\ref{thm::transience}]
We have already given the proof for $\kappa \in (0,4)$ in Section~\ref{subsec::transience}.  The result for $\kappa' > 4$ and $\rho \in (-2,\tfrac{\kappa'}{2}-2)$ follows from Lemma~\ref{lem::radial_loop}.  The result for $\kappa' >4$ and $\rho \geq \tfrac{\kappa'}{2}-2$ follows by invoking Theorem~\ref{thm::lightcone_alpha}.  In particular, if a radial $\SLE_{\kappa'}(\rho)$ process $\eta'$ in $\D$ did not satisfy $\lim_{t \to \infty} \eta'(t) = 0$, then its left and right boundaries would not be continuous near the origin.
\end{proof}

We finish this subsection by computing the maximal number of times that a counterflow line can hit a given point or the domain boundary and then relate the dimension of the various types of self-intersection points of counterflow lines to the dimension of the intersection of GFF flow lines starting from the boundary.


\begin{proposition}
\label{prop::counterflow_maximal_intersection}
Fix $\alpha > -\tfrac{\chi}{2}$ and $\beta \in \R$.  Suppose that $h$ is a GFF on $\C$ and $h_{\alpha \beta} = h-\alpha \arg(\cdot) - \beta\log|\cdot|$, viewed as a distribution defined up to a global multiple of $2\pi(\chi+\alpha)$.  Let $\eta'$ be the counterflow line of $h_{\alpha \beta}$ starting from $\infty$.  Let 
\[ I(\kappa') = 
	\begin{cases}
		2 &\quad\text{if}\quad\quad \kappa' \in (4,8),\\
		3 &\quad\text{if}\quad\quad \kappa' \geq 8.
	\end{cases}\]
The maximal number of times that $\eta'$ can hit any given point (i.e., maximal multiplicity) is given by 
\begin{equation}
\label{eqn::counterflow_intersection_alpha}
 \max\left( \left\lceil \frac{\kappa'}{2(\kappa'+2\alpha \sqrt{\kappa'}-4)} \right\rceil , I(\kappa') \right) \quad\text{almost surely.}
\end{equation}
In particular, the maximal number of times that a radial or whole-plane $\SLE_{\kappa'}^\mu(\rho)$ process with $\rho > \tfrac{\kappa'}{2}-4$ and $\mu \in \R$ can hit any given point is given by
\begin{equation}
\label{eqn::counterflow_intersection_rho}
 \max\left( \left\lceil \frac{\kappa'}{2(2+\rho)} \right\rceil , I(\kappa') \right) \quad\text{almost surely.}
\end{equation}
Finally, the maximal number of times that a radial $\SLE_{\kappa'}^\mu(\rho)$ process with $\rho > \tfrac{\kappa'}{2}-4$ and $\mu \in \R$ process can hit the domain boundary is given by
\begin{equation}
\label{eqn::counterflow_boundary_intersection_rho}
\max\left( \left\lceil \frac{\kappa'}{2(2+\rho)} \right\rceil -1, \left\lceil \frac{\kappa'-4}{2(2+\rho)} \right\rceil, I(\kappa')-1 \right).
\end{equation}
\end{proposition}
The value $I(\kappa')$ in the statement of Proposition~\ref{prop::counterflow_maximal_intersection} gives the maximal number of times that a chordal $\SLE_{\kappa'}$ or $\SLE_{\kappa'}(\ul{\rho})$ process can hit an interior point.  In particular, chordal $\SLE_{\kappa'}$ and $\SLE_{\kappa'}(\ul{\rho})$ processes have double but not triple points or higher order self-intersections for $\kappa' \in (4,8)$ and have triple points but not higher order self-intersections for $\kappa' \geq 8$.  The other expressions in the maximum function in~\eqref{eqn::counterflow_intersection_alpha},~\eqref{eqn::counterflow_intersection_rho} give the number of intersections which can arise from the path winding around its target point and then hitting itself.  In particular, note that the first expression in the maximum in~\eqref{eqn::counterflow_intersection_rho} is identically equal to $1$ for $\rho \geq \tfrac{\kappa'}{2}-2$, i.e.\ $\rho$ is at least the critical value for such a process to have this type of self-intersection (Lemma~\ref{lem::radial_critical_for_hitting}).  Similarly, the first expression in the maximum in~\eqref{eqn::counterflow_boundary_intersection_rho} vanishes for $\rho \geq \tfrac{\kappa'}{2}-2$.

\begin{proof}[Proof of Proposition~\ref{prop::counterflow_maximal_intersection}]
As we explained before the proof, there are two sources of self-intersections in the interior of the domain:
\begin{enumerate}
\item The double ($\kappa' > 4$) and triple ($\kappa' \geq 8$) points which arise from the segments of $\eta'$ restricted to each of the complementary components of the left and right boundaries of $\eta'$ that it visits, and
\item The intersections of the left and right boundaries of $\eta'$.
\end{enumerate}
Since chordal $\SLE_{\kappa'}(\tfrac{\kappa'}{2}-4;\tfrac{\kappa'}{2}-4)$ is a boundary-filling, continuous process, it is easy to see that the set of double or triple points of the first type mentioned above which are also contained in the left and right boundaries of $\eta'$ is countable.  (If $z$ is in the left or right boundary of $\eta'$, $\tau_z^1$ is the first time that $\eta'$ hits $z$ and $\tau_z^2$ is the second, then the interval $[\tau_z^1,\tau_z^2]$ must contain a rational.)  Consequently, the two-self intersection sets are almost surely disjoint as $\rho > \tfrac{\kappa'}{2}-4$.  Moreover, we note that the self-intersections which are contained in the left and right boundaries arise by the path either making a series of clockwise or counterclockwise loops.  In $j-1$ such successive loops, the points that $\eta'$ hits $j$ times correspond to the points where the left (resp.\ right) side of $\eta'$ hits the right (resp.\ left) side $j-1$ times.  Theorem~\ref{thm::whole_plane_duality} then tells us that the height difference of the aforementioned intersection set is given by $2\pi(j-1)(\chi+\alpha) - \pi \chi$.  Then~\eqref{eqn::counterflow_intersection_alpha},~\eqref{eqn::counterflow_boundary_intersection_rho} follow by solving for the value of $j$ that makes this equal to $2\lambda-\pi \chi$ and applying Theorem~\ref{thm::flow_line_interaction}.

We will now prove~\eqref{eqn::counterflow_boundary_intersection_rho}.  For concreteness, we will assume that $\eta'$ is a radial $\SLE_{\kappa'}^\beta(\rho)$ process in $\D$ starting from $-i$ with a single boundary force point of weight $\rho$ located at $(-i)^-$, i.e.\ immediately to the left of $-i$.  We can think of $\eta'$ as the counterflow line of $h+\alpha \arg(\cdot) + \beta \log|\cdot|$ where $h$ is a GFF on $\D$ so that the sum has the boundary data illustrated in Figure~\ref{fig::radial_bd_cfl}.  Then $\eta'$ wraps around and hits a boundary point for the $j$th time for $j \geq 1$ with either a height difference of $2\pi(\chi-\alpha)(j-1) - 2\pi \alpha + \pi \chi$ (if the first intersection occurs to the left of the force point) or with a height difference of $2\pi(\chi-\alpha)(j-1) + 2\lambda'- \pi \chi$ otherwise (see Figure~\ref{fig::radial_bd_cfl}).  (The reason that we now have $\chi-\alpha$ in place of $\chi + \alpha$ as in the previous paragraph is because here we have added $\alpha \arg(\cdot)$ rather than subtracted it.)  Solving for the value of $j$ that makes this equal to $2\lambda-\pi \chi$ and then applying Theorem~\ref{thm::flow_line_interaction} yields the first two expressions in the maximum.  The final expression in the maximum comes from the multiple boundary intersections that can occur when the process interacts with the boundary before passing through the force point (as well as handles the possibility that the process can only hit the boundary one time after passing through the force point).
\end{proof}

\begin{proposition}
\label{prop::counterflow_intersection_dimension}
Suppose that we have the same setup as in Proposition~\ref{prop::counterflow_maximal_intersection} and assume that the assumption~\eqref{eqn::boundary_intersection_dimension} of Proposition~\ref{prop::intersection_dimension} holds.  For each $j \in \N$, let $\CI_j'$ be the set of interior points that $\eta'$ hits exactly $j$ times which are contained in the left and right boundaries of $\eta'$ and let
\begin{equation}
\label{eqn::counterflow_selfintersection_dimension_angle}
 \theta_j' =  2\pi (j-1) \left(1+\frac{\alpha}{\chi}\right) - \pi = \frac{\pi(2j(2+\rho)-2\rho -\kappa')}{\kappa'-4}.
\end{equation}
Then
\[ \p[\dim_\CH(\CI_j') = d(\theta_j')] = 1 \quad\text{for each}\quad j \geq 2\]
where $d(\theta_j')$ is the dimension of the intersection of boundary emanating GFF flow lines with an angle gap of $\theta_j'$.
\end{proposition}
Note that the angle $\theta_j'$ in~\eqref{eqn::counterflow_selfintersection_dimension_angle} is equal to the critical angle~\eqref{eqn::critical_angle} when $\rho=\tfrac{\kappa'}{2}-2$ and $j=2$ and exceeds the critical angle for larger values of $j$.
\begin{proof}[Proof of Proposition~\ref{prop::counterflow_intersection_dimension}]
This follows from the argument of Proposition~\ref{prop::counterflow_maximal_intersection} as well as from the argument of Proposition~\ref{prop::intersection_dimension}.
\end{proof}

\begin{proposition}
\label{prop::counterflow_boundary_dimension}
Fix $\alpha < \tfrac{\chi}{2}$ and $\beta \in \R$.  Suppose that $h_{\alpha \beta} = h+\alpha \arg(\cdot) + \beta\log|\cdot|$ where $h$ is a GFF on $\D$ so that $h_{\alpha \beta}$ has the boundary data depicted in Figure~\ref{fig::radial_bd_cfl} with $O_0 = W_0^-$ (i.e., $O_0$ is immediately to the left of $W_0$).  Let $\eta'$ be the counterflow line of $h_{\alpha \beta}$ starting from $W_0$ so that $\eta' \sim \SLE_{\kappa'}^\beta(\rho)$ with $\rho = \kappa'-6-2\pi \alpha/\lambda' > \tfrac{\kappa'}{2}-4$.  Assume that assumption~\eqref{eqn::boundary_intersection_dimension2} of Proposition~\ref{prop::boundary_intersection_dimension} holds.  Let $\CJ_{j,L}'$ (resp.\ $\CJ_{j,R}'$) be the set of points that $\eta'$ hits on $\partial \D$ exactly $j$ times where the first intersection occurs to the left (resp.\ right) of the force point.  For each $j \in \N$, let
\begin{align*}
\phi_{j,L}
&= 2\pi(j-1)\left(1-\frac{\alpha}{\chi}\right)-\frac{2\pi \alpha}{\chi}+\pi
= \frac{\pi(4-\kappa'+2j(2+\rho))}{\kappa'-4} \quad\text{and}\\
\phi_{j,R}
&= 2\pi(j-1)\left(1-\frac{\alpha}{\chi}\right)+\frac{2\lambda'}{\chi}-\pi
 = \frac{\pi(4-\kappa'-2\rho+2j(2+\rho))}{\kappa'-4}.
\end{align*}
For each $j \in \N$, we have
\begin{align*}
 \p[\dim_\CH(\CJ_{j,L}') = b(\phi_{j,L})\giv\CJ_{j,L}' \neq\emptyset ] &= 1 \quad\text{and}\\
 \p[\dim_\CH(\CJ_{j,R}') = b(\phi_{j,R})\giv\CJ_{j,R}' \neq \emptyset ] &= 1
\end{align*}
where $b(\theta)$ is the dimension of the intersection of a boundary emanating flow line in $\h$ with $\partial \h$ where the path hits with an angle gap of $\theta$.
\end{proposition}
\begin{proof}
The case for $j \geq 2$ follows from the argument at the end of the proof of Proposition~\ref{prop::counterflow_maximal_intersection} as well as the argument of the proof of Proposition~\ref{prop::counterflow_intersection_dimension}.  The angle gaps for $j=1$ are computed similarly and can be read off from Figure~\ref{fig::radial_bd_cfl}.
\end{proof}

\subsection{Space-filling $\SLE_{\kappa'}$}

\begin{figure}[htb!]
\begin{center}
\includegraphics[scale=0.85]{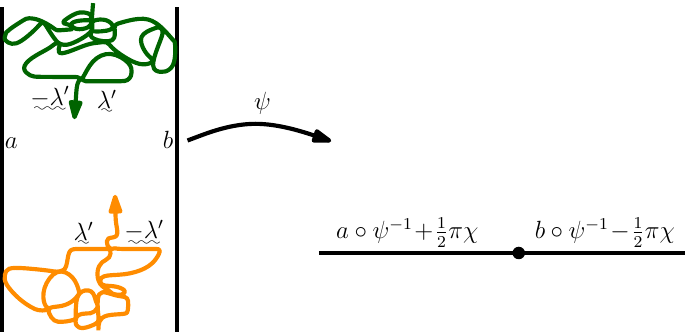}
\end{center}
\caption{\label{fig::spacefillingboundarydata}
Consider a GFF $h$ on the infinite vertical strip $\vstrip = [-1,1] \times \R$ whose boundary values are given by some function $a$ on the left side and some function $b$ on the right side.  We assume that $a,b$ are piecewise constant and change values only a finite number of times. Provided that $\|a\|_\infty < \lambda' + \tfrac{\pi \chi}{2}=\lambda$ and $\|b\|_\infty < \lambda' + \tfrac{\pi \chi}{2}=\lambda$ we can draw the orange counterflow line from the bottom to the top of $\vstrip$ with the boundary conditions shown.  This can be seen by conformally mapping to the upper half-plane (via the coordinate change in Figure~\ref{fig::coordinatechange}) as shown and noting that the corresponding boundary values are strictly greater than $-\lambda'$ on the negative real axis and strictly less than $+\lambda'$ on the positive real axis.  A symmetric argument shows that we can also draw the green counterflow line from the top of $\vstrip$ to the bottom.  Note that, as illustrated, the angles of the flow lines which describe the outer boundary of the orange and green counterflow lines are the same.  This is the observation that leads to the reversibility of space-filling $\SLE_{\kappa'}(\rho_1;\rho_2)$.}
\end{figure}

\begin{figure}[htb!]
\begin{center}
\includegraphics[scale=0.85]{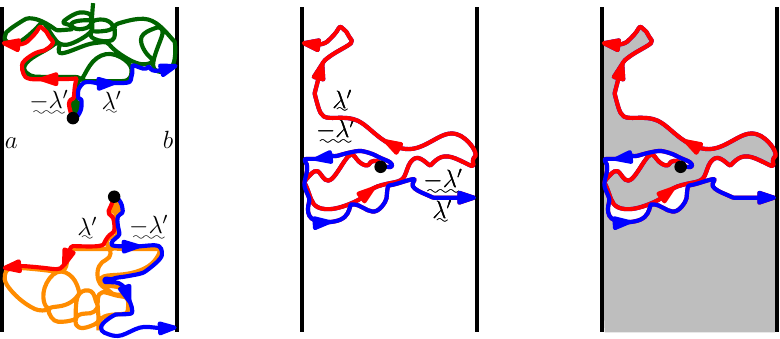}
\end{center}
\caption{\label{fig::spacefillingboundarydata2}
Same setup as in Figure~\ref{fig::spacefillingboundarydata}.  The first figure illustrates the left and right boundary paths of the orange and green curves, shown as red and blue lines stopped when they first hit $\partial \vstrip$.  At the hitting point, the boundary data is as shown ($\pm \lambda'$ on the two sides of the paths as they approach the boundaries horizontally, appropriately modified by winding).  The middle figure shows the flow lines with angles $\tfrac{\pi}{2}$ and $-\tfrac{\pi}{2}$ from a generic point $z$.  The red (resp.\ blue) path is stopped the first time it hits the left (resp.\ right) boundary with the appropriate $\pm \lambda'$ boundary conditions (as opposed to these values plus an integer multiple of $2 \pi \chi$).  In the third figure, we consider the complement of the red and blue paths.  We color gray the points in components of this complement whose boundaries are formed by the left side of a red segment or the right side of a blue segment.  The remaining points are left white.  Fix a countable dense set $(z_i)$ in $\vstrip$ including $z=z_0$, and consider the gray points ``after'' and the white points ``before $z$.''  By drawing the figure for all $z$, we almost surely obtain a total ordering of these $z_i$.  Up to monotone reparameterization, there is a unique continuous curve $\eta$ that hits the $(z_i)$ in order and has the property that $\eta^{-1} (z_i)$ is a dense set of times.  This curve is called the {\bf space-filling counterflow line} from~$+\infty$~to~$-\infty$.}
\end{figure}

In this section, we will prove the existence and continuity of the space-filling $\SLE_{\kappa'}$ processes for $\kappa' > 4$, thus proving Theorem~\ref{thm::space_filling_sle_existence}.  Recall that these processes are defined in terms of an induced ordering on a dense set as described in Section~\ref{subsubsec::branching_space_filling}; see also Figure~\ref{fig::space_filling_ordering} as well as Figure~\ref{fig::spacefillingboundarydata} and Figure~\ref{fig::spacefillingboundarydata2}. As we will explain momentarily, Theorem~\ref{thm::space_filling_reversibility} follows immediately from Theorem~\ref{thm::space_filling_sle_existence} since the definition of space-filling $\SLE_{\kappa'}(\rho_1;\rho_2)$ has time-reversal symmetry built in.  Theorem~\ref{thm::space_filling_reversibility} in turn implies Theorem~\ref{thm::bigger_than_8_reversibility}.  For the convenience of the reader, we restate these results in the following theorem.  We remind the reader that the range of $\rho_i$ values considered in Theorem~\ref{thm::space_filling_reversibility2} is summarized in Figure~\ref{fig::sfrd} and Figure~\ref{fig::sfrd2}.

\begin{theorem}
\label{thm::space_filling_reversibility2}
Let $D$ be a Jordan domain and fix $x,y \in \partial D$ distinct.  Suppose that $\kappa' > 4$ and $\rho_1, \rho_2 \in (-2, \tfrac{\kappa'}{2} - 2)$ and let $h$ be a GFF on $D$ whose counterflow line from $y$ to $x$ is an $\SLE_{\kappa'}(\rho_1;\rho_2)$ process (which is fully branchable).  Then space-filling $\SLE_{\kappa'}(\rho_1;\rho_2)$ in $D$ from $y$ to $x$ exists, is well-defined, and almost surely is a continuous path when parameterized by area.  The path is almost surely determined by $h$ and the path almost surely determines $h$.  When $\kappa' \geq 8$ and $\rho_1,\rho_2 \in (-2, \tfrac{\kappa'}{2}-4]$, space-filling $\SLE_{\kappa'}(\rho_1;\rho_2)$ describes the same law as chordal $\SLE_{\kappa'}(\rho_1;\rho_2)$.  Moreover, the time-reversal of a space-filling $\SLE_{\kappa'}(\rho_1;\rho_2)$ is a space-filling $\SLE_{\kappa'}(\wt{\rho}_1;\wt{\rho}_2)$ in $D$ from $x$ to $y$ where $\wt{\rho}_i$ is the reflection of $\rho_i$ about the $\tfrac{\kappa'}{4}-2$ line.  In particular, space-filling $\SLE_{\kappa'}(\tfrac{\kappa'}{4}-2;\tfrac{\kappa'}{4}-2)$ has time-reversal symmetry.
\end{theorem}

We are now going to explain how to extract Theorem~\ref{thm::cle_locally_finite} from Theorem~\ref{thm::space_filling_reversibility2}.

\begin{proof}[Proof of Theorem~\ref{thm::cle_locally_finite}]
Fix $\kappa' \in (4,8)$.  Suppose that $h$ is a GFF on a Jordan domain $D$ such that its counterflow line $\eta'$ growing from a point $y \in \partial D$ is an $\SLE_{\kappa'}(\kappa'-6)$ process with a single force point located at $y^-$.  Then the branching counterflow line $\eta'$ of $h$ starting from $y$ targeted at a countable dense set of interior points describes the same coupling of radial $\SLE_{\kappa'}(\kappa'-6)$ processes used to generate the $\CLE_{\kappa'}$ exploration tree in $D$ rooted at $y$ \cite{SHE_CLE}.  It follows from the construction of space-filling $\SLE_{\kappa'}(\kappa'-6)$ that the branch of $\eta'$ targeted at a given interior point $z$ agrees with the space-filling $\SLE_{\kappa'}(\kappa'-6)$ process coupled with $h$ starting from $y$ parameterized by $\log$ conformal radius as seen from $z$.  Consequently, the result follows from the continuity of space-filling $\SLE_{\kappa'}(\kappa'-6)$ proved in Theorem~\ref{thm::space_filling_reversibility2}.  Indeed, if $\CLE_{\kappa'}$ was not almost surely locally finite, then there would exist $\epsilon > 0$ such that the probability that there are an infinite number of loops with diameter at least $\epsilon > 0$ is positive.  Since the space-filling $\SLE_{\kappa'}(\kappa'-6)$ process traces the boundary of each of these loops in a disjoint time interval, it would follow that there would be an infinite number of pairwise disjoint time intervals $I_j$ whose image under the space-filling $\SLE_{\kappa'}(\kappa'-6)$ have diameter at least $\epsilon$.  Since the total area of the domain was assumed to be finite and space-filling $\SLE$ is parameterized by area, it follows that the interval of time on which it is defined is finite.  Therefore there must exist a sequence $(j_k)$ in $\N$ such that the length of the intervals $I_{j_k}$ tends to $0$ as $k \to \infty$.  This contradicts the almost sure continuity of space-filling $\SLE$, which proves the result.
\end{proof}

\begin{remark}
\label{rem::space_filling_general_boundary_data}
 One can extend the definition of the space-filling $\SLE_{\kappa'}(\rho_1;\rho_2)$ processes to the setting of many boundary force points.  These processes make sense and are defined in the same way provided the underlying GFF has fully branchable boundary data.  Moreover, there are analogous continuity and reversibility results, though we will not establish these here.  The former can be proved by generalizing our treatment of the two force point case given below using arguments which are very similar to those used to establish the almost sure continuity of the chordal $\SLE_{\kappa'}(\ul{\rho}^L;\ul{\rho}^R)$ processes in \cite[Section~7]{MS_IMAG} and the reversibility statement is immediate from the definition once continuity has been proved.  In this more general setting, the time-reversal of a space-filling $\SLE_{\kappa'}(\ul{\rho}^L;\ul{\rho}^R)$ process is a space-filling $\SLE_{\kappa'}(\wt{\ul{\rho}}^L;\wt{\ul{\rho}}^R)$ process where the vectors of weights $\wt{\ul{\rho}}^L,\wt{\ul{\rho}}^R$ are chosen so that, for each $k$ and $q \in \{L,R\}$, $\sum_{j=1}^k \wt{\rho}^{j,q}$ is the reflection of $\sum_{j=1}^{n^q-k+1} \rho^{j,q}$ about the $\tfrac{\kappa'}{4}-2$ line where $n^q = |\ul{\rho}^q| = |\wt{\ul{\rho}}^q|$.
\end{remark}

We remind the reader that the range of $\rho_i$ values considered in Theorem~\ref{thm::space_filling_reversibility2} is summarized in Figure~\ref{fig::sfrd} and Figure~\ref{fig::sfrd2}.  We will now explain how to derive the time-reversal component of Theorem~\ref{thm::space_filling_reversibility2} (see also Figure~\ref{fig::space_filling_ordering} as well as Figure~\ref{fig::spacefillingboundarydata} and Figure~\ref{fig::spacefillingboundarydata2}).  Consider a GFF $h$ on a vertical strip $\vstrip = [-1,1] \times \R$ in $\C$ and assume that the boundary value function $f$ for $h$ is piecewise constant (with finitely many discontinuities) and satisfies
\begin{equation}
\label{eqn::sidebound}
\| f \|_\infty < \lambda'+\frac{\pi\chi}{2} = \lambda.
\end{equation}
These boundary conditions ensure that we can draw a counterflow line from the bottom to the top of $\vstrip$ as well as from the top to the bottom, as illustrated in Figure~\ref{fig::spacefillingboundarydata}.  In each case the counterflow line is an $\SLE_{\kappa'}(\ul{\rho})$ process where $\sum_{j=1}^k \rho^{j,q} \in (-2,\tfrac{\kappa'}{2}-2)$ for $1 \leq k \leq |\ul{\rho}^q|$ with $q \in \{L,R\}$.  When the boundary conditions are equal to a constant $a$ (resp.\ $b$) on the left (resp.\ right) side of $\vstrip$, the counterflow line starting from the bottom of the strip is an $\SLE_{\kappa'}(\rho_1;\rho_2)$ process with
\begin{equation}
\label{eqn::forward_counterflow_rho}
 \rho_1= \frac{a}{\lambda'} + \left(\frac{\kappa'}{4}-2\right) \quad\text{and}\quad \rho_2 = -\frac{b}{\lambda'} + \left(\frac{\kappa'}{4}-2\right)
\end{equation}
(see Figure~\ref{fig::spacefillingboundarydata}).
In this case, the restriction~\eqref{eqn::sidebound} is equivalent to $\rho_1, \rho_2 \in (-2, \tfrac{\kappa'}{2} - 2)$.  The counterflow line from the top to the bottom of $\vstrip$ is an $\SLE_{\kappa'}(\wt{\rho}_1;\wt{\rho}_2)$ where
\begin{equation}
\label{eqn::reverse_counterflow_rho}
\wt{\rho}_1 = -\frac{a}{\lambda'} + \left(\frac{\kappa'}{4}-2\right) \quad\text{and}\quad \wt{\rho}_2 = \frac{b}{\lambda'} + \left(\frac{\kappa'}{4}-2\right).
\end{equation}
That is, $\wt{\rho}_i$ for $i=1,2$ is the reflection of $\rho_i$ about the $\tfrac{\kappa'}{4}-2$ line and $\wt{\rho}_1,\wt{\rho}_2 \in (-2,\tfrac{\kappa'}{2}-2)$.  These boundary conditions are fully branchable and both the counterflow line of the field from the top to the bottom and for the counterflow line of the field from the bottom to the top of $\vstrip$ hit both sides of $\vstrip$.  We note that the order in which the two counterflow lines hit points is as described in Figure~\ref{fig::space_filling_ordering}.  Moreover,~\eqref{eqn::forward_counterflow_rho} and~\eqref{eqn::reverse_counterflow_rho} together imply that Theorem~\ref{thm::space_filling_reversibility} follows once we prove Theorem~\ref{thm::space_filling_sle_existence}.

What remains to be shown is the almost sure continuity of space-filling $\SLE_{\kappa'}(\rho_1;\rho_2)$ and that the process is well-defined (i.e., the resulting curve does not depend on the choice of countable dense set).  Recall that the ordering which defines space-filling $\SLE_{\kappa'}(\rho_1;\rho_2)$ was described in Section~\ref{subsubsec::branching_space_filling}; see also Figure~\ref{fig::space_filling_ordering} as well as Figure~\ref{fig::spacefillingboundarydata} and Figure~\ref{fig::spacefillingboundarydata2}. By applying a conformal transformation, we may assume without loss of generality that we are working on a bounded Jordan domain $D$.  We then fix a countable dense set $(z_k)$ of $D$ (where we take $z_0 = x$) and, for each $k$, let $\eta_k^L$ (resp.\ $\eta_k^R$) be the flow line of $h$ starting from $z_k$ with angle $\tfrac{\pi}{2}$ (resp.\ $-\tfrac{\pi}{2}$).  For distinct indices $i,j \in \N$, we say that $z_i$ comes before $z_j$ if $z_i$ lies in a connected component of $D \setminus (\eta_j^L \cup \eta_j^R)$ part of whose boundary is traced by either the right side of $\eta_j^L$ or the left side of $\eta_j^R$.  For each $n \in \N$, the sets $\eta_1^L \cup \eta_1^R,\ldots, \eta_n^L \cup \eta_n^R$ divide $D$ into $n+1$ pockets $P_0^n,\ldots,P_n^n$.  For each $1 \leq k \leq n$, $P_k^n$ consists of those points in connected components of $D \setminus \bigcup_{j=1}^n (\eta_j^L \cup \eta_j^R)$ part of whose boundary is traced by a non-trivial segment of either the right side of $\eta_k^L$ or the left side of $\eta_k^R$ (see Figure~\ref{fig::spacefillingboundarydata2} for an illustration) before either path merges into some $\eta_j^L$ or $\eta_j^R$ for $j \neq k$ and $P_0^n$ consists of those points in $D \setminus \cup_{j=1}^n P_j^n$.  For each $0 \leq k \leq n$, let $\sigma_n(k)$ denote the index of the $k$th point in $\{z_0,\ldots,z_n\}$ in the induced order.   We then take $\eta_n'$ to be the piecewise linear path connecting $z_{\sigma_n(0)},\ldots,z_{\sigma_n(n)}$ where the amount of time it takes $\eta_n'$ to travel from $z_{\sigma_n(k)}$ to $z_{\sigma_n(k+1)}$ is equal to the area of $P_{\sigma_n(k)}^n$.  Let $d_k^n$ be the diameter of $P_k^n$.  Then to prove that the approximations $\eta_n'$ to space-filling $\SLE_{\kappa'}(\rho_1;\rho_2)$ are almost surely Cauchy with respect to the metric of uniform convergence of paths on the interval whose length is equal to the area of $D$, it suffices to show that
\begin{equation}
\label{eqn::pocket_diameter_to_zero}
 \max_{0 \leq k \leq n} d_k^n \to 0 \quad\text{as}\quad n \to \infty.
\end{equation}
Moreover, this implies that the limiting curve is well-defined because if $(\wt{z}_k)$ is another countable dense set and $(w_k)$ is the sequence with $w_{2k} = z_k$ and $w_{2k+1} = \wt{z}_k$ then it is clear from~\eqref{eqn::pocket_diameter_to_zero} that the limiting curve associated with $(w_k)$ is the same as the corresponding curve for both $(z_k)$ and $(\wt{z}_k)$.

For each $1 \leq k \leq n$ and $q \in \{L,R\}$, we let $d_k^{n,q}$ denote the diameter of $\eta_k^q$ stopped at the first time that it either merges with one of $\eta_j^q$ for $1 \leq j \leq n$ with $j \neq k$ or hits $\partial D$ with the appropriate height difference (as described in Figure~\ref{fig::spacefillingboundarydata2}).  We let $d_0^{n,L}$ (resp.\ $d_0^{n,R}$) denote the diameter of the counterclockwise (resp.\ clockwise) segment of $\partial D$ starting from $y$ towards $x$ that stops the first time it hits one of the $\eta_j^L$ (resp.\ $\eta_j^R$) for $1 \leq j \leq n$ at a point where the path terminates in $\partial D$.  Finally, we define $d_n^{n+1,L}$ and $d_n^{n+1,R}$ analogously except starting from $x$ with the former segment traveling in the clockwise direction and the latter counterclockwise direction.  Then it follows that
\begin{equation}
\label{eqn::pocket_diameter_bound}
 \max_{0 \leq k \leq n} d_k^n \leq 4 \max_{0 \leq k \leq n+1} \max_{q \in \{L,R\}} d_k^{n,q}.
\end{equation}
Consequently, to prove~\eqref{eqn::pocket_diameter_to_zero} it suffices to show that the right side of~\eqref{eqn::pocket_diameter_bound} almost surely converges to $0$ as $n \to \infty$.  We are going to prove this by first showing that an analog of this statement holds in the setting of the whole-plane (Proposition~\ref{prop::merge_into_net} in Section~\ref{subsub::pocket_diameter}) and then use a series of conditioning arguments to transfer our whole-plane statements to the setting of a bounded Jordan domain.  This approach is similar in spirit to our proof of the almost sure continuity of chordal $\SLE_{\kappa'}(\ul{\rho})$ processes given in \cite[Section~7]{MS_IMAG} and is carried out in Section~\ref{subsub::conditioning}.  As we mentioned in Remark~\ref{rem::space_filling_general_boundary_data}, this approach can be extended to establish the almost sure continuity of the many-force-point space-filling $\SLE_{\kappa'}(\ul{\rho}^L;\ul{\rho}^R)$ processes by reusing more ideas from \cite[Section~7]{MS_IMAG}, though we will not provide a treatment here.

\subsubsection{Pocket diameter estimates in the whole-plane}
\label{subsub::pocket_diameter}

Let $h$ be a whole-plane GFF viewed as a distribution defined up to a global multiple of $2\pi \chi$.  We will now work towards proving an analog of the statement that the right side of~\eqref{eqn::pocket_diameter_bound} converges to $0$ as $n \to \infty$ which holds for the flow lines of $h$ started in a fine grid of points.   Specifically, we fix $\epsilon > 0$ and let $\CD_\epsilon$ be the grid of points in $\epsilon \Z^2$ which are entirely contained in $[-2,2]^2$.  For each $z \in \C$, let $\eta_z$ be the flow line of $h$ starting from $z$.  Fix $K \geq 5$; we will eventually take $K$ to be large (though not changing with $\epsilon$).  Fix $z_0 \in [-1,1]^2$ and let $\eta = \eta_{z_0}$.  For each $z \in \C$ and $n \in \N$, let $\tau_z^n$ be the first time that $\eta_z$ leaves $B(z_0,K n\epsilon)$.  Note that $\tau_z^n = 0$ if $z \notin B(z_0,K n \epsilon)$.  Let $\tau^n = \tau_{z_0}^n$.

\begin{proposition}
\label{prop::merge_into_net}
There exists constants $C > 0$ and $K_0 \geq 5$ such that for every $K \geq K_0$ and $n \in \N$ with $n \leq (K\epsilon)^{-1}$ the following is true.  The probability that $\eta|_{[0,\tau^n]}$ does not merge with any of $\eta_z|_{[0,\tau_z^n]}$ for $z \in \CD_\epsilon \cap B(z_0,K n \epsilon)$ is at most $e^{-C n}$.
\end{proposition}

For each $n \in \N$, we let $w_n \in \CD_\epsilon \setminus B(z_0,(K n+1) \epsilon )$ be a point such that $|\eta(\tau^n) - w_n| \leq 2\epsilon$ (where we break ties according to some fixed but unspecified convention).  Let $\theta_0 = \lambda - \tfrac{\pi}{2}\chi = \lambda'$, i.e., half of the critical angle.  By Theorem~\ref{thm::flow_line_interaction}, a flow line with angle $\theta_0$ almost surely hits a flow line of zero angle started at the same point on its left side.  For each $n \in \N$, let $\gamma_n$ be the flow line of $h$ starting from $\eta(\tau^n)$ with angle $\theta_0$ and let $\sigma^n$ be the first time that $\gamma_n$ leaves $B(z_0,(K n+4)\epsilon)$.  Let $\CF_n$ be the $\sigma$-algebra generated by $\eta|_{[0,\tau^n]}$ as well as $\gamma_i|_{[0,\sigma^i]}$ for $i=1,\ldots,n-1$.  Let $E_n$ be the event that $A_n = \eta([0,\tau^{n+1}]) \cup \gamma_n([0,\sigma^n])$ separates $w_n$ from $\infty$ and that the harmonic measure of the left side of $\eta([0,\tau^{n+1}])$ as seen from $w_n$ in the connected component $P_n$ of $\C \setminus A_n$ which contains $w_n$ is at least~$\tfrac{1}{4}$.  See Figure~\ref{fig::pocket_generation2} for an illustration of the setup as well as the event $E_n$.  We will make use of the following lemma in order to show that it is exponentially unlikely that fewer than a linear number of the events $E_n$ occur for $1 \leq n \leq (K\epsilon)^{-1}$.

\begin{lemma}
\label{lem::make_good_pockets_ideal}
Fix $x^L \leq 0 \leq x^R$ and $\rho^L,\rho^R > -2$.  Suppose that $h$ is a GFF on $\h$ with boundary data so that its flow line $\eta$ starting from $0$ is an $\SLE_\kappa(\rho^L;\rho^R)$ process with $\kappa \in (0,4)$ and force points located at $x^L,x^R$.  Fix $\theta > 0$ such that the flow line $\eta_\theta$ of $h$ starting from $0$ with angle $\theta > 0$ almost surely does not hit the continuation threshold and almost surely intersects $\eta$.  Let $\tau$ (resp.\ $\tau_\theta$) be the first time that $\eta$ (resp.\ $\eta_\theta$) leaves $B(0,2)$.  Let $E$ be the event that $A = \eta([0,\tau]) \cup \eta_\theta([0,\tau_\theta])$ separates $i$ from $\infty$ and that the harmonic measure of the left side of $\eta$ as seen from $i$ in $\h \setminus A$ is at least $\tfrac{1}{4}$.  There exists $p_0 > 0$ depending only on $\kappa$, $\rho^L$, $\rho^R$, and $\theta$ (but not $x^L$ and $x^R$) such that $\p[E] \geq p_0$.
\end{lemma}
\begin{proof}
That $\p[E] > 0$ for a fixed choice of $x^L \leq 0 \leq x^R$ follows from the analogies of Lemma~\ref{lem::path_close} and Lemma~\ref{lem::path_close_hit} which are applicable for boundary emanating flow lines.  That $\p[E]$ is uniformly positive for all $x^L \in [-4,0^-]$ and $x^R \in [0^+,4]$ follows since the law of an $\SLE_\kappa(\ul{\rho})$ process is continuous in the location of its force points \cite[Section~2]{MS_IMAG}.  When either $x^L \notin [-4,0^-]$ or $x^R \notin [0^+,4]$, we can use absolute continuity to compare to the case that either $\rho^L = 0$, $\rho^R = 0$, or both.  Indeed, suppose for example that $x^L < -4$ and $x^R > 4$.  Let $f$ be the function which is harmonic in $\h$ and whose boundary values are given by
\[ \begin{cases}
   \lambda \rho^L \quad&\text{in}\quad(-\infty,x^L],\\
   0\quad&\text{in}\quad (x^L,x^R], \quad\quad \text{and}\\
   -\lambda \rho^R \quad&\text{in}\quad (x^R,\infty).
   \end{cases}\]
Then $h+f$ is a GFF on $\h$ whose boundary values are $-\lambda$ (resp.\ $\lambda$) on $(-\infty,0)$ (resp.\ $(0,\infty)$) so that its flow line starting from $0$ is an ordinary $\SLE_\kappa$ process.  Moreover, if $g \in C^\infty$ with $g|_{B(0,2)} \equiv 1$ and $g|_{\C \setminus B(0,3)} \equiv 0$ then $\| f g\|_\nabla$ is uniformly bounded in $x^L < -4$ and $x^R > 4$ and the flow line of $h+fg$ stopped upon exiting $B(0,2)$ is equal to the corresponding flow line of $h+f$.  Therefore the claim that we get a uniform lower bound on $\p[E]$ as $x^L < -4$ and $x^R > 4$ vary follows from \cite[Remark~3.5]{MS_IMAG}.  The other possibilities follow from a similar argument.
\end{proof}

\begin{figure}
\begin{center}
\includegraphics[scale=0.85]{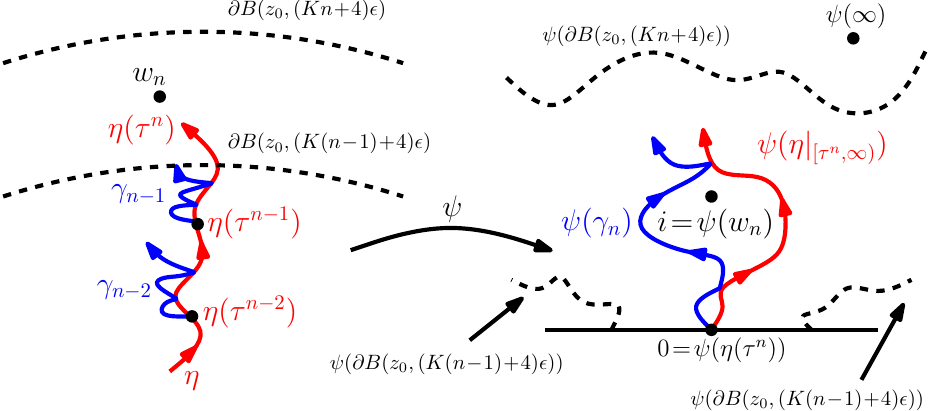}
\end{center}
\caption{\label{fig::pocket_generation2}
Setup for the proof of Lemma~\ref{lem::make_good_pockets}.  The right panel illustrates the event $E_n$.}
\end{figure}

\begin{lemma}
\label{lem::make_good_pockets}
There exists constants $K_0 \geq 5$ and $p_1 > 0$ such that for every $n \in \N$ and $K \geq K_0$ we have
\[ \p[ E_n | \CF_{n-1}]  \geq p_1.\]
\end{lemma}
\begin{proof}
Let $\psi$ be the conformal map which takes the unbounded connected component $U$ of $\C \setminus  (\eta([0,\tau^n]) \cup \bigcup_{j=1}^{n-1} \gamma_j([0,\sigma^j]))$ to $\h$ with $\psi(\eta(\tau^n)) = 0$ and $\psi(w_n) = i$.  The Beurling estimate \cite[Theorem~3.69]{LAW05} and the conformal invariance of Brownian motion together imply that if we take $K \geq 5$ sufficiently large then the images of $\cup_{j=1}^{n-1} \gamma_j([0,\sigma^j])$ and $\infty$ under $\psi$ lie outside of $B(0,100)$.  Consequently, the law of $\wt{h} = h \circ \psi^{-1} - \chi \arg (\psi^{-1})'$ restricted to $B(0,50)$ is mutually absolutely continuous with respect to the law of a GFF on $\h$ restricted to $B(0,50)$ whose boundary data is chosen so that its flow line starting from $0$ is a chordal $\SLE_\kappa(2-\kappa)$ process with a single force point located at the image under $\psi$ of the most recent intersection of $\eta|_{[0,\tau^n]}$ with itself (or just a chordal $\SLE_\kappa$ process if there is no such self-intersection point which lies in $\psi^{-1}(B(0,50))$).  The result then follows from Lemma~\ref{lem::make_good_pockets_ideal} (and the argument at the end of the proof of Lemma~\ref{lem::make_good_pockets_ideal} implies that we get a lower bound which is uniform in the location of the images $\psi(\gamma_j)$).
\end{proof}

On $E_n$, let $F_n$ be the event that $\eta_{w_n}$ merges with $\eta$ upon exiting $P_n$.  Let $\CF = \sigma(\CF_n : n \in \N)$.

\begin{lemma}
\label{lem::good_pocket_merge}
There exists $p_2 > 0$ and $K_0 \geq 5$ such that for every $K \geq K_0$ and $n \in \N$ we have
\[ \p[ F_n | \CF] \one_{E_n} \geq p_2 \one_{E_n}.\]
\end{lemma}
\begin{proof}
This follows from Lemma~\ref{lem::clean_merge} as well as by the absolute continuity argument given in the proof of Lemma~\ref{lem::make_good_pockets_ideal}.
\end{proof}

\begin{proof}[Proof of Proposition~\ref{prop::merge_into_net}]
Assume that $K_0$ has been chosen sufficiently large so that Lemma~\ref{lem::make_good_pockets} and Lemma~\ref{lem::good_pocket_merge} both apply.  Lemma~\ref{lem::make_good_pockets} implies that it is exponentially unlikely that we have fewer than $\tfrac{1}{2} p_1 n$ of the events $E_j$ occur and Lemma~\ref{lem::good_pocket_merge} implies that it is exponentially unlikely that we have fewer than a $\tfrac{1}{2} p_2$ fraction of these in which $F_j$ occurs.
\end{proof}

\subsubsection{Conditioning arguments}
\label{subsub::conditioning}

\begin{figure}[ht!]
\begin{center}
\includegraphics[scale=0.85]{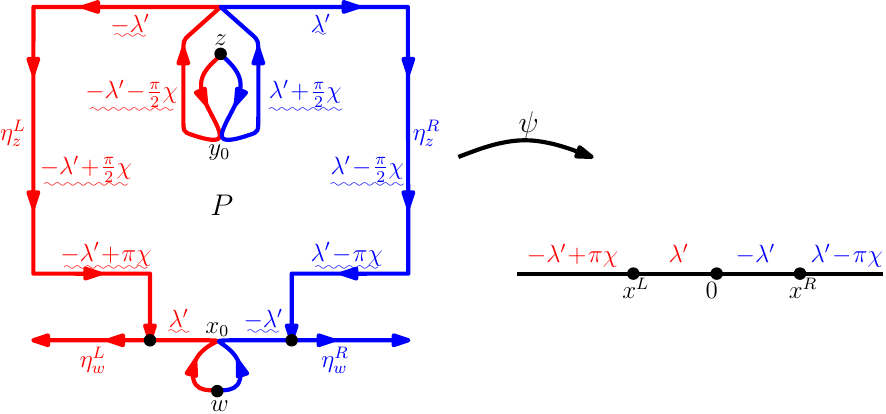}
\end{center}
\caption{\label{fig::whole_plane_to_finite} 
Suppose that $h$ is a whole-plane GFF viewed as a distribution defined up to a global multiple of $2\pi \chi$.  Fix $z,w \in \C$ distinct and let $\eta_u^L$ (resp.\ $\eta_u^R$) be the flow line of $h$ starting from $u$ with angle $\tfrac{\pi}{2}$ (resp.\ $-\tfrac{\pi}{2}$) for $u \in \{z,w\}$.  By Theorem~\ref{thm::flow_line_interaction} and Theorem~\ref{thm::merge_cross}, the flow lines $\eta_u^q$ for $u \in \{z,w\}$ and $q \in \{L,R\}$ form a pocket $P$, as illustrated above.  (It need not be true that $z,w \in \partial P$.)  The boundary data for the conditional law of $h$ in $P$ is as shown, up to a global additive constant in $2\pi \chi \Z$.  (The $\lambda'$ and $-\lambda'$ on the bottom of the figure indicate the heights along the horizontal segments of $\eta_w^L$ and $\eta_w^R$, respectively.)  The opening (resp.\ closing) point of $P$ is the first point on $\partial P$ traced by the right (resp.\ left) side of one of $\eta_u^L$ for $u \in \{z,w\}$, as indicated by $x_0$ (resp.\ $y_0$) in the illustration.  These points can be defined similarly in terms of $\eta_u^R$ for $u \in \{z,w\}$.  We also note that, in general, $\eta_z^q$ for $q \in \{L,R\}$ may have to wind around $z$ several times after first hitting $\eta_w^q$ before the paths merge.  Let $\psi \colon P \to \h$ be a conformal map with $\psi(x_0) = 0$ and $\psi(y_0) = \infty$.  Then the counterflow line of the GFF $h \circ \psi^{-1} - \chi \arg (\psi^{-1})'$ in $\h$ from $0$ to $\infty$ is an $\SLE_{\kappa'}(\tfrac{\kappa'}{2}-4;\tfrac{\kappa'}{2}-4)$ process where the force points $x^L$ and $x^R$ are located at the images under $\psi$ of the points where $\eta_z^L,\eta_w^L$ and $\eta_z^R, \eta_w^R$ merge, respectively.
}
\end{figure}

\begin{figure}[ht!]
\begin{center}
\includegraphics[scale=0.85]{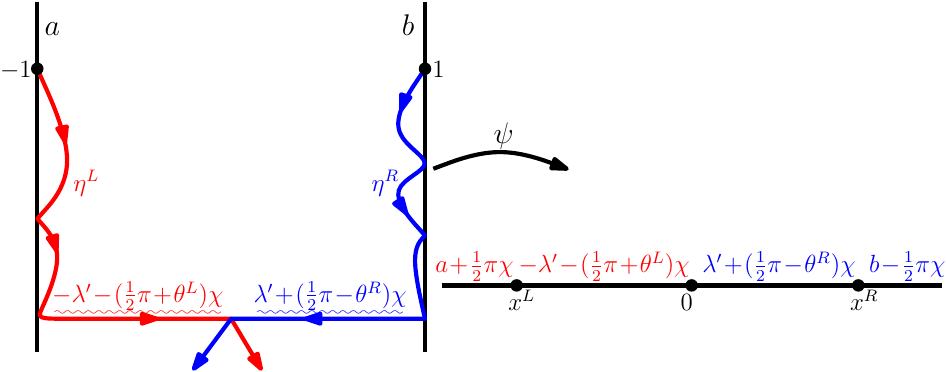}
\end{center}
\caption{\label{fig::conditioning_trick}  Suppose that $h$ is a GFF on $\vstrip = [-1,1] \times \R$ with boundary conditions given by a constant $a$ (resp.\ $b$) on the left (resp.\ right) side of $\partial \vstrip$.  Then $h$ is compatible with a coupling with a space-filling $\SLE_{\kappa'}(\rho_1;\rho_2)$ process $\eta'$ from the bottom to the top of $\vstrip$ with $\rho_1,\rho_2$ determined by $a,b$ as in~\eqref{eqn::forward_counterflow_rho}.  Taking $a,b$ so that $\rho_1=\rho_2=\tfrac{\kappa'}{2}-4$ we have that $\eta'$ is almost surely continuous by Step 1 in the proof of Theorem~\ref{thm::space_filling_reversibility2}.  By conditioning on flow lines $\eta^L,\eta^R$ of angles $\theta^L,\theta^R$ starting from $-1,1$, respectively, we can deduce the almost sure continuity of space-filling $\SLE_{\kappa'}(\rho_1;\rho_2)$ processes for $\rho_1,\rho_2 \in (-2,\tfrac{\kappa'}{2}-4]$ provided we choose $\theta^L,\theta^R$ appropriately.  Since the time-reversal of a space-filling $\SLE_{\kappa'}(\rho_1;\rho_2)$ process is a space-filling $\SLE_{\kappa'}(\wt{\rho}_1;\wt{\rho}_2)$ process where $\wt{\rho}_i$ is the reflection of $\rho_i$ about the $\tfrac{\kappa'}{4}-2$ line, we get the almost sure continuity when $\rho_1,\rho_2 \in [0,\tfrac{\kappa'}{2}-2)$.  By using this argument a second time except with $a,b$ chosen so that the corresponding space-filling $\SLE_{\kappa'}(\wh{\rho}_1;\wh{\rho}_2)$ process has weights $\wh{\rho}_i = \max(\rho_i,0)$, we get the almost sure continuity for all $\rho_1,\rho_2 \in (-2,\tfrac{\kappa'}{2}-2)$.
}
\end{figure}

In this section, we will reduce the continuity of space-filling $\SLE_{\kappa'}(\rho_1;\rho_2)$ processes to the statement given in Proposition~\ref{prop::merge_into_net} thus completing the proof of Theorem~\ref{thm::space_filling_reversibility2}.  The first step is the following lemma.

\begin{lemma}
\label{lem::make_space_filling_pocket}
Suppose that $h$ is a whole-plane GFF viewed as a distribution defined up to a global multiple of $2\pi \chi$.  Fix $z,w \in (-1,1)^2$ distinct.  Let $\eta_z^L$ (resp.\ $\eta_z^R$) be the flow line of $h$ starting from $z$ with angle $\tfrac{\pi}{2}$ (resp.\ $-\tfrac{\pi}{2}$) and define $\eta_w^L,\eta_w^R$ analogously.  Then the probability that the pocket $P$ formed by these flow lines (as described in Figure~\ref{fig::whole_plane_to_finite}) is contained in $[-1,1]^2$ is positive.
\end{lemma}
\begin{proof}
This follows by first applying Lemma~\ref{lem::path_close} and then applying Lemma~\ref{lem::path_close_hit} three times.  See also the argument of Lemma~\ref{lem::clean_merge}.
\end{proof}

\begin{proof}[Proof of Theorem~\ref{thm::space_filling_reversibility2}]
We are going to prove the almost sure continuity of space-filling $\SLE_{\kappa'}(\rho_1;\rho_2)$ in three steps.  In particular, we will first establish the result for $\rho_1,\rho_2 = \tfrac{\kappa'}{2}-4$, then extend to the case that $\rho_1,\rho_2 \in (-2,\tfrac{\kappa'}{2}-4]$, and then finally extend to the most general case that $\rho_1,\rho_2 \in (-2,\tfrac{\kappa'}{2}-2)$.
$\ $\newline

\noindent{\bf Step 1.} $\rho_1 = \rho_2 = \tfrac{\kappa'}{2}-4$. 
Let $h$ be a whole-plane GFF viewed as a distribution defined up to a global multiple of $2\pi \chi$.
Fix $z,w \in (-1,1)^2$ distinct and let $P$ be the pocket formed by the flow lines of $h$ with angles $\tfrac{\pi}{2}$ and $-\tfrac{\pi}{2}$ starting from $z,w$, as described in Lemma~\ref{lem::make_space_filling_pocket} (the particular choice is not important).  Let $\psi$ be a conformal map which takes $P$ to $\D$ with the opening (resp.\ closing) point of $P$ taken to $-i$ (resp.\ $i$).  Note that $\wt{h} = h \circ \psi^{-1} - \chi \arg (\psi^{-1})'$ is a GFF on $\D$ whose boundary data is such that its counterflow line starting from $-i$ is an $\SLE_{\kappa'}(\tfrac{\kappa'}{2}-4;\tfrac{\kappa'}{2}-4)$ process with force points located at the images $x^L,x^R$ of the points where $\eta_z^L,\eta_w^L$ and $\eta_z^R,\eta_w^R$ merge; see Figure~\ref{fig::pocket_generation2}.  Note that $x^L$ (resp.\ $x^R$) is contained in the clockwise (resp.\ counterclockwise) segment of $\partial \D$ from $-i$ to $i$.

Proposition~\ref{prop::merge_into_net} and a union bound implies that the following is true.  The maximal diameter of those flow lines started in $[-1,1]^2 \cap (\epsilon \Z)^2$ with angles $\tfrac{\pi}{2}$ and $-\tfrac{\pi}{2}$ stopped upon merging with a flow line of the same angle started in $[-2,2]^2 \cap (\epsilon \Z)^2$ goes to zero almost surely as $\epsilon \to 0$.  Consequently, conditionally on the positive probability event that $P \subseteq [-1,1]^2$, the maximal diameter of the pockets formed by the flow lines with angles $\tfrac{\pi}{2}$ and $-\tfrac{\pi}{2}$ starting from the points in $\psi(\CD_\epsilon)$ tends to zero with (conditional) probability one.  Indeed, this follows because we have shown that the right side of~\eqref{eqn::pocket_diameter_bound} tends to zero with (conditional) probability one.  This implies that there exists $x^L$ (resp.\ $x^R$) in the clockwise (resp.\ counterclockwise) segment of $\partial \D$ from $-i$ to $i$ and a countable, dense set $\CD$ of $\D$ such that space-filling $\SLE_{\kappa'}(\tfrac{\kappa'}{2}-4;\tfrac{\kappa'}{2}-4)$ in $\D$ from $-i$ to $i$ with force points located at $x^L,x^R \in \partial \D$ generated from flow lines starting at points in $\CD$ is almost surely continuous.  Once we have shown the continuity for one fixed choice of countable dense set, it follows for all countable dense sets by the merging arguments of Section~\ref{sec::interior_flowlines}.  (Recall, in particular, Lemma~\ref{lem::clean_tail_merge}.)  Moreover, we can in fact take $x^L = (-i)^-$ and $x^R = (-i)^+$ by further conditioning on the flow lines of $\wt{h}$ starting from $x^L$ and $x^R$ with angle $-\tfrac{3\pi}{2}$ and $\tfrac{3\pi}{2}$, respectively, which are reflected towards $-i$ and then restricting the path to the complementary component which contains $i$; see also Figure~\ref{fig::conditioning_trick}.  (Equivalently, in the setting of Figure~\ref{fig::whole_plane_to_finite}, we can reflect the flow lines $\eta_z^L$ and $\eta_z^R$ towards the opening point of the pocket $P$.)
$\ $\newline

\noindent{\bf Step 2.} $\rho_1,\rho_2 \in (-2,\tfrac{\kappa'}{2}-4]$.  Suppose that $h$ is a GFF on $\vstrip$ as in Figure~\ref{fig::conditioning_trick} where we take the boundary conditions to be $a=\tfrac{1}{4}\lambda' (\kappa'-8)$ and $b=-\tfrac{1}{4}\lambda'(\kappa'-8) = -a$.  By~\eqref{eqn::forward_counterflow_rho}, $h$ is compatible with a coupling with a space-filling $\SLE_{\kappa'}(\tfrac{\kappa'}{2}-4;\tfrac{\kappa'}{2}-4)$ process $\eta'$ from $-\infty$ to $+\infty$.  Let $\eta^L$ (resp.\ $\eta^R$) be the flow line of $h$ starting from $-1$ (resp.\ $1$) with angle $\theta^L \in (-\tfrac{3\pi}{2},-\tfrac{\pi}{2})$ (resp.\ $\theta^R \in (\tfrac{\pi}{2},\tfrac{3\pi}{2})$).  Note that as $\theta^L \downarrow -\tfrac{3\pi}{2}$, $\eta^L$ converges to the half-infinite vertical line starting starting from $-1$ to $-\infty$ and that when $\theta^L \uparrow -\tfrac{\pi}{2}$, $\eta^L$ ``merges'' with the right side of $\partial \vstrip$.  The angles $\tfrac{3\pi}{2}$ and $\tfrac{\pi}{2}$ have analogous interpretations for $\theta^R$.  Let $U$ be the unbounded connected component of $\vstrip \setminus (\eta^L \cup \eta^R)$ whose boundary contains $+\infty$ and let $\psi \colon U \to \h$ be a conformal map which takes the first intersection point of $\eta^L$ and $\eta^R$ to $0$ and sends $+\infty$ to $\infty$.  Then $\psi(\eta')$ is a space-filling $\SLE_{\kappa'}(\rho^{1,L},\rho^{2,L};\rho^{1,R},\rho^{2,R})$ process in $\h$ (recall Remark~\ref{rem::space_filling_general_boundary_data}) where
\begin{align*}
\rho^{1,L} &= \left(\frac{\kappa'}{2}-2 \right)\left(-\frac{1}{2} - \frac{\theta^L}{\pi} \right) - 2,\quad\quad \rho^{1,L} + \rho^{2,L} = \frac{\kappa'}{2}-4,\\
\rho^{1,R} &= \left(\frac{\kappa'}{2}-2 \right)\left(-\frac{1}{2} + \frac{\theta^R}{\pi} \right) - 2,\quad\quad \rho^{1,R} + \rho^{2,R} = \frac{\kappa'}{2}-4. \notag
\end{align*}
Moreover, the force points associated with $\rho^{1,L}$ and $\rho^{1,R}$ are immediately to the left and to the right of the initial point of the path.  For each $r > 0$, let $\psi_r = r \psi$.  Fix $R > 0$.  Since the law of $h \circ \psi_r^{-1} - \chi \arg (\psi_r^{-1})'$ restricted to $B(0,R) \cap \h$ converges in total variation as $r \to \infty$ to the law of a GFF on $\h$ whose boundary conditions are compatible with a coupling with space-filling $\SLE_{\kappa'}(\rho^{1,L};\rho^{1,R})$ process, also restricted to $B(0,R)$, we get the almost sure continuity of the latter process stopped the first time it exits $B(0,R)$.  By adjusting the angles $\theta^L$ and $\theta^R$, we can take $\rho^{1,L}, \rho^{1,R}$ to be any pair of values in $(-2,\tfrac{\kappa'}{2}-4]$.  This proves the almost sure continuity of space-filling $\SLE_{\kappa'}(\rho_1;\rho_2)$ for $\rho_1,\rho_2 \in (-2,\tfrac{\kappa'}{2}-4]$ from $0$ to $\infty$ in $\h$ stopped upon exiting $\partial B(0,R)$ for each $R > 0$.  To complete the proof of this step, we just need to show that if $\eta'$ is a space-filling $\SLE_{\kappa'}(\rho_1;\rho_2)$ process with $\rho_1,\rho_2 \in (-2,\tfrac{\kappa'}{2}-4]$ then $\eta'$ is almost surely transient: that is, $\lim_{t \to \infty} \eta'(t) = \infty$ almost surely.  This can be seen by observing that, almost surely, infinitely many of the flow line pairs $A_k$ starting from $2^k i$ for $k \in \N$ with angles $\tfrac{\pi}{2}$ and $-\tfrac{\pi}{2}$ stay inside of the annulus $B(0,2^{k+1}) \setminus B(0,2^{k-1})$ (the range of $\eta'$ after hitting $2^k i$ is almost surely contained in the closure of the unbounded connected component of $\h \setminus A_k$).  Indeed, this follows from Lemma~\ref{lem::path_close_hit} and that the total variation distance between the law of $h|_{\h \setminus B(0,s)}$ given $A_1,\ldots,A_k$ and that of $h|_{\h \setminus B(0,s)}$ (unconditionally) converges to zero when $k$ is fixed and $s \to \infty$.  (See the proof of Lemma~\ref{lem::flow_line_ring} for a similar argument.)
$\ $\newline

\noindent{\bf Step 3.}  $\rho_1,\rho_2 \in (-2,\tfrac{\kappa'}{2}-2)$.  By time-reversal, Step 2 implies the almost sure continuity of space-filling $\SLE_{\kappa'}(\wt{\rho}_1;\wt{\rho}_2)$ where $\wt{\rho}_i$ is the reflection of $\rho_i \in (-2,\tfrac{\kappa'}{2}-4]$ about the $\tfrac{\kappa'}{4}-2$ line.  In particular, we have the almost sure continuity of space-filling $\SLE_{\kappa'}(\rho_1;\rho_2)$ for all $\rho_1,\rho_2 \in [0,\tfrac{\kappa'}{2}-2)$.  We are now going to complete the proof by repeating the argument of Step 2.  Fix $\rho_1,\rho_2 \in (-2,\tfrac{\kappa'}{2}-2)$.  Suppose that $\wh{h}$ is a GFF on $\vstrip$ whose boundary conditions are chosen so that the associated space-filling $\SLE_{\kappa'}(\wh{\rho}_1;\wh{\rho}_2)$ process $\eta'$ from $-\infty$ to $+\infty$ satisfies $\wh{\rho}_i \geq \max(\rho_i,0)$ for $i=1,2$.  Explicitly, this means that the boundary data for $\wh{h}$ is equal to some constant $\wh{a}$ (resp.\ $\wh{b}$) on the left (resp.\ right) side of $\vstrip$ with
\begin{align*}
   \wh{a} = \lambda' \left( \wh{\rho}_1 - \left(\frac{\kappa'}{4}-2\right)\right)\quad\text{and}\quad \wh{b} = \lambda'\left(- \wh{\rho}_2+  \left(\frac{\kappa'}{4}-2\right)\right).
\end{align*}
We let $\wh{\eta}^L$ (resp.\ $\wh{\eta}^R$) be the flow line of $h$ starting from $-1$ (resp.\ $1$) with angle $\theta^L \in ( -(\lambda'+\wh{a})/\chi-\pi, -\tfrac{\pi}{2})$ (resp.\ $\theta^R \in (\tfrac{\pi}{2},(\lambda'-\wh{b})/\chi)+\pi$).  The range of angles is such that the flow line $\eta^L$ (resp.\ $\eta^R$) is almost surely defined and terminates upon hitting the right (resp.\ left) side of $\partial \vstrip$ or $-\infty$.  In particular, neither path tends to $+\infty$.  The conditional law of the space-filling $\SLE_{\kappa'}(\wh{\rho}_1;\wh{\rho}_2)$ process $\wh{\eta}'$ associated with $\wh{h}$ in the unbounded connected component of $\vstrip \setminus (\wh{\eta}^L \cup \wh{\eta}^R)$ with $+\infty$ on its boundary is a space-filling $\SLE_{\kappa'}(\rho^{1,L},\rho^{2,L};\rho^{1,R},\rho^{2,R})$ process with
\begin{align*}
\rho^{1,L} &= \left(\frac{\kappa'}{2}-2 \right)\left(-\frac{1}{2} - \frac{\theta^L}{\pi}\right) - 2,\quad\quad \rho^{1,L} + \rho^{2,L} = \wh{\rho}_1,\\
\rho^{1,R} &= \left(\frac{\kappa'}{2}-2 \right)\left(-\frac{1}{2} + \frac{\theta^R}{\pi}\right) - 2,\quad\quad \rho^{1,R} + \rho^{2,R} = \wh{\rho}_2.
\end{align*}
In particular, when $\theta^L$ (resp.\ $\theta^R$) takes on its minimal (resp.\ maximal) value, $\rho^{1,L} = \wh{\rho}_1$ (resp.\ $\rho^{1,R} = \wh{\rho}_2$).  When $\theta^L$ (resp.\ $\theta^R$) takes on its maximal (resp.\ minimal) value, $\rho^{1,L} = -2$ (resp.\ $\rho^{1,R} = -2$).  By adjusting $\theta^L$ and $\theta^R$ (as in Step 2), we can arrange it so that $\rho^{1,L} = \rho_1$ and $\rho^{1,R} = \rho_2$.  The proof is completed by using the scaling and transience argument at the end of Step~2.
\end{proof}

\section{Whole-plane time-reversal symmetries}
\label{sec::timereversal}

In this section we will prove Theorem~\ref{thm::whole_plane_reversibility}.  The proof is based on related arguments that appeared in \cite[Section~4]{MS_IMAG2}.  In \cite[Section~4]{MS_IMAG2}, the authors considered a pair of chordal flow lines $\eta_1$ and $\eta_2$ in a domain $D$.  The starting and ending points for the $\eta_i$ are the same but the paths have different angles.  It was observed that when $\eta_1$ is given, the conditional law of $\eta_2$ is that of a certain type of $\SLE_\kappa(\ul{\rho})$ process in the appropriate component of $D \setminus \eta_1$.  A similar statement holds with the roles of $\eta_1$ and $\eta_2$ reversed.  It is proved in \cite{MS_IMAG2} that these {\em conditional} laws (of one $\eta_i$ given the other) actually determine the overall {\em joint} law of the pair $(\eta_1, \eta_2)$.  We will derive Theorem~\ref{thm::whole_plane_reversibility} as a consequence of the following analog of the result from \cite[Section~4]{MS_IMAG2}.  (Note that the first half is just a restatement of Theorem~\ref{thm::conditional_law} and Theorem~\ref{thm::transience}.  The uniqueness statement in the final sentence is the new part.)

\begin{theorem}
\label{thm::unique_law_for_two_whole_plane_paths}
Suppose that $h$ is a whole-plane GFF, $\alpha > -\chi$, $\beta \in \R$, and let $h_{\alpha \beta} = h - \alpha \arg(\cdot) - \beta \log|\cdot|$, viewed as a distribution defined up to a global multiple of $2\pi(\chi+\alpha)$.  Fix $\theta \in (0,2\pi(1 + \alpha/\chi))$.  Let $\eta_1$ (resp.\ $\eta_2$) be the flow line of $h_{\alpha \beta}$ with angle $0$ (resp.\ $\theta$) starting from $0$.  Let $\mu_{\alpha \beta}$ be the joint law of the pair $(\eta_1, \eta_2)$ defined in this way.  The pair $(\eta_1,\eta_2)$ has the following properties:
\begin{enumerate}[(i)]
\item\label{unique::prop::lifting} Almost surely, $\eta_1$ and $\eta_2$ have liftings to the universal cover of $\C \setminus \{0 \}$ that are simple curves which do not cross each other (though, depending on $\alpha$ and $\theta$, they may intersect each other).  Moreover, almost surely neither curve traces the other (i.e., neither curve intersects the other for any entire open interval of time).
\item\label{unique::prop::transience} The $\eta_i$ are transient: $\lim_{t \to \infty} \eta_i(t) = \infty$ almost surely.
\item\label{unique::prop::conditional} Given $\eta_2$, the conditional law of the portion of $\eta_1$ intersecting each component $S$ of $\C \setminus \eta_2$ is given by an independent chordal $\SLE_\kappa(\rho_1;\rho_2)$ process in $S$ (starting at the first point of $\partial S$ hit by $\eta_2$ and ending at the last point of $\partial S$ hit by $\eta_2$) with
\begin{equation}
\label{eqn::conditional_rho_values}
 \rho_1 = \frac{\theta \chi}{\lambda} -2\quad\text{and}\quad \rho_2 = \frac{(2\pi - \theta) \chi + 2\pi \alpha}{\lambda}-2.
\end{equation}
A symmetric statement holds with the roles of $\eta_1$ and $\eta_2$ reversed.
\end{enumerate}
Every probability measure $\mu$ on path pairs $(\eta_1,\eta_2)$ satisfying~\eqref{unique::prop::lifting}--\eqref{unique::prop::conditional} can be expressed uniquely as
\begin{equation}
\label{eqn::ergodic_decomposition}
 \mu = \int_{\R} \mu_{\alpha \beta} d\nu(\beta)
\end{equation}
where $\nu$ is a probability measure on $\R$.
\end{theorem}

As we will explain in Section~\ref{subsec::whole_plane_reversibility_proof}, Theorem~\ref{thm::whole_plane_reversibility} is almost an immediate consequence of Theorem~\ref{thm::unique_law_for_two_whole_plane_paths}.  The goal of the rest of this section is to prove Theorem~\ref{thm::unique_law_for_two_whole_plane_paths}.  The most obvious approach would be as follows: imagine that we fix some given initial pair $(\eta_1, \eta_2)$.  Then ``resample'' $\eta_1$ from the conditional law given $\eta_2$.  Then ``resample'' $\eta_2$ from its conditional law given $\eta_1$.  Repeat this procedure indefinitely, and show that regardless of the initial values of $(\eta_1, \eta_2)$, the law of the pair of paths after $n$ resamplings ``mixes'', i.e., converges to some stationary distribution, as $n \to \infty$.

And indeed, the proof of the analogous result in \cite[Section~4]{MS_IMAG2} is based on this idea.  Through a series of arguments, it was shown to be sufficient to consider the mixing problem in a slightly modified context in which the endpoints of $\eta_1$ were near to but distinct from the endpoints of $\eta_2$.  The crucial step in solving the modified mixing problem was to show that if we consider an arbitrary initial pair $(\eta_1, \eta_2)$ and a distinct pair $(\wt \eta_1, \wt \eta_2)$, then we can always couple together the resampling procedures so that after a finite number of resamplings, there is a positive probability that the two pairs are the same.

In this section, we will extend the bi-chordal mixing arguments from \cite[Section~4]{MS_IMAG2} to a whole-plane setting.  In the whole-plane setting, we consider a pair $(\eta_1, \eta_2)$ of flow lines from $0$ to $\infty$ in $\C$ of different angles of $h_{\alpha \beta} = h -\alpha \arg(\cdot) - \beta \log|\cdot|$, viewed as a distribution defined up to a global multiple of $2\pi (\chi+\alpha)$, starting from the origin with given values of $\alpha > - \chi$ and $\beta \in \R$, as described in Theorem~\ref{thm::unique_law_for_two_whole_plane_paths}.  We will show that these paths are characterized (up to the $\beta$ term) by the nature of the conditional law of one of the $\eta_i$ given the other (in addition to some mild technical assumptions).  Since this form of the conditional law of one path given the other has time-reversal symmetry (i.e., is symmetric under a conformal inversion of $\C$ that swaps $0$ and $\infty$), this characterization will imply that the joint law of $(\eta_1, \eta_2)$ has time-reversal symmetry.   (The fact that this characterization implies time-reversal symmetry was already observed in \cite{MS_IMAG2}.)

The reader who has not done so will probably want to read (or at least look over) \cite[Section~4]{MS_IMAG2} before reading this section.  Some of the bi-chordal mixing arguments in \cite[Section~4]{MS_IMAG2} carry through to the current setting with little modification.  However, there are some topological complications arising from the fact that paths can wind around and hit themselves and each other in complicated ways.  Section~\ref{subsec::untangling} will derive the topological results needed to push through the arguments in \cite{MS_IMAG2}.  Also, \cite[Section~4]{MS_IMAG2} is able to reduce the problem to a situation in which the starting and ending points of $\eta_1$ and $\eta_2$ are distinct --- the reduction involves first fixing $\eta_1$ and $\eta_2$ up to some stopping time and then fixing a segment $\eta'$ of a counterflow line starting from the terminal point, so that $\eta_1$ and $\eta_2$ exit $D \setminus \eta'$ at different places.  This is a little more complicated in the current setting, because drawing a counterflow line from infinity may not be possible in the usual sense (if the angle $2\pi(1 + \alpha/\chi)$ is too small) and one has to keep track of the effect of the height changes after successive windings around the origin, due to the argument singularity.  Section~\ref{subsec::annulusmixing} will provide an alternative construction that works in this setting.

The constructions and figures in this section may seem complicated, but they should be understood as our attempt to give the simplest (or at least one reasonably simple) answer to the following question: ``What modifications are necessary and natural for extending the methods of \cite[Section~4]{MS_IMAG2} to the context of Theorem~\ref{thm::unique_law_for_two_whole_plane_paths}?''

\subsection{Untangling path ensembles in an annulus}
\label{subsec::untangling}

\begin{figure}[ht!]
\begin{center}
\includegraphics[scale = 0.85]{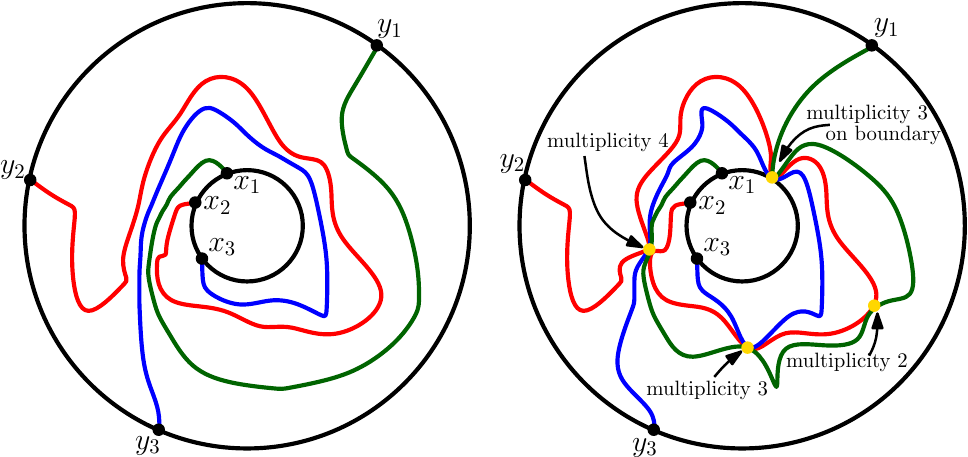}
\end{center}
\caption{\label{fig::1-tangle}
A $(3,-1,1)$-tangle (left) and a $(3,-1,4)$-tangle (right).  The $(3,-1,4)$-tangle has several points of higher multiplicity (shown as gold dots), including a point of multiplicity $4$ in the interior of the annulus $A$ and a point of multiplicity $3$ on the boundary $\partial A$.}
\end{figure}

In order to get the mixing argument to work, we need to show that the relevant space of paths is in some sense connected.  The difficulty which is present in the setting we consider here in contrast to that of \cite[Section~4]{MS_IMAG2} is the paths may wrap around the origin and intersect themselves and each other many times.

Fix a closed annulus $A$ with distinct points $x_1, \ldots, x_k$ (ordered counterclockwise) in the interior annulus boundary and distinct points $y_1, \ldots, y_k$ (ordered counterclockwise) on the exterior annulus boundary and an integer~$\ell$.  Define an {\bf $(k,\ell,m)$-tangle} (as illustrated in Figure~\ref{fig::1-tangle}) to be a collection of~$k$ distinct continuous curves $\gamma_1, \ldots, \gamma_k$ in $A$ such that
\begin{enumerate}
\item Each $\gamma_i \colon [0,1] \to A$ starts at $x_i$ and ends at $y_i$.
\item The lifting of each $\gamma_i$ to the universal cover of $A$ is a simple curve (i.e., a continuous path that does not intersect itself).
\item The (necessarily closed) set of times $t$ for which $\gamma_i$ intersects another curve (or intersects its own past/future) is a set with empty interior and no $\gamma_i$ {\em crosses} itself or any distinct $\gamma_j$.
\item The lifting of $\gamma_1$ to the universal cover of $A$ winds around the interior annulus boundary a net $\ell$ times (rounded down to the nearest integer).  This fixes a homotopy class for $\gamma_1$ (and by extension for all of the $\gamma_j$).
\item Every $a \in A$ has {\em multiplicity} at most $m$, where the {\bf multiplicity} of $a$ is the number of pairs $(i,t)$ for which $\gamma_i(t) = a$.
\item Each $a \in \partial A$ has multiplicity at most $m-1$ and no path hits one of the endpoints of another path.
\end{enumerate}

Let $G_{k,\ell,m}$ be the graph whose vertex set is the set of all $(k,\ell,m)$-tangles and in which two $(k,\ell,m)$-tangles $(\gamma_1, \ldots, \gamma_k)$ and $(\eta_1, \ldots, \eta_k)$ are adjacent if and only if for some $i$ we have
\begin{enumerate}
\item $\gamma_j = \eta_j$ (up to monotone reparameterization) for all $j \not = i$
\item  There is one component $C$ of $$A \setminus \cup_{j \not = i} \gamma_j([0,1])$$ and a pair of times $s,t \in [0,1]$ (after monotone reparameterization if necessary) such that $\gamma_i(s) = \eta_i(s) \in C$, $\gamma_i(t) = \eta_i(t) \in C$, and $\gamma_i$ and $\eta_i$ agree outside of $(s,t)$.
\end{enumerate}

\begin{lemma}
\label{lem::gmconnected}
The graph $G_{k,\ell,m}$ is connected.  In other words, we can get from any $(k,\ell,m)$-tangle $(\gamma_1, \ldots, \gamma_k)$ to any other $(k,\ell,m)$-tangle $(\eta_1, \ldots, \eta_k)$ by finitely many steps in~$G_{k,\ell,m}$.
\end{lemma}
\begin{proof}
We will first argue that regardless of the value of $m$, we can get from any $(k,\ell,m)$-tangle to some $(k,\ell,1)$-tangle in finitely many steps.  We refer to the connected components of $A \setminus \cup_{j=1}^k \gamma_j([0,1])$ as {\bf pockets} and observe that for topological reasons the boundary of each pocket is comprised of at most two path segments (and perhaps parts of the boundary of $A$).  The (at most two) {\em endpoints} of the pocket are those points common to these two segments.  Boundary points of a pocket that are not endpoints are called {\bf interior boundary points} of the pocket.

\begin{figure}[ht!]
\begin{center}
\includegraphics[scale = 0.85]{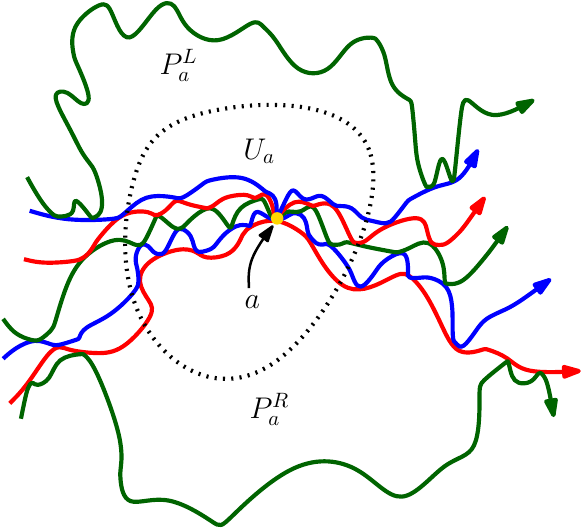}
\end{center}
\caption{\label{fig::ua}
A portion of an $m$-tangle (with $m=5$) is shown above.  The interior point $a$ has multiplicity five.  The neighborhood $U_a$ (boundary shown with dotted lines) contains five intersecting directed path segments, all of which pass through $a$.  The pocket left of this bundle of path segments is denoted $P^L_a$, while the pocket right of this bundle of path segments is denoted $P^R_a$.  (Note that in general it is possible to have $P^L_a = P^R_a$.)  There may be infinitely many pockets contained in $U_a$.  Indeed, it is possible that there are infinitely many small pockets in every neighborhood of $a$.  However, $P^L_a$ and $P^R_a$ are the only pockets that lie either left of all or right of all $5$ path segments through $U_a$ that pass through $a$.  This implies that $P^L_a$ and $P^R_a$ are the only pockets that have interior boundary points (within $U_a$) of multiplicity $m=5$.}
\end{figure}

For every multiplicity $m$ point $a$, we can find a small neighborhood $U_a$ of $a$ whose pre-image in $\{1,2,\ldots,k\} \times [0,1]$ consists of $m$ disjoint open intervals.  The images of these segments in $A$ are simple path segments that do not cross each other, and that all come together at $a$.  The leftmost such segment is part of the right boundary of a single pocket $P_a^L$ and the rightmost such segment is part of the left boundary of a single pocket $P_a^R$, as illustrated in Figure~\ref{fig::ua}.

We now claim that at most finitely many pockets have a multiplicity-$m$ boundary point that is {\em not} one of the two endpoints of the pocket.  If there were infinitely many such pockets, then (by compactness of $A$) we could find a sequence $a_1, a_2, \ldots$ of corresponding multiplicity $m$ points (each a non-endpoint boundary of a {\em different} pocket) converging to a point $a \in A$.  Clearly $a$ has multiplicity at least $m$ so it must be an interior point in $A$ and we can construct the neighborhood $U_a$ containing $a$ as described above; but the only two pockets within this neighborhood that can have interior multiplicity-$m$ boundary points are the pockets $P_a^L$ and $P_a^R$, as illustrated in Figure~\ref{fig::ua}.  Since only two pockets intersecting $U_a$ have multiplicity-$m$ interior boundary points, there cannot be an infinite sequence of such pockets intersecting $U_a$ and converging to $a$, so we have established a contradiction and verified the claim.

Now, within each pocket, we can move the left or right boundary away (getting rid of all multiplicity $m$ points on its boundary) in single step without introducing any new multiplicity $m$ points.  Repeating this for each of the pockets with multiplicity $m$ boundary points allows us to remove all multiplicity $m$ points in finitely many steps.

A similar procedure allows us to remove any multiplicity $m-1$ boundary points from the boundary of $A$.  (A multiplicity $m-1$ boundary point on $\partial A$ is essentially a multiplicity $m$ boundary point if we interpret a boundary arc of $\partial A$ as one of the paths.  Every such point is on the interior boundary of exactly one pocket and the same argument as above shows that there are only finitely pockets with interior.)

We have now shown that we can move from any $(k,\ell,m)$-tangle to an $(k,\ell,m-1)$-tangle in finitely many steps.  Repeating this procedure allows us to get to a $(k,\ell,1)$-tangle in finitely many steps.  In a $(k,\ell,1)$-tangle all of the $k$ paths are disjoint and they intersect $\partial A$ only at their endpoints.  For the remainder of the proof, it suffices to show that one can get from any $(k,\ell,1)$-tangle $(\gamma_1, \ldots, \gamma_k)$ to any other $(k,\ell,1)$-tangle $(\eta_1, \ldots, \eta_k)$ with finitely many steps in $G_{k,\ell,1}$.  It is easy to see that one can continuously deform $\gamma_1$ to $\eta_1$ since the two are homotopically equivalent.  Similarly one can continuously deform all of $A$ in such a way that $\gamma_1$ gets deformed to $\eta_1$ and the other paths get mapped to continuous paths.  One can then fix the first path and deform the domain to take the second path to $\eta_2$, and so forth.  Ultimately we obtain a continuous deformation of $(\gamma_1, \ldots, \gamma_k)$ to $(\eta_1, \ldots, \eta_k)$ within the space of $(k,\ell,1)$-tangles.  By compactness, the minimal distance that one path gets from another during this deformation is greater than some $\delta>0$. We can now write the continuous deformation as a finite sequence of steps such that each path moves by at most $\delta/2$ (in Hausdorff distance) during each step. Thus we can move the paths one at a time through these steps without their interfering with each other.  Repeating this for each step, we can get from $(\gamma_1, \ldots, \gamma_k)$ to $(\eta_1, \ldots, \eta_k)$ with finitely many moves in $G_{k,\ell,1}$.
\end{proof}

\subsection{Bi-chordal annulus mixing}
\label{subsec::annulusmixing}

\begin{figure}[ht!]
\begin{center}
\includegraphics[scale=.85]{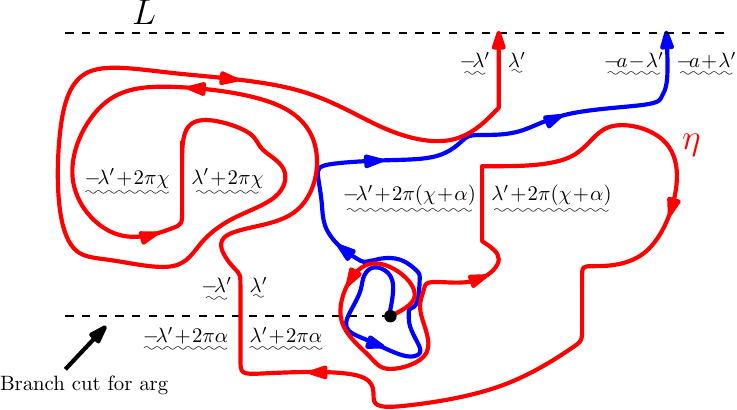}
\end{center}
\caption{\label{fig::interior_path_pair}
Suppose that $h_{\alpha \beta} = h - \alpha \arg(\cdot) - \beta \log|\cdot|$, viewed as a distribution defined up to a global multiple of $2\pi(\chi+\alpha)$, where $h$ is a whole-plane GFF.  The blue and red paths are flow lines of $h_{\alpha\beta}$ starting from the origin with angles $0$ and $a/\chi$, respectively (recall Figure~\ref{fig::interior_path_bd2}).}
\end{figure}

\begin{figure}[ht!]
\begin{center}
\includegraphics[scale=.85]{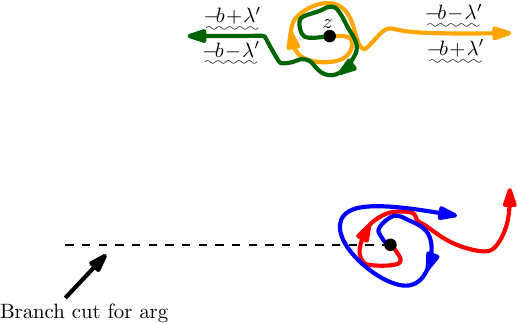}
\end{center}
\caption{\label{fig::interior_pair_with_z}
The purpose of this figure is to motivate the construction for the bi-chordal mixing argument used to prove Theorem~\ref{thm::unique_law_for_two_whole_plane_paths}.  Red and blue paths are as in Figure~\ref{fig::interior_path_pair}, except that these paths are stopped at some positive and finite stopping time.  A second pair of orange and green flow lines is drawn starting from a far away point $z$: the initial angle of the green path is chosen uniformly at random from $[0,2\pi)$, while the angle of the second is opposite that of the first.  Because of the randomness in the green/orange angles, the conditional law of the green/orange pair, given $h_{\alpha\beta}$, depends only on the choice of $h_{\alpha\beta}$ {\em modulo additive constant} (recall Remark~\ref{rem::alpha_beta_determined}).  If we draw the green/orange paths out to $\infty$, then the conditional law of $h_{\alpha \beta}$ in the complementary connected component containing origin is that of a GFF with $\alpha$-flow line boundary conditions.  The first step in the proof of Theorem~\ref{thm::unique_law_for_two_whole_plane_paths} is to construct paths $\gamma_1,\gamma_2$ that play the roles of the green/orange pair, though this will be done indirectly since there is not an ambient GFF defined on all of $\C$ in the setting of the second part of Theorem~\ref{thm::unique_law_for_two_whole_plane_paths}.}
\end{figure}

In this section, we will complete the proof of Theorem~\ref{thm::unique_law_for_two_whole_plane_paths}.  Throughout, we shall assume that $\eta_1,\eta_2$ are paths satisfying~\eqref{unique::prop::lifting}--\eqref{unique::prop::conditional} of Theorem~\ref{thm::unique_law_for_two_whole_plane_paths}.  In order to get the argument of \cite[Section~4]{MS_IMAG2} to work, we will need to condition on an initial segment of each of $\eta_1$ and $\eta_2$ as well as some additional paths which start far from~$0$.  This conditioning serves to separate the initial and terminal points of $\eta_1$ and $\eta_2$.  The idea is to imagine that $\eta_1,\eta_2$ are coupled with an ambient GFF $h$ on $\C$.  We then pick some point $z$ which is far away from zero and draw flow lines $\gamma_1$, $\gamma_2$ of $h$ starting from $z$ where the angle of $\gamma_1$ is chosen uniformly from $[0,2\pi)$ (recall Remark~\ref{rem::alpha_beta_determined}) and $\gamma_2$ points in the opposite direction of $\gamma_1$, i.e.\ the angle of $\gamma_2$ is $\pi$ relative to that of $\gamma_1$ (see Figure~\ref{fig::interior_pair_with_z}).  Conditioning on $\gamma_1$ and $\gamma_2$ as well as initial segments of $\eta_1$ and $\eta_2$ then puts us into the setting of Section~\ref{subsec::untangling}.

We cannot carry this out directly because \emph{a priori} (assuming only the setup of the second part of Theorem~\ref{thm::unique_law_for_two_whole_plane_paths}) we do not have a coupling of $\eta_1,\eta_2$ with a GFF on $\C$.  We circumvent this difficulty as follows.  Conditioned on the pair $\eta = (\eta_1,\eta_2)$, we let $h$ be an instance of the GFF on $\C \setminus (\eta_1 \cup \eta_2)$.  We let $h$ have $\alpha$-flow line boundary conditions on $\eta_1$ and $\eta_2$ where the value of $\alpha$ and the angles on each are determined so that if $\eta_1,\eta_2$ \emph{were} flow lines of a GFF with an $\alpha \arg$ singularity and these angles then they would have the same resampling property, as in the statement of Theorem~\ref{thm::unique_law_for_two_whole_plane_paths}.  (Note that $\rho_1,\rho_2$ as in~\eqref{eqn::conditional_rho_values} determine $\alpha$ and $\theta$.)

We then use the GFF $h$ to determine the law of $\gamma = (\gamma_1,\gamma_2)$, at least until one hits $\eta_1$ or $\eta_2$.  As in Figure~\ref{fig::interior_pair_with_z_mapped_over}, we would ultimately like to continue $\gamma_1$ and $\gamma_2$ all the way to $\infty$.  We accomplish this by following the rule that when one of the $\gamma_i$ hits one of the $\eta_j$, it either reflects off $\eta_j$ or immediately crosses $\eta_j$, depending on what it would do if $\eta_j$ \emph{were} in fact a GFF flow line hit at the same angle (as described in Theorem~\ref{thm::flow_line_interaction}).  To describe the construction more precisely, we will first need the following lemma which serves to rule out the possibility that $\gamma_i$ hits one of the $\eta_j$ at a self-intersection point of $\eta_j$ or at a point in $\eta_1 \cap \eta_2$.  

\begin{lemma}
\label{lem::pocket_harmonic_measure_intersection}
Suppose that $P$ is a pocket of $\C \setminus (\eta_1 \cup \eta_2)$.  Then the harmonic measure of each of the sets
\begin{enumerate}[(i)]
\item\label{item::self_intersection} the self-intersection of points of each $\eta_j$,
\item\label{item::intersection} $\eta_1 \cap \eta_2$
\end{enumerate}
as seen from any point in $P$ is almost surely zero.
\end{lemma}
\begin{proof}
This follows from the resampling property and Lemma~\ref{lem::sle_kappa_rho_boundary_intersection}.  Indeed, fix $z \in \C \setminus \{0\}$ and let $P$ be the pocket of $\C \setminus (\eta_1 \cup \eta_2)$ which contains $z$.  Then it suffices to show that the statement of the lemma holds for $P$ since every pocket contains a point with rational coordinates.  Let $P_1$ be pocket of $\C \setminus \eta_1$ which contains $z$.  Then the resampling property implies that the conditional law of the segment of $\eta_2$ which traverses $P_1$ is that of a chordal $\SLE_\kappa(\rho^L;\rho^R)$ process with $\rho^L,\rho^R > -2$.  Lemma~\ref{lem::sle_kappa_rho_boundary_intersection} thus implies that the harmonic measure of the points in $\partial P$ as described in~\eqref{item::self_intersection} and~\eqref{item::intersection} which are contained in the segment traced by $\eta_2$ is almost surely zero as seen from any point in $P$ because these points are in particular contained in the intersection of the restriction of $\eta_2$ to $\ol{P}_1$ with $\partial P_1$.  Swapping the roles of $\eta_1$ and $\eta_2$ thus implies the lemma.
\end{proof}

We will now give a precise construction of $\gamma = (\gamma_1,\gamma_2)$.  Recall that $h$ is a GFF on $\C \setminus (\eta_1 \cup \eta_2)$ with $\alpha$-flow line boundary conditions with angles and the value of $\alpha$ determined by the resampling property of $\eta_1,\eta_2$.  For each $i=1,2$ we inductively define path segments $\gamma_{i,k}$ and stopping times $\tau_{i,k}$ as follows.  We let $\gamma_{i,1}$ be the flow line of $h$ starting from $z$ with uniformly chosen angle in $[0,2\pi(1+\alpha/\chi))$.  Let $\tau_{i,1}$ be the first time that $\gamma_{i,1}$ hits $\eta_1 \cup \eta_2$ with a height difference such that $\gamma_{i,1}$ would cross either $\eta_1$ or $\eta_2$ at $\gamma_{i,1}(\tau_{i,1})$ if the boundary segment \emph{were} a GFF flow line (recall Theorem~\ref{thm::flow_line_interaction}).  If there is no such time, then we take $\tau_{i,1} = \infty$.  On the event $\{\tau_{i,1} < \infty\}$, Lemma~\ref{lem::pocket_harmonic_measure_intersection} and Lemma~\ref{lem::exit_harmonic_measure} imply that $\gamma_{i,1}(\tau_{i,1})$ is almost surely not a self-intersection point of one of the $\eta_j$'s and is not in $\eta_1 \cap \eta_2$.  Suppose that $\gamma_{i,1},\ldots,\gamma_{i,k}$ and $\tau_{i,1},\ldots,\tau_{i,k}$ have been defined.  Assume that we are working on the event that the latter are all finite and $\gamma_{i,k}(\tau_{i,k})$ is not a self-intersection point of one of the $\eta_j$'s and is not in $\eta_1 \cap \eta_2$.  Then $\gamma_{i,k}(\tau_{i,k})$ is almost surely contained in the boundary of precisely two pockets of $\C \setminus (\eta_1 \cup \eta_2)$, say $P$ and $Q$ and the range of $\gamma_{i,k}$ just before time $\tau_{i,k}$ is contained in one of the pockets, say $P$.  We then let $\gamma_{i,k+1}$ be the flow line of $h$ in $Q$ starting from $\gamma_{i,k}(\tau_{i,k})$ with the angle determined by the intersection of $\gamma_{i,k}$ with $\eta_1 \cup \eta_2$ at time $\tau_{i,k}$.  We also let $\tau_{i,k+1}$ be the first time that $\gamma_{i,k+1}$ intersects $\eta_1 \cup \eta_2$ at a height at which it can cross.  If $\gamma_{i,k+1}$ does not intersect $\eta_1 \cup \eta_2$ with such a height difference, we take $\tau_{i,k+1} = \infty$.  On $\{\tau_{i,k+1} < \infty\}$, Lemma~\ref{lem::pocket_harmonic_measure_intersection} and Lemma~\ref{lem::exit_harmonic_measure} imply that $\gamma_{i,k+1}(\tau_{i,k+1})$ is almost surely not contained in either a self-intersection point of one of the $\eta_j$'s or $\eta_1 \cap \eta_2$.

Let $k_0$ be the first index $k$ such that $\tau_{i,k_0} = \infty$ and let $P$ be the connected component of $\C \setminus (\eta_1 \cup \eta_2)$ whose closure contains $\eta_{i,k_0}$.  Let $x$ (resp.\ $y$) be the first (resp.\ last) point of $\partial P$ drawn by $\eta_1,\eta_2$ as they trace $\partial P$.  Then $\gamma_{i,k_0}$ terminates in $\ol{P}$ at $y$.  We then take $\gamma_{i,k_0+1}$ to be the concatenation of the flow lines of $h$ with the appropriate angle in the pockets of $\C \setminus (\eta_1 \cup \eta_2)$ which lie after $P$ in their natural ordering; we will show in Lemma~\ref{lem::gamma_behavior} that this yields an almost surely continuous path.  Finally, we let $\gamma$ be the concatenation of $\gamma_{i,1},\ldots,\gamma_{i,k_0+1}$.  As in the case of flow lines defined using a GFF on all of $\C$, it is not hard to see that each $\gamma_i$ almost surely crosses each $\eta_j$ at most finitely many times (recall Theorem~\ref{thm::merge_cross}):

\begin{lemma}
\label{lem::gamma_behavior}
Each $\gamma_i$ as defined above is an almost surely continuous path which crosses each $\eta_j$ at most finitely many times after which it visits the connected components of $\C \setminus (\eta_1 \cup \eta_2)$ according to their natural order, i.e.\ the order that their boundaries are drawn by $\eta_1$ and $\eta_2$.  Similarly, each $\gamma_i$ crosses any given flow line of $h$ at most a finite number of times.
\end{lemma}
\begin{proof}
The assertion regarding the number of times that the $\gamma_i$ cross the $\eta_j$ or any given flow line of $h$ is immediate from the construction and the same argument used to prove Theorem~\ref{thm::merge_cross}.  To see that $\gamma_i$ is continuous, we note that we can write $\gamma_i$ as a local uniform limit of curves as follows.  Fix $T > 0$.  For each $n \in \N$, we note that there are only a finite number of bounded connected components of $\C \setminus (\eta_1([0,T]) \cup \eta_2([0,T]))$ whose diameter is at least $\tfrac{1}{n}$ by the continuity of $\eta_1,\eta_2$.  We thus let $\gamma_{i,n,T}$ be the concatenation of $\gamma_{i,1},\ldots,\gamma_{i,k_0}$ along with the segments of $\gamma_{i,k_0+1}$ which traverse bounded pockets of $\C \setminus (\eta_1([0,T]) \cup \eta_2([0,T]))$ whose diameter is at least $\tfrac{1}{n}$ and the segments which traverse pockets of diameter less than $\tfrac{1}{n}$ are replaced by the part of $\eta_1$ which traces the side of the pocket that it visits first.  Then $\gamma_{i,n,T}$ is a continuous path since it can be thought of as concatenating a finite collection of continuous paths with the path which arises by taking $\eta_1$ and then replacing a finite collection of disjoint time intervals $[s_1,t_1],\ldots,[s_k,t_k]$ with other continuous paths which connect $\eta_1(s_i)$ to $\eta_1(t_i)$.  Moreover, it is immediate from the definition that the sequence $(\gamma_{i,n,T} : n \in \N)$ is Cauchy in the space of continuous paths $[0,1] \to \C$ defined modulo reparameterization with respect to the $L^\infty$ metric.  Therefore $\gamma_i$ stopped upon entering the unbounded connected component of $\C \setminus (\eta_1([0,T]) \cup \eta_2([0,T]))$ is almost surely continuous.  Since this holds for each $T > 0$, this completes the proof in the case that $\C \setminus (\eta_1 \cup \eta_2)$ does not have an unbounded connected component.  If there is an unbounded component, then we have proved the continuity of~$\gamma_i$ up until it first enters such a component, say~$P$.  If $\gamma_i$ did not cross into $P$, then the continuity follows since its law in $P$ is given by that of an $\SLE_\kappa(\rho^L;\rho^R)$ process with $\rho^L,\rho^R > -2$.  The analysis is similar if $\gamma_i$ crossed into~$P$, which completes the proof.
\end{proof}

Once we have fixed one of the $\eta_i$, we can sample $\eta_j$ for $j \neq i$ by fixing a GFF $h_i$ on $\C \setminus \eta_i$ with $\alpha$-flow line boundary conditions on $\eta_i$ and then taking $\eta_j$ to be a flow line of $h_i$ starting from the origin with the value of $\alpha$ and the angle of $\eta_j$ determined by the resampling property of $\eta_1,\eta_2$.  (In the case that $\eta_i$ is self-intersecting, we take $\eta_j$ to be a concatenation of flow lines of $h_i$ starting at the pocket opening points with the appropriate angle.)  We let $\gamma^i = (\gamma_1^i,\gamma_2^i)$ be the pair constructed using the same rules to construct $\gamma$ described above using the GFF $h_i$ in place of $h$.  In the following lemma, we will show that $\gamma = \gamma^i$ almost surely.  This is useful for the mixing argument because it tells us that the conditional law of $\eta_j$ given \emph{both} $\eta_i$ and $\gamma$ can be described in terms of a GFF flow line.  We will keep the proof rather brief because it is similar to some of the arguments in Section~\ref{sec::interior_flowlines}.

\begin{lemma}
\label{lem::gamma_one_gff}
Fix $i \in \{1,2\}$ and assume that the GFFs $h$ and $h_i$ described above have been coupled together so that $h = h_i$ on $\C \setminus (\eta_1 \cup \eta_2)$.  Then $\gamma^i = \gamma$ almost surely.  In particular, the conditional law of $\eta_1$ given $\eta_2$, $\gamma$ and the heights of $h_2$ along $\gamma$ is given by a flow line of a GFF on $\C \setminus (\eta_2 \cup \gamma_1 \cup \gamma_2)$ whose boundary data agrees with that of $h_2$ conditional on $\eta_2,\gamma$.  A symmetric statement holds when the roles of $\eta_1$ and $\eta_2$ are swapped.
\end{lemma}
\begin{proof}
That $\gamma_k^i$ agrees with $\gamma_k$ until its first crossing with one of $\eta_1,\eta_2$ follows from Theorem~\ref{thm::uniqueness}.  In particular, $\gamma_k^i$ almost surely does not cross at a self-intersection point of either of the $\eta_j$'s or a point in $\eta_1 \cap \eta_2$.  Whenever $\gamma_k^i$ crosses into a new pocket of $\C \setminus (\eta_1 \cup \eta_2)$, it satisfies the same coupling rules with the GFF, so that the paths continue to agree follows from the uniqueness theory for boundary emanating GFF flow lines \cite[Theorem~1.2]{MS_IMAG}.  The same is likewise true once $\gamma_k^i$ (resp.\ $\gamma_k$) starts to follow the pockets of $\C \setminus (\eta_1 \cup \eta_2)$ in order, which completes the proof of the lemma.
\end{proof}

\begin{figure}[ht!!]
\begin{center}
\includegraphics[scale=.85]{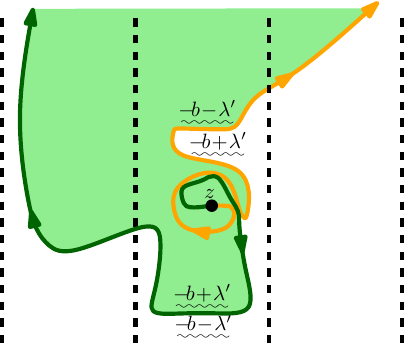}
\end{center}
\caption{\label{fig::lifted_pair_with_z_one_branch} A pair of opposite-going paths, as in Figure~\ref{fig::interior_pair_with_z}, lifted to the universal cover of $\C \setminus \{0\}$ and conformally mapped to $\C$ via the map $z \to i \log z$ (so that each copy of $\C$ maps to one vertical strip, with the origin mapping to the bottom of the strip and $\infty$ mapping to the top of the strip).  In this lifting, the paths cannot cross one another (though they may still touch one another as shown; a further lifting to the universal cover of the complement of $z$ would make the paths simple).  The region cut off from $-\infty$ by the pair of paths is shaded in light green.}
\end{figure}

\begin{figure}[ht!!]
\begin{center}
\includegraphics[scale=.85]{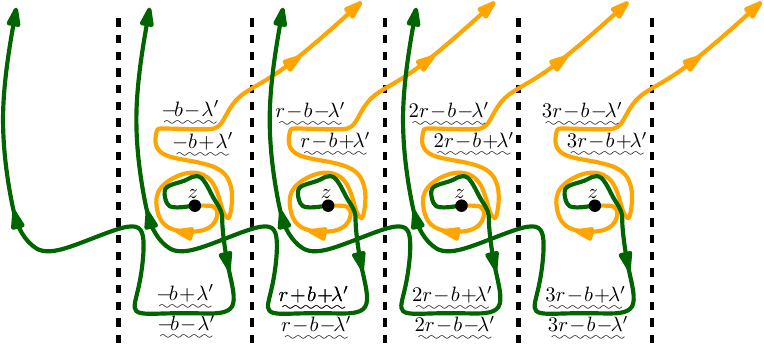}
\end{center}
\caption{\label{fig::lifted_pair_with_z}  This figure is the same as Figure~\ref{fig::lifted_pair_with_z_one_branch} except that all of possible the liftings of the path to the universal cover are shown.  We claim that the boundary of the complementary connected component of the paths which contains the origin (reached as an infinite limit in the down direction) can be expressed as a disjoint union consisting of at most one segment from each path.  Observe that if one follows the trajectory of a single (say green) path, once it crosses one of the green/yellow pairs, it never recrosses it.  We may consider the transformed image of $h$ (under the usual conformal coordinate transformation) to be a single-valued (generalized) function that increases by $r= 2\pi(\chi + \alpha)$ as one moves from one strip to the next.}
\end{figure}

\begin{figure}[ht!]
\begin{center}
\includegraphics[scale=.85]{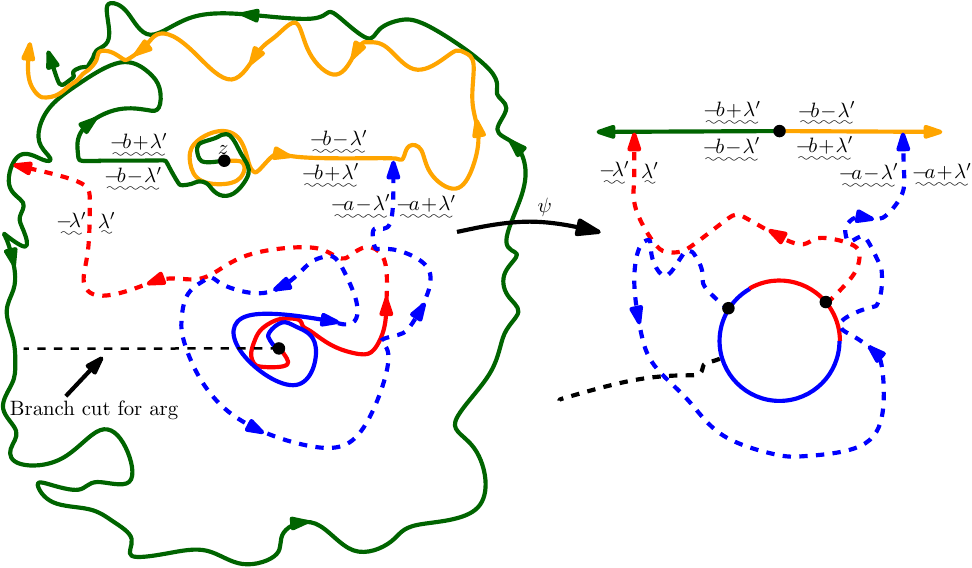}
\end{center}
\caption{\label{fig::interior_pair_with_z_mapped_over}
The left is the same as Figure~\ref{fig::interior_pair_with_z} except that the orange and green paths from $z$ (which we call $\gamma_1$ and $\gamma_2$) are continued to $\infty$.  The analysis of Figure~\ref{fig::lifted_pair_with_z} shows that $\C \setminus (\gamma_1 \cup \gamma_2)$ has a simply connected component containing $0$ and the boundary of this component has (one or) two arcs: a single side (left or right) of an arc of $\gamma_1$, and a single side of an arc of $\gamma_2$.   Take $|z|$ large and draw the solid red and blue curves up to some small stopping times before they hit the orange/green curves.  Let $\psi$ conformally map the annular region (the component of the complement of the four solid curves whose boundary intersects all four) to $-\h \setminus \ol D$, for some closed disk $\ol D$, in such a way that the orange and green boundary segments are mapped to complementary semi-infinite intervals of $\R$, both paths directed toward $\infty$.  In the figure shown, the dotted red (resp.\ blue) path may {\em cross} $\R$ where it intersects if and only if $|b| < \tfrac{\pi \chi}{2}$ (resp.\ $|a-b| < \tfrac{\pi \chi}{2}$).  If the dotted blue and red paths cannot cross upon hitting $\R$ as shown, they may wind around $\ol D$  one or more times (picking up multiples of $r = 2\pi(\chi + \alpha)$ as they go) before crossing.  A crossing after some number of windings is possible for the red (resp.\ blue) curve if and only if $-b+ r \Z$ (resp.\ $b- a+ r \Z$) contains a point in $(-\tfrac{\pi\chi}{2}, \tfrac{\pi \chi}{2})$.  Otherwise, the red (blue) curve reaches $\infty$ without crossing $\R$.}
\end{figure}

\begin{figure}[ht!]
\begin{center}
\includegraphics[scale=.85]{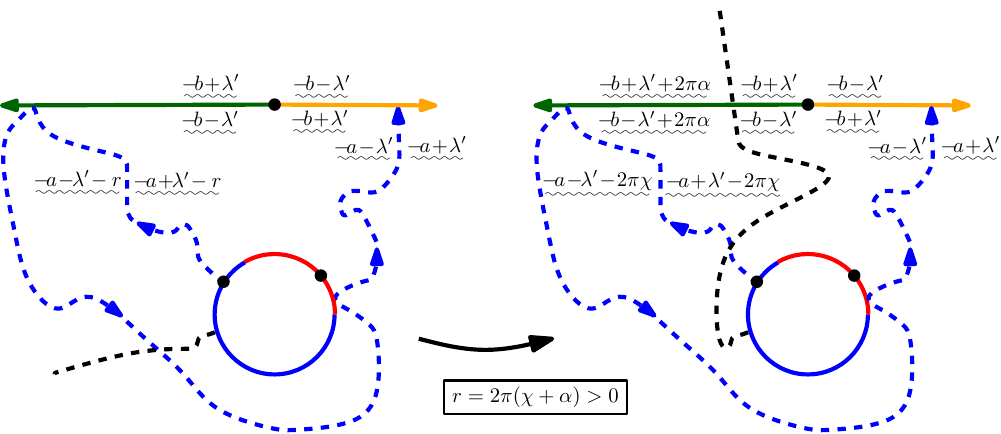}
\end{center}
\caption{\label{fig::interior_blue_reaching_endpoint}  The left figure is the same as Figure~\ref{fig::interior_pair_with_z_mapped_over}, but in the right figure we change the position of the branch cut for the argument (adjusting the heights accordingly so that the paths remain the same).  We see that the blue path accumulates at $\infty$ with positive probability (without merging into or crossing the green and orange lines, and without crossing the branch cut on the right an additional time) when $b-a > \tfrac{\pi}{2}\chi$ and $b-a-r < - \tfrac{\pi}{2}\chi$.  Thus, this can happen with positive probability after some number of windings if and only if $(b-a) + r \Z$ fails to intersect $[-\tfrac{\pi}{2}\chi, \tfrac{\pi}{2}\chi]$.  In fact, if $(b-a) + r \Z$ fails to intersect $[-\tfrac{\pi}{2}\chi, \tfrac{\pi}{2}\chi]$, then the path cannot cross/merge after any number of windings, so it must almost surely accumulate at $\infty$.  Conversely, if $(b-a)+r \Z$ {\em does} intersect $[-\tfrac{\pi}{2}\chi, \tfrac{\pi}{2}\chi]$, then the blue path almost surely merges into or crosses $\R$ after some (not necessarily deterministic) number of windings around $D$.}
\end{figure}

\begin{lemma}
\label{lem::gamma_component_containing_zero}
Let $U$ be the connected component of $\C \setminus (\gamma_1 \cup \gamma_2)$ which contains~$0$.  Then
\begin{enumerate}[(i)]
\item $U$ can be expressed as a (possibly degenerate) disjoint union of one segment of $\gamma_1$ and one segment of $\gamma_2$; see Figure~\ref{fig::interior_pair_with_z_mapped_over}.
\item Almost surely, $\dist(\partial U,0) \to \infty$ as $|z| \to \infty$ (where $z$ is the starting point of $\gamma_1,\gamma_2$).
\end{enumerate}
\end{lemma}
\begin{proof}
The first assertion follows from Lemma~\ref{lem::gamma_behavior} as well as the argument described in Figure~\ref{fig::lifted_pair_with_z_one_branch} and Figure~\ref{fig::lifted_pair_with_z}.  To see the second assertion of the lemma, we first condition on $\eta_1$ and then consider two possibilities.  Either:
\begin{enumerate}
\item The range of $\eta_1$ contains self-intersection points with arbitrarily large modulus.
\item The range of $\eta_1$ does not contain self-intersection points with arbitrarily large modulus.
\end{enumerate}

In the former case, the transience of $\eta_1$ implies that for every $r > 0$ there exists $R > r$ such that if $|z| \geq R$ then the distance of the connected component $P_z$ of $\C \setminus \eta_1$ containing $z$ to $0$ is at least $r$.  By increasing $R > 0$ if necessary the same is also true for all of the pockets of $\C \setminus \eta_1$ which come after $P_z$ in the order given by the order in which $\eta_1$ traces part of the boundary of such a pocket.  Say that two pockets $P,Q$ of $\C \setminus \eta_1$ are adjacent if the intersection of their boundaries contains the image under $\eta_1$ of a non-trivial interval.  Fix $k \in \N$ and let $P_z^k$ denote the union of the pockets of $\C \setminus \eta_1$ that can be reached from $P_z$ by jumping to adjacent pockets at most $k$ times.  By the same argument, there exists $R > r$ such that if $|z| > R$ then $\dist(P_z^k,0) \geq r$.  The same likewise holds for the pockets which come after those which make up $P_z^k$.  Combining this with the first assertion of the lemma implies the second assertion in this case.

The argument for the latter case is similar to the proof of \cite[Proposition~7.33]{MS_IMAG}.  We let $V$ be the unbounded connected component of $\C \setminus \eta_1$ and $\psi \colon V \to \h$ be a conformal map which fixes $\infty$ and sends the final self-intersection point of $\eta_1$ to $0$.  We let $\wt{h}_1 = h_1 \circ \psi^{-1} - \chi \arg(\psi^{-1})'$ where $h_1$ is the GFF on $\C \setminus \eta_1$ used to define $\eta_2$ and $\gamma = \gamma^1$ as in Lemma~\ref{lem::gamma_behavior}.  We note that the boundary conditions of $\wt{h}_1$ are piecewise constant, changing only once at $0$.  Let $\wt{\gamma} = \psi(\gamma)$.  Then it suffices to show that the diameter of the connected component $\wt{U}$ of $\h \setminus \wt{\gamma}$ which contains $0$ becomes unbounded as $\wt{z} = \psi(z)$ tends to $\infty$ in $\h$.  To see this, we let $\wt{\eta}_1,\ldots,\wt{\eta}_{k_0}$ be flow lines of $\wt{h}$ starting from $0$ with equally spaced angles such that $\wt{\eta}_j$ almost surely intersects both $\wt{\eta}_{j-1}$ and $\wt{\eta}_{j+1}$ for each $1 \leq j \leq k_0$.  Here, we take $\wt{\eta}_0 = (-\infty,0)$ and $\wt{\eta}_{k_0+1} = (0,\infty)$.  By Lemma~\ref{lem::gamma_behavior} there exists $m_0$ such that $\wt{\gamma}$ can cross each of the $\wt{\eta}_j$ at most $m_0$ times.  Let $n_0 = k_0 m_0$.

Say that two connected components $\wt{P}, \wt{Q}$ of $\h \setminus \cup_{j=1}^{k_0} \wt{\eta}_j$ are adjacent if the intersection of the boundaries of $\psi^{-1}(\wt{P})$ and $\psi^{-1}(\wt{Q})$ contains the image of a non-trivial interval of $\psi^{-1}(\wt{\eta}_j)$ for some $0 \leq j \leq k+1$.  We also say that $\wt{Q}$ comes after $\wt{P}$ if there exists $0 \leq j \leq k_0$ such that both $\wt{P}$ and $\wt{Q}$ lie between $\wt{\eta}_j$ and $\wt{\eta}_{j+1}$ and the boundary of $\wt{Q}$ is traced by $\wt{\eta}_j$ after it traces the boundary of $\wt{P}$.  Let $\wt{P}_{\wt{z}}$ be the connected component of $\h \setminus \cup_{j=1}^{k_0} \wt{\eta}_j$ which contains $\wt{z}$ and let $\wt{\CP}_{\wt{z}}$ be the closure of the union of the connected components which can be reached in at most $n_0$ steps starting from $\wt{P}_{\wt{z}}$ or comes after such a component.  By Lemma~\ref{lem::gamma_behavior}, $\wt{\gamma} \subseteq \wt{\CP}_{\wt{z}}$, so it suffices to show that
\begin{equation}
\label{eqn::pockets_to_infinity}
\dist(\wt{\CP}_{\wt{z}},0) \to \infty \quad\text{as}\quad|\wt{z}| \to \infty\quad\text{almost surely}.
\end{equation}
Since each $\wt{\eta}_j$ is almost surely transient as a chordal $\SLE_\kappa(\rho^L;\rho^R)$ process in $\h$ from $0$ to $\infty$ with $\rho^L,\rho^R \in (-2,\tfrac{\kappa}{2}-2)$ (where the weights depend on $j$) and almost surely has intersections with both of its neighbors with arbitrarily large modulus, it follows that \begin{equation}
\dist(\wt{P}_{\wt{z}},0) \to \infty\quad\text{as}\quad |\wt{z}| \to \infty\quad\text{almost surely.}
\end{equation}
The same is likewise true for all of the pockets which can be reached from $\wt{P}_{\wt{z}}$ in at most $n_0$ steps as well as for the pockets which come after these.  This proves~\eqref{eqn::pockets_to_infinity}, hence the second assertion of the lemma.
\end{proof}

By Lemma~\ref{lem::gamma_one_gff}, we know how to resample $\eta_1$ given $(\gamma,\eta_2)$.  Similarly, we know how to resample $\eta_2$ given $(\gamma,\eta_1)$.  Indeed, in each case $\eta_i$ is given by a flow line of a GFF on $\C \setminus (\gamma_1 \cup \gamma_2 \cup \eta_j)$ for $j \neq i$.  We will now argue that given $\gamma$ as well as initial segments of $\eta_1$ and $\eta_2$, the conditional law of $\eta_1,\eta_2$ until crossing $\gamma$ is uniquely determined by the resampling property and can be described in terms of GFF flow lines (see Figure~\ref{fig::interior_pair_with_z_mapped_over}).

\begin{center}
\begin{tabular} {|l|l|}
\hline
{\bf Range of values for $b-a$} & {\bf Behavior of blue path in Figure~\ref{fig::interior_pair_with_z_mapped_over}} \\ & {\bf \hspace{.2in} if it hits $\R$ as shown} \\
\hline
 $b-a \leq -\tfrac{\pi}{2}\chi - 2\lambda'$ & Cannot hit $\R$ (without going around the disk). \\
\hline
 $b-a \in (- \tfrac{\pi}{2}\chi-2\lambda', -\tfrac{\pi}{2}\chi)$ & Can hit green only, reflects left afterward. \\
\hline
 $b-a = -\tfrac{\pi}{2}\chi$ & Can hit green only, merges with green. \\
\hline
 $b-a \in (-\tfrac{\pi}{2}\chi, \tfrac{\pi}{2}\chi)$ & Can hit either color, crosses $\R$ afterward.\\
\hline
$b-a = \tfrac{\pi}{2}\chi$ & Can hit orange only, merges with orange. \\
\hline
$b-a \in (\tfrac{\pi}{2}\chi, \tfrac{\pi}{2}\chi+2\lambda')$ & Can hit orange only, reflects right. \\
\hline
$b-a \geq \tfrac{\pi}{2}\chi+2\lambda'$ & Cannot hit $\R$ (without going around the disk). \\
\hline
\end{tabular}
\captionof{table}{\label{tab::paths_reach_zero}}
\end{center}

\begin{lemma}
\label{lem::conditional_law_given_gamma_and_stubs}
Suppose that $\tau_i$ for $i=1,2$ is an almost surely positive and finite stopping time for $\eta_i$.  Let $E$ be the event that $\cup_{i=1}^2 \eta_i([0,\tau_i])$ is contained in the connected component $U$ of $\C \setminus (\gamma_1 \cup \gamma_2)$ which contains $0$.  On $E$, let $A$ be the connected component of $U \setminus \cup_{i=1}^2 \eta_i([0,\tau_i])$ whose boundary intersects $\gamma$.  Then the conditional law of $\eta_1,\eta_2$ stopped upon exiting $\ol{U}$ given $E$, $A$, $h|_{\partial A}$ (where $h$ is the GFF on $\C \setminus (\eta_1 \cup \eta_2)$ used to generated $\gamma$) is that of a pair of flow lines of a GFF on $A$ whose boundary behavior agrees with that of $h|_{\partial A}$ and with angles as implied by the resampling property for $\eta_1$ and $\eta_2$.
\end{lemma}
\begin{proof}
The proof is similar to that of \cite[Theorem~4.1]{MS_IMAG2}.
Suppose that $\wh{h}$ is a GFF on $A$ whose boundary conditions are as described in the statement of the lemma (given $E$) and let $\wh{\eta} = (\wh{\eta}_1,\wh{\eta}_2)$ be the flow lines of $\wh{h}$ starting from $\eta_i(\tau_i)$ with the same angles as (are implied for) $\eta_1,\eta_2$.  Then for $i,j=1,2$ and $j \neq i$, we know that the conditional law of $\wh{\eta}_i$ given $(\wh{\eta}_j,\gamma)$ and $\wh{h}|_{\partial A}$ for $j \neq i$ is the same as that of $\eta_i$ given $(\eta_j,\gamma)$ and $h|_{\partial A}$.  Moreover, $\wh{\eta}$ is homotopic to $\eta$ since the boundary conditions for $\wh{h}$ force the net winding of $\wh{\eta}_1,\wh{\eta}_2$ around the inner boundary of $A$ to be the same as as that of $\eta_1,\eta_2$ (where both pairs are stopped upon exiting $\ol{U}$).

The resampling property for $(\wh{\eta}_1,\wh{\eta}_2)$ implies that it is a stationary distribution for the following Markov chain. Its state space consists of pairs of continuous, non-crossing  paths $(\vartheta_1,\vartheta_2)$ in $A$ where $\vartheta_i$ connects $\eta_i(\tau_i)$ to $\partial U$ for $i=1,2$.  The transition kernel is given by first picking $i \in \{1,2\}$ uniformly and then resampling $\vartheta_i$ by:
\begin{enumerate}
\item Picking a GFF on $A \setminus \vartheta_j$ for $j \in \{1,2\}$ distinct from $i$ with boundary data agrees with $\wh{h}|_{\partial A}$ and has $\alpha$-flow line boundary conditions with the same angle as (is implied for) $\eta_j$ along $\vartheta_j$.
\item Taking the flow line starting from $\eta_i(\tau_i)$ with the same angle as (is implied for) $\eta_i$ stopped upon first exiting $\ol{U}$.
\end{enumerate}

As explained in \cite[Section~4]{MS_IMAG2}, any such ergodic measure $\nu$ which arises in the ergodic decomposition of either the law of $\eta$ or $\wh{\eta}$ must be supported on path pairs which are:
\begin{enumerate}
\item homotopic to $(\eta_1,\eta_2)$ in $A$
\item there exists $m < \infty$ (possibly random) such that the number of times a path hits any point in $\ol{A}$ is almost surely at most $m$.
\end{enumerate}
Indeed, as we mentioned earlier, $\wh{\eta}$ almost surely satisfies the first property due to the boundary data of $\wh{h}$.  The second property is satisfied for $\eta$ by transience and continuity.  The discussion in Section~\ref{subsec::critical_angle} implies that $\wh{\eta}$ also satisfies the second property.  To complete the proof it suffices to show that there is only one such ergodic measure.  Suppose that $\nu,\wt{\nu}$ are such ergodic measures and that $\vartheta = (\vartheta_1,\vartheta_2)$ (resp.\ $\wt{\vartheta} = (\wt{\vartheta}_1,\wt{\vartheta}_2)$) is distributed according to $\nu$ (resp.\ $\wt{\nu}$).  Then it suffices to show that we can construct a coupling $(\vartheta,\wt{\vartheta})$ such that $\p[\vartheta = \wt{\vartheta}] > 0$ since this implies that $\nu$ and $\wt{\nu}$ are not be mutually singular, hence equal by ergodicity.

As explained in Figure~\ref{fig::interior_pair_with_z_mapped_over} and Figure~\ref{fig::interior_blue_reaching_endpoint} as well as in Table~\ref{tab::paths_reach_zero}, it might be that the strands of $\vartheta$ (resp.\ $\wt{\vartheta}$) exit $\ol{U}$ the same point or distinct points, depending on the boundary data along $\gamma$.  Moreover, in the former case the strands exit at the point $y_0$ on $\partial U$ which is last drawn by the strands of $\gamma$ (the ``closing point'' of the pocket; in the right side of Figure~\ref{fig::interior_pair_with_z_mapped_over} this point corresponds to $\infty$ in $-\h$).  Thus by possibly drawing a segment of a counterflow line starting from $y_0$, we may assume without loss of generality that we are in the latter setting.  Indeed, this is similar to the trick used in \cite[Section~4]{MS_IMAG2}.

Recall that we can describe the conditional law of $\vartheta_i$ given $(\gamma,\vartheta_j)$ (and the boundary heights) in $A$ in terms of a flow line of a GFF on $A \setminus \vartheta_j$ and the same is likewise true with $\wt{\vartheta}_1,\wt{\vartheta}_2$ in place of $\vartheta_1,\vartheta_2$.  Thus by Lemma~\ref{lem::gmconnected}, Lemma~\ref{lem::path_close}, and Lemma~\ref{lem::path_close_hit}, it follows that we can couple $\vartheta$ and $\wt{\vartheta}$ together such that there exists a positive probability event $F$ on which each is a $(k,\ell,1)$-tangle in $A$ and $\wt{\vartheta}_i$ is much closer to $\vartheta_i$ for $i=1,2$ than $\vartheta_i$ is to $\vartheta_j$, $j \neq i$.    Thus by working on $F$ and first resampling $\vartheta_1,\wt{\vartheta}_1$ given $\vartheta_2,\wt{\vartheta}_2$, absolute continuity for the GFF implies that we can recouple the paths together so that $\vartheta_1 = \wt{\vartheta}_1$ with positive probability (see \cite[Lemma~4.2]{MS_IMAG2}).  On this event, the resampling property for $\vartheta_2$ (resp.\ $\wt{\vartheta}_2$) given $\vartheta_1$ (resp.\ $\wt{\vartheta}_1$) implies that we can couple the laws together so that $\wt{\vartheta} = \vartheta$ with positive conditional probability.  This proves the existence of the desired coupling, which completes the proof.
\end{proof}

We can now complete the proof of Theorem~\ref{thm::unique_law_for_two_whole_plane_paths}.

\begin{proof}[Proof of Theorem~\ref{thm::unique_law_for_two_whole_plane_paths}]
Fix $\epsilon > 0$ and let $\tau_i^\epsilon$ for $i=1,2$ be the first time that $\eta_i$ hits $\partial B(0,\epsilon)$.  Fix $R > 0$ very large and $z \in \C$ with $|z| \geq R$ sufficiently large so that (by Lemma~\ref{lem::gamma_component_containing_zero}) it is unlikely that the connected component $U$ of $\C \setminus (\gamma_1 \cup \gamma_2)$ containing $0$ intersects $B(0,R)$.  Let $h$ be the GFF on $\C \setminus (\eta_1 \cup \eta_2)$ used to generate $\gamma$.  Let $A$ and $E$ be as in Lemma~\ref{lem::conditional_law_given_gamma_and_stubs} where we take the stopping times for $\eta_i$ as above.  By Lemma~\ref{lem::conditional_law_given_gamma_and_stubs}, we know that the conditional law of $\eta_1,\eta_2$ given $\eta_i|_{[0,\tau_i^\epsilon]}$ for $i=1,2$, $\gamma$, $h|_{\partial A}$, and $E$ is described in terms of a pair of flow lines of a GFF $\wt{h}$ on $A$.  Let $\psi$ be a conformal transformation which takes $A$ to an annulus and let $\wh{h} = \wt{h} \circ \psi^{-1} - \chi \arg (\psi^{-1})'$.  Then we can write $\wh{h} = \wh{h}_0 - \alpha \arg(\cdot) + \wh{f}_0$ where $\wh{h}_0$ is a zero-boundary GFF and $\wh{f}_0$ is a harmonic function.  The boundary conditions for $\wh{f}_0$ are given by $a_\epsilon$ (resp.\ $b_R$) on the inner (resp.\ outer) boundary of the annulus, up to a bounded additive error which does not depend on $\epsilon > 0$ or $R > 0$.  (The error comes from $\chi$ times the winding of the two annulus boundaries, additive terms of $\pm \lambda'$ depending on whether the boundary segment is the image of the left or the right side of one of the $\eta_i$ or $\gamma_i$, and finally from the angles of the different segments.)  The value of $\alpha$ is determined by the resampling property for $\eta_1,\eta_2$.  In particular, away from the annulus boundary it is clear that $\wh{f}_0$ is well-approximated by an affine transformation of the $\log$ function.  Indeed, this follows because $\wh{f}_0$ is well-approximated by the function which is harmonic in the annulus with boundary values on the annulus boundaries given by the corresponding average of $\wh{f}_0$ and the functions which are harmonic in an annulus and take on a constant value on the inner and outer annulus boundaries (i.e., radially symmetric) are exactly the affine transformations of the $\log$ function.  Thus by sending $R \to \infty$ and $\epsilon \to 0$, we see that $\wh{f}_0$ converges to a multiple of the $\log$ function (modulo additive constant).  The measure $\nu$ in the statement of the theorem is exactly given by the law of the multiple of the $\log$ function.
\end{proof}

\subsection{Proof of Theorem~\ref{thm::whole_plane_reversibility}}
\label{subsec::whole_plane_reversibility_proof}
Fix $\kappa \in (0,4)$, $\alpha > -\chi$, and let $\rho = 2-\kappa+2\pi\alpha / \lambda$.  By adjusting the value of $\alpha$, we note that $\rho$ can take on any value in $(-2,\infty)$.
Let $h$ be a whole-plane GFF and $h_\alpha = h-\alpha \arg(\cdot)$, viewed as a distribution defined up to a global multiple of $2\pi(\chi+\alpha)$. By Theorem~\ref{thm::alphabeta}, the flow line $\eta$ of $h_\alpha$ starting from $0$ with zero angle is a whole-plane $\SLE_\kappa(\rho)$ process.  Let $\eta_1 = \eta$ and let $\eta_2$ be the flow line of $h_\alpha$ starting from $0$ with angle $\theta = \pi(1 + \tfrac{\alpha}{\chi})$.  Note that this choice of $\theta$ lies exactly in the middle of the available range.  For $i=1,2$, let $\CR(\eta_i)$ denote the time-reversal of $\eta_i$.  By Theorem~\ref{thm::unique_law_for_two_whole_plane_paths} and the main result of \cite{MS_IMAG2}, we know that the conditional law of $\eta_1$ given $\eta_2$ is the same as that of $\CR(\eta_1)$ given $\CR(\eta_2)$ and the same also holds when the roles of $\eta_1$ and $\eta_2$ are swapped.  Consequently, it follows that the joint law of the image of the pair of paths $(\CR(\eta_1),\CR(\eta_2))$ under $z \mapsto 1/z$ is described (up to reparameterization) by
\[ \int_{\R} \mu_{\alpha \beta} d\nu(\beta)\]
where $\nu$ is a probability measure on $\R$ and $\mu_{\alpha \beta}$ is as defined in Theorem~\ref{thm::unique_law_for_two_whole_plane_paths}.  In order to complete the proof of Theorem~\ref{thm::whole_plane_reversibility} for $\kappa \in (0,4)$, we need to show that $\nu(\{0\}) = 1$.  This in turn is a consequence of the following proposition.

\begin{proposition}
\label{prop::beta_determined}
Suppose that $\kappa \in (0,4)$, $\rho > -2$, $\beta \in \R$, and that $\vartheta$ is a whole-plane $\SLE_\kappa^\beta(\rho)$ process.  For each $k \in \N$, let $\tau_k$ (resp.\ $\sigma_k$) be the first (resp.\ last) time that $\vartheta$ hits $\partial(k\D)$.  For each $j, k \in \N$ with $j < k$, let $N_{j,k}$ be the number of times that $\vartheta|_{[\sigma_j,\tau_k]}$ winds around $0$ (rounded down to the nearest integer).  For each $j \in \N$ we almost surely have that
\[ \beta = 2\pi\big(\chi + \alpha\big) \left(\lim_{k \to \infty} \frac{N_{j,k}}{\log k} \right).\]
In particular, the value of $\beta$ is almost surely determined by $\vartheta$ and is invariant under time-reversal/inversion.
\end{proposition}

The statement of Proposition~\ref{prop::beta_determined} is natural in view of Theorem~\ref{thm::alphabeta} and Proposition~\ref{prop::alphabeta_existence}.  We emphasize that the winding is counted positively (resp.\ negatively) when $\vartheta$ travels around the origin in the counterclockwise (resp.\ clockwise) direction.  The main step in the proof of Proposition~\ref{prop::beta_determined} is the following lemma, which states that the harmonic extension of the winding of a curve upon getting close to (and evaluated at) a given point is well approximated by the winding number at this point.

\begin{lemma}
\label{lem::winding_twisting}
There exists a constant $C > 0$ such that the following is true.  Suppose that $\vartheta$ is a continuous curve in $\ol{\D}$ connecting $\partial \D$ to $0$ with continuous radial Loewner driving function $W$.  Fix $\epsilon \in (0,\tfrac{1}{2})$, let $\tau_\epsilon = \inf\{t \geq 0 : |\vartheta(t)| = \epsilon\}$, and let $N_\epsilon$ be the number of times that $\vartheta|_{[0,\tau_\epsilon]}$ winds around $0$ (rounded down to the nearest integer).  We have that
\[ |2\pi N_\epsilon - \arg(W_{\tau_\epsilon})| \leq C.\]
\end{lemma}

The quantity $\arg(W_{\tau_\epsilon})$ in the statement of Lemma~\ref{lem::winding_twisting} is called the {\bf twisting} of $\vartheta$ upon hitting $\partial (\epsilon \D)$.  An estimate very similar to Lemma~\ref{lem::winding_twisting} was proved in an unpublished work of Schramm and Wilson \cite{SW_TWISTING}.

\begin{proof}[Proof of Lemma~\ref{lem::winding_twisting}]
Let $\wt{\vartheta}$ be the concatenation of $\vartheta|_{[0,\tau_\epsilon]}$ with the curve that travels along the straight line segment starting at $\vartheta(\tau_\epsilon)$ towards $0$ until hitting $\partial(\tfrac{\epsilon}{2} \D)$ and then traces (all of) $\partial(\tfrac{\epsilon}{2} \D)$ in the counterclockwise direction.  Let $\wt{\tau}_\epsilon$ be the time that $\wt{\vartheta}$ finishes tracing $\partial(\tfrac{\epsilon}{2} \D)$ and let $\wt{N}_\epsilon$ be the number of times that $\wt{\vartheta}|_{[0,\wt{\tau}_\epsilon]}$ winds around $0$.  Then
\begin{equation}
\label{eqn::winding_twisting1}
 |N_\epsilon - \wt{N}_\epsilon| \leq 1.
\end{equation}
Let $(\wt{g}_t)$ be the radial Loewner evolution associated with $\wt{\vartheta}$, $\wt{W}$ its radial Loewner driving function, and $\wt{f}_t = \wt{W}_t^{-1} \wt{g}_t$.  Note that
\[  \arg (\wt{f}_t'(0)) = -\arg (\wt{W}_t) \quad\text{for each} \quad t \geq 0.\]
That is, $\arg(\wt{W}_t)$ is equal to the value of the harmonic function $z \mapsto -\arg ( \wt{f}_t'(z) )$ evaluated at $z = 0$.  We claim that there exists a constant $C_1 > 0$ such that
\begin{equation}
\label{eqn::winding_twisting2}
|2\pi \wt{N}_\epsilon - \arg(\wt{W}_{\wt{\tau}_\epsilon})| \leq C_1.
\end{equation}
If $\wt{\vartheta}$ is a piecewise smooth curve then the boundary values of $\arg(\wt{f}_{\wt{\tau}_\epsilon}')$ along $\partial (\tfrac{\epsilon}{2} \D)$ differ from $-2\pi \wt{N}_\epsilon$ by at most a constant $C_0 > 0$.  Thus the claim follows in this case since $\arg(\wt{f}_{\wt{\tau}_\epsilon}')$ is harmonic in $\tfrac{\epsilon}{2} \D$.  The claim for general continuous curves follows by approximation and \cite[Proposition~4.43]{LAW05}.  Observe that 
\begin{equation}
\label{eqn::winding_twisting3}
 \arg(\wt{W}_{\wt{\tau}_{\epsilon}})  = 2\pi + \arg(W_{\tau_\epsilon}).
\end{equation}
Combining~\eqref{eqn::winding_twisting1},~\eqref{eqn::winding_twisting2}, and~\eqref{eqn::winding_twisting3} gives the result.
\end{proof}

\begin{proof}[Proof of Proposition~\ref{prop::beta_determined}]
The second assertion of the proposition is an immediate consequence of the first, so we will focus our attention on the latter.  

We begin by letting $\wt{N}_{j,k}$ be the number of times that $\vartheta|_{[\tau_j,\tau_k]}$ winds around $0$ (rounded down to the nearest integer).  By the scale invariance of whole-plane $\SLE_\kappa^\beta(\rho)$, the law of the number of times that $\vartheta|_{[\sigma_j,\tau_j]}$ winds around $0$ does not depend on $j$.  Moreover, by the transience of whole-plane $\SLE_\kappa^\beta(\rho)$ (Theorem~\ref{thm::transience}) we have that this quantity is finite almost surely.  Consequently, it is not difficult to see that
\[ \lim_{k \to \infty} \frac{N_{j,k} - \wt{N}_{j,k}}{\log k} = 0\]
almost surely.  In particular, it suffices to prove the result with $\wt{N}_{j,k}$ in place of $N_{j,k}$.

Suppose that $\wh{h}_{\alpha \beta} = \wh{h} + \alpha \arg(\cdot) + \beta \log|\cdot|$ where $\wh{h}$ is a GFF on $\D$ such that $\wh{h}_{\alpha \beta}$ has the same boundary values as illustrated in the left side of Figure~\ref{fig::radial_bd} where we take $W_0 = -i$.  Let $\wh{\vartheta}$ be the flow line of $\wh{h}_{\alpha \beta}$ starting from $-i$ and $\psi_\epsilon(z) = \epsilon/z$.  As explained in the proof of Proposition~\ref{prop::alphabeta_existence}, the random curve $\psi_\epsilon(\wh{\vartheta})$ converges to a whole-plane $\SLE_\kappa^\beta(\rho)$ process as $\epsilon \to 0$.  Consequently, it suffices to prove the result with $\wh{\vartheta}$ in place of $\vartheta$ and the hitting times $\wh{\tau}_j = \inf\{t \geq 0 : |\wh{\vartheta}(t)| = \tfrac{1}{j}\}$ in place of $\tau_j$.  Let $\wh{W}$ denote the radial Loewner driving function associated with $\wh{\vartheta}$ and, for each $j \in \N$, let $X_j = \arg(\wh{W}_{\wh{\tau}_j})$.  By Lemma~\ref{lem::winding_twisting}, it suffices to show that
\[ \beta = -\big(\chi+\alpha\big)\left( \lim_{k \to \infty} \frac{X_k - X_j}{\log k} \right) \quad\text{for every}\quad j \in \N\quad\text{almost surely}\]
(the reason for the sign difference from the statement of Proposition~\ref{prop::beta_determined} is that the inversion $z \mapsto z^{-1}$ reverses the direction in which the path winds).  For each $j \in \N$ let $Y_j$ denote the average of $\wh{h}_{\alpha \beta}$ on $\partial (\tfrac{1}{j} \D)$.  The conditional law of $Y_j$ given $\wh{\vartheta}|_{[0,\wh{\tau}_j]}$ is that of a Gaussian random variable with mean $(\chi+\alpha) X_j + O(1)$ and bounded variance (see \cite[Proposition~3.2]{DS08}).  Consequently, it suffices to show that
\[ \beta = -\lim_{k \to \infty} \frac{Y_k -Y_j}{\log k} \quad\text{for every}\quad j \in \N\quad\text{almost surely.}\]
This follows because for each $k > j$, $Y_k - Y_j$ is equal in law to a Gaussian random variable with mean $-\beta \log(k/j) + O(1)$ and variance $O(\log(k/j))$ (see \cite[Proposition~3.2]{DS08}).
\end{proof}

We will now complete the proof of Theorem~\ref{thm::whole_plane_reversibility} for $\kappa' \in (4,8]$.  Suppose that $\eta'$ is a whole-plane $\SLE_{\kappa'}(\rho)$ process for $\kappa' \in (4,8]$ and $\rho > \tfrac{\kappa'}{2}-4$.
Theorem~\ref{thm::whole_plane_duality} implies that the outer boundary of $\eta'$ is described by a pair of whole-plane GFF flow lines, say $\eta^L$ and $\eta^R$ with angle gap $\pi$.  Consequently, it follows from Theorem~\ref{thm::whole_plane_reversibility} applied for $\kappa=16/\kappa' \in [2,4)$ that we can construct a coupling of $\eta'$ with a whole-plane $\SLE_{\kappa'}(\rho)$ process $\wt{\eta}'$ from $\infty$ to $0$ such that the left and right boundaries of $\wt{\eta}'$ are almost surely equal to $\eta^L,\eta^R$.  Theorem~\ref{thm::whole_plane_duality} implies that the conditional law of $\eta'$ given $\eta^L$ and $\eta^R$ in each of the connected components of $\C \setminus (\eta^L \cup \eta^R)$ which lie between $\eta^L$ and $\eta^R$ is independently that of a chordal $\SLE_{\kappa'}(\tfrac{\kappa'}{2}-4;\tfrac{\kappa'}{2}-4)$ process going from the first point on the component boundary drawn by $\eta^L$ and $\eta^R$ to the last.  The same is also true for $\wt{\eta}'$ but with the roles of the first and last points swapped.  Consequently, it follows from the main result of \cite{MS_IMAG3} that we can couple $\eta'$ and $\wt{\eta}'$ together so that $\wt{\eta}'$ is almost surely the time-reversal of $\eta'$.  This completes the proof for $\kappa \in (4,8]$ for $\rho > \tfrac{\kappa'}{2}-4$.  The result for $\rho = \tfrac{\kappa'}{2}-4$ follows by taking a limit $\rho \downarrow \tfrac{\kappa'}{2}-4$, which completes the proof of Theorem~\ref{thm::whole_plane_reversibility}.
\qed

\begin{remark}
\label{rem::reversible_beta}
The same proof applies if we add a multiple of $\beta \log |z|$.  It implies that the whole-plane SLE path drawn from~$0$ to~$\infty$ with non-zero~$\beta$ drift (and possibly non-zero~$\alpha$) has a law that is preserved when we reverse time (up to monotone parameterization) and map the plane to itself via $z \to 1/\bar z$.
\end{remark}

\bibliographystyle{hmralphaabbrv}
\addcontentsline{toc}{section}{References}
\bibliography{sle_kappa_rho}

\bigskip
\filbreak
\begingroup
\small
\parindent=0pt

\bigskip
\vtop{
\hsize=5.3in
Microsoft Research\\
One Microsoft Way\\
Redmond, WA, USA }

\bigskip
\vtop{
\hsize=5.3in
Department of Mathematics\\
Massachusetts Institute of Technology\\
Cambridge, MA, USA } \endgroup \filbreak
\end{document}